\documentclass[10pt]{amsart}

\RequirePackage{etex}
\usepackage[utf8]{inputenc}
\usepackage[T1]{fontenc}
\usepackage[english]{babel}

\usepackage[mathscr]{euscript}
\usepackage{caption}
\usepackage{subcaption}
\usepackage{datetime}
\usepackage{url}
\usepackage{verbatim}
\usepackage{float}
\usepackage{lipsum}

\usepackage{titletoc}
\usepackage{textcomp}
\usepackage{anysize}
\usepackage{fancyhdr}
\usepackage{calrsfs}

\usepackage{amssymb}
\usepackage{amsmath,mathtools}
\usepackage{amsfonts}
\usepackage{amscd}
\usepackage{amsthm}
\usepackage{dsfont}

\usepackage{latexsym}
\usepackage{enumerate}
\usepackage[shortlabels]{enumitem}
\usepackage{array}
\usepackage{stackrel}
\usepackage{stmaryrd}
\usepackage{stackengine}

\usepackage{graphicx}
\graphicspath{ {Images/} }
\usepackage{color,graphics}
\usepackage{tikz}
	\usetikzlibrary{arrows,automata,patterns,decorations.pathmorphing,decorations.pathreplacing, decorations.markings,shapes.misc}
\usepackage{tikz-cd}
	\usetikzlibrary{matrix,arrows,calc,automata,trees}
\usepackage[all,cmtip]{xy}
\usepackage{tkz-graph}
\usepackage{pgfplots,caption}
\pgfplotsset{compat=1.9}

\usepackage{lscape}
\usepackage[active]{srcltx}
\usepackage{rotating}

\usepackage[hypertexnames=false]{hyperref}

\AtBeginDocument{
   \def\MR#1{}
}

\makeatletter
\renewcommand{\tocsection}[3]{
  \indentlabel{\@ifnotempty{#2}{\bfseries\ignorespaces#1 #2\quad}}\bfseries#3}
\renewcommand{\tocsubsection}[3]{
  \indentlabel{\@ifnotempty{#2}{\ignorespaces#1 #2\quad}}#3}
\renewcommand{\tocsubsubsection}[3]{
  \indentlabel{\@ifnotempty{#2}{\ignorespaces#1 #2\quad}}#3}

\newcommand\@dotsep{4.5}
\def\@tocline#1#2#3#4#5#6#7{\relax
  \ifnum #1>\c@tocdepth 
  \else
    \par \addpenalty\@secpenalty\addvspace{#2}%
    \begingroup \hyphenpenalty\@M
    \@ifempty{#4}{%
      \@tempdima\csname r@tocindent\number#1\endcsname\relax
    }{%
      \@tempdima#4\relax
    }%
    \parindent\z@ \leftskip#3\relax \advance\leftskip\@tempdima\relax
    \rightskip\@pnumwidth plus1em \parfillskip-\@pnumwidth
    #5\leavevmode\hskip-\@tempdima{#6}\nobreak
    \leaders\hbox{$\m@th\mkern \@dotsep mu\hbox{.}\mkern \@dotsep mu$}\hfill
    \nobreak
    \hbox to\@pnumwidth{\@tocpagenum{\ifnum#1=1\bfseries\fi#7}}\par
    \nobreak
    \endgroup
  \fi}
\AtBeginDocument{%
\expandafter\renewcommand\csname r@tocindent0\endcsname{0pt}
}
\def\l@subsection{\@tocline{2}{0pt}{2.5pc}{5pc}{}}

\makeatother

\fancypagestyle{my_own_style}{
		\fancyhf{}
		\fancyfoot[C]{Page \thepage}
}

\newcommand{\N}{{\mathbb N}}
\newcommand{\Z}{{\mathbb Z}}

\newcommand{\C}{{\mathbb C}}

\newcommand{\F}{{\mathbb F}}

\DeclareMathAlphabet{\pazocal}{OMS}{zplm}{m}{n}

\newcommand{\calA}{{\pazocal A}}
\newcommand{\calB}{{\pazocal B}}

\newcommand{\calG}{{\pazocal G}}
\newcommand{\calH}{{\pazocal H}}

\newcommand{\calJ}{{\pazocal J}}

\newcommand{\calO}{{\pazocal O}}
\newcommand{\calP}{{\pazocal P}}

\newcommand{\calS}{{\pazocal S}}
\newcommand{\calT}{{\pazocal T}}
\newcommand{\calU}{{\pazocal U}}

\newcommand{\calW}{{\pazocal W}}

\newcommand{\calZ}{{\pazocal Z}}

\newcommand{\OV}{{\calO V}}

\newcommand{\ra}{\rightarrow}

\newcommand{\xra}{\xrightarrow}

\newcommand{\lra}{\longrightarrow}

\newcommand{\ol}{\overline}

\newcommand{\wt}{\widetilde}

\newcommand{\sred}{s_{\mathrm{red}}}

\newcommand{\Lab}{L^{{\rm ab}}}
\newcommand{\LVab}{LV^{{\rm ab}}}

\newcommand{\xddots}{%
  \raise 4pt \hbox {.}
  \mkern 6mu
  \raise 1pt \hbox {.}
  \mkern 6mu
  \raise -2pt \hbox {.}
}

\numberwithin{equation}{section}

\theoremstyle{plain}

\newtheorem{theorem}{Theorem}[section]
\newtheorem*{theorem*}{Theorem}
\newtheorem{lemma}[theorem]{Lemma}
\newtheorem{proposition}[theorem]{Proposition}

\newtheorem*{claim*}{Claim}

\theoremstyle{definition}

\newtheorem{definition}[theorem]{Definition}
\newtheorem{example}[theorem]{Example}
\newtheorem{examples}[theorem]{Examples}

\newtheorem{remark}[theorem]{Remark}
\newtheorem*{remark*}{Remark}
\newtheorem{remarks}[theorem]{Remarks}
\newtheorem*{remarks*}{Remarks}

\newtheorem*{assumption*}{Assumption}

\newtheorem{construction}[theorem]{Construction}

\newtheorem{notation}[theorem]{Notation}

\tikzset{
  on each segment/.style={
    decorate,
    decoration={
      show path construction,
      moveto code={},
      lineto code={
        \path [#1]
        (\tikzinputsegmentfirst) -- (\tikzinputsegmentlast);
      },
      curveto code={
        \path [#1] (\tikzinputsegmentfirst)
        .. controls
        (\tikzinputsegmentsupporta) and (\tikzinputsegmentsupportb)
        ..
        (\tikzinputsegmentlast);
      },
      closepath code={
        \path [#1]
        (\tikzinputsegmentfirst) -- (\tikzinputsegmentlast);
      },
    },
  },
  mid arrow/.style={postaction={decorate,decoration={
        markings,
        mark=at position .5 with {\arrow[#1]{stealth}}
      }}},
}

\tikzset{
	cross/.style={cross out, draw=black, minimum size=2*(#1-\pgflinewidth), inner sep=0pt, outer sep=0pt},
	cross/.default={1pt}
}

\tikzset{
  symbol/.style={
    draw=none,
    every to/.append style={
      edge node={node [sloped, allow upside down, auto=false]{$#1$}}}
  }
}

\title[A correspondence between surjective local homeomorphisms and a family of separated graphs]{A correspondence between surjective local homeomorphisms and a family of separated graphs}
\author{Pere Ara}
\address[P. Ara]{Departament de Matem\`atiques, Edifici Cc, Universitat Aut\`onoma de Barcelona, 08193 Cerdanyola del Vall\`es (Barcelona), Spain, and}
\address{Centre de Recerca Matem\`atica, Edifici Cc, Campus de Bellaterra, 08193 Cerdanyola del Vall\`es (Barcelona), Spain.}
\email{Pere.Ara@uab.cat}
\author{Joan Claramunt}
\address[J. Claramunt]{Departamento de Matem\'aticas, Universidad Carlos III de Madrid, Av. de la Universidad 30, 28911 Legan\'es (Madrid), Spain.}
\email{jclaramu@math.uc3m.es}


\thanks{Both authors were partially supported by the project PID2020-113047GB-I00/AEI/10.13039/501100011033. The first-named author was also partially supported by the Spanish State Research Agency, through the Severo Ochoa and Mar\'ia de Maeztu Program for Centers and Units of Excellence in R$\&$D (CEX2020-001084-M). The second named author was also partially supported by the ERC Starting Grant ``Limits of Structures in Algebra and Combinatorics'' No. 805495.}

\date{\today}

\begin{document}

\pagestyle{plain}
 
\begin{abstract}
We present a graph-theoretic model for dynamical systems $(X,\sigma)$ given by a surjective local homeomorphism $\sigma$ on a totally disconnected compact metrizable space $X$. In order to make the dynamics appear explicitly in the graph, we use two-colored Bratteli separated graphs as the graphs used to encode the information. In fact, our construction gives a \textit{bijective} correspondence between such dynamical systems and a subclass of separated graphs which we call \textit{l-diagrams}. This construction generalizes the well-known shifts of finite type, and leads naturally to the definition of a \textit{generalized finite shift}. It turns out that any dynamical system $(X,\sigma)$ of our interest is the inverse limit of a sequence of generalized finite shifts. We also present a detailed study of the corresponding Steinberg and $C^*$ algebras associated with the dynamical system $(X,\sigma)$, and we use the above approximation of $(X,\sigma)$ to write these algebras as colimits of the associated algebras of the corresponding generalized finite shifts, which we call \textit{generalized finite shift algebras}.
\end{abstract}

\maketitle

\tableofcontents

\normalsize

\section{Introduction}\label{section-introduction}

Dynamical systems consisting of a pair $(X,f)$, where $X$ is a topological space and $f\colon X\to X$  is a continuous map, have attracted a great deal of attention along the years. In particular, an important line of investigation consists of approximating a dynamical system by means of other dynamical systems defined by combinatorial data. In many occasions, invariants associated to the original dynamical system can be computed using these combinatorial approximations. In this respect, symbolic dynamics has shown to provide a very useful source of approximating systems, see \cite{Bruin2022, DK2018, LM2021}.
An early example occurs with the existence of a Markov partition on several classes of dynamical systems, which is proved to imply that the system under study is a factor of a corresponding shift associated to the Markov partition, see \cite[Proposition 6.5.8]{LM2021}.

It has recently been shown in \cite{GM2020} that any dynamical system $(X,f)$, where $X$ is a compact Hausdorff space, can be expressed as an inverse limit of one-sided shifts of finite type, and that the important property of shadowing holds for $(X,f)$ if and only if the inverse system satisfies the Mittag-Leffler property. See also \cite{shimomura2014} and \cite{shimomura2020} for results concerning the approximation of zero-dimensional systems in terms of various combinatorial structures.

In a different setting, Bratteli diagrams were used in \cite{Bratteli1972} to classify approximately finite-dimensional $C^*$-algebras. Herman, Putnam and Skau \cite{HPS1992} showed that there is a bijective correspondence (modulo suitable equivalence relations) between essentially minimal homeomorphisms $f$ on the Cantor set $X$ and a class of {\it ordered} Bratteli diagrams.  These results were later used to obtain classification theorems for the corresponding crossed product $C^*$-algebras $C^*(X)\rtimes \Z$ associated to the action of $\Z$ on $X$ induced by the homeomorphism $f$, see for instance \cite{GPS1995}. This kind of combinatorial structure has been studied in more generality, for some dynamical systems not necessarily essentially minimal, see for example \cite{BNS2021}. In addition, the above correspondences have been recently upgraded to suitable category equivalences by Amini, Elliott and Golestani, see \cite{AEG2015} and \cite{AEG2021}.

In this paper we develop a new approximation method for any dynamical system $(X,\sigma)$, where $\sigma$ is a surjective local homeomorphism and $X$ is a totally disconnected compact metrizable space. This is a very large class of dynamical systems, including all homeomorphisms on totally disconnected compact metrizable spaces, in particular all minimal Cantor systems \cite{Putnam2018}, all one-sided shifts of finite type, all two-sided shifts, covering maps of compact totally disconnected metrizable spaces (see \cite{D95}, \cite{EV06} and Remark \ref{rem:covering-map}), the one-sided shifts associated to topological graphs $E= (E^0,E^1,r,s)$, where $E^0$ and $E^1$ are compact totally disconnected metrizable spaces and $s$ is surjective (see \cite{Katsura2}, \cite{Katsura4}, \cite{KatsuraCKTOPLGR}, \cite{KL17}), one-sided shifts on infinite path spaces of finite ultragraphs with no sinks \cite{GR2019}, and the covers of one-sided subshifts \cite{BC2020}.

We show that our approximation is well-behaved with respect to the associated $C^*$-algebras. Specifically, we prove that the groupoid $C^*$-algebra $C^*(\calG (X,\sigma))$ of the Deaconu-Renault groupoid $\calG (X,\sigma)$ associated to $(X,\sigma)$ is naturally isomorphic to the colimit of the groupoid $C^*$-algebras of the approximating systems, see Theorem \ref{intro-theorem-combined.FINAL.RESULT} below. This groupoid $C^*$-algebra generalizes the crossed product construction mentioned above. A new basic building block emerges from our approach, generalizing both one-sided and two-sided shifts of finite type, which we have termed a {\it generalized shift of finite type}. The main combinatorial tool employed is the theory of separated graphs. A \textit{separated graph} $(E,C)$ consists of a graph $E= (E^0,E^1,r,s)$ together with a prescribed partition $C_v$ of the set of edges arriving at $v$, for each vertex $v \in E^0$. In our work, we use \textit{separated Bratteli diagrams}, which are Bratteli diagrams which are also separated graphs, so that the edges on every layer of the diagram are given a specific separation.

It is worth to mention that for a large class of dynamical systems, including the ones considered in the present paper, it has been shown in \cite{ABCE2023} that one can characterize conjugacy of the dynamical systems in terms of the associated $C^*$-algebras and also in terms of the associated Deaconu-Renault groupoids. See also \cite{BCS2020}, \cite{BCS2023}, \cite{CRST2021} and \cite{GTR2023} for other studies on the interconnections between local homeomorphisms, groupoids and $C^*$-algebras.

Let \textsf{LHomeo} be the family of dynamical systems $(X,\sigma)$ consisting of a totally disconnected compact metrizable space $X$, together with a surjective local homeomorphism $\sigma : X \ra X$.
In the spirit of the work done in \cite{HPS1992}, we provide a bijective correspondence (modulo suitable  equivalence relations) between \textsf{LHomeo} and a class of separated Bratteli diagrams, which we call \textit{$l$-diagrams}, see Definition \ref{definition-L.separated.Bratteli}.

Let $E$ be a finite graph having neither sinks nor sources. It is well-known that out of it one can construct a dynamical system $(X_E,\sigma_E) \in \textsf{LHomeo}$ as follows. The space $X_E$ is taken to be the set of one-sided infinite paths in the graph, endowed with the topology generated by the cylinder sets, and then $\sigma_E$ is defined as the one-sided shift. These are the well-known one-sided shifts of finite type. If we consider the space of bi-infinite paths instead of $X_E$, we obtain the definition of a two-sided shift of finite type, which is a homeomorphism. These are basic objects in symbolic dynamics.

From an arbitrary directed graph $E$ one can construct two important $\C$-algebras in a universal way. The first one is a purely algebraic construction, called the \textit{Leavitt path algebra} of the directed graph $E$, denoted by $L_{\C}(E)$. Leavitt path algebras were independently  introduced by Abrams and Aranda Pino \cite{AA2005} and by the first-named author, Moreno and Pardo \cite{AMP2007}. The second construction, which  chronologically precedes the first, also contains analytic information and is called the \textit{graph $C^*$-algebra} of the directed graph $E$, denoted by $C^*(E)$. Graph $C^*$-algebras were first studied in detail in the papers \cite{KPRR1997} and \cite{KPR1998}. It turns out that $C^*(E)$ is the enveloping $C^*$-algebra of $L_{\C}(E)$ for each directed graph $E$, in the sense of Definition \ref{definition-Steinberg.algebra}.

The connection with the theory of topological groupoids was already explored in \cite{KPRR1997}, where it is shown that the graph $C^*$-algebra of a locally finite graph is isomorphic to the $C^*$-algebra of a suitably defined \textit{graph groupoid} $\calG_E$ associated to $E$. This was later generalized to arbitrary graphs. When the graph is finite and has neither sinks nor sources, the graph groupoid $\calG_E$ turns out to be the Deaconu-Renault groupoid of the one-sided shift $(X_E,\sigma_E)$, and thus we obtain  the isomorphism $C^*(E)\cong C^*(\calG(X_E,\sigma_E))$. This connection was extended to the purely algebraic situation in \cite{Steinberg2010} and \cite{CFST2014} through the consideration of the \textit{Steinberg algebra} $A_{\C}(\calG)$ of an ample \'etale groupoid $\calG$. Concretely, it is shown in \cite[Proposition 4.3]{CFST2014} that there is a natural $*$-algebra isomorphism $L_{\C}(E) \cong A_{\C}(\calG_E)$ for each directed graph $E$. There are multiple  interesting connections between these groupoid algebras and symbolic dynamics, see for instance the recent papers \cite{ABCE2023} and \cite{CDE2020}.

The first-named author and Goodearl introduced in \cite{AG2011} and \cite{AG2012} the Leavitt path algebra $L_{\C}(E,C)$ and the graph $C^*$-algebras $C^*(E,C)$ associated to a separated graph $(E,C)$. These are defined  in a similar way to the Leavitt path algebras and graph $C^*$-algebras of ordinary directed graphs, but taking into account the partitions $\{C_v \mid v \in E^0\}$ in the defining relations of such algebras (see Definitions \ref{definition-leavitt.alg} and \ref{definition-C*.alg} for the precise definitions). Thenceforth there has been a considerable amount of work towards the theory of separated graph algebras \cite{Ara2012,Ara2022,ABC2023,AE2015,AL2018,GR2015,Lolk2017,Lolk2017.2}, not only because they provide a whole rich class of examples of algebras not being realizable as graph algebras, but because they proved to be useful in solving, either in the positive or in the negative, important conjectures (see e.g. \cite{ABP2020,ABP2020.2,ABPS2021,AE2014}).

The relation between dynamical systems and separated graphs was first demonstrated in a paper by the first-named author and Exel \cite{AE2014}. The authors attach to each finite bipartite separated graph $(E,C)$ a partial dynamical system $(\Omega(E,C),\F,\theta)$, where $\F = \F(E^1)$ is the free group on the set of edges, in such a way that the algebraic crossed product $C_{\C}(\Omega(E,C)) \rtimes_{\theta^*}^{\text{alg}} \F$ is canonically isomorphic to a well-determined quotient $L_{\C}(E,C)/J$ of $L_{\C}(E,C)$. They also found similar results in the $C^*$-algebraic setting. As an application of these results, the authors were able to construct non-paradoxical topological actions of discrete groups on Cantor spaces which do not admit an invariant finitely additive measure, thus providing counterexamples to the topological version of Tarski's alternative, see \cite[Section 7]{AE2014} for details. In connection with this work, it was shown in \cite{AL2018} that the study of the above mentioned dynamical systems associated to separated graphs is equivalent to the study of \textit{convex subshifts}, a notion coined by the first-named author and Lolk in \cite{AL2018}, introduced as a far-reaching unification proposal for the well-known one-sided and two-sided subshifts.

The first main result of the paper is the following:

\begin{theorem}[Theorem \ref{theorem-correspondence}]\label{intro-theorem-correspondence}
Out of a dynamical system $(X,\sigma) \in \emph{\textsf{LHomeo}}$, one can construct an $l$-diagram $(F,D)$. Conversely, out of an $l$-diagram $(F,D)$, one can construct a dynamical system $(X,\sigma) \in \emph{\textsf{LHomeo}}$.

Moreover, these correspondences are inverse of each other, so we obtain, modulo appropriate equivalence relations, a bijective correspondence between \emph{\textsf{LHomeo}} and \emph{\textsf{LDiag}}, the class of $l$-diagrams.
\end{theorem}

In particular, if we restrict our attention to homeomorphisms, the corresponding diagrams have a simpler form. We call such separated Bratteli diagrams \textit{$h$-diagrams} (see Definition \ref{definition-H.separated.Bratteli}) and we denote by \textsf{Homeo} the subclass of \textsf{LHomeo} consisting of those dynamical systems $(X,\sigma)$ such that $\sigma$ is a homeomorphism.

\begin{theorem}[Theorem \ref{theorem-equivalence.homeo}]\label{intro-theorem-equivalence.homeo}
The restriction of the above correspondences gives a bijective correspondence between the subclass of homeomorphisms $\emph{\textsf{Homeo}}$ and the subclass \emph{\textsf{HDiag}} of $h$-diagrams.
\end{theorem}

It follows from Theorem \ref{intro-theorem-equivalence.homeo} and \cite[Theorem 4.7]{HPS1992} that there is a bijective map between equivalence classes of essentially simple ordered Bratteli diagrams and a subset of equivalence classes of $h$-diagrams. It would be interesting to determine explicitly the range of this bijective map. Moreover, one can try to build a {\it category} of $h$-diagrams, in the spirit of \cite{AEG2021}, and then extend the above bijection to a suitable equivalence of categories. Another line of future research is extending the techniques developed in the present paper to cover the case of partial (local) homeomorphisms, and relate this extension with the work done by Giordano, Gonçalves and Starling in \cite{GGS2017}.

\medskip

Let $(E,C)$ be a finite bipartite separated graph satisfying the conditions needed to be the first layer of an $l$-diagram. These conditions are explicitly stated in Definition \ref{definition-general.finite.shift}, where  these separated graphs are named {\it generalized finite shift graphs}. By applying a process called the \textit{canonical resolution} of the bipartite separated graph we obtain an $l$-diagram $(F,D)$ (Proposition \ref{proposition-multiresolution.is.Ldiagram}) which, by an application of Theorem \ref{intro-theorem-correspondence}, gives rise to a dynamical system $(X,\sigma) \in \textsf{LHomeo}$. Those dynamical systems arising by applying this construction are named \textit{generalized finite shifts}, since its class includes both one-sided and two-sided finite shifts. Our next main result states that any dynamical system $(X,\sigma) \in \textsf{LHomeo}$ can be approximated, in a sense made precise below, by generalized finite shifts. This is done by working with an explicit realization of the generalized finite shifts, which is based on the work done in \cite{AL2018}.

\begin{theorem}[Theorem \ref{theorem-inverse.limit}]\label{intro-theorem-inverse.limit}
Given $(X,\sigma) \in \emph{\textsf{LHomeo}}$, there exists an inverse system $\{(X_n,\sigma_n),\psi_{n,m}\}$ in \emph{\textsf{LHomeo}} consisting of generalized finite shift graphs such that
$$(X,\sigma) \cong \varprojlim (X_n,\sigma_n).$$
\end{theorem}

Let $(X,\sigma) \in \textsf{LHomeo}$ and $(F,D)$ be the corresponding $l$-diagram. In the last section of the paper we apply the results obtained in order to study the relationship between the Steinberg algebra and the $C^*$-algebra of the Deaconu-Renault groupoid associated with $(X,\sigma)$ and a certain corner of an inductive limit of Leavitt path algebras and graph $C^*$-algebras corresponding to the finite levels of $(F,D)$. This study culminates in Theorems \ref{theorem-main.iso.theorem} and \ref{theorem-isomorphism.for.Cstars}, where it is shown that the aforementioned algebras are isomorphic. 

We close the paper by further using Theorem \ref{intro-theorem-inverse.limit} and various other tools to obtain a similar approximation result at the level of the Steinberg algebras and the groupoid $C^*$-algebras. Here $A_K(\calG(X,\sigma))$ denotes the Steinberg algebra of $\calG(X,\sigma)$, the Deaconu-Renault groupoid associated with $(X,\sigma)$, with coefficients in an arbitrary field with involution $(K,*)$.

\begin{theorem}[Theorem \ref{theorem-combined.FINAL.RESULT}]\label{intro-theorem-combined.FINAL.RESULT}
Let $(X,\sigma) \in \emph{\textsf{LHomeo}}$ and write it as an inverse limit of a sequence of generalized finite shifts $\{(X_n,\sigma_n)\}_n$ as in Theorem \ref{intro-theorem-inverse.limit}. Then
$$A_K(\calG(X,\sigma)) \cong \varinjlim A_K(\calG(X_n,\sigma_n)).$$
At the level of $C^*$-algebras, we also get
$$C^*(\calG(X,\sigma)) \cong \varinjlim C^*(\calG(X_n,\sigma_n)).$$
\end{theorem}

The results of this paper, in conjunction with the results in \cite{AE2014} and \cite{AE2015}, will be used in forthcoming work by the authors, to study both the stable and the non-stable $K$-theory of the algebras attached to the dynamical systems $(X, \sigma)\in \textsf{LHomeo}$, as well as the type semigroup of $(X,\sigma)$. 

The paper is organized as follows. In Section \ref{section-preliminaries} we provide the reader some basic concepts such as the definition of a separated graph, and its associated Leavitt path algebra and $C^*$-algebra. We also recall the construction of the Deaconu-Renault groupoid associated with a dynamical system $(X,\sigma)$ formed by a surjective local homeomorphism $\sigma$ on a totally disconnected compact metrizable space $X$, and the construction of its groupoid $C^*$-algebra, and its Steinberg algebra over an arbitrary field with involution. We present in Theorem \ref{theorem-correspondence} the bijective correspondence between surjective local homeomorphisms on totally disconnected compact metrizable spaces and a certain subfamily of the family of separated graphs, called $l$-diagrams, in Section \ref{section-correspondence.local.homeo}. Section \ref{section-homeos} is then devoted to characterize the class of $l$-diagrams which arises from the class of dynamical systems given by homeomorphisms on totally disconnected compact metrizable spaces, called $h$-diagrams. In Section \ref{section-general.shifts.finite.type} we show how shifts of finite type fit in our theory, and we present a generalization of them, which we call \textit{generalized finite shifts}. We provide yet another different point of view of such dynamical systems using the work from \cite{AL2018}, culminating in Theorem \ref{theorem-inverse.limit} which shows that any dynamical system of our kind is an inverse limit of a sequence of generalized finite shits. We use this `approximation theorem' in order to give similar characterizations, in terms of colimits, of the Steinberg algebra (respectively, $C^*$-algebra) of the Deaconu-Renault groupoid associated with the dynamical system in terms of the sequence of Steinberg algebras (respectively, $C^*$-algebras) of the Deaconu-Renault groupoids associated with the generalized finite shifts (Theorem \ref{theorem-combined.FINAL.RESULT}). This is done in Section \ref{section-approx.algebras}.

\section{Preliminaries}\label{section-preliminaries}

In this section we provide background definitions and concepts needed throughout the paper.

\subsection{Separated graphs}\label{subsection-separated.graphs}

We follow \cite{AE2014}. A \textit{separated graph} is a pair $(E,C)$ where $E = (E^0,E^1,r,s)$ is a directed graph ($E^0$ is the set of vertices, $E^1$ the set of edges, and $r,s\colon E^1\to E^0$ are the range and source maps, respectively), and $C = \bigsqcup_{v \in E^0} C_v$, where each $C_v$ is a partition of the set of edges $r^{-1}(v)$ into pairwise disjoint, non-empty subsets, for each vertex $v \in E^0$. We set $C_v= \emptyset$ in case $r^{-1}(v) = \emptyset$. We say that $(E,C)$ is \textit{finitely separated} if all the sets in $C$ are finite, that is, if $|X| <\infty$ for all $X\in C$.

A \textit{bipartite separated graph} is a separated graph $(E,C)$ where we can write $E^0 = E^{0,0} \sqcup E^{0,1}$ in such a way that $s(E^1) = E^{0,1}$ and  $r(E^1) = E^{0,0}$. 

We will follow the conventions from \cite{AE2014} and \cite{AL2018} concerning sources and ranges of edges of a graph, and the relations (E), (SCK1) and (SCK2) below. In particular, a non-trivial path in the graph $E$ will be a sequence $e_1e_2\cdots e_n$ of edges in $E$ such that $s(e_i)= r(e_{i+1})$ for $i=1,\dots , n-1$. (A trivial path consists of just one vertex $v$ in $E$.) We warn the reader that different conventions are used in various other sources.   

We can now define our algebras of interest, namely the Leavitt path algebra and the $C^*$-algebra of a separated graph $(E,C)$.

\begin{definition}\label{definition-leavitt.alg}
The \textit{Leavitt path algebra} of the separated graph $(E,C)$ with coefficients in a field $K$ with involution $*$, denoted by $L_K(E,C)$, is the $*$-algebra with generators $E^0 \sqcup E^1 \sqcup (E^1)^*$ together with the relations
\begin{enumerate}
\item[(V)] $vw = \delta_{v,w} v$, and $v = v^*$ for $v,w \in E^0$,
\item[(E)] $r(e)e = e s(e) = e$ for $e \in E^1$,
\item[(SCK1)] $e^* f = \delta_{e,f} s(e)$ for $e,f \in Y$, for all $Y\in C$,
\item[(SCK2)] $v = \sum_{e \in Y} ee^*$ for every $Y \in C_v$ finite, $v \in E^0$.
\end{enumerate}
\end{definition}

\begin{definition}\label{definition-C*.alg}
The \textit{$C^*$-algebra} of the separated graph $(E,C)$, denoted by $C^*(E,C)$, is the universal $C^*$-algebra with generators $E^0 \sqcup E^1 \sqcup (E^1)^*$ together with the relations (V),(E),(SCK1) and (SCK2) from Definition \ref{definition-leavitt.alg}. In other words, $C^*(E,C)$ is the enveloping $C^*$-algebra of $L_{\C}(E,C)$.
\end{definition}

\subsection{Groupoids and associated algebras}\label{subsection-groupoids.algebras}

Given a groupoid $\calG$, we will write $\calG^{(2)}$ for the set of pairs of composable elements, and by $\calG^{(0)}$ for its set of units. The range and source maps will be denoted by $r : \calG \ra \calG^{(0)}$ and $s : \calG \ra \calG^{(0)}$, respectively.

In this paper, a \textit{topological groupoid} will be a groupoid $\calG$ possessing a locally compact Hausdorff topology, in which the range and source maps $r,s$, the inverse map and the multiplication map are all continuous maps. We refer to \cite{Renault1980, SSW2020, Steinberg2010, CFST2014} for general background on topological groupoids and their associated $C^*$-algebras.

\begin{definition}\label{definition-local.homeo}
Let $X$ and $Y$ be topological spaces. A \textit{local homeomorphism} from $X$ to $Y$ is a continuous map $h\colon X\to Y$ such that every $x\in X$ has an open neighborhood $U$ such that $h(U)$ is open in $Y$ and $h|_U$ is a homeomorphism from $U$ onto its image $h(U)$. Whenever $X=Y$ we say that $h$ is a local homeomorphism on $X$. 
\end{definition}

A topological groupoid $\calG$ is \textit{étale} if the range map $r\colon \calG \to  \calG^{(0)}$ is a local homeomorphism. Then the source map is also a local homeomorphism and $\calG^{(0)}$ is an open subset of $\calG$ (see \cite[Lemma 8.4.2]{SSW2020}).

An open subset $\calW \subseteq \calG$ is called an \textit{open bisection} if both $r(\calW),s(\calW)$ are open subsets of $\calG^{(0)}$, and $r|_{\calW}, s|_{\calW}$ are homeomorphisms onto their respective images. Note that a topological groupoid is étale if and only if its topology has a basis of open bisections. 

A topological groupoid $\calG$ is called \textit{ample} if its topology admits a basis consisting of compact open bisections. Clearly all ample groupoids are étale.

From now on $\calG$ will denote a second countable, Hausdorff and ample groupoid, thus admitting a countable basis for the topology consisting of compact open bisections.

\begin{definition}\label{definition-Steinberg.algebra}
Let $(K,*)$ be a field endowed with an involution $*$. The \textit{Steinberg algebra} $A_K(\calG)$ of $\calG$ is the $*$-algebra of all the locally constant functions $f \colon \calG \to K$ with compact support. The multiplication is just the convolution product
$$(f \ast g)(\gamma) = \sum_{s(\alpha)=s(\gamma)} f(\gamma \alpha^{-1}) g(\alpha)$$
for $f,g \in A_K(\calG), \gamma \in \calG$, and the involution is given by the rule
$$f^*(\gamma) = f(\gamma^{-1})^*$$
for $f \in A_K(\calG), \gamma \in \calG$.

The \textit{full groupoid $C^*$-algebra} of $\calG$, denoted by $C^*(\calG)$, is the enveloping $C^*$-algebra of $A_{\C}(\calG)$, in the sense that there is an injective $*$-homomorphism $\iota \colon A_{\C}(\calG)\to C^*(\calG)$ such that any $*$-homomorphism from $A_{\C}(\calG)$ to a $C^*$-algebra $B$ uniquely extends to a $*$-homomorphism from $C^*(\calG)$ to $B$. For second countable, ample Hausdorff groupoids, this definition agrees with the usual definition of the full groupoid $C^*$-algebra given by Renault in \cite{Renault1980}, by \cite[Lemma 9.2.4]{SSW2020} and \cite[Theorem 7.1(1)]{CZ2022}.
\end{definition}

Every element of $A_K(\calG)$ can be written as a finite linear combination of characteristic functions $\mathds{1}_{U_i}$, where $\{U_i\}$ are mutually disjoint compact open bisections. If $U$ and $V$ are two compact open bisections in $\calG$, then we have
$$\mathds{1}_U \cdot \mathds{1}_V = \mathds{1}_{UV}$$
in $A_K(\calG)$, where
$$UV := \{uv \in \calG \mid u \in U,v \in V, s(u) = r(v)\}$$
is again a compact open bisection.

\subsection{Dynamical systems}\label{subsection-dynamical.systems}

We will concentrate on dynamical systems $(X,\sigma)$ formed by a totally disconnected compact metrizable space $X$ together with a surjective local homeomorphism $\sigma \colon X \to X$ (see Definition \ref{definition-local.homeo}).

\begin{definition}\label{definition-partition}
By a \textit{partition} of $X$ we mean a finite family $\calP$ consisting on non-empty, pairwise disjoint clopen subsets of $X$, such that $X = \bigsqcup_{Z \in \calP} Z$.
\end{definition}

Given two partitions $\calP_1,\calP_2$ of $X$, we say that $\calP_2$ is \textit{finer} than $\calP_1$, written as $\calP_1 \precsim \calP_2$, if every element $Z \in \calP_2$ is contained in a (necessarily unique) element $Z' \in \calP_1$, that is $Z \subseteq Z'$. Given two partitions $\calP_1,\calP_2$ of $X$, there is a canonical way of constructing a new partition which refines both $\calP_1,\calP_2$; this new partition is given by the \textit{wedge} of $\calP_1$ and $\calP_2$, and is defined by
$$\calP_1 \vee \calP_2 := \{Z_1 \cap Z_2 \mid Z_1 \in \calP_1,Z_2 \in \calP_2 \text{ with } Z_1 \cap Z_2 \neq \emptyset\}.$$
The \textit{diameter} of an element $Z \in \calP$, being $\calP$ a partition of $X$, is defined by
$$\text{diam}(Z) := \sup\{ d(x,y) \mid x,y \in Z\},$$
where $d$ is the distance witnessing the metrizability of $X$. The \textit{diameter} of the partition $\calP$ is thus
$$\text{diam}(\calP) := \max\{\text{diam}(Z) \mid Z \in \calP\}.$$

We now recall the definition of the {\it Deaconu-Renault groupoid} associated with the dynamical system $(X,\sigma)$. Define
$$\calG(X,\sigma) := \bigcup_{m,n \in \N_0}\{(x,m-n,y) \mid x,y \in X, \sigma^m(x) = \sigma^n(y)\} \subseteq X \times \Z \times X.$$
This set is given the structure of a groupoid via the following. The set of composable elements will be
$$\calG(X,\sigma)^{(2)} = \{((x,p,y),(z,q,t)) \mid y = z\} \subseteq \calG(X,\sigma)^2.$$
The multiplication is then given by $(x,p,y) \cdot (y,q,z) = (x,p+q,z)$, and the inverse of an element is simply $(x,p,y)^{-1} = (y,-p,x)$. Thus the unit space of this groupoid is
$$\calG(X,\sigma)^{(0)} = \{(x,0,x) \mid x \in X\} \simeq X.$$
The range and source maps $r,s : \calG(X,\sigma) \ra X$ are given, respectively, by $r(x,p,y) = x$ and $s(x,p,y) = y$.

This groupoid becomes a topological Hausdorff groupoid in the topology generated by the basis of cylinder sets $\calB = \{Z(U,m,n,V) \mid U,V \subseteq X \text{ open, and }m,n \in \N_0\}$, where each cylinder $Z(U,m,n,V)$ is defined by
$$Z(U,m,n,V) := \{(x,m-n,y) \mid x \in U, y \in V, \sigma^m(x) = \sigma^n(y)\}.$$
In fact, $\calG(X,\sigma)$ becomes an étale groupoid. More specifically, by fixing $m,n \in \N_0$ and letting $U,V \subseteq X$ be open subsets such that both $\sigma^m|_U, \sigma^n|_V$ are homeomorphisms onto their images, then $Z = Z(U,m,n,V)$ is such that $r|_Z,s|_Z$ are homeomorphisms onto open subsets of $\calG(X,\sigma)$. In particular, they are also local homeomorphisms onto open subsets of $\calG(X,\sigma)^{(0)} = X$, and so $Z$ is in fact an open bisection.

Moreover, since we can take $U,V$ to be clopen subsets of $X$ (as $X$ is totally disconnected), we can take $Z(U,m,n,V)$ to be also compact. As a consequence, $\calG(X,\sigma)$ is an ample, Hausdorff groupoid, with a basis of open compact bisections consisting of the cylinder sets $Z(U,m,n,V)$ such that $m,n \in \N_0$ and $U,V \subseteq X$ are clopen with the property that $\sigma^m|_U, \sigma^n|_V$ are homeomorphisms onto their images.

\section{The correspondence between local homeomorphisms and \texorpdfstring{$l$}{}-diagrams}\label{section-correspondence.local.homeo}

In this section we establish a bijective correspondence between a specific family of separated graphs and surjective local homeomorphisms of a second countable totally disconnected compact space $X$.

We first describe the type of separated graphs we consider. It is a subclass of the class of all separated Bratteli diagrams considered in \cite{AL2018}. We recall the definition of these graphs.

\begin{definition}\label{definition-separated.Bratteli}
A \textit{separated (or colored) Bratteli diagram} is an infinite separated graph $(F,D)$ with the following properties:
\begin{enumerate}[(a),leftmargin=0.7cm]
\item The vertex set $F^0$ is the union of finite, non-empty, pairwise disjoint sets $F^{0,j}$, $j \ge 0$.
\item The edge set $F^1$ is the union of finite, non-empty, pairwise disjoint sets $F^{1,j}$, $j\ge 0$.
\item The range and source maps satisfy $r(F^{1,j})=F^{0,j}$ and $s(F^{1,j})=F^{0,j+1}$ for all $j\ge 0$, respectively.
\end{enumerate}
\end{definition}
Note that a Bratteli diagram is just a separated Bratteli diagram with the trivial separation.

Now we can define our class of graphs.

\begin{definition}\label{definition-L.separated.Bratteli}
Let $(F,D)$ be a separated Bratteli diagram. We say that $(F,D)$ is an \textit{$l$-diagram} if the following hold:
\begin{enumerate}[(a),leftmargin=0.7cm]
\item (Separations for the even layers) For each $j \geq 0$ and $v \in F^{0,2j}$, the separation $D_v$ is given by
$$D_v= \{ B_v,R_v \}$$
for some (non-empty) subsets $B_v,R_v \subseteq r^{-1}(v)$. The elements of $B_v$ will be called \textit{blue edges}, and the elements of $R_v$ \textit{red edges}.
\item (Separations for the odd layers) For each $j \ge 1$ and $v\in F^{0,2j-1}$, the separation $D_v$ is given by
$$D_v = \{B_v\}\cup \{ R(f) \mid f \in R_{r(f)}, s(f) = v\}$$
for some (non-empty) subsets $B_v,R(f) \subseteq r^{-1}(v)$. The elements of $B_v$ will be called \textit{blue edges}. We write $R_v= \bigcup _{f\in s^{-1}(v)\cap R_{r(f)}} R(f)$, and call the elements of $R_v$ generically \textit{red edges}.
\item (Vertices at even layers) For $j \geq 1$, we have
$$F^{0,2j} = \bigsqcup _{v\in F^{0,2j-1}} s(B_v) = \bigsqcup_{v\in F^{0,2j-1}} s(R_v).$$
Moreover, for any $v \in F^{0,2j-1}$, we have $s(e) \ne s(e')$ for all distinct edges $e,e' \in B_v$, and similarly $s(f) \ne s(f')$ for all distinct edges $f,f' \in R_v$.
\item (Vertices at the first odd layer) We have
$$(1) \text{ } F^{0,1} = \bigsqcup_{v\in F^{0,0}} s(B_v), \quad \text{ and also } \quad (2) \text{ } F^{0,1} = \bigcup_{v\in F^{0,0}} s(R_v).$$
Moreover, $s(e) \ne s(e')$ for all distinct edges $e,e'$ with $e\in B_{v_1}$, $e'\in B_{v_2}$, for any $v_1,v_2 \in F^{0,2j}$ with $j \geq 0$.
\item (Compatibility at even layers) For each $j \ge 0$ and $v \in F^{0,2j}$, we have
$$\bigsqcup _{e\in B_v} s(B_{s(e)}) = \bigsqcup_{f\in R_v} s(R(f)).$$
\item  (Rombs at odd layers, red edges) For each $j\ge 1$, each $v \in F^{0,2j-1}$ and each pair $(f_0,f_1)$ of red edges such that $r(f_0) = v$ and $s(f_0) = r(f_1)$, there exists a (unique) pair $(e_0,e_1)$ of blue edges such that $r(e_0) = r(f_0) = v$, $s(e_0) = r(e_1)$ and $s(e_1) = s(f_1)$.
\item (Rombs at odd layers, blue edges) For each $j\ge 1$, each $v \in F^{0,2j-1}$, each red edge $g \in R_{r(g)}$ such that $s(g)=v$, and each pair $(e_0,e_1)$ of blue edges such that $r(e_0) = v$ and $s(e_0) = r(e_1)$, there exists a unique pair $(f_0,f_1)$ of red edges such that $f_0\in R(g)$, $r(f_0) = r(e_0) = v$, $s(f_0) = r(f_1)$ and $s(f_1) = s(e_1)$. 
\end{enumerate}
The family of $l$-diagrams will be denoted by \textsf{LDiag}.
\end{definition}

From the very definition of an $l$-diagram we can infer two properties: the first one is that the decompositions given in condition (d) above can be extended to any odd layer; the second one is that one can automatically construct rombs at even layers, using condition (e) above together with the first property already mentioned. These are proven in the next proposition.

\begin{proposition}\label{proposition-decomposition.odd.rombs.even}
Let $(F,D)$ be an $l$-diagram.
\begin{enumerate}[(i),leftmargin=0.7cm]
\item \emph{(Vertices at odd layers)} For $j \geq 1$, we have
$$F^{0,2j-1} = \bigsqcup_{v\in F^{0,2j-2}} s(B_v) = \bigcup_{v\in F^{0,2j-2}} s(R_v).$$
\item \emph{(Rombs at even layers)} Let $j \geq 0$ and $v \in F^{0,2j}$. Given a pair $(e_0,e_1)$ of blue edges such that $r(e_0) = v$ and $r(e_1) = s(e_0)$, there is a unique pair $(f_0,f_1)$ of red edges such that $r(f_0) = v$, $r(f_1) = s(f_0)$, $s(f_1) = s(e_1)$ and $f_1 \in R(f_0)$.

Conversely, given a pair $(f_0,f_1)$ of red edges such that $r(f_0) = v$ and $f_1 \in R(f_0)$, there is a unique pair $(e_0,e_1)$ of blue edges such that $r(e_0) = v$, $r(e_1) = s(e_0)$ and $s(e_1)= s(f_1)$.
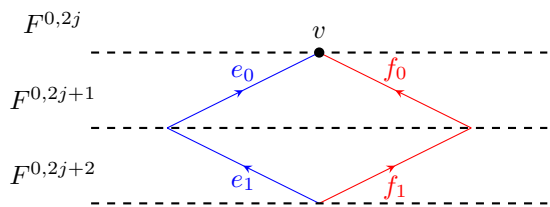
\begin{figure}[H]
\begin{tikzpicture}
 \path [draw=blue,postaction={on each segment={mid arrow=blue}}]
 (0,0) -- node[below,blue]{$e_1$} (-2,1)
 (-2,1) -- node[above,blue]{$e_0$} (0,2)
 ;
 \path [draw=red,postaction={on each segment={mid arrow=red}}]
 (0,0) -- node[below,red]{$f_1$} (2,1)
 (2,1) -- node[above,red]{$f_0$} (0,2)
 ;
 \node[circle,fill=black,scale=0.4,label=$v$] (T) at (0,2) {};
 \draw[black,thick,dashed] (-3,0) -- (3,0);
 \draw[black,thick,dashed] (-3,1) -- (3,1);
 \draw[black,thick,dashed] (-3,2) -- (3,2);
 \node[label=$F^{0,2j}$] (N) at (-3.5,2) {};
 \node[label=$F^{0,2j+1}$] (N+1) at (-3.5,1) {};
 \node[label=$F^{0,2j+2}$] (N+2) at (-3.5,0) {};
\end{tikzpicture}
\label{figure-schematics1}
\end{figure}
\end{enumerate}
\end{proposition}
\begin{proof}
\begin{enumerate}[(i),leftmargin=0.6cm]
\item We can prove this by induction. The case $j = 1$ holds by definition, so assume the statement is true for all indices up to some $j-1 \geq 1$ and let us prove it for $j$. Given a vertex $v \in F^{0,2j-1}$, being $(F,D)$ a separated Bratteli diagram implies that there exists some edge $e \in F^{1,2j-2}$ with source $v$, which can be either blue or red.

In case $e = e_1 \in B_{r(e_1)}$ is a blue edge, then condition (c) above applied to $F^{0,2j-2}$ tells us that there is another blue edge $e_0 \in B_{r(e_0)}$ such that $s(e_0) = r(e_1)$. Moreover $r(e_0) \in F^{0,2j-3}$, thus by the induction hypothesis there is some red edge $g \in R_{r(g)}$ such that $s(g) = r(e_0)$. Now we can apply condition (g) to get a unique pair $(f_0,f_1)$ of red edges such that $f_0 \in R(g)$, $r(f_0) = r(e_0)$, $s(f_0) = r(f_1)$ and $s(f_1) = s(e_1) = v$. This says precisely that $v$ is the source of some red edge from $F^{1,2j-2}$.

In case $e = f_1 \in R_{r(f_1)}$ is a red edge, then condition (c) again tells us that there is another red edge $f_0 \in R_{r(f_0)}$ such that $s(f_0) = r(f_1)$. We apply condition (f) to get a unique pair $(e_0,e_1)$ of blue edges such that $r(e_0) = r(f_0)$, $s(e_0) = r(e_1)$ and $s(e_1) = s(f_1) = v$. This says precisely that $v$ is the source of some blue edge from $F^{1,2j-2}$.

This proves that
$$F^{0,2j-1} = \bigcup_{v\in F^{0,2j-2}} s(B_v) = \bigcup_{v\in F^{0,2j-2}} s(R_v).$$
The disjointness of the first union follows from the last part of property (d).

\item Let $(e_0,e_1)$ be a pair of blue edges such that $r(e_0) = v$ and $r(e_1) = s(e_0)$. By condition (e) we have
$$s(e_1) \in \bigsqcup_{e \in B_v} s(B_{s(e)}) = \bigsqcup_{f \in R_v} s(R(f)),$$
so there exists a pair of red edges $(f_0,f_1)$ such that $s(e_1) = s(f_1)$, with $f_1 \in R (f_0)$ and $f_0 \in R_v$ (so that $r(f_1) = s(f_0)$ and $r(f_0) = v$). The pair $(f_0,f_1)$ satisfies the required properties. To prove uniqueness, assume that there exists another pair $(f'_0,f'_1)$ of red edges such that $r(f'_0)=v$, $s(f'_0)=r(f'_1)$, $s(f'_1) = s(e_1) = s(f_1)$ and $f_1'\in R(f_0')$. Condition (c) ensures that $f'_1 = f_1$; in particular $s(f'_0) = r(f'_1) = r(f_1) = s(f_0)$. Since $f_1'=  f_1\in R(f_0)\cap R(f_0')$, we conclude that $f_0 = f_0'$. That is, the pair $(f_0,f_1)$ is unique.

The proof of the converse is analogous, although the last part of condition (d) is needed. 
\end{enumerate}
\end{proof}

\begin{notation}\label{notation-romb}
We will sometimes write $(e_0,e_1,f_0,f_1)$ when referring to the romb associated to either the blue pair of edges $(e_0,e_1)$ or to the red pair of edges $(f_0,f_1)$.
\end{notation}

\begin{remarks}\label{remark-about.def.of.Ldiagram}
\begin{enumerate}[(1),leftmargin=0.7cm]
\item By condition (c) in Definition \ref{definition-L.separated.Bratteli}, we deduce that for each $j \ge 1$ and $v \in F^{0,2j}$ there are exactly two different edges $e,f$ such that $s(e) = s(f) = v$, one of them, say $e$, belongs to a blue set $e \in B_{r(e)}$ and the other, $f$, belongs to $R(g)$ for a unique red edge $g$, that is, there is a unique $g\in R_{r(g)}$ such that $s(g)= r(f)$ and $f\in R(g)$.

On the other hand, by part (i) of Proposition \ref{proposition-decomposition.odd.rombs.even} and the last part of condition (d), we deduce that for each $j \geq 1$ and $v \in F^{0,2j-1}$ there is exactly one blue edge $e \in B_{r(e)}$ such that $s(e) = v$. There also exists a red edge $f \in R_{r(f)}$ such that $s(f) = v$, but such red edge may \textit{not} be the unique red edge departing from $v$.
\item We have exactly 2 ``colors'' at each vertex $v$ from an even layer $F^{0,2j}$, namely the blue edges $B_v$ and the red edges $R_v$.

On the other hand, by remark (1) above, we have $|s^{-1}(v)|$ different ``colors'' at each vertex $v \in F^{2j-1}$, $j \geq 1$, each one corresponding to a different edge departing from $v$.

\item The disjointness of both unions in condition (e) of Definition \ref{definition-L.separated.Bratteli} is automatic, due to condition (c) of the same definition and remark (1) above.

\item The uniqueness of the pair $(e_0,e_1)$ in condition (f) of Definition \ref{definition-L.separated.Bratteli} is also automatic, due to remark (1) above.

\item An $l$-diagram cannot have double red edges, except possibly at the first layer. More concretely, for $n \geq 2$, given vertices $v \in F^{0,n-1}$ and $w \in F^{0,n}$, there is at most one red edge $f$ such that $s(f) = w$ and $r(f) = v$. This is clear if $n$ is even by remark (1). Assume then that $n = 2j+1$ for some $j \geq 1$, and suppose that there are two red edges $f_1,f'_1$ such that $s(f_1) = s(f'_1) = w$ and $r(f_1) = r(f'_1) = v$. Let $f_0$ be the unique red edge such that $s(f_0) = v$, so there is a unique red edge $g \in R_{r(g)}$ such that $f_0 \in R(g)$ (see again remark (1)). By condition (f) in Definition \ref{definition-L.separated.Bratteli}, there is a unique pair $(e_0,e_1)$ of blue edges such that $r(e_0) = r(f_0), s(e_0) = r(e_1)$ and $s(e_1) = s(f_1) = w$. But now $(f_0,f_1)$ and $(f_0,f'_1)$ are two pairs of red edges such that $f_0 \in R(g)$, $r(f_0) = r(e_0), s(f_0) = r(f_1) = r(f'_1)$ and $s(f_1) = s(f'_1) = s(e_1)$, and this contradicts the uniqueness statement in condition (g) of Definition \ref{definition-L.separated.Bratteli}. This shows that there is at most one red edge from $w$ to $v$.
\end{enumerate}
\end{remarks}

\begin{example}\label{example-ldiagram}

Figure \ref{figure-example0} presents the first three layers of an $l$-diagram. Here we have, for example, that the separation $D_u$ associated with the vertex $u \in F^{0,1}$ is written as
$$D_u = \{B_u,R(f_1),R(f_2)\},$$
where the set $B_u$ consists of the two blue edges having range $u$, $R(f_1)$ consists of the unique straight red edge having range $u$ (which we call $g_1$), and $R(f_2)$ consists of the unique dotted red edge having range $u$. Similarly,
$$D_v = \{B_v,R_v\},$$
where here $B_v$ consists of the unique blue edge having range $v$, and $R_v$ consists of the two dotted red edges having range $v$. It is straightforward to check that conditions (a)-(g) from Definition \ref{definition-L.separated.Bratteli} are fulfilled, at least for these first three levels.
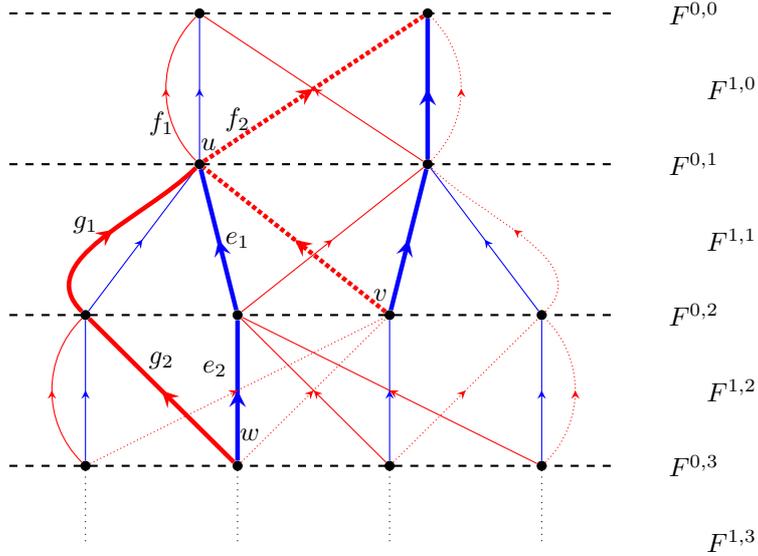
\begin{figure}[ht]
\begin{tikzpicture}
	\draw[black,thick,dashed] (-4,3) -- (4,3);
	\draw[black,thick,dashed] (-4,1) -- (4,1);
	\draw[black,thick,dashed] (-4,-1) -- (4,-1);
	\draw[black,thick,dashed] (-4,-3) -- (4,-3);
	\node[circle,fill=black,scale=0.4] (A1) at (-1.5,3) {};
	\node[scale=0.5,label=$f_1$] (F1) at (-2,1.2) {};
	\node[scale=0.5,label=$f_2$] (F2) at (-1,1.25) {};
	\node[scale=0.5,label=$g_1$] (F3) at (-3,-0.1) {};
	\node[scale=0.5,label=$e_1$] (F4) at (-1,-0.3) {};
	\node[scale=0.5,label=$g_2$] (F5) at (-2,-1.9) {};
	\node[scale=0.5,label=$e_2$] (F6) at (-1.3,-2) {};
	\node[circle,fill=black,scale=0.4] (A2) at (1.5,3) {};
	\node[circle,fill=black,scale=0.4,label={[xshift=0.12cm, yshift=0.0cm]$u$}] (B1) at (-1.5,1) {};
	\node[circle,fill=black,scale=0.4] (B2) at (1.5,1) {};
	\node[circle,fill=black,scale=0.4] (C1) at (-3,-1) {};
	\node[circle,fill=black,scale=0.4] (C2) at (-1,-1) {};
	\node[circle,fill=black,scale=0.4,label={[xshift=-0.12cm, yshift=0.0cm]$v$}] (C3) at (1,-1) {};
	\node[circle,fill=black,scale=0.4] (C4) at (3,-1) {};
	\node[circle,fill=black,scale=0.4] (D1) at (-3,-3) {};
	\node[circle,fill=black,scale=0.4,label={[xshift=0.16cm, yshift=0.15cm]$w$}] (D2) at (-1,-3) {};
	\node[circle,fill=black,scale=0.4] (D3) at (1,-3) {};
	\node[circle,fill=black,scale=0.4] (D4) at (3,-3) {};
	\draw[-,blue,postaction={on each segment={mid arrow=blue}}] (B1) to (A1);
	\draw[line width=1.7pt,-,blue,postaction={on each segment={mid arrow=blue}}] (B2) to (A2);
	\draw[-,blue,postaction={on each segment={mid arrow=blue}}] (C1) to (B1);
	\draw[line width=1.7pt,-,blue,postaction={on each segment={mid arrow=blue}}] (C2) to (B1);
	\draw[line width=1.7pt,-,blue,postaction={on each segment={mid arrow=blue}}] (C3) to (B2);
	\draw[-,blue,postaction={on each segment={mid arrow=blue}}] (C4) to (B2);
	\draw[-,blue,postaction={on each segment={mid arrow=blue}}] (D1) to (C1);
	\draw[line width=1.7pt,-,blue,postaction={on each segment={mid arrow=blue}}] (D2) to (C2);
	\draw[-,blue,postaction={on each segment={mid arrow=blue}}] (D3) to (C3);
	\draw[-,blue,postaction={on each segment={mid arrow=blue}}] (D4) to (C4);
	\draw[-,red,postaction={on each segment={mid arrow=red}}] (B1) to [out=135,in=225] (A1);
	\draw[line width=1.7pt,densely dotted,red,postaction={on each segment={mid arrow=red}}] (B1) to (A2);
	\draw[-,red,postaction={on each segment={mid arrow=red}}] (B2) to (A1);
	\draw[densely dotted,red,postaction={on each segment={mid arrow=red}}] (B2) to [out=45,in=315] (A2);
	\draw[line width=1.7pt,-,red,postaction={on each segment={mid arrow=red}}] (C1) to [out=135,in=225] (B1);
	\draw[-,red,postaction={on each segment={mid arrow=red}}] (C2) to (B2);
	\draw[line width=1.7pt,densely dotted,red,postaction={on each segment={mid arrow=red}}] (C3) to (B1);
	\draw[densely dotted,red,postaction={on each segment={mid arrow=red}}] (C4) to [out=45,in=315] (B2);
	\draw[-,red,postaction={on each segment={mid arrow=red}}] (D1) to [out=135,in=225] (C1);
	\draw[line width=1.7pt,-,red,postaction={on each segment={mid arrow=red}}] (D2) to (C1);
	\draw[densely dotted,red,postaction={on each segment={mid arrow=red}}] (D1) to (C3);
	\draw[densely dotted,red,postaction={on each segment={mid arrow=red}}] (D2) to (C3);
	\draw[-,red,postaction={on each segment={mid arrow=red}}] (D3) to (C2);
	\draw[densely dotted,red,postaction={on each segment={mid arrow=red}}] (D3) to (C4);
	\draw[-,red,postaction={on each segment={mid arrow=red}}] (D4) to (C2);
	\draw[densely dotted,red,postaction={on each segment={mid arrow=red}}] (D4) to [out=45,in=315] (C4);
	\node at (5,3){$F^{0,0}$};
	\node at (5.5,2){$F^{1,0}$};
	\node at (5,1){$F^{0,1}$};
	\node at (5.5,0){$F^{1,1}$};
	\node at (5,-1){$F^{0,2}$};
	\node at (5.5,-2){$F^{1,2}$};
	\node at (5,-3){$F^{0,3}$};
	\node at (5.5,-4){$F^{1,3}$};
	\draw[black,dotted] (-3,-4) -- (-3,-3);
	\draw[black,dotted] (-1,-4) -- (-1,-3);
	\draw[black,dotted] (1,-4) -- (1,-3);
	\draw[black,dotted] (3,-4) -- (3,-3);
\end{tikzpicture}
\caption{The first three levels of an $l$-diagram.}
\label{figure-example0}
\end{figure}
In the same Figure \ref{figure-example0} we show two examples of rombs. For instance, the first romb, based on the vertex $w \in F^{0,3}$ and ending at the vertex $u \in F^{0,1}$, corresponds to condition (f) for the pair of red edges $(g_1,g_2)$ determining a unique pair of blue edges $(e_1,e_2)$ satisfying the romb condition, and also corresponds to condition (g) for $g = f_1$ and the pair of blue edges $(e_1,e_2)$ determining a unique pair of red edges $(g_1,g_2)$ satisfying the romb condition and such that $g_1 \in R(f_1)$. The second romb, based on the vertex $v$, is a consequence of Proposition \ref{proposition-decomposition.odd.rombs.even}.
\end{example}

Our correspondence will associate, to every $l$-diagram $(F,D)$, a dynamical system $(X,\sigma)$ where $X$ is a totally disconnected metrizable compact space, and $\sigma : X \ra X$ is a surjective local homeomorphism; and, conversely, an $l$-diagram $(F,D)$ will be associated to every such dynamical system $(X,\sigma)$.

\begin{notation}\label{notation-Lhomeo}
For ease of notation, we will denote by $\textsf{LHomeo}$ the family of dynamical systems $(X,\sigma)$ where $X$ is a totally disconnected metrizable compact space, and $\sigma : X \ra X$ is a surjective local homeomorphism.
\end{notation}

Of course, the aforementioned constructions will hold `up to equivalences' in appropriate senses.

\begin{definition}\label{definition-equiv.top.conj}
Let $\sigma_1 : X_1 \ra X_1$, $\sigma_2 : X_2 \ra X_2$ be two surjective local homeomorphisms defined on totally disconnected metrizable compact spaces $X_1,X_2$. We say that $(X_1,\sigma_1)$ and $(X_2,\sigma_2)$ are \textit{topologically conjugate} if there exists a homeomorphism $h : X_1 \ra X_2$ intertwining $\sigma_1$ and $\sigma_2$, that is $h \circ \sigma_1 = \sigma_2 \circ h$.
\end{definition}

We write $(X_1,\sigma_1) \sim_{t.c.} (X_2,\sigma_2)$ in case the two dynamical systems are topologically conjugate. It is clear that $\sim_{t.c.}$ is an equivalence relation.

\begin{remark}
The definition of conjugacy from \cite[Definition 2.1]{ABCE2023} requires the additional condition that $\sigma _1 \circ h^{-1} = h^{-1}\circ \sigma_2 $. However this condition is automatic in our situation, as can be seen by pre- and post-composing the identity $h \circ \sigma_1 = \sigma_2\circ h$ with $h^{-1}$ (see also \cite[Lemma 2.6]{ABCE2023}).
\end{remark}

For $l$-diagrams, the notion of equivalence that fits in our context (apart from the natural notion of `up to isomorphism' which we will define later) is the notion of \textit{telescoping}, or \textit{contraction}. For an $l$-diagram $(F,D)$ and a given pair $j < l$ of non-negative integers, we denote
$$P_{j,l} = \{(e_j,e_{j+1},\dots,e_{l-1}) \mid e_i \in F^{1,i}, s(e_i) = r(e_{i+1})\} \subseteq \prod_{i=j}^{l-1} F^{1,i}.$$
The source of $\ol{e} = (e_j,e_{j+1},\dots,e_{l-1}) \in P_{j,l}$ is $s(\ol{e}) := s(e_{l-1})$, and its range is $r(\ol{e}) := r(e_j)$.

\begin{definition}\label{definition-equiv.telesc.l.diagrams}
Given an $l$-diagram $(F,D)$ and a sequence $(m_n)_{n \geq 0}$, $0 \le  m_0 < m_1 < m_2 < \cdots$ in $\N \cup \{0\}$, with $m_0$ even and satisfying the property
\begin{equation}\tag{CR}
m_{n+1} - m_n \equiv 1 \text{ }(\text{mod }2),
\end{equation}
we define the \textit{telescoping} (or \textit{contraction}) of $(F,D)$ with respect to $(m_n)_{n \geq 0}$ as the separated graph $(F',D')$ where:
\begin{enumerate}[(a),leftmargin=0.7cm]
\item $(F')^{0,n} := F^{0,m_n}$ for $n \geq 0$.
\item $(F')^{1,n} := P_{m_n,m_{n+1}}^B \sqcup P_{m_n,m_{n+1}}^R$ for $n \geq 0$, where for a given pair $j<l$ of non-negative integers,
$$P^B_{j,l} := \{(e_j,\dots,e_{l-1}) \in P_{j,l} \mid e_i \in B_{r(e_i)}, j \leq i \leq l-1 \},$$
$$P^R_{j,l} := \{(f_j,\dots,f_{l-1}) \in P_{j,l} \mid  f_i \in R_{r(f_i)}, \text{ }j \leq i \leq l-1 \text{ and } f_{i+1} \in R(f_i), \text{ $i$ even} \}.$$
\item For $v \in (F')^{0,2j}$ with $j \geq 0$, $D'_v := \{B'_v, R'_v\}$, where
$$B'_v = \{ \ol{e} \in P_{m_{2j},m_{2j+1}}^B \mid r(\ol{e}) = v\}, \quad R'_v = \{ \ol{f} \in P_{m_{2j},m_{2j+1}}^R \mid r(\ol{f}) = v\}.$$
\item For $v \in (F')^{0,2j+1}$ with $j \geq 0$, $D'_v := \{B'_v\} \cup \{R'(\hat{f}) \mid \hat{f} \in (F')^{1,2j}, s(\hat{f}) = v, \hat{f} \in R'_{r(\hat{f})}\}$, where
$$B'_v = \{ \ol{e} \in P_{m_{2j+1},m_{2j+2}}^B \mid r(\ol{e}) = v\},$$
$$R'(\hat{f}) = \{ \ol{f} = (f_{m_{2j+1}},\dots,f_{m_{2j+2}-1}) \in P_{m_{2j+1},m_{2j+2}}^R \mid f_{m_{2j+1}} \in R(f'_{m_{2j+1}-1})\},$$
where $\hat{f} = (f'_{m_{2j}},\dots,f'_{m_{2j+1}-1})$.
\end{enumerate}
\end{definition}
We note that both $m_n$ and $n$ have the same parity due to condition (CR) and the fact that $m_0$ is even. 

\begin{lemma}\label{lemma-F'D'.L.diagram}
$(F',D')$ as defined in Definition \ref{definition-equiv.telesc.l.diagrams} is an $l$-diagram.
\end{lemma}
\begin{proof}
First observe that $(F',D')$ is a separated Bratteli diagram since $(F,D)$ is. To prove that it is indeed an $l$-diagram, we need to check conditions (a)-(g) from Definition \ref{definition-L.separated.Bratteli}. Conditions (a) and (b) follow by definition.
\begin{enumerate}[(a),start=3,leftmargin=0.6cm]
\item We first show that, for $j \geq 1$,
$$\bigsqcup_{v \in F^{0,m_{2j-1}}} s(B'_v) \mathrel{\overset{\makebox[0pt]{\mbox{\normalfont\tiny\sffamily (c.1)}}}{=}} F^{0,m_{2j}} \mathrel{\overset{\makebox[0pt]{\mbox{\normalfont\tiny\sffamily (c.2)}}}{=}} \bigsqcup_{v \in F^{0,m_{2j-1}}} s(R'_v).$$
The inclusion $\subseteq$ in {\normalfont\tiny\sffamily (c.1)} is obvious. For the other one we use condition (c) of Definition \ref{definition-L.separated.Bratteli} and part (i) of Proposition \ref{proposition-decomposition.odd.rombs.even}. Given $v \in F^{0,m_{2j}}$, there exists $v_1 \in F^{0,m_{2j}-1}$ and $e_{m_{2j}-1} \in B_{v_1}$ such that $v = s(e_{m_{2j}-1})$. Now since $v_1 \in F^{0,m_{2j}-1}$, by part (i) of Proposition \ref{proposition-decomposition.odd.rombs.even} there exists $v_2 \in F^{0,m_{2j}-2}$ and $e_{m_{2j}-2} \in B_{v_2}$ such that $v_1 = s(e_{m_{2j}-2})$. Proceeding by induction, we construct a sequence $(e_{m_{2j}-i})_{i=1}^{m_{2j}-m_{2j-1}}$ of blue edges satisfying $s(e_{m_{2j}-i-1}) = r(e_{m_{2j}-i})$ and $v = s(e_{m_{2j}-1})$. Then $e := (e_{m_{2j-1}},\dots,e_{m_{2j}-1}) \in P^B_{m_{2j-1},m_{2j}}$ is such that $s(e) = v$, so the inclusion is now proved. Now in order to prove that the union is indeed disjoint, take $v \in s(B'_{v_1}) \cap s(B'_{v_2})$ with $v_1 \neq v_2 \in F^{0,m_{2j-1}}$. There exist then $e_1 \in B'_{v_1}$ and $e_2 \in B'_{v_2}$ with $s(e_1) = v = s(e_2)$. Let $N$ be the integer such that 
$(e_1)_N \neq (e_2)_N$ and $(e_1)_j = (e_2)_j$ for  $j>N$.
It exists since $v_1 \neq v_2$, so necessarily $e_1 \neq e_2$. But then $s((e_1)_N) = s((e_2)_N)$ and $(e_1)_N \neq (e_2)_N$, and this contradicts either condition (c) of Definition \ref{definition-L.separated.Bratteli} or part (i) of Proposition \ref{proposition-decomposition.odd.rombs.even} at the $N^{\text{th}}$ layer, depending on the parity of $N$.

The same arguments hold exactly for {\normalfont\tiny\sffamily (c.2)}, but special care needs to be taken when proving disjointness of the union. This is because we need to make a distinction in terms of the parity of $N$ which is of different nature to the above distinction. Given $v \in s(R'_{v_1}) \cap s(R'_{v_2})$ with $v_1 \neq v_2 \in F^{0,m_{2j-1}}$, there exist $f_1 \in R'_{v_1}$ and $f_2 \in R'_{v_2}$ with $s(f_1) = v = s(f_2)$. We can again consider the integer $N$ such that 
$(f_1)_N \neq (f_2)_N$ and $(f_1)_j = (f_2)_j$ for  $j>N$. Thus $s((f_1)_N) = s((f_2)_N)$ and $(f_1)_N \neq (f_2)_N$. If $N$ is odd, the contradiction arises due to condition (c) of Definition \ref{definition-L.separated.Bratteli}. If $N$ is even, then $s((f_1)_N) = s((f_2)_N)$ belongs to an odd layer, thus $(f_1)_{N+1} = (f_2)_{N+1} \in R((f_1)_N) \cap R((f_2)_N)$ by definition of $P_{m_{2j-1},m_{2j}}^B$. This is a contradiction because $R((f_1)_N) \cap R((f_2)_N) = \emptyset$ since $(f_1)_N \neq (f_2)_N$.\\

Now for the second part. Take any $v \in (F')^{0,2j-1} = F^{0,m_{2j-1}}$ and $e_1,e_2 \in B'_v$. If it were the case that $s(e_1) = s(e_2)$, then by applying properties (c) and (d) of Definition \ref{definition-L.separated.Bratteli} at the different layers in between $F^{0,m_{2j-1}}$ and $F^{0,m_{2j}}$ we would deduce that $e_1 = e_2$. An argument similar to the one given in the above paragraph shows the same property for $f_1,f_2 \in R'_v$, analyzing the parity of the different layers.

\item The first two parts (1) and (2) follow \textit{mutatis mutandis} the same arguments as in (c) above, using conditions (c) and (d) of Definition \ref{definition-L.separated.Bratteli} and part (i) of Proposition \ref{proposition-decomposition.odd.rombs.even} at the different layers.

The second property follows from the same results again, applied to the different layers.

\item We need to show that, for $v \in F^{0,m_{2j}}$,
$$\bigsqcup_{e \in B'_v} s(B'_{s(e)}) = \bigsqcup_{f \in R'_v} s(R'(f)).$$
Note that we do not need to prove disjointness of the unions, since these will follow by Remark \ref{remark-about.def.of.Ldiagram} (3). Let $w = s(e^{(1)})$, with $e^{(1)} = (e_{m_{2j+1}}^{(1)},\dots,e_{m_{2j+2}-1}^{(1)}) \in B'_{s(e^{(0)})}$ and $e^{(0)} = (e_{m_{2j}}^{(0)},\dots,e_{m_{2j+1}-1}^{(0)}) \in B'_v$. Note that $(e^{(0)},e^{(1)}) \in P^B_{m_{2j},m_{2j+2}}$, and
$$m_{2j+2} - m_{2j} \equiv (2j+2)-2j \equiv 0 \text{ (mod }2),$$
so we have an \textit{even} number of edges from $w$ to $v$ through the blue path $(e^{(0)},e^{(1)})$. We can then write the path as $(e^{(0)},e^{(1)}) = (e_{1,0},e_{1,1},\dots,e_{n_j,0},e_{n_j,1})$, being $n_j$ the natural number satisfying $2n_j = m_{2j+2}-m_{2j}$. For each layer $1 \leq i \leq n_j$, we can make use of condition (e) of Definition \ref{definition-L.separated.Bratteli} to conclude that there exist pairs of \textit{red} edges $(f_{i,0},f_{i,1})$ such that $f_{i,0} \in R_{r(e_{i,0})}, f_{i,1} \in R(f_{i,0})$, and $s(f_{i,1}) = s(e_{i,1})$. We then construct two red paths
$$(f^{(0)},f^{(1)}) = (f_{1,0},f_{1,1},\dots,f_{n_j,0},f_{n_j,1})$$
of lengths $m_{2j+2}-m_{2j+1}$ and $m_{2j+1}-m_{2j}$ respectively, satisfying $f^{(0)} \in R'_v$ and $f^{(1)} \in R'(f^{(0)})$. Since $w = s(f^{(1)})$, we see that $w \in s(R'(f^{(0)}))$ with $f^{(0)} \in R'_v$, thus completing the proof of the first inclusion. The other inclusion is completely analogous, and so we omit it.

\item Let $v \in (F')^{0,2j-1} = F^{0,m_{2j-1}}$ and $(f^{(0)},f^{(1)})$ a pair of red edges such that $r(f^{(0)}) = v$ and $s(f^{(0)}) = r(f^{(1)}) \in F^{0,m_{2j}}$. Note that $w := s(f^{(1)}) \in F^{0,m_{2j+1}}$ and 
$$m_{2j+1} - m_{2j-1} \equiv (2j+1)-(2j-1) \equiv 0 \text{ (mod }2),$$
so we have an \textit{even} number of edges from $w$ to $v$ through the red path $(f^{(0)},f^{(1)})$. We can then write the path as $(f^{(0)},f^{(1)}) = (f_{1,0},f_{1,1},\dots,f_{n_j,0},f_{n_j,1})$, being $n_j$ the natural number satisfying $2n_j = m_{2j+1}-m_{2j-1}$. For each layer $1 \leq i \leq n_j$, we can make use of condition (f) of Definition \ref{definition-L.separated.Bratteli} to conclude that there exist pairs of \textit{blue} edges $(e_{i,0},e_{i,1})$ such that $r(e_{i,0}) = r(f_{i,0})$, $s(e_{i,0}) = r(e_{i,1})$ and $s(e_{i,1}) = s(f_{i,1}) = r(f_{i+1,0}) = r(e_{i+1,0})$. We then construct two blue paths
$$(e^{(0)},e^{(1)}) = (e_{1,0},e_{1,1},\dots,e_{n_j,0},e_{n_j,1})$$
of lengths $m_{2j+1}-m_{2j}$ and $m_{2j}-m_{2j-1}$ respectively, satisfying $r(e^{(0)}) = r(e_{1,0}) = r(f_{1,0}) = v$, $s(e^{(0)}) = r(e^{(1)})$ and $s(e^{(1)}) = s(e_{n_j,1}) = s(f_{n_j,1}) = w$. As noticed in Remark \ref{remark-about.def.of.Ldiagram} (4), it is not necessary to prove uniqueness of this blue pair of edges $(e^{(0)},e^{(1)})$.
\item Let $v \in (F')^{0,2j-1} = F^{0,m_{2j-1}}$, a red edge $g \in R'_{r(g)}$ such that $s(g) = v$, and $(e^{(0)},e^{(1)})$ a pair of blue edges such that $r(e^{(0)}) = v$ and $s(e^{(0)}) = r(e^{(1)}) \in F^{0,m_{2j}}$. Note that $w := s(e^{(1)}) \in F^{0,m_{2j+1}}$, and again 
$$m_{2j+1} - m_{2j-1} \equiv (2j+1)-(2j-1) \equiv 0 \text{ (mod }2),$$
so we have an \textit{even} number of edges from $w$ to $v$ through the blue path $(e^{(0)},e^{(1)})$. We can then write the path as $(e^{(0)},e^{(1)}) = (e_{1,0},e_{1,1},\dots,e_{n_j,0},e_{n_j,1})$, being $n_j$ the natural number satisfying $2n_j = m_{2j+1}-m_{2j-1}$. For each layer $1 \leq i \leq n_j$, we can make use of condition (g) of Definition \ref{definition-L.separated.Bratteli} to conclude that there exist unique pairs of \textit{red} edges $(f_{i,0},f_{i,1})$ such that $f_{1,0} \in R(g_0), f_{i,0} \in R(f_{i-1,1})$ for $2 \leq i \leq n_j$, and $r(f_{i,0}) = r(e_{i,0}), s(f_{i,0}) = r(f_{i,1})$ and $s(f_{i,1}) = s(e_{i,1})$. Here $g_0$ is the first red edge of the path $g$, i.e. the one having source $v$. We then construct two red paths
$$(f^{(0)},f^{(1)}) = (f_{1,0},f_{1,1},\dots,f_{n_j,0},f_{n_j,1})$$
with $f^{(0)}\in P^R_{m_{2j-1},m_{2j}}$ and $f^{(1)}\in P^R_{m_{2j},m_{2j+1}}$, satisfying $r(f^{(0)}) = r(f_{1,0}) = r(e_{1,0}) = v$, $s(f^{(0)}) = r(f^{(1)})$ and $s(f^{(1)}) = s(f_{n_j,1}) = s(e_{n_j,1}) = w$. Moreover, $f^{(0)} \in R'(g)$ by the definition of $R'(g)$ (Definition \ref{definition-equiv.telesc.l.diagrams} (d)). This is the desired red pair.

To prove uniqueness, assume that $(\wt{f}^{(0)},\wt{f}^{(1)})$ is another pair of red edges satisfying the same conditions as the pair $(f^{(0)},f^{(1)})$, namely that $\wt{f}^{(0)} \in R'(g)$, $r(\wt{f}^{(0)}) = r(e^{(0)}) = v$, $s(\wt{f}^{(0)}) = r(\wt{f}^{(1)})$ and $s(\wt{f}^{(1)}) = s(e^{(1)})$. Write it as
$$(\wt{f}^{(0)},\wt{f}^{(1)}) = (\wt{f}_{1,0},\wt{f}_{1,1},\dots,\wt{f}_{n_j,0},\wt{f}_{n_j,1}).$$
We apply condition (f) of Definition \ref{definition-L.separated.Bratteli} to each pair of red edges $(\wt{f}_{i,0},\wt{f}_{i,1})$ to conclude that there exist pairs $(\wt{e}_{i,0},\wt{e}_{i,1})$ of blue edges such that $r(\wt{e}_{i,0}) = r(\wt{f}_{i,0})$, $s(\wt{e}_{i,0}) = r(\wt{e}_{i,1})$ and $s(\wt{e}_{i,1}) = s(\wt{f}_{i,1})$. Now,
$$s(\wt{e}_{n_j,1}) = s(\wt{f}^{(1)}) = s(e^{(1)}) = s(e_{n_j,1}) \in F^{0,m_{2j+1}}.$$
By condition (d) of Definition \ref{definition-L.separated.Bratteli}, we necessarily have $\wt{e}_{n_j,1} = e_{n_j,1}$. Thus in particular $s(\wt{e}_{n_j,0}) = r(\wt{e}_{n_j,1}) = r(e_{n_j,1}) = s(e_{n_j,0}) \in F^{0,m_{2j+1}-1}$. Now condition (c) of Definition \ref{definition-L.separated.Bratteli} ensures that $\wt{e}_{n_j,0} = e_{n_j,0}$ necessarily. Therefore $r(\wt{f}_{n_j,0}) = r(\wt{e}_{n_j,0}) = r(e_{n_j,0}) = r(f_{n_j,0})$. We continue in this way to get that $\wt{e}_{i,j}= e_{i,j}$ for all $i,j$, and thus 
$$r(f_{i,0})= r(e_{i,0})= r(\wt{e}_{i,0}) = r(\wt{f}_{i,0}) \quad \text{ for all } 1 \leq i \leq n_j.$$
Now observe that $\wt{f}_{i,0}\in R(g_0)$, and thus 
by the uniqueness part of condition (g) of Definition \ref{definition-L.separated.Bratteli}, we get that the pairs $(\wt{f}_{1,0},\wt{f}_{1,1})$ and $(f_{1,0},f_{1,1})$ are the same. By applying this argument inductively (using that $f^{(0)},f^{(1)},\wt{f}^{(0)}, \wt{f}^{(1)}$ belong to $P^R_{j,l}$ for suitable $j,l$), we end up concluding that $\wt{f}^{(0)} = f^{(0)}$ and $\wt{f}^{(1)} = f^{(1)}$. This proves the uniqueness property we wanted to show.
\end{enumerate}
The lemma is proved.
\end{proof}

\begin{definition}\label{definition-equiv.rel.Ldiagrams}
We let $\sim_l$ denote the equivalence relation on $l$-diagrams generated by `up to isomorphism' equivalence and `up to contraction' equivalence, defined as follows.
\begin{enumerate}[(a),leftmargin=0.7cm]
\item (Isomorphism of $l$-diagrams) The $l$-diagrams $(F_1,D_1)$ and $(F_2,D_2)$ are said to be \textit{isomorphic}, written $(F_1,D_1) \cong (F_2,D_2)$, if there exists $\phi : (F_1,D_1) \ra (F_2,D_2)$ an isomorphism of separated diagrams (see \cite{AG2012}) such that, for any $j \geq 0$, we have
$$\phi_1( (B_1)_v ) = (B_2)_{\phi_0(v)} \quad \text{ and } \quad \phi_1( (R_1)_v ) = (R_2)_{\phi_0(v)}$$
in case $v \in (F_1)^{0,2j}$, and
$$\phi_1( (B_1)_v ) = (B_2)_{\phi_0(v)} \quad \text{ and } \quad \phi_1( R_1(f) ) = R_2(\phi_1(f))$$
in case $v \in (F_1)^{0,2j+1}$, with $f \in R_{r(f)}$ and $s(f) = v$. Here $R_1(f) \subseteq (R_1)_v$ and $R_2(\phi_1(f)) \subseteq (R_2)_{\phi_0(v)}$.
\item (Contraction of $l$-diagrams) The $l$-diagrams $(F_1,D_1)$ and $(F_2,D_2)$ are said to be \textit{equivalent by contraction}, written $(F_1,D_1) \sim_{\text{ctr}} (F_2,D_2)$, if there exists a sequence of $l$-diagrams $\{(E_i,C_i)\}_{i = 1}^n$ with $(E_1,C_1) = (F_1,D_1)$ and $(E_n,C_n) = (F_2,D_2)$ so that $(E_{i+1},C_{i+1})$ is obtained from $(E_i,C_i)$ by contraction (as defined in Definition \ref{definition-equiv.telesc.l.diagrams}), or vice versa. Symbolically, we will write
\begin{equation*}
\begin{tikzcd}[every arrow/.append style={shift left}]
 (E_1,C_1) \arrow{r} & (E_2,C_2) \arrow{l} \arrow{r} & \cdots \arrow{l} \arrow{r} & (E_n,C_n) \arrow{l}
\end{tikzcd}
\end{equation*}
in case $(E_1,C_1) \sim_{\text{ctr}} (E_n,C_n)$.
\end{enumerate}
\end{definition}

It is clear that both relations $\cong$ and $\sim_{\text{ctr}}$ define equivalence relations, so that $\sim_l$ is well-defined.

\begin{lemma}\label{lemma-composition.of.contractions.Ldiagram}
$(E_1,C_1) \sim_{\emph{ctr}} (E_n,C_n)$ if and only if there exists a sequence of $l$-diagrams $\{(E_i,C_i)\}_{i=2}^{n-1}$ such that
\begin{equation*}
\begin{tikzcd}
 (E_1,C_1) \arrow{r} & (E_2,C_2) & (E_3,C_3) \arrow{l} \arrow{r} & \cdots \arrow[r,shift left=1ex,dotted]{} & (E_n,C_n) \arrow[l,shift left=0.5ex,dotted]{}
 \end{tikzcd}
\end{equation*}
That is, each $(E_{2i},C_{2i})$ can be obtained by contracting $(E_{2i-1},C_{2i-1})$ and also by contracting $(E_{2i+1},C_{2i+1})$.
\end{lemma}
\begin{proof}
It suffices to prove that the composition of two contractions is again a contraction. To this aim, suppose that $(F_3,D_3)$ is a contraction of $(F_2,D_2)$ and $(F_2,D_2)$ is a contraction of $(F_1,D_1)$. Thus there exist strictly increasing sequences $(m_n)_{n\geq 0}$ and $(m'_n)_{n \geq 0}$ in $\N \cup \{0\}$, with $m_0$ and $m'_0$ even and satisfying (CR), and such that $(F_2,D_2) = (F'_1,D'_1)$ and $(F_3,D_3) = (F'_2,D'_2)$. Then the sequence $M_n := m'_{m_n}$ also satisfies (CR):
$$M_{n+1} - M_n = m'_{m_{n+1}} - m'_{m_n} \equiv m_{n+1} - m_n \equiv 1 \text{ (mod }2).$$
Moreover, $M_0 = m'_{m_0} \equiv m_0 \equiv 0 \text{ (mod }2)$, and thus $M_0$ is also even. The rest is now clear.
\end{proof}

Our first aim is to prove the following theorem.

\begin{theorem}\label{theorem-correspondence}
There exists a bijective correspondence between equivalence classes of $l$-diagrams under $\sim_l$ and topological conjugacy classes of surjective local homeomorphisms on totally disconnected, compact metrizable spaces.

More concretely, there are well-defined maps
$$\Phi : \emph{\textsf{LHomeo}}/\sim_{t.c.} \lra \emph{\textsf{LDiag}}/\sim_l, \quad \Psi : \emph{\textsf{LDiag}}/\sim_l \lra \emph{\textsf{LHomeo}}/\sim_{t.c.}$$
which are inverse of each other.
\end{theorem}

The rest of the section is devoted to the proof of Theorem \ref{theorem-correspondence}.

\subsection{From \textsf{LHomeo} to \textsf{LDiag}}\label{subsection-from.local.homeo.to.Ldiagram}

We start by associating an $l$-diagram to a surjective local homeomorphism. Let $(X,\sigma) \in \textsf{LHomeo}$. First, some general remarks about $\sigma$.

\begin{lemma}\label{lemma-properties.local.homeo}
There always exists a partition $\calP$ of $X$ such that the restriction $\sigma |_Z$ is a homeomorphism for any $Z \in \calP$. Moreover, each $\sigma(Z)$ is also clopen.

As a consequence, the set $\sigma^{-1}(\{x\})$ is finite for any $x \in X$, and in fact uniformly bounded by $|\calP|$ above.
\end{lemma}
\begin{proof}
Since $\sigma$ is a local homeomorphism, for any $x \in X$ we can find an open set $U_x$ such that $\sigma(U_x)$ is open in $X$ and $\sigma |_{U_x}$ is a homeomorphism. In fact, since the space $X$ is totally disconnected, we can assume that each $U_x$ is clopen. Then $X = \bigcup_{x \in X} U_x$, so by compactness we can find a finite subfamily $\{U_1,\dots,U_n\}$ of $\{U_x\}_{x \in X}$ such that $X = \bigcup_{i=1}^n U_i$, each $U_i$ is clopen and $\sigma|_{U_i}$ is a homeomorphism. By disjontifying $\{U_1,\dots,U_n\}$, we can construct another finite family $\{V_1,\dots,V_m\}$ such that its elements are clopen and pairwise disjoint, $X = \bigcup_{i=1}^m V_i$, and each $\sigma|_{V_i}$ is a homeomorphism. This family is our desired partition $\calP$, so $|\calP| = m$.

Note that $\sigma(V_i)$ is open in $X$ for $1 \leq i \leq m$. Moreover, since each $V_i$ is closed in $X$, it is in particular compact, thus $\sigma(V_i)$ is also compact. Since $X$ is a Hausdorff space, this implies that $\sigma(V_i)$ is closed. That is, each $\sigma(V_i)$ is clopen.

To conclude, assume that there is some $x \in X$ with at least $m+1$ different pre-images $y_1,\dots,y_{m+1} \in X$. Since $X = \bigcup_{i=1}^m V_i$, there are at least two pre-images of $x$ that lie inside the same $V_i$, say $y_1,y_2 \in V_i$. But this contradicts the fact that $\sigma$ is injective when restricted to $V_i$. The lemma now follows. 
\end{proof}

Let $N(\sigma) := \text{sup}\{|\sigma^{-1}(\{x\})| \mid x \in X\}$, which is a finite quantity by the lemma above. For each $1 \leq i \leq N(\sigma)$, define
$$\calU_i := \{x \in X \mid |\sigma^{-1}(\{x\})| = i\}.$$
Since $\sigma$ is surjective, the family $\{\calU_1,\dots,\calU_{N(\sigma)}\}$ covers $X$. In fact, they form a partition of $X$ into clopen sets, as we prove in the following lemma.

\begin{lemma}\label{lemma-each.Ui.clopen}
The family $\{\calU_1,\dots,\calU_{N(\sigma)}\}$ forms a partition of $X$ into clopen sets.
\end{lemma}
\begin{proof}
It is clear that the sets $\calU_1,\dots , \calU_{N(\sigma)}$ are  disjoint and cover $X$. Thus it is enough to show that each $\calU_i$ is clopen. We will in fact show that each set
$$\calU_{\geq i} := \{x \in X \mid |\sigma^{-1}(\{x\})| \geq i\},$$
for $1 \leq i \leq N(\sigma)$, is clopen, from which the result will follow immediately since $\calU_i = \calU_{\geq i} \backslash \calU_{\geq i+1}$ for $1 \leq i < N(\sigma)$, and $\calU_{N(\sigma)} = \calU_{\geq N(\sigma)}$.

By letting $\calP$ be any partition satisfying the properties of Lemma \ref{lemma-properties.local.homeo}, we see that
$$\calU_{\geq i} = \bigcup_{\substack{(Z_1,\dots,Z_i) \in \calP^i \\ \text{all different}}} \sigma(Z_1) \cap \cdots \cap \sigma(Z_i).$$
Indeed, given $x \in \calU_{\geq i}$, there exist at least $i$ different pre-images of $x$, namely $y_1,\dots,y_i$. Since $\sigma$ restricts to a homeomorphism on each $Z \in \calP$, it is clear that $y_j \in Z_j$ for different $Z_j$'s. Thus $x \in \sigma(Z_1) \cap \cdots \cap \sigma(Z_i)$ for $(Z_1,\dots,Z_i) \in \calP^i$ all different. The converse inclusion is trivial, since any element in $\sigma(Z_1) \cap \cdots \cap \sigma(Z_i)$ for $(Z_1,\dots,Z_i) \in \calP^i$ all different has, by definition, at least $i$ different pre-images.

The result follows since each $\sigma(Z)$ is clopen.
\end{proof}

\begin{remark}\label{rem:covering-map}
Any covering map of topological spaces (in the sense of \cite{clark2023}) is a surjective local homeomorphism, and any surjective local homeomorphism on a compact space is a covering map. Hence the results from \cite{D95} and \cite{EV06} apply to the family of dynamical systems studied in the present paper. In particular it follows from \cite[Theorem 9.1]{EV06} that, if $(X,\sigma)\in \textsf{LHomeo}$, then the groupoid $C^*$-algebra $C^*(\calG (X,\sigma))$ is isomorphic to the crossed product $C(X)\rtimes _{\alpha, \mathcal L} \mathbb N$, where $\alpha \colon C(X)\to C(X)$ is the $*$-homomorphism given by $\alpha (f)= f\circ \alpha$ and $\mathcal L \colon C(X) \to C(X)$ is the transfer operator
$$\mathcal L (f)(x) = \frac{1}{|\sigma^{-1}(x)|} \sum_{y\in \sigma^{-1}(x)} f(y).$$
Note that in \cite[Example 4.5]{clark2023} it is constructed an example of a surjective local homeomorphism $f\colon X\to Y$ which is not a covering map, where $X$ and $Y$ are totally disconnected, locally compact, second-countable spaces. Take $Z:= \bigsqcup _{n=0}^{\infty} Y_n \sqcup \bigsqcup_{m=0}^{\infty} X_m$, where $Y_n= Y$ for all $n\ge 0$ and $X_m= X$ for all $m\ge 0$.  Denoting by $y_n$ the copy of $y\in Y$ in $Y_n$, and by $x_m$ the copy of $x\in X$ in $X_m$, we define $\sigma \colon Z \to Z$ by $\sigma (y_m) = y_{m+1}$ for $m\ge 0$, $\sigma (x_n)= x_{n-1}$ for all $n\ge 1$, and $\sigma (x_0) = f(x)_0$. Then $Z$ is a totally disconnected, locally compact, second countable space, and $\sigma$ is a surjective local homeomorphism on $Z$ which is not a covering map. Another nice example of a surjective local homeomorphism which is not a covering map is the map $f\colon \C \setminus \{0\} \to \C\setminus \{0\}$ defined by $f(z) = e^z$.
\end{remark}

\begin{definition}\label{definition-new.refined.partitions}
For a partition $\calP$ of $X$ such that $\sigma$ restricts to a homeomorphism on each $Z \in \calP$, we define new partitions
\begin{enumerate}[(a),leftmargin=0.7cm]
\item $\calP^{\sigma} := \bigcup_{k=1}^{N(\sigma)} \{ \sigma(Z_1) \cap \cdots \cap \sigma(Z_k) \backslash \calU_{\geq k+1} \neq \emptyset \mid Z_i \in \calP \text{ all different}\}$;
\item $\sigma^{-1}(\calP) := \{\sigma^{-1}(Z) \mid Z \in \calP\}$.
\end{enumerate}
\end{definition}

The next lemma shows that $\calP^{\sigma}$ and $\sigma^{-1}(\calP)$ are indeed well-defined.

\begin{lemma}\label{lemma-new.refined.partitions}
The sets $\calP^{\sigma}$ and $\sigma^{-1}(\calP)$ constitute partitions of $X$. Also, if $\calP_1$ and $\calP_2$ are partitions of $X$, then
$$(\calP_1 \vee \calP_2)^{\sigma} = \calP_1^{\sigma} \vee \calP_2^{\sigma} \quad \text{ and } \quad \sigma^{-1}(\calP_1 \vee \calP_2) = \sigma^{-1}(\calP_1) \vee \sigma^{-1}(\calP_2).$$
\end{lemma}
\begin{proof}
The statement for $\sigma^{-1}(\calP)$ is clear. Let us prove it for $\calP^{\sigma}$.

First, note that each non-empty $\ol{Z} = \sigma(Z_1) \cap \cdots \cap \sigma(Z_k) \backslash \calU_{\geq k+1} \in \calP^{\sigma}$ is clopen in $X$, since $\calU_{\geq k+1}$ is clopen and each $\sigma(Z_i)$ is clopen too.

Second, take $\ol{Z}_1 = \sigma(Z_1) \cap \cdots \cap \sigma(Z_k) \backslash \calU_{\geq k+1}$ and $\ol{Z}_2 = \sigma(Z'_1) \cap \cdots \cap \sigma(Z'_{k'}) \backslash \calU_{\geq k'+1}$ elements from $\calP^{\sigma}$, so that $(Z_1,\dots,Z_k)$ are all different, and so are $(Z'_1,\dots,Z'_{k'})$. If say $k > k'$, then $(Z_1,\dots,Z_k)$ constitute $k > k'$ different pairwise disjoint clopen sets, thus $\sigma(Z_1) \cap \cdots \cap \sigma(Z_k) \subseteq \calU_{\geq k'+1}$. But then
$$\ol{Z}_1 \cap \ol{Z}_2 = \sigma(Z_1) \cap \cdots \cap \sigma(Z_k) \cap \sigma(Z'_1) \cap \cdots \cap \sigma(Z'_{k'}) \backslash (\calU_{\geq k+1} \cup \calU_{\geq k'+1}) \subseteq \calU_{\geq k'+1} \backslash \calU_{\geq k'+1} = \emptyset.$$
The argument for $k' > k$ is similar, so we may assume that $k' = k$. But then
$$\ol{Z}_1 \cap \ol{Z}_2 = \sigma(Z_1) \cap \cdots \cap \sigma(Z_k) \cap \sigma(Z'_1) \cap \cdots \cap \sigma(Z'_k) \backslash \calU_{\geq k+1}.$$
If some $Z_i$ is different from all $Z'_j$'s, then $(Z_i,Z'_1,\dots,Z'_k)$ are all different, thus $\sigma(Z_i) \cap \sigma(Z'_1) \cap \cdots \cap \sigma(Z'_k) \subseteq \calU_{\geq k+1}$. This gives, as before, $\ol{Z}_1 \cap \ol{Z}_2 = \emptyset$. This proves that the elements of $\calP^{\sigma}$ are pairwise disjoint.

To conclude, simply note that $\{ \sigma(Z_1) \cap \cdots \cap \sigma(Z_k) \backslash \calU_{\geq k+1} \neq \emptyset \mid Z_i \in \calP \text{ all different}\}$ gives a partition of $\calU_k$, thus $\bigcup_{\ol{Z} \in \calP^{\sigma}} \ol{Z} = \calU_1 \sqcup \cdots \sqcup \calU_{N(\sigma)} = X$.

The rest of the lemma is straightforward.
\end{proof}

We now define the key concept of a $\sigma$-refined sequence of partitions.

\begin{definition}\label{definition-sigma.refined.partition}
We say that a sequence of partitions $\{\calP_n\}_{n \ge 0}$ of $X$ is a \textit{$\sigma$-refined sequence of partitions} of $X$ if the following properties hold:  
\begin{enumerate}[(a),leftmargin=0.7cm]
\item $\sigma$ restricts to a homeomorphism $\sigma |_Z$ for any $Z \in \calP_0$ (the existence of such a partition is guaranteed by Lemma \ref{lemma-properties.local.homeo}).
\item $\calP_n \precsim \calP_{n+1}$, that is $\calP_{n+1}$ is finer than $\calP_n$, for all $n \geq 0$.
\item The union $\bigcup_{n \geq 0} \calP_n$ generates the topology of $X$.
\item $\text{diam}(\calP_n)$ tends to $0$ as $n \ra \infty$.
\item We have
$$\calP_{2j} \vee \calP_{2j}^{\sigma} \precsim \calP_{2j+1}, \quad \text{and} \quad \calP_{2j+1} \vee \sigma^{-1}(\calP_{2j+1}) \precsim \calP_{2j+2}$$
for all $j \geq 0$.
\end{enumerate}	
Note that conditions (a) and (b) together imply that, for any $n \geq 0$, $\sigma$ restricts to a homeomorphism $\sigma |_Z$ for any $Z \in \calP_n$.
\end{definition}

\begin{remark}\label{rem:tame-defining-sequences}
Since $X$ is compact and all our partitions are finite, it is straightforward to check that any $\sigma$-refined sequence of partitions of $X$ in the sense of Definition \ref{definition-sigma.refined.partition} is a  complete, tame defining sequence in the sense of \cite[Definitions 3.1.1 and 3.1.8]{DGS2021}.
\end{remark}

\begin{lemma}\label{lemma-refining.partition.local}
Let $(X,\sigma) \in \emph{\textsf{LHomeo}}$. Then $X$ admits a $\sigma$-refined sequence of partitions.
\end{lemma}
\begin{proof}
Take a sequence of partitions $\calP'_n$ satisfying conditions (a)-(d) in Definition \ref{definition-sigma.refined.partition}. It always exists by our assumptions on the space $X$. Now define new partitions inductively by
\begin{align*}
& \calP_0 := \calP'_0; \\
& \calP_{2j+1} := \calP'_{2j+1} \vee \calP_{2j} \vee \calP_{2j}^{\sigma} & \text{ for all } j \geq 0; \\
& \calP_{2j+2} := \calP'_{2j+2} \vee \calP_{2j+1} \vee \sigma^{-1}(\calP_{2j+1}) & \text{ for all } j \geq 0.
\end{align*}
It is clear that the sequence $\{\calP_n\}_{n \geq 0}$ satisfies the required properties.
\end{proof}

Thus from now on we will assume that a $\sigma$-refined sequence of partitions $\{\calP_n\}_{n \geq 0}$ is given.\\

\begin{examples}\label{examples-cantor.1}
\begin{enumerate}[1., leftmargin=0.7cm]
\item Take $X = \{0,1\}^{\Z}$ as a model for the Cantor set, and define the two-sided shift
$$\varphi \colon X \to X, \quad \varphi(x)_i = x_{i+1}.$$
A basis for the topology of $X$ is again given by the cylinder sets: for integers $r,l \geq 0$ and $\varepsilon_{-l},...,\varepsilon_r \in \{0,1\}$, these are given by
$$[\varepsilon_{-l} \cdots \underline{\varepsilon_0} \cdots \varepsilon_r] := \{x \in X \mid x_i = \varepsilon_i \text{ for } -l \leq i \leq r\}.$$
With this topology, $(X,\varphi) \in \textsf{LHomeo}$, and in fact $\varphi$ is even a homeomorphism. The sequence of partitions $\{\calP_n\}_{n \geq 0}$, given, for $n \geq 0$, by
$$\calP_{2n} = \{[\varepsilon_{-n} \cdots \underline{\varepsilon_0} \cdots \varepsilon_n] \mid \varepsilon_i \in \{0,1\} \text{ for } -n \leq i \leq n\},$$
$$\calP_{2n+1} = \{[\varepsilon_{-n-1} \varepsilon_{-n} \cdots \underline{\varepsilon_0} \cdots \varepsilon_n ] \mid \varepsilon_i \in \{0,1\} \text{ for } -n-1 \leq i \leq n\},$$
is a $\varphi$-refined sequence of partitions of $X$.
\item Take $X^+ = \{0,1\}^{\N_0}$ as another model for the Cantor set, but consider now the one-sided shift
$$\sigma \colon X^+ \to X^+, \quad \sigma(x)_i = x_{i+1}.$$
A basis for the topology of $X^+$ is again given by the cylinder sets, where now $l = 0$. Even though $\sigma$ is not a homeomorphism, it is certainly a surjective local homeomorphism, so $(X^+,\sigma) \in \textsf{LHomeo}$. The sequence of partitions $\{\calP_n\}_{n \geq 0}$, given, for $n \geq 0$, by
$$\calP_n = \{[\underline{\varepsilon_0} \cdots \varepsilon_n] \mid \varepsilon_i \in \{0,1\} \text{ for } 0 \leq i \leq n \},$$
is a $\sigma$-refined sequence of partitions of $X^+$.
\end{enumerate}
\end{examples}

The main idea to construct the $l$-diagram $(F,D)$ from $(X,\sigma)$ is the following: the vertices will correspond to the sets $Z \in \calP_n$, with $\calP_n$ a $\sigma$-refined sequence of partitions; the blue edges will correspond to inclusion maps $Z \subseteq Z'$, and the red edges will correspond to the action of $\sigma$ or $\sigma^{-1}$, which in the latter case will give rise to various red edges departing from the same vertex.

\begin{construction}\label{construction-l.diag}
Take $\{\calP_n\}_{n \geq 0}$ to be any $\sigma$-refined sequence of partitions of $X$. We construct $(F,D)$ inductively, as follows.

\begin{enumerate}[(a),leftmargin=0.7cm]

\item (Vertices) For $n \geq 0$ we set $F^{0,n} = \calP_n$, so that the vertices at the $n^{\text{th}}$ layer are the clopen sets in the partition $\calP_n$.

\item (Edges) For $n \geq 0$, we let $F^{1,n} = B^{(n)} \sqcup R^{(n)}$, where the set $B^{(n)}$ will consist on the blue edges, and the set $R^{(n)}$ will be the set of red edges, which we now define. We also set, for $Z \in F^{0,n}$,
$$B_Z = \{ e \in B^{(n)} \mid r(e) = Z\}, \quad R_Z = \{ f \in R^{(n)} \mid r(f) = Z\}.$$

\item[(b.1)] (Blue edges) The blue edges are defined as follows. For any pair $Z \in \calP_n, Z' \in \calP_{n+1}$ such that $Z' \subseteq Z$, we write down a blue edge $e(Z',Z)$ with source $Z'$ and range $Z$. Thus 
$$B^{(n)} = \{ e(Z',Z) \mid Z \in \calP_n,Z'\in \calP_{n+1} \text{ such that } Z' \subseteq Z\}.$$

\item[(b.2)] (Red edges) For the definition of red edges, we need to distinguish between \textit{even} and \textit{odd} layers.

We start with an even layer $n = 2j, j \geq 0$. For each pair $Z \in \calP_{2j}$ and $Z' \in \calP_{2j+1}$ such that $Z' \subseteq \sigma(Z)$, we write down a red edge $f(Z',Z)$ with source $Z'$ and range $Z$. Thus
$$R^{(2j)} = \{ f(Z',Z) \mid Z \in \calP_{2j},Z'\in \calP_{2j+1} \text{ such that } Z' \subseteq \sigma(Z)\}.$$
Note that it may be the case that there are different red edges departing from the same vertex $Z'$, simply because $Z'$ is in principle contained in an intersection $\sigma(Z_1) \cap \cdots \cap \sigma(Z_l)$ for some $Z_i \in \calP_{2j}$.

Now we pay our attention to an odd layer $n = 2j+1, j \geq 0$. For $Z' \in \calP_{2j+1}$ and $Z'' \in \calP_{2j+2}$ such that $\sigma(Z'') \subseteq Z'$ (or equivalently $Z'' \subseteq \sigma^{-1}(Z')$), we write down a red edge $f(Z'',Z')$ with source $Z''$ and range $Z'$. Thus
$$R^{(2j+1)} = \{ f(Z'',Z') \mid Z' \in \calP_{2j+1},Z''\in \calP_{2j+2} \text{ such that } Z'' \subseteq \sigma^{-1}(Z')\}.$$
Note that here, in contrast to what happens for even layers, there is only one red edge departing from each vertex $Z'' \in \calP_{2j+2}$.

\item (Separations) We again distinguish between even and odd layers. For an even layer $n = 2j$, $j \geq 0$, the separation for $Z \in F^{0,2j}$ is defined to be
$$D_Z = \{B_Z,R_Z\}.$$
For an odd layer $n = 2j+1$, $j \geq 0$, the separation for $Z' \in F^{0,2j+1}$ is defined to be
$$D_{Z'} = \{B_{Z'}\} \cup \{R(f) \mid s(f) = Z', f \in R_{r(f)}\},$$
where here $R(f)$ denotes the set of all those red edges $g \in R_{Z'}$ such that $s(g) \subseteq r(f)$. Thus $R_{Z'} = \bigcup_{f \in s^{-1}(Z') \cap R_{r(f)}} R(f)$.
\end{enumerate}
\end{construction}

\begin{figure}[H]
\begin{tikzpicture}
	\draw[black,thick,dashed] (-5,0) -- (5,0);
	\draw[black,thick,dashed] (-5,2) -- (5,2);
	\draw[black,thick,dashed] (-5,4) -- (5,4);
	\node[circle,fill=black,scale=0.4] (B1) at (-5,0) {};
	\node[circle,fill=black,scale=0.4] (B2) at (-3,0) {};
	\node[circle,fill=black,scale=0.4] (B3) at (-2,0) {};
	\node[circle,fill=black,scale=0.4] (B4) at (-1,0) {};
	\node[circle,fill=black,scale=0.4] (B5) at (1,0) {};
	\node[circle,fill=black,scale=0.4] (B6) at (2,0) {};
	\node[circle,fill=black,scale=0.4] (B7) at (3,0) {};
	\node[circle,fill=black,scale=0.4] (B8) at (5,0) {};
	\node[circle,fill=black,scale=0.4] (M1) at (-4,2) {};
	\node[circle,fill=black,scale=0.4] (M2) at (-1,2) {};
	\node[circle,fill=black,scale=0.4] (M3) at (0,2) {};
	\node[circle,fill=black,scale=0.4] (M4) at (1,2) {};
	\node[circle,fill=black,scale=0.4] (M5) at (4,2) {};
	\node[circle,fill=black,scale=0.4] (T1) at (-2.5,4) {};
	\node[circle,fill=black,scale=0.4] (T2) at (0,4) {};
	\node[circle,fill=black,scale=0.4] (T3) at (2.5,4) {};
	\node[label=$Z \in F^{0,2j}$] (2N) at (-5,4) {};
	\node[label=$Z' \in F^{0,2j+1}$] (2N+1) at (-5,2) {};
	\node[label=$Z'' \in F^{0,2j+2}$] (2N+2) at (-5,0) {};
	%
	\draw[-,blue,postaction={on each segment={mid arrow=blue}}] (M1) to (T1);
	\draw[-,blue,postaction={on each segment={mid arrow=blue}}] (M2) to (T1);
	\draw[-,blue,postaction={on each segment={mid arrow=blue}}] (M3) to (T2);
	\draw[-,blue,postaction={on each segment={mid arrow=blue}}] (M4) to (T3);
	\draw[-,blue,postaction={on each segment={mid arrow=blue}}] (M5) to (T3);
	\draw[-,blue,postaction={on each segment={mid arrow=blue}}] (B1) to (M1);
	\draw[-,blue,postaction={on each segment={mid arrow=blue}}] (B2) to (M1);
	\draw[-,blue,postaction={on each segment={mid arrow=blue}}] (B3) to (M2);
	\draw[-,blue,postaction={on each segment={mid arrow=blue}}] (B4) to (M3);
	\draw[-,blue,postaction={on each segment={mid arrow=blue}}] (B5) to (M3);
	\draw[-,blue,postaction={on each segment={mid arrow=blue}}] (B6) to (M4);
	\draw[-,blue,postaction={on each segment={mid arrow=blue}}] (B7) to (M5);
	\draw[-,blue,postaction={on each segment={mid arrow=blue}}] (B8) to (M5);
	\draw[-,red,postaction={on each segment={mid arrow=red}}] (M3) to (T1);
	\draw[-,red,postaction={on each segment={mid arrow=red}}] (M5) to (T1);
	\draw[densely dotted,red,postaction={on each segment={mid arrow=red}}] (M1) to (T2);
	\draw[densely dotted,red,postaction={on each segment={mid arrow=red}}] (M2) to (T2);
	\draw[densely dotted,red,postaction={on each segment={mid arrow=red}}] (M4) to (T2);
	\draw[densely dotted,red,postaction={on each segment={mid arrow=red}}] (M5) to (T2);
	\draw[dashed,red,postaction={on each segment={mid arrow=red}}] (M1) to (T3);
	\draw[dashed,red,postaction={on each segment={mid arrow=red}}] (M3) to (T3);
	\draw[-,red,postaction={on each segment={mid arrow=red}}] (B1) to (M5);
	\draw[-,red,postaction={on each segment={mid arrow=red}}] (B2) to (M3);
	\draw[-,red,postaction={on each segment={mid arrow=red}}] (B3) to (M5);
	\draw[densely dotted,red,postaction={on each segment={mid arrow=red}}] (B4) to (M2);
	\draw[densely dotted,red,postaction={on each segment={mid arrow=red}}] (B5) to (M4);
	\draw[dashed,red,postaction={on each segment={mid arrow=red}}] (B6) to (M1);
	\draw[dashed,red,postaction={on each segment={mid arrow=red}}] (B7) to (M3);
	\draw[dashed,red,postaction={on each segment={mid arrow=red}}] (B8) to (M1);
  \draw[decorate,decoration={brace,amplitude=10pt,mirror}] 
    (5.2,0) -- (5.2,2);
	\node at (8,1.33){$Z'' \begingroup\color{blue} \ra \endgroup Z' \text{ if } Z'' \subseteq Z'$};
	\node at (8,0.67){$Z'' \begingroup\color{red} \ra \endgroup Z' \text{ if } \sigma(Z'') \subseteq Z'$};	
	\draw[decorate,decoration={brace,amplitude=10pt,mirror}] 
    (5.2,2) -- (5.2,4);
	\node at (8,3.33){$Z' \begingroup\color{blue} \ra \endgroup Z \text{ if } Z' \subseteq Z$};
	\node at (8,2.67){$Z' \begingroup\color{red} \ra \endgroup Z \text{ if } Z' \subseteq \sigma(Z)$};
\end{tikzpicture}
\caption{Construction of the edges of $(F,D)$.}
\label{figure-schematics3}
\end{figure}
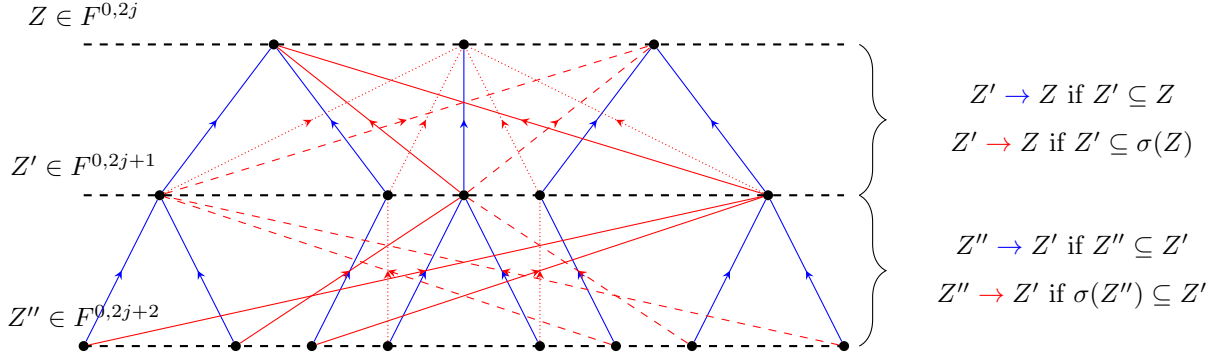

In this way we build a separated Bratteli diagram $(F,D)$. We now show that $(F,D)$ is an $l$-diagram. Indeed, it is an $l$-diagram of a special type, since by definition it does not have double red edges in any of its layers. This motivates the following definition.

\begin{definition}\label{definition-refined.Ldiagram}
Let $(F,D)$ be an $l$-diagram. We say that $(F,D)$ is a {\it refined $l$-diagram} if for any two vertices $w\in F^{0,n}$ and $v\in F^{0,n-1}$, $n\ge 1$, there is at most one red edge $f$ such that $s(f)= w$ and $r(f)=v$. By Remark \ref{remark-about.def.of.Ldiagram} (5), this is equivalent to asking the property only at the first layer, that is, for $n=1$.
\end{definition}

The following lemma will be useful in the sequel.

\begin{lemma}\label{lemma-group.of.f}
Let $\{ \calP_n \}_{n \geq 0}$ be a  $\sigma$-refined sequence of partitions of $X$.
\begin{enumerate}[i),leftmargin=0.7cm]
\item For $j\ge 0$, let $Z_1 \in \calP_{2j}$, and let $g \in R_{Z_1}$, with $s(g) = Z$. Then $Z\subseteq \sigma(Z_1)$, and we have
$$(\sigma|_{Z_1})^{-1}(Z) = \sigma^{-1}(Z)\cap Z_1 = \bigsqcup_{f\in R(g)} s(f).$$
\item For $j\ge 1$, let $Z \in \calP_{2j-1}$, $Z'' \in \calP_{2j+1}$ and $Z_1 \in \calP_{2j-2}$ be such that
$$Z'' \subseteq Z \subseteq \sigma(Z_1).$$
Then there is a unique $Z' \in \calP_{2j}$ such that
$$Z'\subseteq Z_1 \qquad \text{and} \qquad Z'' \subseteq \sigma(Z') \subseteq Z\subseteq \sigma(Z_1).$$
\end{enumerate}
\end{lemma}
\begin{proof}
\begin{enumerate}[i),leftmargin=0.6cm]
\item First note that $Z\subseteq \sigma(Z_1)$ by definition of the red edges at even layers.

Now note that since $\calP_{2j} \vee \sigma^{-1}(\calP _{2j+1}) \precsim \calP_{2j+1} \vee \sigma^{-1}(\calP _{2j+1}) \precsim \calP_{2j+2}$, we can write the set $\sigma^{-1}(Z)\cap Z_1$ as a disjoint union of sets $Z'' \in \calP_{2j+2}$. If $Z'' \in \calP_{2j+2}$ and $Z''\subseteq \sigma^{-1}(Z)\cap Z_1$, then there is a unique red edge $f$ such that $s(f)=Z''$ and $r(f)= Z$, and moreover $s(f)= Z''\subseteq Z_1 = r(g)$, so that $f\in R(g)$. Conversely any $f\in R(g)$ gives rise to a set $Z'':=s(f)\in \calP_{2j+2}$ with $Z''\subseteq \sigma^{-1}(Z)\cap Z_1$. This shows the result.

\item Take $j \geq 1$, $Z \in \calP_{2j-1}$, $Z'' \in \calP_{2j+1}$, and $Z_1 \in \calP_{2j-2}$ such that $Z'' \subseteq Z \subseteq \sigma(Z_1)$. Since $\calP_{2j}^{\sigma}$ is coarser than $\calP_{2j+1}$, we can find sets $Z'_1, Z'_2,\dots, Z'_k \in \calP_{2j}$ such that
$$Z'' \subseteq \sigma(Z'_1) \cap \cdots \cap \sigma(Z'_k) \setminus \calU_{\ge k+1}.$$
We have that $Z'' \subseteq \sigma(Z'_i)$ for $1 \leq i \leq k$ and $Z'' \cap \sigma(\ol{Z}) = \emptyset$ for all $\ol{Z} \in \calP_{2j}$ distinct from $Z'_1, \dots, Z'_k$. On the other hand note that, since $Z \subseteq \sigma(Z_1)$, there is a unique red edge $g$ such that $s(g) = Z$ and $r(g) = Z_1$.  By i), we have
$$(\sigma|_{Z_1})^{-1}(Z) = \sigma^{-1}(Z) \cap Z_1 =  \bigsqcup_{f \in R(g)} s(f).$$
Since $\sigma$ is injective on all sets of the partitions $\calP_n$, we obtain that
$$Z = \bigsqcup_{f \in R(g)} \sigma(s(f)).$$
Since $Z'' \subseteq Z$, it follows that $Z'' \cap \sigma(s(f)) \ne \emptyset$ for some $f \in R(g)$. But by the property mentioned above, this means that $s(f)$ is one of the sets $Z_1',\dots, Z_k'$ and thus $Z'' \subseteq \sigma(s(f))$. Since the sets $\sigma(s(f))$ are mutually disjoint for $f \in R(g)$, it follows that there is exactly one $f \in R(g)$ such that $Z'' \subseteq \sigma(s(f))$. Let $f_0$ be this unique edge in $R(g)$, and set $Z' := s(f_0)$. Since $f_0 \in R(g)$ we have $Z' = s(f_0) \subseteq r(g)= Z_1$ and moreover
$$Z'' \subseteq \sigma(s(f_0))= \sigma(Z') \subseteq Z \subseteq \sigma(Z_1),$$
as desired.
\end{enumerate}
\end{proof}

\begin{lemma}\label{lemma-FD.is.Ldiagram}
The separated graph $(F,D)$ associated to a $\sigma$-refined sequence of partitions $\{\calP_n\}_{n \geq 0}$ is a refined $l$-diagram.
\end{lemma}

\begin{proof} The condition of not having double red edges follows directly from the definition. We need to check conditions (a)-(g) from Definition \ref{definition-L.separated.Bratteli}. Conditions (a) and (b) follow by definition.
\begin{enumerate}[(a),start=3,leftmargin=0.6cm]
\item We first show that, for $j \geq 1$,
$$\bigsqcup_{Z \in \calP_{2j-1}} s(B_{Z}) \mathrel{\overset{\makebox[0pt]{\mbox{\normalfont\tiny\sffamily (c.1)}}}{=}} \calP_{2j} \mathrel{\overset{\makebox[0pt]{\mbox{\normalfont\tiny\sffamily (c.2)}}}{=}} \bigsqcup_{Z \in \calP_{2j-1}} s(R_{Z}).$$
We prove the non-trivial inclusion in {\normalfont\tiny\sffamily (c.1)}. Given $Z' \in \calP_{2j}$, we can find a (unique) $Z \in \calP_{2j-1}$ such that $Z' \subseteq Z$ since $\calP_{2j-1}$ is coarser than $\calP_{2j}$. Therefore there exists, by definition, a blue edge $e(Z',Z)$ with source $Z'$ and range $Z$. This proves {\normalfont\tiny\sffamily (c.1)}. Disjointness of the left union follows from the uniqueness of the set $Z$. Indeed, given $Z' \in s(B_{Z_1}) \cap s(B_{Z_2})$ with $Z_1 \neq Z_2 \in \calP_{2j-1}$, we have $Z' \subseteq Z_1 \cap Z_2 = \emptyset$, a contradiction.

Let us now prove the non-trivial inclusion in {\normalfont\tiny\sffamily (c.2)}. Since $\sigma^{-1}(\calP_{2j-1})$ is a partition coarser than $\calP_{2j}$, given $Z' \in \calP_{2j}$ we can find $Z \in \calP_{2j-1}$ such that $Z' \subseteq \sigma^{-1}(Z)$, equivalently $\sigma(Z') \subseteq Z$. This says that there is a red edge $f(Z',Z)$ with source $Z'$ and range $Z$. This proves the desired inclusion. The disjointness of the right union follows by analogous arguments as in the above paragraph.

The second part of condition (c) follows by the definition of the edges in $(F,D)$.

\item The fact that
$$\bigsqcup_{Z \in \calP_0} s(B_Z) = \calP_1 = \bigcup_{Z \in \calP_0} s(R_Z)$$
follows from identical arguments as in (c), using now the fact that the partition $\calP_0$ is coarser than the partition $\calP_1$. Disjointness of the left union is proved in a similar way as in (c).

The second part of condition (d) follows by the definition of the blue edges in $(F,D)$.

\item We need to show that, for $j \geq 0$ and $Z \in \calP_{2j}$,
$$\bigcup_{e \in B_Z} s(B_{s(e)}) = \bigcup_{f \in R_Z} s(R(f)).$$
Disjointness of the unions will follow from condition (c) (see Remark \ref{remark-about.def.of.Ldiagram} (3)).

Let $Z'' = s(e_1)$ for some $e_1 \in B_{s(e_0)}$ and some $e_0 \in B_Z$, so that $Z'' \subseteq s(e_0) \subseteq Z$. Since $\sigma^{-1}(\calP_{2j+1})$ is a coarser partition than $\calP_{2j+2}$ we can find $Z' \in \calP_{2j+1}$ such that $\sigma(Z'') \subseteq Z'$. Thus there exists a red edge $f_1 = f(Z'',Z') \in R_{Z'}$ such that $Z'' = s(f_1)$. Note now that $Z' \cap \sigma(Z) \neq \emptyset$, so by the construction of the partition $\calP_{2j}^{\sigma}$ and the fact that $\calP_{2j}^{\sigma}$ is coarser than $\calP_{2j+1}$, we necessarily have $Z' \subseteq \sigma(Z)$. Therefore there exists another red edge $f_0 = f(Z',Z) \in R_{Z}$ with $Z' = s(f_0)$. Since $s(f_1) = Z'' \subseteq Z = r(f_0)$, we have $f_1 \in R(f_0)$. This proves the inclusion $\subseteq$.

For the other inclusion, let $Z'' = s(f_1)$ for some $f_1 \in R(f_0)$ and some $f_0 \in R_Z$, so that $Z'' = s(f_1) \subseteq r(f_0) = Z$. Since $\calP_{2j+1}$ is coarser than $\calP_{2j+2}$, we can find $Z' \in \calP_{2j+1}$ with $Z'' \subseteq Z'$. But then $Z' \cap Z \neq \emptyset$, so we necessarily have that $Z' \subseteq Z$ by the fact that $\calP_{2j}$ is coarser than $\calP_{2j+1}$. This means precisely that there exist blue edges $e_0 \in B_Z$ with $Z' = s(e_0)$ and $e_1 \in B_{s(e_0)}$ with $Z'' = s(e_1)$. This concludes the proof of the desired inclusion.

\item Take $j \geq 1$, $Z \in \calP_{2j-1}$ and a pair $(f_0,f_1)$ of red edges such that $f_0 \in R_Z, f_1 \in R_{s(f_0)}$. By letting $Z'' := s(f_1)$, we have by definition $Z'' \subseteq \sigma(\ol{Z}') \subseteq Z$ where $\ol{Z}' := s(f_0) = r(f_1)$. Since $\calP_{2j}$ is coarser than $\calP_{2j+1}$, we can find $Z' \in \calP_{2j}$ with $Z'' \subseteq Z'$. But then $Z' \cap Z \neq \emptyset$, so necessarily $Z' \subseteq Z$ since $\calP_{2j-1}$ is coarser than $\calP_{2j}$. This means precisely that there exist blue edges $e_0 \in B_Z$ with $Z' = s(e_0)$ and $e_1 \in B_{s(e_0)}$ with $Z'' = s(e_1)$. The pair of blue edges $(e_0,e_1)$ is unique by Remark \ref{remark-about.def.of.Ldiagram} (4). This shows condition (f).

\item Take $j \geq 1$, $Z \in \calP_{2j-1}$, a red edge $g \in R_{r(g)}$ with $Z = s(g)$ and $Z_1 := r(g)$, and a pair $(e_0,e_1)$ of blue edges such that $e_0 \in B_Z$ and $s(e_0) = r(e_1)$. By letting $Z'' := s(e_1)$, we have by definition $Z'' \subseteq \ol{Z}' \subseteq Z$ where $\ol{Z}' := s(e_0) = r(e_1)$. We obtain the chain
$$Z'' \subseteq \ol{Z}' \subseteq Z \subseteq \sigma (Z_1).$$
By Lemma \ref{lemma-group.of.f} (ii), there is a unique $Z' \in \calP_{2j}$ such that $Z' \subseteq Z_1$ and
$$Z'' \subseteq \sigma(Z') \subseteq Z \subseteq \sigma(Z_1).$$
Let $f_0$ be the red edge such that $Z' = s(f_0)$ and $Z = r(f_0)$, and observe that $f_0 \in R(g)$ because $s(f_0) \subseteq r(g)$. We then have $Z'' \subseteq \sigma (s(f_0))$, and this means that there is a unique red edge $f_1$ such that $Z'' = s(f_1)$ and $Z' = r(f_1) = s(f_0)$. We have that $r(f_0) = Z = r(e_0)$ and $s(f_1) = Z'' = s(e_1)$. Therefore we have proved that there is a unique pair $(f_0,f_1)$ of red edges such that $f_0 \in R(g)$, $r(f_0) = r(e_0)$, $s(f_0) = r(f_1)$ and $s(f_1) = s(e_1)$, as required.

\end{enumerate}
This proves the lemma.
\end{proof}

\begin{examples}\label{examples-cantor.2}
This is a continuation of Examples \ref{examples-cantor.1}.

\begin{enumerate}[1., leftmargin=0.7cm]
\item Take $X = \{0,1\}^{\Z}$ and $\varphi \colon X \to X$ to be the two-sided shift. Using the $\varphi$-refined sequence of partitions $\{\calP_n\}_{n \geq 0}$ given in Examples \ref{examples-cantor.1} 1), we obtain the $l$-diagram shown in Figure \ref{figure-example.cantor.1}.

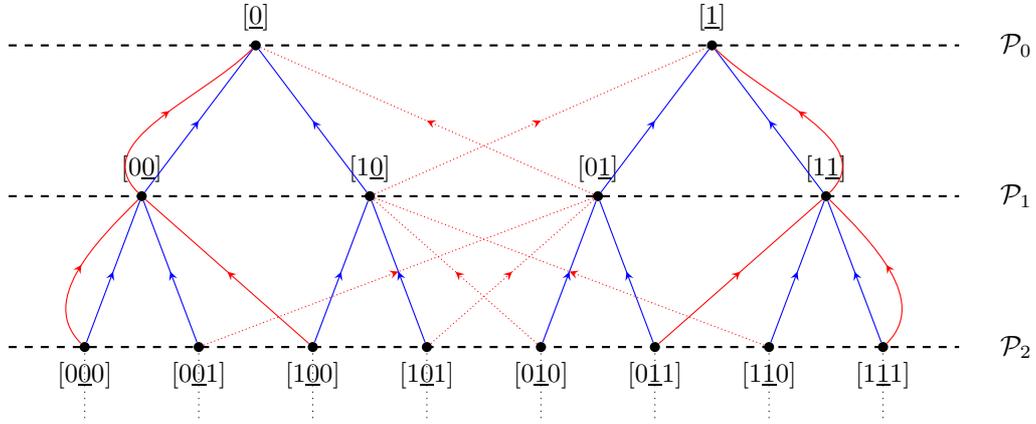
\begin{figure}[ht]
\begin{tikzpicture}
	\draw[black,thick,dashed] (-6.25,3) -- (6.25,3);
	\draw[black,thick,dashed] (-6.25,1) -- (6.25,1);
	\draw[black,thick,dashed] (-6.25,-1) -- (6.25,-1);
	\node[circle,fill=black,scale=0.4,label={[xshift=0.0cm, yshift=0.0cm]$[\underline{0}]$}] (0) at (-3,3) {};
	\node[circle,fill=black,scale=0.4,label={[xshift=0.0cm, yshift=0.0cm]$[\underline{1}]$}] (1) at (3,3) {};
	\node[circle,fill=black,scale=0.4,label={[xshift=0.0cm, yshift=0.0cm]$[0\underline{0}]$}] (00) at (-4.5,1) {};
	\node[circle,fill=black,scale=0.4,label={[xshift=0.0cm, yshift=0.0cm]$[1\underline{0}]$}] (01) at (-1.5,1) {};
	\node[circle,fill=black,scale=0.4,label={[xshift=0.0cm, yshift=0.0cm]$[0\underline{1}]$}] (10) at (1.5,1) {};
	\node[circle,fill=black,scale=0.4,label={[xshift=0.0cm, yshift=0.0cm]$[1\underline{1}]$}] (11) at (4.5,1) {};
	\node[circle,fill=black,scale=0.4,label=below:{$[0\underline{0}0]$}] (000) at (-5.25,-1) {};
	\node[circle,fill=black,scale=0.4,label=below:{$[0\underline{0}1]$}] (100) at (-3.75,-1) {};
	\node[circle,fill=black,scale=0.4,label=below:{$[1\underline{0}0]$}] (001) at (-2.25,-1) {};
	\node[circle,fill=black,scale=0.4,label=below:{$[1\underline{0}1]$}] (101) at (-0.75,-1) {};
	\node[circle,fill=black,scale=0.4,label=below:{$[0\underline{1}0]$}] (010) at (0.75,-1) {};
	\node[circle,fill=black,scale=0.4,label=below:{$[0\underline{1}1]$}] (110) at (2.25,-1) {};
	\node[circle,fill=black,scale=0.4,label=below:{$[1\underline{1}0]$}] (011) at (3.75,-1) {};
	\node[circle,fill=black,scale=0.4,label=below:{$[1\underline{1}1]$}] (111) at (5.25,-1) {};
	\draw[-,blue,postaction={on each segment={mid arrow=blue}}] (00) to (0);
	\draw[-,blue,postaction={on each segment={mid arrow=blue}}] (01) to (0);
	\draw[-,blue,postaction={on each segment={mid arrow=blue}}] (10) to (1);
	\draw[-,blue,postaction={on each segment={mid arrow=blue}}] (11) to (1);
	\draw[-,blue,postaction={on each segment={mid arrow=blue}}] (000) to (00);
	\draw[-,blue,postaction={on each segment={mid arrow=blue}}] (100) to (00);
	\draw[-,blue,postaction={on each segment={mid arrow=blue}}] (001) to (01);
	\draw[-,blue,postaction={on each segment={mid arrow=blue}}] (101) to (01);
	\draw[-,blue,postaction={on each segment={mid arrow=blue}}] (010) to (10);
	\draw[-,blue,postaction={on each segment={mid arrow=blue}}] (110) to (10);
	\draw[-,blue,postaction={on each segment={mid arrow=blue}}] (011) to (11);
	\draw[-,blue,postaction={on each segment={mid arrow=blue}}] (111) to (11);
	\draw[-,red,postaction={on each segment={mid arrow=red}}] (00) to [out=135,in=225] (0);
	\draw[densely dotted,red,postaction={on each segment={mid arrow=red}}] (10) to (0);
	\draw[densely dotted,red,postaction={on each segment={mid arrow=red}}] (01) to (1);
	\draw[-,red,postaction={on each segment={mid arrow=red}}] (11) to [out=45,in=315] (1);
	\draw[-,red,postaction={on each segment={mid arrow=red}}] (000) to [out=135,in=225] (00);
	\draw[-,red,postaction={on each segment={mid arrow=red}}] (001) to (00);
	\draw[densely dotted,red,postaction={on each segment={mid arrow=red}}] (010) to (01);
	\draw[densely dotted,red,postaction={on each segment={mid arrow=red}}] (011) to (01);
	\draw[densely dotted,red,postaction={on each segment={mid arrow=red}}] (100) to (10);
	\draw[densely dotted,red,postaction={on each segment={mid arrow=red}}] (101) to (10);
	\draw[-,red,postaction={on each segment={mid arrow=red}}] (110) to (11);
	\draw[-,red,postaction={on each segment={mid arrow=red}}] (111) to [out=45,in=315] (11);
	\node at (7,3){$\calP_0$};
	\node at (7,1){$\calP_1$};
	\node at (7,-1){$\calP_2$};
	\draw[black,dotted] (-5.25,-1) -- (-5.25,-2);
	\draw[black,dotted] (-3.75,-1) -- (-3.75,-2);
	\draw[black,dotted] (-2.25,-1) -- (-2.25,-2);
	\draw[black,dotted] (-0.75,-1) -- (-0.75,-2);
	\draw[black,dotted] (0.75,-1) -- (0.75,-2);
	\draw[black,dotted] (2.25,-1) -- (2.25,-2);
	\draw[black,dotted] (3.75,-1) -- (3.75,-2);
	\draw[black,dotted] (5.25,-1) -- (5.25,-2);
\end{tikzpicture}
\caption{First two levels of the $l$-diagram associated with $(X,\varphi)$.}
\label{figure-example.cantor.1}
\end{figure}

\item Take $X^+ = \{0,1\}^{\N_0}$ and $\sigma \colon X^+ \to X^+$ to be the one-sided shift. Using the $\sigma$-refined sequence of partitions $\{\calP_n\}_{n \geq 0}$ given in Examples \ref{examples-cantor.1} 2), we obtain the $l$-diagram shown in Figure \ref{figure-example.cantor.2}.

\begin{figure}[ht]
\begin{tikzpicture}
	\draw[black,thick,dashed] (-6.25,3) -- (6.25,3);
	\draw[black,thick,dashed] (-6.25,1) -- (6.25,1);
	\draw[black,thick,dashed] (-6.25,-1) -- (6.25,-1);
	\node[circle,fill=black,scale=0.4,label={[xshift=0.0cm, yshift=0.0cm]$[\underline{0}]$}] (0) at (-3,3) {};
	\node[circle,fill=black,scale=0.4,label={[xshift=0.0cm, yshift=0.0cm]$[\underline{1}]$}] (1) at (3,3) {};
	\node[circle,fill=black,scale=0.4,label={[xshift=0.0cm, yshift=0.0cm]$[\underline{0}0]$}] (00) at (-4.5,1) {};
	\node[circle,fill=black,scale=0.4,label={[xshift=0.0cm, yshift=0.0cm]$[\underline{0}1]$}] (01) at (-1.5,1) {};
	\node[circle,fill=black,scale=0.4,label={[xshift=0.0cm, yshift=0.0cm]$[\underline{1}0]$}] (10) at (1.5,1) {};
	\node[circle,fill=black,scale=0.4,label={[xshift=0.0cm, yshift=0.0cm]$[\underline{1}1]$}] (11) at (4.5,1) {};
	\node[circle,fill=black,scale=0.4,label=below:{$[\underline{0}00]$}] (000) at (-5.25,-1) {};
	\node[circle,fill=black,scale=0.4,label=below:{$[\underline{0}01]$}] (001) at (-3.75,-1) {};
	\node[circle,fill=black,scale=0.4,label=below:{$[\underline{0}10]$}] (010) at (-2.25,-1) {};
	\node[circle,fill=black,scale=0.4,label=below:{$[\underline{0}11]$}] (011) at (-0.75,-1) {};
	\node[circle,fill=black,scale=0.4,label=below:{$[\underline{1}00]$}] (100) at (0.75,-1) {};
	\node[circle,fill=black,scale=0.4,label=below:{$[\underline{1}01]$}] (101) at (2.25,-1) {};
	\node[circle,fill=black,scale=0.4,label=below:{$[\underline{1}10]$}] (110) at (3.75,-1) {};
	\node[circle,fill=black,scale=0.4,label=below:{$[\underline{1}11]$}] (111) at (5.25,-1) {};
	\draw[-,blue,postaction={on each segment={mid arrow=blue}}] (00) to (0);
	\draw[-,blue,postaction={on each segment={mid arrow=blue}}] (01) to (0);
	\draw[-,blue,postaction={on each segment={mid arrow=blue}}] (10) to (1);
	\draw[-,blue,postaction={on each segment={mid arrow=blue}}] (11) to (1);
	\draw[-,blue,postaction={on each segment={mid arrow=blue}}] (000) to (00);
	\draw[-,blue,postaction={on each segment={mid arrow=blue}}] (001) to (00);
	\draw[-,blue,postaction={on each segment={mid arrow=blue}}] (010) to (01);
	\draw[-,blue,postaction={on each segment={mid arrow=blue}}] (011) to (01);
	\draw[-,blue,postaction={on each segment={mid arrow=blue}}] (100) to (10);
	\draw[-,blue,postaction={on each segment={mid arrow=blue}}] (101) to (10);
	\draw[-,blue,postaction={on each segment={mid arrow=blue}}] (110) to (11);
	\draw[-,blue,postaction={on each segment={mid arrow=blue}}] (111) to (11);
	\draw[-,red,postaction={on each segment={mid arrow=red}}] (00) to [out=135,in=225] (0);
	\draw[densely dotted,red,postaction={on each segment={mid arrow=red}}] (00) to (1);
	\draw[-,red,postaction={on each segment={mid arrow=red}}] (01) to [out=160,in=270] (0);
	\draw[densely dotted,red,postaction={on each segment={mid arrow=red}}] (01) to (1);
	\draw[-,red,postaction={on each segment={mid arrow=red}}] (10) to (0);
	\draw[densely dotted,red,postaction={on each segment={mid arrow=red}}] (10) to [out=20,in=270] (1);
	\draw[-,red,postaction={on each segment={mid arrow=red}}] (11) to (0);
	\draw[densely dotted,red,postaction={on each segment={mid arrow=red}}] (11) to [out=45,in=315] (1);
	\draw[-,red,postaction={on each segment={mid arrow=red}}] (000) to [out=135,in=225] (00);
	\draw[-,red,postaction={on each segment={mid arrow=red}}] (001) to (01);
	\draw[-,red,postaction={on each segment={mid arrow=red}}] (010) to (10);
	\draw[-,red,postaction={on each segment={mid arrow=red}}] (011) to (11);
	\draw[densely dotted,red,postaction={on each segment={mid arrow=red}}] (100) to (00);	
	\draw[densely dotted,red,postaction={on each segment={mid arrow=red}}] (101) to (01);
	\draw[densely dotted,red,postaction={on each segment={mid arrow=red}}] (110) to (10);
	\draw[densely dotted,red,postaction={on each segment={mid arrow=red}}] (111) to [out=45,in=360-45] (11);
	\node at (7,3){$\calP_0$};
	\node at (7,1){$\calP_1$};
	\node at (7,-1){$\calP_2$};
	\draw[black,dotted] (-5.25,-1) -- (-5.25,-2);
	\draw[black,dotted] (-3.75,-1) -- (-3.75,-2);
	\draw[black,dotted] (-2.25,-1) -- (-2.25,-2);
	\draw[black,dotted] (-0.75,-1) -- (-0.75,-2);
	\draw[black,dotted] (0.75,-1) -- (0.75,-2);
	\draw[black,dotted] (2.25,-1) -- (2.25,-2);
	\draw[black,dotted] (3.75,-1) -- (3.75,-2);
	\draw[black,dotted] (5.25,-1) -- (5.25,-2);
\end{tikzpicture}
\caption{First two levels of the $l$-diagram associated with $(X^+,\sigma)$.}
\label{figure-example.cantor.2}
\end{figure}
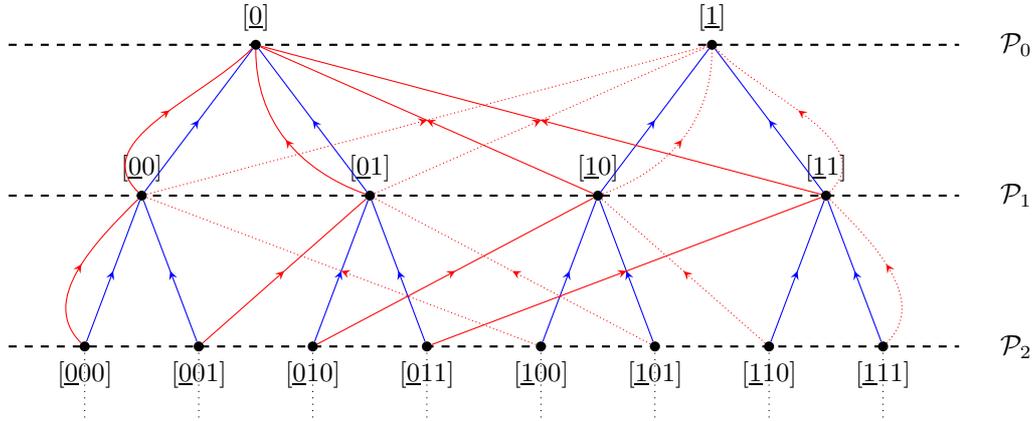

\end{enumerate}
\end{examples}

We now proceed by showing that the equivalence class, under $\sim_l$, of the $l$-diagram constructed from $(X,\sigma)$ does not depend on the choice of the $\sigma$-refined sequence of partitions $\{\calP_n\}_{n \geq 0}$. First, a proposition.

\begin{proposition}\label{proposition-iso.telescoping.Ldiagram}
Let $\{\calP_n\}_{n \geq 0}$ be a $\sigma$-refined sequence of partitions of $X$. Let $(F,D)$ be the $l$-diagram obtained by applying the construction \ref{construction-l.diag}. Let also $(m_n)_{n \geq 0}$ be a sequence in $\N \cup \{0\}$ with $0 \leq m_0 < m_1 < \cdots$, $m_0$ even and satisfying \emph{(CR)}, and let
\begin{enumerate}[$(I)$,leftmargin=0.7cm]
\item $(F',D')$ be the $l$-diagram obtained by telescoping $(F,D)$ with respect to $(m_n)_{n \geq 0}$;
\item $(F'',D'')$ be the $l$-diagram obtained from the $\sigma$-refined sequence of partitions $\{\calP_{m_n}\}_{n \geq 0}$ by applying the construction \ref{construction-l.diag}.
\end{enumerate}
Then $(F',D') \cong (F'',D'')$ as $l$-diagrams.
\end{proposition}
\begin{proof} First of all, note that $(F,D)$ and $(F',D')$ make sense by Lemmas \ref{lemma-FD.is.Ldiagram} and \ref{lemma-F'D'.L.diagram} respectively. Also $(F'',D'')$ makes sense by Lemma \ref{lemma-FD.is.Ldiagram} since the sequence $(m_n)_{n \geq 0}$ is strictly increasing, and thus $\{\calP_{m_n}\}_{n \geq 0}$ is indeed a $\sigma$-refined sequence of partitions.

Clearly the vertex sets for $(F',D')$ and $(F'',D'')$ are the same, since $(F')^{0,n} = F^{0,m_n} = \calP_{m_n} = (F'')^{0,n}$. We now define a bijective correspondence between $(F')^{1,n} = P^B_{m_n,m_{n+1}} \sqcup P^R_{m_n,m_{n+1}}$ and $(F'')^{1,n} = B''^{(n)} \sqcup R''^{(n)}$. 

\begin{enumerate}[(i),leftmargin=0.6cm]
\item Take $e = (e_{m_n},\dots,e_{m_{n+1}-1}) \in P^B_{m_n,m_{n+1}}$. Put $Z = r(e) = r(e_{m_n}) \in \calP_{m_n}$ and $Z' = s(e) = s(e_{m_{n+1}-1}) \in \calP_{m_{n+1}}$. By construction, there exists a sequence $(Z_0,Z_1,\dots,Z_{m_{n+1}-m_n})$ of clopen sets, with $Z_i \in \calP_{m_{n+1}-i}$, satisfying
$$Z' = Z_0 \subseteq Z_1 \subseteq \cdots \subseteq Z_{m_{n+1}-m_n-1} \subseteq Z_{m_{n+1}-m_n} = Z,$$
and also $s(e_{m_{n+1}-i}) = Z_{i-1}$, $r(e_{m_{n+1}-i}) = Z_i$ for each $1 \leq i \leq m_{n+1}-m_n$. Therefore $Z' \subseteq Z$, so $e$ corresponds to a \textit{unique} blue edge $\ol{e} \in (F'')^{1,n}$ with source $Z'$ and range $Z$. This defines an injective map
$$P_{m_n,m_{n+1}}^B \ra B''^{(n)}, \quad e \mapsto \ol{e}.$$
To show bijectivity of this map, take $\ol{e} \in B''^{(n)}$ with source $Z' \in \calP_{m_{n+1}}$ and range $Z \in \calP_{m_n}$. Since $\calP_m$ is coarser than $\calP_{m+1}$ for all $m \geq 0$, there exists a \textit{unique} sequence $(Z_0,\dots,Z_{m_{n+1}-m_n})$ of clopen sets, with $Z_i \in \calP_{m_{n+1}-i}$, such that
$$Z' = Z_0 \subseteq Z_1 \subseteq \cdots \subseteq Z_{m_{n+1}-m_n-1} \subseteq Z_{m_{n+1}-m_n} = Z.$$
This corresponds to a \textit{unique} blue path $e = (e_{m_n},\dots,e_{m_{n+1}-1})$ with $r(e_{m_{n+1}-i}) = Z_i = s(e_{m_{n+1}-i-1})$ for all $1 \leq i \leq m_{n+1}-m_n-1$, and $Z = r(e), Z' = s(e)$. Therefore $e \in P^B_{m_n,m_{n+1}}$. It is clear that $e$ maps to $\ol{e}$ under the above map $P^B_{m_n,m_{n+1}} \ra B''^{(n)}$. This proves the desired bijectivity.

\item We now deal with the red edges for the even layers, so $n = 2j$ for $j \geq 0$. Note that in this case $m_n$ is also even. Take $f = (f_{m_{2j}},\dots,f_{m_{2j+1}-1}) \in P^R_{m_{2j},m_{2j+1}}$. Put $Z = r(f) = r(f_{m_{2j}}) \in \calP_{m_{2j}}$ and $Z' = s(f) = s(f_{m_{2j+1}-1}) \in \calP_{m_{2j+1}}$. By construction, there exists a sequence $(Z_0,Z_1,\dots,Z_{m_{2j+1}-m_{2j}})$ of clopen sets, with $Z_i \in \calP_{m_{2j+1}-i}$, satisfying
$$Z' = Z_0 \subseteq \sigma(Z_1) \subseteq Z_2 \subseteq \cdots \subseteq Z_{m_{2j+1}-m_{2j}-1} \subseteq \sigma(Z_{m_{2j+1}-m_{2j}}) = \sigma(Z),$$
$s(f_{m_{2j+1}-i}) = Z_{i-1}$ and $r(f_{m_{2j+1}-i}) = Z_i$ for each $1 \leq i \leq m_{2j+1}-m_{2j}$, and also
$$Z_1 \subseteq Z_3 \subseteq \cdots \subseteq Z_{m_{2j+1}-m_{2j}-2} \subseteq Z_{m_{2j+1}-m_{2j}} = Z.$$
Therefore $Z' \subseteq \sigma(Z)$, so $f$ corresponds to a \textit{unique} red edge $\ol{f} \in (F'')^{1,2j}$ with source $Z'$ and range $Z$. This defines an injective map
$$P_{m_{2j},m_{2j+1}}^R \ra R''^{(2j)}, \quad f \mapsto \ol{f}.$$
We now proceed to show its bijectivity. Take $\ol{f} \in R''^{(2j)}$ with source $Z' \in \calP_{m_{2j+1}}$ and range $Z \in \calP_{m_{2j}}$.

\begin{claim*}
There exists a \textit{unique} sequence $(Z_0,\dots,Z_{m_{2j+1}-m_{2j}})$ of clopen sets $Z_i \in \calP_{m_{2j+1}-i}$ such that
$$Z' = Z_0 \subseteq \sigma(Z_1) \subseteq Z_2 \subseteq \cdots \subseteq Z_{m_{2j+1}-m_{2j}-1} \subseteq \sigma(Z_{m_{2j+1}-m_{2j}}) = \sigma(Z)$$
and
$$Z_1 \subseteq Z_3 \subseteq \cdots \subseteq Z_{m_{2j+1}-m_{2j}-2} \subseteq Z_{m_{2j+1}-m_{2j}} = Z.$$
\end{claim*}
\begin{proof}
First we set $Z_0:= Z'$ and $Z_{m_{2j+1}-m_{2j}}:=Z$. Now since the partition $\calP_{m}$ is coarser than the partition $\calP_{m+1}$ for all $m \geq 0$, we see that there are unique $Z_2 \in \calP_{m_{2j+1}-2}$, $Z_4 \in \calP_{m_{2j+1}-4}, \dots, Z_{m_{2j+1}-m_{2j}-1} \in \calP_{m_{2j}+1}$ with 
$$Z' = Z_{0}\subseteq Z_2 \subseteq Z_4 \subseteq \cdots \subseteq Z_{m_{2j+1}-m_{2j}-1}.$$
Since $Z' \subseteq \sigma(Z)$ and $\calP_{m_{2j}}^{\sigma}$ is coarser than $\calP_{m_{2j}+1}$, it follows that $Z_{m_{2j+1}-m_{2j}-1} \subseteq \sigma(Z)$. Therefore we have the inclusions
$$Z_{m_{2j+1}-m_{2j}-3} \subseteq Z_{m_{2j+1}-m_{2j}-1} \subseteq \sigma(Z).$$
By Lemma \ref{lemma-group.of.f} (ii), there is a unique $Z_{m_{2j+1}-m_{2j}-2} \in \calP_{m_{2j}+2}$ such that $Z_{m_{2j+1}-m_{2j}-2} \subseteq Z$ and
$$Z_{m_{2j+1}-m_{2j}-3} \subseteq \sigma(Z_{m_{2j+1}-m_{2j}-2}) \subseteq Z_{m_{2j+1}-m_{2j}-1} \subseteq \sigma(Z_{m_{2j+1}-m_{2j}}) = \sigma(Z).$$
Iteration of this argument produces the desired sequence $(Z_0,\dots,Z_{m_{2j+1}-m_{2j}})$. This proves the claim.
\end{proof}

The sequence just constructed corresponds to a \textit{unique} red path $f = (f_{m_{2j}},\dots,f_{m_{2j+1}-1})$ with $r(f_{m_{2j+1}-i}) = Z_i = s(f_{m_{2j+1}-i-1})$ for all $1 \leq i \leq m_{2j+1} - m_{2j} - 1$, and $Z = r(f), Z' = s(f)$. Therefore $f \in P^R_{m_{2j},m_{2j+1}}$. It is clear that $f$ maps to $\ol{f}$ under the above map $P_{m_{2j},m_{2j+1}}^R \ra R''^{(2j)}$. Bijectivity now follows.

\item Now for the red edges at the odd layers, so $n = 2j-1$ for $j \geq 1$. Note that in this case $m_n$ is also odd. Take $f = (f_{m_{2j-1}},\dots,f_{m_{2j}-1}) \in P^R_{m_{2j-1},m_{2j}}$. Put $Z = r(f) = r(f_{m_{2j-1}}) \in \calP_{m_{2j-1}}$ and $Z' = s(f) = s(f_{m_{2j}-1}) \in \calP_{m_{2j}}$. Take also $g = (g_{m_{2j-2}},\dots,g_{m_{2j-1}-1}) \in P^R_{m_{2j-2},m_{2j-1}}$ to be such that $f \in R'(g)$. By construction, there exists a sequence $(Z_0,Z_1,\dots,Z_{m_{2j}-m_{2j-1}})$ of clopen sets, with $Z_i \in \calP_{m_{2j}-i}$, satisfying
$$\sigma(Z') = \sigma(Z_0) \subseteq Z_1 \subseteq \sigma(Z_2) \subseteq \cdots \subseteq \sigma(Z_{m_{2j}-m_{2j-1}-1}) \subseteq Z_{m_{2j}-m_{2j-1}} = Z,$$
$s(f_{m_{2j}-i}) = Z_{i-1}$ and $r(f_{m_{2j}-i}) = Z_i$ for each $1 \leq i \leq m_{2j}-m_{2j-1}$, and also
$$Z' = Z_0 \subseteq Z_2 \subseteq Z_4 \subseteq \cdots \subseteq Z_{m_{2j}-m_{2j-1}-1}.$$
Therefore $\sigma(Z') \subseteq Z$, so $f$ corresponds to a \textit{unique} red edge $\ol{f} \in (F'')^{1,2j-1}$ with source $Z'$ and range $Z$. This defines an injective map
$$P_{m_{2j-1},m_{2j}}^R \ra R''^{(2j-1)}, \quad f \mapsto \ol{f}.$$
Moreover, since $g \in P^R_{m_{2j-2},m_{2j-1}}$, there exists a sequence $(\wt{Z}_0, \wt{Z}_1,\dots , \wt{Z}_{m_{2j-1}-m_{2j-2}})$ of clopen sets $\wt{Z}_i \in \calP_{m_{2j-1}-i}$ such that
$$Z = \wt{Z}_0 \subseteq \sigma(\wt{Z}_1) \subseteq \wt{Z}_2 \subseteq \cdots \subseteq \wt{Z}_{m_{2j-1}-m_{2j-2}-1} \subseteq \sigma(\wt{Z}_{m_{2j-1}-m_{2j-2}}) = \sigma(\wt{Z})$$
and
$$\wt{Z}_1 \subseteq \wt{Z}_3 \subseteq \cdots \subseteq \wt{Z}_{m_{2j-1}-m_{2j-2}} = \wt{Z},$$
where $Z = s(g)$, $\wt{Z} = r(g)$, and $s(g_{m_{2j-1}-i}) = \wt{Z}_{i-1}$, $r(g_{m_{2j-1}-i}) = \wt{Z}_i$ for each $1 \leq i \leq m_{2j-1}-m_{2j-2}$. Since $f \in R'(g)$, it follows by the definition of $R'(g)$ (Definition \ref{definition-equiv.telesc.l.diagrams} (d)) that $f_{m_{2j-1}} \in R(g_{m_{2j-1}-1})$, and so
$$Z_{m_{2j}-m_{2j-1}-1} = s(f_{m_{2j-1}}) \subseteq r(g_{m_{2j-1}-1}) = \wt{Z}_1.$$
Combining this with the previous inclusions we obtain
$$s(\ol{f}) = Z' =Z_0 \subseteq Z_2 \subseteq \cdots \subseteq Z_{m_{2j}-m_{2j-1}-1} \subseteq \wt{Z}_{1} \subseteq \cdots \subseteq \wt{Z} = r(\ol{g}),$$
where $\ol{g}$ is the unique red edge in $(F'')^{1,2j-2}$ constructed from $g$ (see (ii) above). This shows that $\ol{f}\in R''(\ol{g})$ in the $l$-diagram $(F'',D'')$.

We now show bijectivity of the map $P_{m_{2j-1},m_{2j}}^R \ra R''^{(2j-1)}$. Take $\ol{f} \in R''^{(2j-1)}$ with source $Z' \in \calP_{m_{2j}}$ and range $Z \in \calP_{m_{2j-1}}$. Take also $\ol{g} \in R''^{(2j-2)}$ such that $\ol{f} \in R''(\ol{g})$. Since $m_{2j-2}$ is even, we know from (ii) above that there is a unique sequence $(\wt{Z}_0, \wt{Z}_1,\dots , \wt{Z}_{m_{2j-1}-m_{2j-2}})$ of clopen sets $\wt{Z}_i \in \calP_{m_{2j-1}-i}$ such that
$$Z = \wt{Z}_0 \subseteq \sigma(\wt{Z}_1) \subseteq \wt{Z}_2 \subseteq \cdots \subseteq \wt{Z}_{m_{2j-1}-m_{2j-2}-1} \subseteq \sigma(\wt{Z}_{m_{2j-1}-m_{2j-2}}) = \sigma(\wt{Z})$$
and
$$\wt{Z}_1 \subseteq \wt{Z}_3 \subseteq \cdots \subseteq \wt{Z}_{m_{2j-1}-m_{2j-2}} = \wt{Z},$$
where $Z = s(\ol{g})$, $\wt{Z} = r(\ol{g})$. Let $g \in P^R_{m_{2j-2},m_{2j-1}}$ be the corresponding red edge.

\begin{claim*}
There exists a \textit{unique} sequence $(Z_0,\dots,Z_{m_{2j}-m_{2j-1}})$ of clopen sets $Z_i \in \calP_{m_{2j}-i}$ such that
$$\sigma(Z') = \sigma(Z_0) \subseteq Z_1 \subseteq \sigma(Z_2) \subseteq \cdots \subseteq \sigma(Z_{m_{2j}-m_{2j-1}-1}) \subseteq Z_{m_{2j}-m_{2j-1}} = Z$$
and
$$Z' = Z_0 \subseteq Z_2 \subseteq Z_4 \subseteq \cdots \subseteq Z_{m_{2j}-m_{2j-1}-1} \subseteq \wt{Z}_1.$$
\end{claim*}
\begin{proof}
Note that the hypothesis that $\ol{f}\in R''(\ol{g})$ gives us the condition
$$Z' = s(\ol{f})\subseteq r(\ol{g}) = \wt{Z}.$$
Since $\sigma^{-1} (\calP_{m_{2j}-1}) $ is coarser than $\calP_{m_{2j}}$, there is a unique $Z_1 \in \calP_{m_{2j}-1}$ such that $\sigma (Z') \subseteq Z_1$. We can now find $Z_{3},Z_5,\dots , Z_{m_{2j}-m_{2j-1}}$ such that $Z_{2l+1} \in \calP_{m_{2j}-(2l+1)}$ and
$$Z_1 \subseteq Z_3 \subseteq Z_5 \subseteq \cdots \subseteq Z_{m_{2j}-m_{2j-1}}.$$
Observe that $\sigma(Z') \subseteq Z_1 \subseteq Z_{m_{2j}-m_{2j-1}}$ and also $\sigma(Z') \subseteq Z$ by hypothesis. Since $\calP_{m_{2j-1}}$ is a partition, we must have $Z_{m_{2j}-m_{2j-1}}= Z$.

Now note that we have inclusions
$$Z_{m_{2j}-m_{2j-1}-2} \subseteq Z \subseteq \sigma(\wt{Z}_1).$$
By Lemma \ref{lemma-group.of.f} (ii), we get a unique $Z_{m_{2j}-m_{2j-1}-1} \in \calP_{m_{2j-1}+1}$ such that $Z_{m_{2j}-m_{2j-1}-1} \subseteq \wt{Z}_1$ and
$$Z_{m_{2j}-m_{2j-1}-2} \subseteq \sigma(Z_{{m_2j}-m_{2j-1}-1}) \subseteq Z \subseteq \sigma(\wt{Z}_1).$$
Continuing in this way, we find unique clopen sets $Z_{2}, Z_4, \dots , Z_{m_{2j}-m_{2j-1}-3}$ such that	$Z_{2l} \in \calP_{m_{2j} -2l}$ for $1 \leq l \leq \frac{m_{2j}-m_{2j-1}-3}{2}$ and
$$Z_{2} \subseteq Z_4 \subseteq \cdots \subseteq Z_{m_{2j}-m_{2j-1}-3} \subseteq Z_{m_{2j}-m_{2j-1}-1},$$
$$Z_1 \subseteq \sigma(Z_2) \subseteq Z_3 \subseteq \cdots \subseteq \sigma(Z_{m_{2j}-m_{2j-1}-1}) \subseteq Z \subseteq \sigma(\wt{Z}_1).$$

We show that $Z' \subseteq Z_2$. Indeed, we have $\sigma(Z') \subseteq Z_1 \subseteq \sigma(Z_2)$. Moreover we have that $Z' \subseteq \wt{Z}$ by our hypothesis that $\ol{f} \in R''(\ol{g})$, and also we have the inclusions
$$Z_2 \subseteq Z_4 \subseteq \cdots \subseteq Z_{m_{2j}-m_{2j-1}-1} \subseteq \wt{Z}_1 \subseteq \wt{Z}_3 \subseteq \cdots \subseteq \wt{Z}_{m_{2j-1}-m_{2j-2}} = \wt{Z}.$$
Hence both $Z'$ and $Z_2$ are subsets of $\wt{Z}$. Since the restriction of $\sigma$ to $\wt{Z}$ is injective and $\sigma(Z') \subseteq \sigma(Z_2)$, we get that $Z'\subseteq Z_2$. Hence taking $Z_0:= Z'$ we complete our proof of the existence and uniqueness of the sequence $(Z_0,\dots , Z_{m_{2j}-m_{2j-1}})$, thus proving the claim.
\end{proof}

The sequence just constructed corresponds to a \textit{unique} red path $f = (f_{m_{2j-1}},\dots,f_{m_{2j}-1})$ with $r(f_{m_{2j}-i}) = Z_i = s(f_{m_{2j}-i-1})$ for all $1 \leq i \leq m_{2j}-m_{2j-1}-1$, and $Z = r(f), Z' = s(f)$. Therefore $f \in P^R_{m_{2j-1},m_{2j}}$, and moreover $f \in R'(g)$ because $Z_{m_{2j}-m_{2j-1}-1} \subseteq \wt{Z}_1$. It is clear that $f$ maps to $\ol{f}$ under the map $P_{m_{2j-1},m_{2j}}^R \ra R''^{(2j-1)}$.
\end{enumerate}
One can easily check that all the defined bijective maps preserve the separations as in Definition \ref{definition-equiv.rel.Ldiagrams} (a). The result follows.
\end{proof}

\begin{theorem}\label{theorem-independence.of.partitions.Ldiagram}
The equivalence class, under $\sim_l$, of the $l$-diagram constructed from $(X,\sigma)$ by using the construction \ref{construction-l.diag} does not depend on the choice of the $\sigma$-refined sequence of partitions $\{\calP_n\}_{n \geq 0}$.
\end{theorem}
\begin{proof}
Suppose we have two $\sigma$-refined sequences of partitions $\{\calP_n\}_{n \geq 0}$, $\{\wt{\calP}_n\}_{n \geq 0}$. Passing to subsequences (satisfying (CR) and with the first element of the sequence being an even number) does not affect the equivalence class of the resulting $l$-diagram by Proposition \ref{proposition-iso.telescoping.Ldiagram}, so we may assume that
$$\calP_n \precsim \wt{\calP}_{n+1}\precsim \calP_{n+2} \quad \text{for all }n \geq 0.$$
We can create a new sequence of partitions $\{\ol{\calP}_n\}_{n \geq 0}$ by
\begin{equation*}
\ol{\calP}_n := \begin{cases}
\calP_{3n} & \text{for }n \equiv 0 \text{ (mod }3); \\
\wt{\calP}_{3n} & \text{otherwise}.
\end{cases}
\end{equation*}
This sequence of partitions also satisfies the required properties:
\begin{enumerate}[(a),leftmargin=0.6cm]
\item Since $\ol{\calP}_0 = \calP_0$, $\sigma$ restricts to a homeomorphism $\sigma|_Z$ for each $Z \in \ol{\calP}_0$.
\item For $n \geq 0$, we have
$$\begin{cases}
\ol{\calP}_{3n} = \calP_{9n} \precsim \wt{\calP}_{9n+1} \precsim \wt{\calP}_{9n+3} = \ol{\calP}_{3n+1}, \\
\ol{\calP}_{3n+1} = \wt{\calP}_{9n+3} \precsim \wt{\calP}_{9n+6} = \ol{\calP}_{3n+2}, \\
\ol{\calP}_{3n+2} = \wt{\calP}_{9n+6} \precsim \calP_{9n+7} \precsim \calP_{9n+9} = \ol{\calP}_{3n+3}.
\end{cases}$$
\item Clearly the union $\bigcup_{n \geq 0} \ol{\calP}_n$ generates the topology of $X$, since both $\bigcup_{n \geq 0} \calP_n$ and $\bigcup_{n \geq 0} \wt{\calP}_n$ generate the topology of $X$.
\item We have $\text{diam}(\ol{\calP}_n) \leq \text{diam}(\calP_{3n+1})$, which tends to zero as $n \ra \infty$.
\item We compute
\begin{equation*}
\ol{\calP}_{2n} \vee \ol{\calP}_{2n}^{\sigma} = \begin{cases} \calP_{6n} \vee \calP_{6n}^{\sigma} \precsim \calP_{6n+1} \precsim \wt{\calP}_{6n+3} = \ol{\calP}_{2n+1} & \text{if } 2n \equiv 0 \text{ (mod }3), \\
\wt{\calP}_{6n} \vee \wt{\calP}_{6n}^{\sigma} \precsim \wt{\calP}_{6n+1} \precsim \wt{\calP}_{6n+3} = \ol{\calP}_{2n+1} & \text{if } 2n \equiv 1 \text{ (mod }3), \\
\wt{\calP}_{6n} \vee \wt{\calP}_{6n}^{\sigma} \precsim \wt{\calP}_{6n+1} \precsim \calP_{6n+3} = \ol{\calP}_{2n+1} & \text{if } 2n \equiv 2 \text{ (mod }3).
\end{cases}
\end{equation*}
In either case
$$\ol{\calP}_{2n} \vee \ol{\calP}_{2n}^{\sigma} \precsim \ol{\calP}_{2n+1}.$$
\item Also,
\begin{equation*}
\ol{\calP}_{2n+1} \vee \sigma^{-1}(\ol{\calP}_{2n+1}) = \begin{cases} \wt{\calP}_{6n+3} \vee \sigma^{-1}(\wt{\calP}_{6n+3}) \precsim \wt{\calP}_{6n+4} \precsim \wt{\calP}_{6n+6} = \ol{\calP}_{2n+2} & \text{if } 2n \equiv 0 \text{ (mod }3), \\
\wt{\calP}_{6n+3} \vee \sigma^{-1}(\wt{\calP}_{6n+3}) \precsim \wt{\calP}_{6n+4} \precsim \calP_{6n+6} = \ol{\calP}_{2n+2} & \text{if } 2n \equiv 1 \text{ (mod }3), \\
\calP_{6n+3} \vee \sigma^{-1}(\calP_{6n+3}) \precsim \calP_{6n+4} \precsim \wt{\calP}_{6n+6} = \ol{\calP}_{2n+2} & \text{if } 2n \equiv 2 \text{ (mod }3).
\end{cases}
\end{equation*}
In either case
$$\ol{\calP}_{2n+1} \vee \sigma^{-1}(\ol{\calP}_{2n+1}) \precsim \ol{\calP}_{2n+2}.$$
\end{enumerate}
We let $(F_1,D_1), (F_2,D_2)$ and $(F,D)$ be the $l$-diagrams constructed from the $\sigma$-refined sequences of partitions $\{\calP_n\}_{n \geq 0}, \{\wt{\calP}_n\}_{n \geq 0}$ and $\{\ol{\calP}_n\}_{n \geq 0}$, respectively, using the construction given in \ref{construction-l.diag}. Consider the two sequences
$$(m_n)_{n \geq 0} := (0,3,6,9,\dots,3n,\dots), \qquad (\wt{m}_n)_{n \geq 0} := (2,5,\dots,3n+2,\dots),$$
which clearly satisfy (CR), and $m_0 \equiv \wt{m}_0 \equiv 0 \text{ (mod }2)$. The contraction of $(F,D)$ with respect to $(m_n)_{n \geq 0}$ gives an $l$-diagram with associated partitions
$$\calP_0, \calP_9, \calP_{18}, \calP_{27},\dots, \calP_{9n},\dots$$
so it is itself a contraction of $(F_1,D_1)$ by Proposition \ref{proposition-iso.telescoping.Ldiagram}. Hence $(F_1,D_1) \sim_{l} (F,D)$. Also, the contraction of $(F,D)$ with respect to $(\wt{m}_n)_{n \geq 0}$ gives an $l$-diagram with associated partitions
$$\wt{\calP}_6, \wt{\calP}_{15}, \wt{\calP}_{24},\dots, \wt{\calP}_{9n+6},\dots$$
so it is itself a contraction of $(F_2,D_2)$ by Proposition \ref{proposition-iso.telescoping.Ldiagram} again. Hence $(F_2,D_2) \sim_{l} (F,D)$.

Altogether we deduce that
$$(F_1,D_1) \sim_l (F,D) \sim_l (F_2,D_2),$$
as required.
\end{proof}

Using the results obtained so far, we obtain a well-defined map
$$\wt{\Phi} : \textsf{LHomeo} \lra \textsf{LDiag}/\sim_l, \quad (X,\sigma) \mapsto [(F,D)]$$
where $(F,D)$ is the $l$-diagram obtained from $(X,\sigma)$ by applying the construction \ref{construction-l.diag}. Indeed, $\wt{\Phi}(X,\sigma)$ does not depend on the choice of the $\sigma$-refined sequence of partitions used to define it by Theorem \ref{theorem-independence.of.partitions.Ldiagram}.

In the next theorem we prove the last thing left in order to properly define $\Phi$, namely that the equivalence class of $(F,D)$ does not depend on the representative chosen for the class of $(X,\sigma)$ under the equivalence relation $\sim_{t.c.}$.

\begin{theorem}\label{theorem-independence.of.top.conj.Ldiagram}
Suppose we have $(X_1,\sigma_1),(X_2,\sigma_2) \in \emph{\textsf{LHomeo}}$ which are topologically conjugate through the homeomorphism $h : X_1 \ra X_2$. Then $\wt{\Phi}(X_1,\sigma_1) = \wt{\Phi}(X_2,\sigma_2)$.
\end{theorem}
\begin{proof}
First, notice that $N(\sigma_1) = N(\sigma_2) =: N$ since $h$ is a homeomorphism. Also, $h(\calU^{(1)}_k) = \calU^{(2)}_k$ for each $1 \leq k \leq N$, being
$$\calU^{(i)}_k = \{x \in X_i \mid |\sigma^{-1}_i(\{x\})| = k\}.$$

Take $\{\calP_n\}_{n \geq 0}$ a $\sigma_1$-refined sequence of partitions of $X_1$, and $\{\calP'_n\}_{n \geq 0}$ a $\sigma_2$-refined sequence of partitions of $X_2$. Let $(F_i,D_i)$ denote the $l$-diagram constructed from $(X_i,\sigma_i)$ using the construction \ref{construction-l.diag}, for $i=1,2$. We define, for $n \geq 0$, new partitions
$$\calP^{(1)}_n := \calP_n \vee h^{-1}(\calP'_n) \text{ for $X_1$, and }\text{ } \calP^{(2)}_n := h(\calP_n) \vee \calP'_n \text{ for $X_2$.}$$
Then the sequence $\{\calP^{(i)}_n\}_{n \geq 0}$ is a $\sigma_{i}$-refined sequence of partitions of $X_i$, for $i=1,2$:
\begin{enumerate}[(a),leftmargin=0.6cm]
\item Since $\sigma_1$ (resp. $\sigma_2$) restricts to a homeomorphism over each $Z \in \calP_0$ (resp. $Z \in \calP'_0$), it is clear that each $\sigma_i$ restricts to a homeomorphism $\sigma_i|_Z$, for $Z \in \calP^{(i)}_0$.
\item $\calP^{(1)}_n \precsim \calP^{(1)}_{n+1}$ and $\calP^{(2)}_n \precsim \calP^{(2)}_{n+1}$, since $\calP_n \precsim \calP_{n+1}$ and $\calP'_n \precsim \calP'_{n+1}$.
\item Their unions generate the topology of $X_1$ and $X_2$, respectively, since this property holds true for the partitions $\{\calP_n\}_{n \geq 0}$ of $X_1$ and $\{\calP_n'\}_{n \geq 0}$ of $X_2$.
\item $\text{diam}(\calP_n^{(1)}) \leq \text{diam}(\calP_n)$ and $\text{diam}(\calP_n^{(2)}) \leq \text{diam}(\calP_n')$, so they both tend to zero as $n \ra \infty$.
\item First, one may notice that $h(\calP^{\sigma_1}) = h(\calP)^{\sigma_2}$ for any $\sigma_1$-refined partition $\calP$ of $X_1$. This follows from the fact that $h(\calU^{(1)}_k) = \calU^{(2)}_k$. Now, for $X_1$ and $n \geq 0$, we compute
\begin{align*}
\calP^{(1)}_{2n} \vee (\calP^{(1)}_{2n})^{\sigma_1} & = \calP_{2n} \vee h^{-1}(\calP'_{2n}) \vee \calP_{2n}^{\sigma_1} \vee (h^{-1}(\calP'_{2n}))^{\sigma_1} \\
& = \calP_{2n} \vee \calP_{2n}^{\sigma_1} \vee h^{-1} ( \calP'_{2n} ) \vee h^{-1} \big( (\calP'_{2n})^{\sigma_2} \big) \\
& = \big( \calP_{2n} \vee \calP_{2n}^{\sigma_1} \big) \vee h^{-1}\big( \calP'_{2n} \vee (\calP'_{2n})^{\sigma_2} \big) \\
& \precsim \calP_{2n+1} \vee h^{-1}(\calP'_{2n+1}) = \calP^{(1)}_{2n+1}
\end{align*}
and
\begin{align*}
\calP^{(1)}_{2n+1} \vee \sigma_1^{-1}(\calP^{(1)}_{2n+1}) & = \calP_{2n+1} \vee h^{-1}(\calP'_{2n+1}) \vee \sigma_1^{-1}(\calP_{2n+1}) \vee \sigma_1^{-1}(h^{-1}(\calP'_{2n+1})) \\
& = \big( \calP_{2n+1} \vee \sigma_1^{-1}(\calP_{2n+1}) \big) \vee h^{-1}\big( \calP'_{2n+1} \vee \sigma_2^{-1}(\calP'_{2n+1}) \big) \\
& \precsim \calP_{2n+2} \vee h^{-1}(\calP'_{2n+2}) = \calP^{(1)}_{2n+2}.
\end{align*}
Similar computations hold for $X_2$, that is, for $n \geq 0$, we have
$$\calP^{(2)}_{2n} \vee (\calP^{(2)}_{2n})^{\sigma_2} \precsim \calP^{(2)}_{2n+1} \quad \text{ and } \quad \calP^{(2)}_{2n+1} \vee \sigma_2^{-1}(\calP^{(2)}_{2n+1}) = \calP^{(2)}_{2n+2}.$$
\end{enumerate}
Note that $h(\calP^{(1)}_n) = h(\calP_n \vee h^{-1}(\calP'_n)) = h(\calP_n) \vee \calP'_n = \calP^{(2)}_n$. Therefore the vertices of $(F_1,D_1)$ are in bijective correspondence with the vertices of $(F_2,D_2)$ at each layer $n \geq 0$. To conclude the theorem, we need to show that the edges $F_1^{1,n}$ are in bijective correspondence with $F_2^{1,n}$ for all $n \geq 0$, also preserving the separations as in Definition \ref{definition-equiv.rel.Ldiagrams} (a). These follow directly from the following considerations:
\begin{enumerate}[(i),leftmargin=0.6cm]
\item For $Z \in \calP^{(1)}_{2n}$ and $Z' \in \calP^{(1)}_{2n+1}$, $Z' \subseteq Z$ if and only if $h(Z') \subseteq h(Z)$. Also $Z' \subseteq \sigma_1(Z)$ if and only if $h(Z') \subseteq h(\sigma_1(Z)) = \sigma_2(h(Z))$.
\item For $Z' \in \calP^{(1)}_{2n+1}$ and $Z'' \in \calP^{(1)}_{2n+2}$, $Z'' \subseteq Z'$ if and only if $h(Z'') \subseteq h(Z')$. Also $\sigma_1(Z'') \subseteq Z'$ if and only if $\sigma_2(h(Z'')) = h(\sigma_1(Z'')) \subseteq h(Z')$.
\end{enumerate}
This concludes the proof of the theorem.
\end{proof}

We finally obtain the desired map $\Phi$.

\begin{theorem}\label{theorem-first.map}
There exists a well-defined map
$$\Phi : \emph{\textsf{LHomeo}}/\sim_{t.c.} \lra \emph{\textsf{LDiag}}/\sim_l, \quad [(X,\sigma)] \mapsto [(F,D)],$$
where $(F,D)$ is the $l$-diagram obtained from $(X,\sigma)$ by applying the construction \ref{construction-l.diag}.
\end{theorem}
\begin{proof}
This follows directly from the definition of $\wt{\Phi}$ and Theorem \ref{theorem-independence.of.top.conj.Ldiagram}.
\end{proof}


\subsection{From \textsf{LDiag} to \textsf{LHomeo}}\label{subsection-from.Ldiagram.to.local.homeo}

We now describe how to construct a surjective local homeomorphism $(X,\sigma)$ from an $l$-diagram $(F,D)$.

Let $(F,D) \in \textsf{LDiag}$. We first construct the space $X$. It will be given by the set of infinite paths
$$X = \{(e_0,e_1,e_2,\dots) \mid e_i \in B_{r(e_i)}, r(e_{i+1}) = s(e_i) \text{ for all } i \geq 0\}.$$
Note that all edges are blue. The topology on $X$ is the one with basis consisting of the cylinder sets
$$Z(e_0,e_1,\dots,e_n) = \{(x_i)_{i \geq 0} \in X \mid x_0 = e_0,\dots,x_n = e_n\},$$
for $(e_0,\dots,e_n) \in P_{0,n}^B$. This is indeed a basis for a topology because the intersection of two cylinder sets is either empty or again a cylinder set. 

\begin{lemma}\label{lemma-infinite.paths.compact.L}
The space $X$ is a totally disconnected, compact metrizable space.
\end{lemma}
\begin{proof}
The decompositions
$$X = \bigsqcup_{(e_0,\dots,e_n) \in P_{0,n}^B} Z(e_0,\dots,e_n),$$
for any $n \geq 0$, imply that each cylinder set is a clopen set. Therefore $X$ is totally disconnected. It is clearly Hausdorff, second countable and regular, hence by Urysohn's Metrization Theorem, $X$ becomes metrizable.

We can regard $X$ as a topological subspace of $\prod_{n \geq 0} F^{1,n}$ endowed with the product topology, where each set $F^{1,n}$ is endowed with the discrete topology. This space is compact by Tychonoff's Theorem, and $X$ is a closed subspace of it. Therefore $X$ is compact.
\end{proof}

We now aim to define a natural surjective local homeomorphism $\sigma : X \ra X$, as follows.

\begin{construction}\label{construction-local.homeo}
Let $(e_0,e_1,e_2,\dots) \in X$. For each $j \ge 0$ we may complete the pair of blue edges $(e_{2j},e_{2j+1})$ to a romb $(e_{2j},e_{2j+1},f_{2j},f_{2j+1})$ using part (ii) of Proposition \ref{proposition-decomposition.odd.rombs.even}, so that $r(f_{2j}) =r(e_{2j})$, $s(f_{2j})=r(f_{2j+1})$, $s(f_{2j+1})=s(e_{2j+1})$ and $f_{2j+1}\in R(f_{2j})$. Note that $s(f_{2j-1}) = r(f_{2j})$ for each $j \ge 1$. So, for $j \ge 1$, we may consider the red pair $(f_{2j-1}, f_{2j})$ and complete it to another romb $(e_{2j-1}', e_{2j}', f_{2j-1}, f_{2j})$ by using condition (f) in Definition \ref{definition-L.separated.Bratteli}. We obtain a well-defined blue path
$$(e_1',e_2',\dots) \in X$$
starting at $v: = r(e_1') \in F^{0,1}$. Now, part (1) of condition (d) in Definition \ref{definition-L.separated.Bratteli} entails that there is $w \in F^{0,0}$ and $e_0' \in B_w$ such that $s(e_0')= v = r(e_1')$. Also, by the last part of this same condition (d) the edge $e_0'$ is unique with this property. Thus
$$(e_0',e_1',e_2',\dots) \in X.$$
We define $\sigma \colon X \to X$ by
$$\sigma(e_0,e_1,e_2,\dots) = (e'_0,e_1',e_2',\dots).$$

\begin{figure}[H]
\begin{tikzpicture}
	\path [draw=blue,postaction={on each segment={mid arrow=blue}}]
	(-1,6) -- node[above,blue]{$e_0$} (0,7.2)
	(0,4.8) -- node[above,blue]{$e_1$} (-1,6)
	(-1,3.6) -- node[above,blue]{$e_2$} (0,4.8)
	(0,2.4) -- node[above,blue]{$e_3$} (-1,3.6)
	(-1,1.2) -- node[above,blue]{$e_4$} (0,2.4)
	(0,0) -- node[above,blue]{$e_5$} (-1,1.2)
	;
	\path [draw=red,postaction={on each segment={mid arrow=red}}]
	(1,6) -- node[above,red]{$f_0$} (0,7.2)
	(0,4.8) -- node[above,red]{$f_1$} (1,6)
	(1,3.6) -- node[above,red]{$f_2$} (0,4.8)
	(0,2.4) -- node[above,red]{$f_3$} (1,3.6)
	(1,1.2) -- node[above,red]{$f_4$} (0,2.4)
	(0,0) -- node[above,red]{$f_5$} (1,1.2)
	;
	\path [draw=blue,dashed,postaction={on each segment={mid arrow=blue}}]
	(1,6) -- node[above,blue]{$e_0'$} (2,7.2)
	(2,4.8) -- node[above,blue]{$e_1'$} (1,6)
	(1,3.6) -- node[above,blue]{$e_2'$} (2,4.8)
	(2,2.4) -- node[above,blue]{$e_3'$} (1,3.6)
	(1,1.2) -- node[above,blue]{$e_4'$} (2,2.4)
	(2,0) -- node[above,blue]{$e_5'$} (1,1.2)
	;
	\draw[black,thick,dashed] (0,-0.5) -- (0,0);
	\draw[black,thick,dashed] (2,-0.5) -- (2,0);
	\draw[black,dashed] (-2,0) -- (3,0);
	\draw[black,dashed] (-2,1.2) -- (3,1.2);
	\draw[black,dashed] (-2,2.4) -- (3,2.4);
	\draw[black,dashed] (-2,3.6) -- (3,3.6);
	\draw[black,dashed] (-2,4.8) -- (3,4.8);
	\draw[black,dashed] (-2,6) -- (3,6);
	\draw[black,dashed] (-2,7.2) -- (3,7.2);
	\node[circle,fill=blue,scale=0.6] (B0) at (0,0) {};
	\node[circle,fill=blue,scale=0.6] (B1) at (-1,1.2) {};
	\node[circle,fill=blue,scale=0.6] (B2) at (0,2.4) {};
	\node[circle,fill=blue,scale=0.6] (B3) at (-1,3.6) {};
	\node[circle,fill=blue,scale=0.6] (B4) at (0,4.8) {};
	\node[circle,fill=blue,scale=0.6] (B5) at (-1,6) {};
	\node[circle,fill=blue,scale=0.6] (B6) at (0,7.2) {};
	\node[fill=black,scale=0.6,label=$v$] (BB'2) at (1,6) {};
	\node[fill=black,scale=0.6] (BB'1) at (1,3.6) {};
	\node[fill=black,scale=0.6] (BB'0) at (1,1.2) {};
	\node[circle,fill=blue,scale=0.6] (B'0) at (2,0) {};
	\node[circle,fill=blue,scale=0.6] (N'1) at (2,2.4) {};
	\node[circle,fill=blue,scale=0.6] (N'2) at (2,4.8) {};
	\node[circle,fill=blue,scale=0.6,label=$w$] (N'3) at (2,7.2) {};
\end{tikzpicture}
\caption{Construction of the map $\sigma$.}
\label{figure-schematics4}
\end{figure}
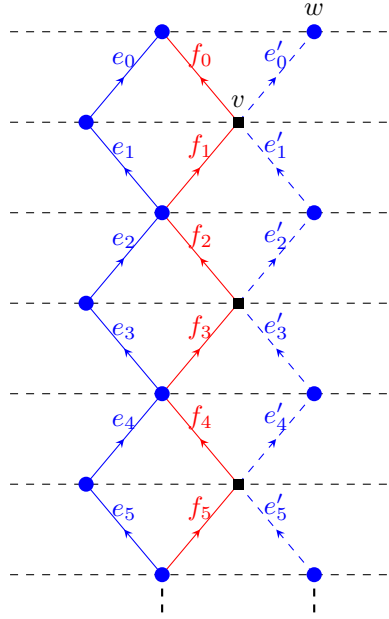
\end{construction}

\begin{lemma}\label{lemma-sigma.surj.local.homeo}
The map $\sigma$ is a surjective local homeomorphism.
\end{lemma}
\begin{proof}
We first show that $\sigma$ is continuous. Let $x = (e_0,e_1,e_2,\dots) \in X$, with $e_i \in B_{r(e_i)}$ for all $i \geq 0$. Let $y = \sigma(x) = (e'_0,e'_1,e'_2,\dots)$ and consider an open neighborhood $Z(e_0',e'_1,e'_2,\dots,e'_{2n})$ of $y$. Then we have
$$\sigma(Z(e_0,e_1,\dots,e_{2n+1})) \subseteq Z(e'_0,e'_1,e'_2,\dots,e'_{2n}),$$
and $Z(e_0,e_1,\dots,e_{2n+1})$ is an open neighborhood of $x$. This shows continuity of $\sigma$.\\

We now show that $\sigma$ is surjective. Indeed we will show that there are a finite number of pre-images of each element $y \in X$.

Let $y =(e_0',e_1',e_2',\dots)\in X$, with $e_i' \in B_{r(e_i')}$ for all $i \geq 0$. By part (2) of condition (d) in Definition \ref{definition-L.separated.Bratteli}, we can choose a red edge $f'_0$ such that $s(f'_0) = s(e_0')$. Note that existence of such $f'_0$ is ensured, but certainly not uniqueness, although we do have a finite number of possibilities for $f'_0$. As we show now, this is the only choice we need.

By recursively applying property (g) in Definition \ref{definition-L.separated.Bratteli}, we construct a sequence of red edges $f'_1,f'_2,\dots$ such that $f'_{2j+1} \in R(f'_{2j})$ and $(e'_{2j+1},e'_{2j+2},f'_{2j+1},f'_{2j+2})$ is a romb for all $j \geq 0$. This can be done since each $s(f'_{2j}) \in F^{0,2j+1}$. Note that the sequence $f'_1,f'_2,\dots$ is uniquely determined by $y$ and by $f'_0$. Finally we use part (ii) of Proposition \ref{proposition-decomposition.odd.rombs.even} to build pairs of blue edges $(e_{2j},e_{2j+1})$ such that $(e_{2j},e_{2j+1},f'_{2j},f'_{2j+1})$ is a romb, for all $j \ge 0$. Then $x = (e_0,e_1,e_2,\dots) \in X$ is an element such that $\sigma(x) = y$, as required. This shows in particular that $\sigma$ is surjective.

Conversely, if $\sigma(x) = y$, with $x = (e_0,e_1,e_2,\dots)$, and if $(f_0,f_1)$ is the unique pair of red edges such that $f_1 \in R(f_0)$ and $(e_0,e_1,f_0,f_1)$ is a romb, then $x$ is the unique pre-image of $y$ determined by the choice of $f_0$ as the red edge $f$ such that $s(f) = s(e_0')$. This means that there exist exactly $|s^{-1}(s(e_0'))|-1$ pre-images of $y$. Observe also that this implies that $\sigma $ is injective when restricted to $Z(e_0,e_1)$, because $f_0$ is determined by $(e_0,e_1)$. 

To show that $\sigma$ is a local homeomorphism, it suffices to show that $\sigma$ is an open map, since we already know that $\sigma$ is continuous and locally injective. To show that $\sigma $ is an open map, it suffices to prove that for any blue path $(e_0,e_1,\dots ,e_{2n-1})\in P^B_{0,2n}$, with $n\ge 1$, we have that $\sigma (Z(e_0,e_1,\dots , e_{2n-1}))$ is an open subset of $X$. Now, if
$y= (e_0',e_1',\dots , e'_{2n},e'_{2n+1},\dots ) = \sigma (x)$ for some $x\in Z(e_0,e_1,\dots e_{2n-1})$, then following the proof of surjectivity given above we can easily see that 
$$Z(e'_0,\dots ,e'_{2n}) \subseteq \sigma (Z(e_0,e_1,\dots ,e_{2n-1})).$$
This shows that $\sigma$ is an open map, hence it is a local homeomorphism. 
\end{proof}


Using the above construction, we obtain a well-defined map
$$\wt{\Psi} : \textsf{LDiag} \lra \textsf{LHomeo}/\sim_{t.c.}, \quad (F,D) \mapsto [(X,\sigma)]$$
where $(X,\sigma)$ is the dynamical system obtained from $(F,D)$ by applying the construction \ref{construction-local.homeo}. In the next theorem we show that the equivalence class of $(X,\sigma)$ does not depend on the representative chosen for the class $(F,D)$ under the equivalence relation $\sim_l$.

\begin{theorem}\label{theorem-independence.of.Lclass}
Let $(F_1,D_1), (F_2,D_2) \in \emph{\textsf{LDiag}}$ be two equivalent $l$-diagrams under $\sim_l$. Write $(X_i,\sigma_i)$ for the dynamical system obtained from $(F_i,D_i)$ by applying the construction \ref{construction-local.homeo}, for $i=1,2$. Then there exists a homeomorphism $h : X_1 \ra X_2$ such that $h \circ \sigma_1 = \sigma_2 \circ h$.
\end{theorem}
\begin{proof}
It suffices to consider the case where $(F_2,D_2)$ is a contraction of $(F_1,D_1)$ with respect to a sequence $(m_n)_{n \geq 0}$ in $\N \cup \{0\}$ with $0 \le m_0 < m_1 < \cdots$, satisfying (CR) and with $m_0$ even. We write $(F,D) := (F_1,D_1), (F',D') := (F_2,D_2)$, and $(X,\sigma) := (X_1,\sigma_1) , (X',\sigma') := (X_2,\sigma_2)$.

We have
\begin{align*}
& X = \{(e_0,e_1,\dots) \mid e_i \in F^{1,i} \text{ blue edges, } r(e_{i+1}) = s(e_i) \text{ for all } i \geq 0\}, \\
& X' = \{((e_{m_0},\dots,e_{m_1-1}),(e_{m_1},\dots,e_{m_2-1}),\dots) \mid (e_{m_0},...,e_{m_1},...,e_{m_2},...) \in X\}.
\end{align*}
There is an obvious identification $X \cong X'$, which is in fact a homeomorphism (it sends cylinder sets to cylinder sets). It clearly gives our desired conjugacy $h$.
\end{proof}

We summarize our construction in the following theorem.

\begin{theorem}\label{theorem-second.map}
There exists a well-defined map
$$\Psi : \emph{\textsf{LDiag}}/\sim_l \lra \emph{\textsf{LHomeo}}/\sim_{t.c.}, \quad [(F,D)] \mapsto [(X,\sigma)],$$
where $(X,\sigma) \in \emph{\textsf{LHomeo}}$ is the dynamical system obtained from $(F,D)$ by applying the construction \ref{construction-local.homeo}.
\end{theorem}
\begin{proof}
This follows directly from the definition of $\wt{\Psi}$ and Theorem \ref{theorem-independence.of.Lclass}.
\end{proof}

\subsection{Equivalence between both constructions}\label{subsection-equiv.Ldiagrams.local.homeos}

In this section we prove that both constructions \ref{construction-l.diag} and \ref{construction-local.homeo} are inverses of each other, modulo the stated equivalences in \textsf{LHomeo} and \textsf{LDiag}. First, a technical lemma. Recall Definition \ref{definition-refined.Ldiagram} for the notion of a refined $l$-diagram.

\begin{lemma}\label{lemma-refined.and.sigma.refined}
Let $(F,D)$ be an $l$-diagram. Then $(F,D)$ is a refined $l$-diagram if and only if $\sigma|_{Z(v)}$ is injective for all $v \in F^{0,0}$. Here $Z(v)$ is defined naturally as the set of those elements $(e_0,e_1,e_2,\dots) \in X$ such that $r(e_0) = v$.
\end{lemma}
\begin{proof}
By Remark \ref{remark-about.def.of.Ldiagram} (5) it is enough to show that $\sigma|_{Z(v)}$ is injective for all $v \in F^{0,0}$ if and only if there are no double red edges from $F^{0,1}$ to $F^{0,0}$, that is, given $w \in F^{0,1}$ and $v \in F^{0,0}$, there is at most one red edge $f$ such that $s(f) = w$ and $r(f) = v$.

Suppose first that there are no double red edges from a vertex in $F^{0,1}$ to a vertex in $F^{0,0}$. Let $y = (e_0',e_1',e_2', \dots )\in X$, where $e_i' \in B_{r(e_i')}$ for all $i \ge 0$. By the proof of Lemma \ref{lemma-sigma.surj.local.homeo}, the set of pre-images of $y$ by $\sigma$ is in bijective correspondence with the set of red edges $f_0$ such that $s(f_0) = s(e_0')$. Hence there is at most one element $x \in Z(v)$ such that $\sigma(x) = y$, for each $v \in F^{0,0}$.

Conversely, if there are two red edges $f_0,f_0'$ such that $s(f_0) = s(f_0') = w \in F^{0,1}$ and $r(f_0) = r(f_0') = v \in F^{0,0}$, then any point $y = (e_0',e_1',e_2',\dots ) \in X$ such that $s(e_0') = w$ has at least two pre-images in $Z(v)$, and thus $\sigma|_{Z(v)}$ is not injective.

This concludes the proof of the lemma.
\end{proof}

With the help of the above lemma, we can now show our main result.

\begin{theorem}[Theorem \ref{theorem-correspondence}]\label{theorem-equivalence}
The maps $\Phi : \emph{\textsf{LHomeo}}/\sim_{t.c.} \ra \emph{\textsf{LDiag}}/\sim_l$ and $\Psi : \emph{\textsf{LDiag}}/\sim_l \ra \emph{\textsf{LHomeo}}/\sim_{t.c.}$ defined in Theorems \ref{theorem-first.map} and \ref{theorem-second.map}, respectively, are mutually inverse maps.

Therefore the above procedures establish a bijective correspondence between equivalence classes of $l$-diagrams (via isomorphism and contraction) and topological conjugacy classes of surjective local homeomorphisms on totally disconnected, compact metrizable spaces.
\end{theorem}
\begin{proof}
We first show that $\Phi \circ \Psi = \text{id}_{\textsf{LDiag}/\sim_l}$. Let $(F,D)$ be an $l$-diagram. By Remark \ref{remark-about.def.of.Ldiagram} (5) the $l$-diagram obtained from $(F,D)$ by using the sequence $(2,3,4,5,\dots)$ is a refined $l$-diagram, which of course is $\sim_l$-equivalent to $(F,D)$. Hence we may assume that $(F,D)$ is a refined $l$-diagram.

Let $(X,\sigma)$ be the dynamical system in \textsf{LHomeo} constructed from $(F,D)$ by applying the construction \ref{construction-local.homeo}. Consider the sequence of partitions given by the cylinder sets
$$\calP_0 := \{Z(v) \mid v \in F^{0,0} \}, \quad \calP_{n} := \{Z(e_0,\dots,e_{n-1}) \mid (e_0,\dots,e_{n-1}) \in P_{0,n}^B\}$$
for $n \geq 1$. Note that there is a natural identification between $\calP_n$ and $F^{0,n}$. Recall that, by definition of $\sigma$, we have
$$\sigma(Z(e_0,\dots,e_{2n+1})) \subseteq Z(e'_0,\dots,e'_{2n})$$
for some blue path $(e_0',\dots, e_{2n}')$ of length $2n+1$.

First of all, we show that $\{\calP_n\}_{n\ge 0}$ is a $\sigma$-refined sequence of partitions of $X$ (see Definition \ref{definition-sigma.refined.partition}).

\begin{enumerate}[(a),leftmargin=0.6cm]
\item By Lemma \ref{lemma-refined.and.sigma.refined}, $\sigma|_{Z(v)}$ is injective for all $v \in F^{0,0}$, since we are assuming that $(F,D)$ is a refined $l$-diagram.
\item $\calP_n \precsim \calP_{n+1}$ because $Z(e_0,\dots,e_n,e_{n+1}) \subseteq Z(e_0,\dots,e_n)$.
\item The union $\bigcup_{n \geq 0} \calP_n$ generates the topology of $X$ by definition.
\item We may take the metric determined by the distance function $d$ given by
$$d((e_i),(e_i')) =\frac{1}{2^N},$$
where $N$ is the smallest non-negative integer such that $e_N\ne e_{N}'$ for two different elements $(e_i),(e_i')\in X$. It is then clear that  
$\text{diam}(Z(e_0,\dots,e_n)) \xra{n \ra \infty} 0$.
\item For each $n \ge 0$ we first show that $\sigma^{-1}(\calP _{2n+1}) \precsim \calP_{2n+2}$. Given a blue path $(e_0, e_1,\dots , e_{2n+1})$ of length $2n+2$, we know that there exists a blue path $(e_0',\dots , e_{2n}')$ of length $2n+1$ such that
$$\sigma(Z(e_0,\dots , e_{2n+1})) \subseteq Z(e'_0,\dots,e'_{2n}).$$
Therefore $Z(e_0,\dots , e_{2n+1}) \subseteq \sigma^{-1}(Z(e'_0,\dots,e'_{2n}))$, and so $\sigma^{-1}(\calP_{2n+1}) \precsim \calP_{2n+2}$.

We now show that $\calP_{2n}^{\sigma} \precsim \calP_{2n+1}$. Take a blue path $(e_0,e_1,\dots , e_{2n})$ of length $2n+1$. As seen in the proof of Lemma \ref{lemma-sigma.surj.local.homeo}, each element of $Z(e_0,\dots , e_{2n})$ has exactly $k$ pre-images by $\sigma$, where $k$ is the number of red edges $f_0$ in $F^{1,0}$ such that $s(f_0) = s(e_0)$. Suppose first that $n\ge 1$. Given such an edge $f_{j,0}$, for $1 \leq j \leq k$, there are unique red edges $f_{j,1},f_{j,2},\dots ,f_{j,2n-1},f_{j,2n}$ such that $f_{j,1}\in R(f_{j,0})$, $(e_{2i-1},e_{2i},f_{j,2i-1},f_{j,2i})$ is a romb for $1 \leq i \leq n$, and $f_{j,2i+1} \in R(f_{j,2i})$ for $1 \leq i \leq n-1$. For $0 \leq i \leq n-1$, let $(e_{j,2i}',e_{j,2i+1}',f_{j,2i},f_{j,2i+1})$ be the romb determined by $(f_{j,2i,},f_{j,2i+1})$. We then have
$$Z(e_0,\dots , e_{2n}) \subseteq  \Big( \bigcap_{j=1}^k \sigma ( Z(e'_{j,0}, \dots , e'_{j,2n-1}))\Big) \setminus \calU_{\ge k+1}.$$
Observing that the set in the right-hand side belongs to $\calP_{2n}^{\sigma}$, we obtain that $\calP_{2n}^{\sigma} \precsim \calP_{2n+1}$, as desired.

If $n=0$, then using that $(F,D)$ is a refined $l$-diagram we get
$$Z(e_0)\subseteq \Big( \bigcap_{j=1}^k \sigma ( Z(v_j))\Big) \setminus \calU_{\ge k+1},$$
where $v_j=r(f_j)$, and $f_1,\dots , f_k$ are the distinct red edges such that $s(f_j)= s(e_0)$. 
\end{enumerate}
Therefore $\{\calP_n\}_{n \geq 0}$ is a $\sigma$-refined sequence of partitions of $X$. Let $(F',D')$ be the $l$-diagram obtained from $(X,\sigma)$ by applying the construction \ref{construction-l.diag}. We aim to show that $(F,D) \sim_l (F',D')$.

It is clear that the vertex sets of $(F,D)$ and of $(F',D')$ are identified: given $v \in F^{0,n+1}$ for $n \geq 0$, there exists a unique blue path $(e_0,e_1,\dots,e_n)$ such that $s(e_n) = v$. Such vertex corresponds exactly to the clopen set $Z(e_0,\dots,e_n) \in \calP_{n+1} = (F')^{0,n+1}$. The vertex $v \in F^{0,0}$ corresponds to the clopen $Z(v) \in \calP_0 = (F')^{0,0}$. This correspondence is clearly bijective.

Now for the edges. We distinguish between even and odd layers.
\begin{enumerate}[(i),leftmargin=0.6cm]
\item For an even layer $2n \geq 0$, and given $v \in F^{0,2n}, w \in F^{0,2n+1}$, there exists a blue edge $e \in F^{1,2n}$ with source $w$ and range $v$ if and only if $v = s(e_{2n-1}), w = s(e_{2n})$ for a unique blue path $(e_0,\dots,e_{2n-1},e_{2n})$, if and only if there exists a blue edge $e(Z',Z) \in (F')^{1,2n}$ with source $Z' = Z(e_0,\dots,e_{2n-1},e_{2n})$ and range $Z = Z(e_0,\dots,e_{2n-1})$.

Similarly, by the definition of $\sigma$, there exists a red edge $f \in F^{1,2n}$ with source $w$ and range $v$ if and only if $v = s(e_{2n-1}), w = s(e'_{2n})$ for unique blue paths $(e_0,\dots,e_{2n-1}), (e'_0,\dots,e'_{2n})$ satisfying $Z(e'_0,\dots,e'_{2n}) \subseteq \sigma (Z(e_0,\dots,e_{2n-1}))$, if and only if there exists a red edge $f(Z',Z) \in (F')^{1,2n}$ with source $Z' = Z(e'_0,\dots,e'_{2n})$ and range $Z = Z(e_0,\dots,e_{2n-1})$ (note that since we are assuming that $(F,D)$ is refined, there is at most one red edge between any two vertices of $F$). Let us show the first of these equivalences (note that, when $n=0$, one has to interpret $Z(e_0,\dots , e_{2n-1})$ as $Z(v)$).

Suppose first that $Z(e'_0,\dots,e'_{2n}) \subseteq \sigma (Z(e_0,\dots,e_{2n-1}))$. Taking
$$y = (e_0',\dots ,e_{2n}', e_{2n+1}',\dots ) \in Z(e_0',\dots ,e_{2n}'),$$
we have that there is some $x = (e_0,\dots ,e_{2n-1},e_{2n},\dots) \in Z(e_0,\dots , e_{2n-1})$ such that $\sigma(x) = y$. Let $(e_{2j}, e_{2j+1}, f_{2j}, f_{2j+1})$ be the romb associated to the blue pair $(e_{2j},e_{2j+1})$, for $j \geq 0$. Since $\sigma(x) = y$, we have by definition of $\sigma$ that $(e'_{2j+1},e'_{2j+2},f_{2j+1},f_{2j+2})$ is the romb associated to the red pair $(f_{2j+1}, f_{2j+2})$, for all $j\ge 0$. In particular we get that $s(f_{2n})= s(e_{2n}') = w$ and that $r(f_{2n})= r(e_{2n})= s(e_{2n-1})= v$. Hence we see that there is a red edge $f \in F^{1,2n}$ from $w$ to $v$. 

To show the converse, suppose there is a red edge $f_{2n} \in F^{1,2n}$ such that $s(f_{2n}) = w$ and $r(f_{2n}) = v$, and let
$$y = (e_0',\dots ,e_{2n}', e_{2n+1}',\dots ) \in Z(e_0',\dots , e_{2n}').$$
By property (g) in Definition \ref{definition-L.separated.Bratteli}, we can inductively build red pairs $(f_{2j+1},f_{2j+2})$ for $j \ge n$, such that $f_{2j+1} \in R(f_{2j})$ and  $(e_{2j+1}',e_{2j+2}',f_{2j+1},f_{2j+2})$ is a romb, for all $j \ge n$. Now, for all $j \ge n$, we can build the rombs $(e_{2j},e_{2j+1},f_{2j},f_{2j+1})$ (see part (ii) of Proposition \ref{proposition-decomposition.odd.rombs.even}). Note that $r(e_{2n}) = r(f_{2n}) = v = s(e_{2n-1})$, so that we can consider the point
$$x = (e_0,\dots ,e_{2n-1},e_{2n}, e_{2n+1},\dots) \in Z(e_{0},\dots , e_{2n-1}).$$
Now we have
$$\sigma(x) = (e_0'', \dots , e_{2n}'', e_{2n+1}',e_{2n+2}',\dots).$$
In particular $(e_0'',\dots , e_{2n}'')$ is a blue path with source $r(e_{2n+1}') = s(e_{2n}') = w$, so by uniqueness of the blue paths we get that $(e_0'', \dots , e_{2n}'') = (e_0',\dots , e_{2n}')$, and consequently $\sigma(x) = y$, with $x \in Z(e_0,\dots e_{2n-1})$. This shows the desired inclusion $Z(e_0',\dots e_{2n}') \subseteq \sigma (Z(e_0,\dots , e_{2n-1}))$.

We have shown that there exist bijective correspondences between blue edges in $F^{1,2n}$ and blue edges in $(F')^{1,2n}$, and between red edges in $F^{1,2n}$ and red edges in $(F')^{1,2n}$.

\item For an odd layer $2n+1 \geq 1$, and given $v \in F^{0,2n+1}, w \in F^{0,2n+2}$, there exists a blue edge $e \in F^{1,2n+1}$ with source $w$ and range $v$ if and only if $v = s(e_{2n}), w = s(e_{2n+1})$ for a unique blue path $(e_0,\dots,e_{2n},e_{2n+1})$, if and only if there exists a blue edge $e(Z',Z) \in (F')^{1,2n+1}$ with source $Z' = Z(e_0,\dots,e_{2n},e_{2n+1})$ and range $Z = Z(e_0,\dots,e_{2n})$.

Similarly, by the definition of $\sigma$, there exists a red edge $f \in F^{1,2n+1}$ with source $w$ and range $v$ if and only if $v = s(e'_{2n}), w = s(e_{2n+1})$ for unique blue paths $(e_0,\dots,e_{2n+1}), (e'_0,\dots,e'_{2n})$ satisfying $Z(e_0,\dots,e_{2n+1}) \subseteq \sigma^{-1}(Z(e'_0,\dots,e'_{2n}))$, if and only if there exists a red edge $f(Z'',Z') \in (F')^{1,2n+1}$ with source $Z'' = Z(e_0,\dots,e_{2n+1})$ and range $Z' = Z(e'_0,\dots,e'_{2n})$.

Indeed, if $\sigma(Z(e_0,\dots,e_{2n+1})) \subseteq Z(e'_0,\dots,e'_{2n})$, then it follows directly from the definition of $\sigma$ that there is a red edge $f \in F^{1,2n+1}$ from $w = s(e_{2n+1})$ to $v = s(e_{2n}')$.

Conversely, suppose that there is a red edge $f \in F^{1,2n+1}$ from $w = s(e_{2n+1})$ to $v = s(e_{2n}')$, and let
$$x = (e_0, \dots ,e_{2n+1}, e_{2n+2},\dots ) \in Z(e_0,\dots ,e_{2n+1}).$$
Since $w \in F^{0,2n+2}$ it follows from condition (c) in Definition \ref{definition-L.separated.Bratteli} that there is a unique red edge with source $w$, which must be equal to $f$. Therefore, if $(e_{2n},e_{2n+1},f_{2n},f_{2n+1})$ is the romb determined by $(e_{2n},e_{2n+1})$, then necessarily $f_{2n+1} = f$. If
$$\sigma(x) = (e_0'',\dots , e_{2n}'',e_{2n+1}'',\dots),$$
then, by definition of $\sigma$, we have $s(e_{2n}') =  v= r(f) = r(f_{2n+1}) = r(e_{2n+1}'') = s(e_{2n}'')$. By uniqueness of blue paths, we get that $(e_0'',\dots , e_{2n}'') = (e_0',\dots , e_{2n}')$, and therefore $\sigma(x) \in Z(e_0',\dots, e_{2n}')$, showing that $\sigma(Z(e_0,\dots , e_{2n+1})) \subseteq Z(e_0',\dots, e_{2n}')$.

We have shown that there exist bijective correspondences between blue edges in $F^{1,2n+1}$ and blue edges in $(F')^{1,2n+1}$, and between red edges in $F^{1,2n+1}$ and red edges in $(F')^{1,2n+1}$.
\end{enumerate}
From these one can easily deduce that $(F,D) \sim_l (F',D')$, and so $[(F,D)] = [(F',D')] = \Phi(\Psi([(F,D)]))$, as we wanted to show.\\

We are left to show that $\Psi \circ \Phi = \text{id}_{\textsf{LHomeo}/\sim_{t.c.}}$. Let $(X,\sigma) \in \textsf{LHomeo}$ and take $\{\calP_n\}_{n \geq 0}$ to be any sequence of $\sigma$-refined partitions of $X$. Denote by $(F,D)$ the refined $l$-diagram constructed from $\{\calP_n\}_{n \geq 0}$ by applying the construction \ref{construction-l.diag}, and by $(X',\sigma')$ the dynamical system in \textsf{LHomeo} constructed from $(F,D)$ by applying the construction \ref{construction-local.homeo}. Recall that
$$X' = \{(e_0,e_1,e_2,\dots) \mid e_i \in B_r(e_i), r(e_{i+1}) = s(e_i) \text{ for all } i \geq 0\},$$
and a basis for the topology of $X'$ is given by the cylinder sets $Z(e_0,e_1,\dots,e_n)$. We aim to show that $(X,\sigma)$ and $(X',\sigma')$ are topologically conjugate.

Let us construct the homeomorphism $h : X \ra X'$ as follows. Given $x \in X$, for each $n \geq 0$ let $Z_n(x)$ be the unique clopen set in the partition $\calP_n$ that contains $x$. Then $x \in \bigcap_{n \geq 0} Z_n(x)$, and since $\text{diam}(Z_n(x))$ tend to zero as $n \ra \infty$, necessarily this intersection is exactly $\{x\}$, that is
$$\bigcap_{n \geq 0} Z_n(x) = \{x\}.$$
Let now $e_n(x)$ be the unique blue edge in $F^{1,n}$ such that $r(e_n(x)) = Z_n(x)$ and $s(e_n(x)) = Z_{n+1}(x)$ (it exists since $x \in Z_n(x) \cap Z_{n+1}(x)$, so necessarily $Z_{n+1}(x) \subseteq Z_n(x)$). We define
$$h : X \ra X', \quad x \mapsto (e_0(x),e_1(x),e_2(x),\dots).$$
We show that $h$ is a homeomorphism, and that $h \circ \sigma = \sigma' \circ h$.
\begin{enumerate}[(i),leftmargin=0.6cm]
\item Suppose $x,y \in X$ are such that $e_n(x) = e_n(y)$ for all $n \geq 0$. Then
$$\{x\} = \bigcap_{n \geq 0} Z_n(x) = \bigcap_{n \geq 0} r(e_n(x)) = \bigcap_{n \geq 0} r(e_n(y)) = \bigcap_{n \geq 0} Z_n(y) = \{y\},$$
i.e. $x = y$. This proves injectivity of $h$.
\item Let $e = (e_0,e_1,e_2,\dots) \in X'$. Putting $Z_n = r(e_n)$, we have a decreasing sequence of non-empty clopen sets $Z_0 \supseteq Z_1 \supseteq Z_2 \supseteq \cdots$. Therefore $\bigcap_{n \geq 0} Z_n \neq \emptyset$ by Cantor's Intersection Theorem, and by the assumption that $\text{diam}(Z_n) \xra{n \ra \infty} 0$, it must consist of a single point $\{x\}$. Then $x \in Z_n$ for all $n \geq 0$, so $e_n(x) = e_n$. This proves surjectivity of $h$.
\end{enumerate}
Thus $h$ is bijective. Since $X,X'$ are compact and Hausdorff, in order to prove that $h$ is a homeomorphism it is enough to prove continuity of $h$.
\begin{enumerate}[(i),start=3,leftmargin=0.6cm]
\item Let $x \in h^{-1}(Z(e_0,e_1,\dots,e_n))$, so that $e_i(x) = e_i$ for $0 \leq i \leq n$. That means $x \in Z_0(x),\dots,Z_n(x),Z_{n+1}(x)$, being $Z_i(x) = r(e_i(x))$ and $Z_{n+1}(x) = s(e_n(x))$. But then for any $y \in \bigcap_{i=0}^{n+1} Z_i(x) = Z_{n+1}(x)$ we have that $e_i(y) = e_i(x) = e_i$ for $0\le i \le n$, i.e. $h(y) \in Z(e_0,e_1,\dots,e_n)$. Therefore $x \in Z_{n+1}(x) \subseteq h^{-1}(Z(e_0,e_1,\dots,e_n))$. This shows continuity of $h$.
\end{enumerate}
Hence $h$ is a homeomorphism. To conclude the theorem, we need to prove that $h$ intertwines the surjective local homeomorphisms $\sigma$ and $\sigma'$.
\begin{enumerate}[(i),start=4,leftmargin=0.6cm]
\item Let $x \in X$. Let also $Z_n(\sigma(x)) \in \calP_n$ be such that $\{\sigma(x)\} = \bigcap_{n \geq 0} Z_n(\sigma(x))$, and consider the blue edges $e_n := e_n(\sigma(x))$, so that $r(e_n) = Z_n(\sigma(x))$ and $s(e_n) = Z_{n+1}(\sigma(x))$.
\noindent Let $\{Z_n(x)\}_n$ and $\{e_n(x)\}_n$ be such that $\{x\} = \bigcap_{n \geq 0} Z_n(x), r(e_n(x)) = Z_n(x)$ and $s(e_n(x)) = Z_{n+1}(x)$.

For each $j \geq 0$, we complete the pair of blue edges $(e_{2j}(x),e_{2j+1}(x))$ to a romb $(e_{2j}(x),e_{2j+1}(x),f_{2j},f_{2j+1})$. For $j \geq 1$ we have $s(f_{2j-1}) = r(f_{2j})$, so we may consider the red pair $(f_{2j-1},f_{2j})$ and complete it to another romb $(e_{2j-1}',e_{2j}',f_{2j-1},f_{2j})$. There exists then a unique blue edge $e_0'$ such that $s(e_0') = r(e_1')$. Then by construction of $\sigma'$ we have $\sigma'(h(x)) = (e_0',e_1',e_2',\dots)$. Let $Z_n' := r(e_n')$ for $n \geq 0$. This is a decreasing sequence of clopen subsets $Z_n' \in \calP_n$. Also, by definition, we have
$$\sigma(Z_{2j+2}(x)) \subseteq Z'_{2j+1} \subseteq \sigma(Z_{2j}(x))$$
for all $j \geq 0$. In particular, $\sigma(x) \in \sigma(Z_{2j+2}(x)) \subseteq Z_{2j+1}'$, so by uniqueness $Z_{2j+1}' = Z_{2j+1}(\sigma(x))$. On the other hand, note that
$$Z_{2j}(\sigma(x)) \cap Z_{2j}' \supseteq Z_{2j+1}(\sigma(x)) \cap Z_{2j+1}' = Z_{2j+1}' \neq \emptyset$$
for all $j \geq 0$. By uniqueness again, $Z_{2j}' = Z_{2j}(\sigma(x))$. We have thus concluded that $Z_n' = Z_n(\sigma(x))$ for all $n \geq 0$, and so
$$\sigma'(h(x)) = (e_0',e_1',e_2',\dots) = h(\sigma(x)),$$
as we wanted to show.
\end{enumerate}
This concludes the proof of the theorem.
\end{proof}

\section{Homeomorphisms}\label{section-homeos}

If we restrict our attention to homeomorphisms $\varphi : X \ra X$ of a totally disconnected, compact metrizable space $X$, then by restricting the correspondence from Theorem \ref{theorem-equivalence} we obtain a bijective correspondence between such dynamical systems and a subfamily of the family \textsf{LDiag}, which we call \textsf{HDiag}. In this section we describe such correspondence. We first define the class \textsf{HDiag}.

\begin{definition}\label{definition-H.separated.Bratteli}
Let $(F,D)$ be a separated Bratteli diagram. We say that $(F,D)$ is an \textit{$h$-diagram} if the following hold:
\begin{enumerate}[(a),leftmargin=0.7cm]
\item (Separation for the layers) For each $n \geq 0$ and $v \in F^{0,n}$, we have
$$D_v= \{ B_v,R_v \}.$$
The elements of $B_v$ will be called \textit{blue edges}, and the elements of $R_v$ \textit{red edges}.
\item (Vertices) For each $n \geq 0$,
$$F^{0,n+1} = \bigsqcup_{v\in F^{0,n}} s(B_v) = \bigsqcup _{v\in F^{0,n}} s(R_v) .$$
Moreover, for any $v \in F^{0,n}$, we have $s(e) \neq s(e')$ for all distinct blue edges $e,e' \in B_v$, and similarly $s(f) \neq s(f')$ for all distinct red edges $f,f' \in R_v$.
\item (Compatibility) For each $n \ge 0$ and $v \in F^{0,n}$, we have 
$$\bigsqcup _{e \in B_v} s(B_{s(e)}) = \bigsqcup_{f \in R_v} s(R_{s(f)}).$$
\end{enumerate}
The family of $h$-diagrams will be denoted by \textsf{HDiag}.
\end{definition}

\begin{remarks}\label{remarks-edges.unique}
\begin{enumerate}[(1),leftmargin=0.7cm]
\item Condition (b) in Definition \ref{definition-H.separated.Bratteli} implies that, for each $n \geq 1$ and $v \in F^{0,n}$, there are exactly two different edges $e \in B_{r(e)}, f \in R_{r(f)}$ such that $s(e) = s(f) = v$.
\item The disjointness of both unions in condition (c) of Definition \ref{definition-H.separated.Bratteli} is automatic, due to condition (b).
\end{enumerate}
\end{remarks}

The class \textsf{HDiag} is a subclass of the diagrams considered in Section \ref{section-correspondence.local.homeo}. To prove this, we need the following technical lemma, which enables us to build rombs at any layer of an $h$-diagram $(F,D)$.

\begin{lemma}\label{lemma-building.rombs}
Let $v \in F^{0,n}$. Given a pair of blue edges $(e_0,e_1)$ such that $r(e_0) = v$ and $r(e_1)=s(e_0)$, there is a unique pair $(f_0,f_1)$ of red edges such that $r(f_0)= v$, $r(f_1)=s(f_0)$ and $s(f_1)= s(e_1)$.

Conversely, given a pair of red edges $(f_0,f_1)$ such that $r(f_0) = v$ and $r(f_1)=s(f_0)$, there is a unique pair $(e_0,e_1)$ of blue edges such that $r(e_0)= v$, $r(e_1)=s(e_0)$ and $s(e_1)= s(f_1)$.
\end{lemma}
\begin{proof}
By condition (c) in Definition \ref{definition-H.separated.Bratteli} we have
$$s(e_1) \in \bigsqcup_{e \in B_v} s(B_{s(e)}) = \bigsqcup_{f \in R_v} s(R_{s(f)}),$$
so there exist a pair of red edges $(f_0,f_1)$ such that $s(e_1) = s(f_1)$, being $f_1 \in R_{s(f_0)}$ and $f_0 \in R_v$. The pair $(f_0,f_1)$ satisfies the required properties. To prove uniqueness, assume that there exists another pair $(f'_0,f'_1)$ of red edges such that $r(f'_0)=v$, $s(f'_0)=r(f'_1)$ and $s(f'_1) = s(e_1) = s(f_1)$. Condition (b) in Definition \ref{definition-H.separated.Bratteli} ensures that $f'_1 = f_1$; in particular $s(f'_0) = r(f'_1) = r(f_1) = s(f_0)$. Applying condition (b) again gives $f'_0 = f_0$. That is, the pair $(f_0,f_1)$ is unique.

The same arguments apply if one starts with a pair of red edges to conclude existence and uniqueness of a pair of blue edges satisfying the desired properties. 
\end{proof}

\begin{proposition}\label{proposition-hdiag.ldiag}
Every $h$-diagram is an $l$-diagram.
\end{proposition}
\begin{proof}
Indeed, if $(F,D)$ is an $h$-diagram, it is also an $l$-diagram of a special kind, namely the separations for the odd layers consist only of two colors. We need to check conditions (a)-(g) of Definition \ref{definition-L.separated.Bratteli}.

\noindent (a) The separations for the even layers are what they should be by definition.

\noindent (b) By Remark \ref{remarks-edges.unique} (1), for each $j \geq 0$ and $v \in F^{0,2j+1}$ there is exactly one red edge $f_v \in R_{r(f_v)}$ with source $v$. Therefore here $R(f_v) = R_v$, and we have the separation
$$D_v = \{B_v,R_v\} = \{B_v,R(f_v)\}.$$

\noindent (c) Follows directly from condition (b) in Definition \ref{definition-H.separated.Bratteli}.

\noindent (d) Parts (1) and (2) follow directly from condition (b) in Definition \ref{definition-H.separated.Bratteli}. Also, the second part follows from the disjointness of the sets in the same condition (b) and from the second part of that same condition.

\noindent (e) Follows directly from condition (c) in Definition \ref{definition-H.separated.Bratteli}.

\noindent (f)-(g) They follow from Lemma \ref{lemma-building.rombs} above.
\end{proof}

Note that the equivalence relation $\sim_l$ in \textsf{LDiag} restricts to an equivalence relation $\sim_h$ in \textsf{HDiag}. More concretely, the contraction of an $h$-diagram is again an $h$-diagram.

\begin{lemma}\label{lemma-F'D'.H.diagram}
The contraction $(F',D')$ of an $h$-diagram, as defined in Definition \ref{definition-equiv.telesc.l.diagrams}, is an $h$-diagram.
\end{lemma}
\begin{proof}
First observe that if $(F,D)$ is an $h$-diagram, then the contraction $(F',D')$ with respect to a sequence $(m_n)_{n \geq 0}$ in $\N \cup \{0\}$ with $0 \leq m_0 < m_1 < \cdots$, $m_0$ even and satisfying (CR) is given by the following data.
\begin{enumerate}[(i),leftmargin=0.6cm]
\item $(F')^{0,n} = F^{0,m_n}$ for $n \geq 0$.
\item $(F')^{1,n} = P_{m_n,m_{n+1}}^B \cup P_{m_n,m_{n+1}}^R$ for $n \geq 0$, where for a given pair $j < l$ of non-negative integers, we have
$$P_{j,l}^B = \{(e_j,\dots,e_{l-1}) \in P_{j,l} \mid e_i \in B_{r(e_i)}, j \leq i \leq l-1\},$$
$$P_{j,l}^R = \{(f_j,\dots,f_{l-1}) \in P_{j,l} \mid f_i \in R_{r(e_i)}, j \leq i \leq l-1\}.$$
\item For $v \in (F')^{0,n}$ with $n \geq 0$, $D_v' = \{B_v',R_v'\}$ being
$$B_v' = \{\ol{e} \in P_{m_n,m_{n+1}}^B \mid r(\ol{e}) = v\}, \quad R_v' = \{\ol{f} \in P_{m_n,m_{n+1}}^R \mid r(\ol{f}) = v\}.$$
\end{enumerate}
Now the proof of the lemma follows identical arguments as in the proof of Lemma \ref{lemma-F'D'.L.diagram}, so we omit it.
\end{proof}

We now describe the correspondence between homeomorphisms on a totally disconnected, compact metrizable space and $h$-diagrams. For ease of notation, we let \textsf{Homeo} be the family of dynamical systems $(X,\varphi)$ being $X$ a totally disconnected, compact metrizable space, and $\varphi : X \ra X$ a homeomorphism. Note that \textsf{Homeo} is a subfamily of \textsf{LHomeo}.

\begin{theorem}\label{theorem-homeo.H.diag}
Let $(X,\varphi) \in \emph{\textsf{Homeo}}$. Then the resulting $l$-diagram $(F,D)$ obtained from $(X,\varphi)$ by applying the construction \ref{construction-l.diag} is in fact an $h$-diagram.
\end{theorem}
\begin{proof}
Take $\{\calP_n\}_{n \geq 0}$ to be any $\varphi$-refined sequence of partitions of $X$. We know that $(F,D)$ so obtained is an $l$-diagram by Lemma \ref{lemma-FD.is.Ldiagram}. We recall its construction.

For $n \geq 0$ we have $F^{0,n} = \calP_n$, and $F^{1,n} = B^{(n)} \cup R^{(n)}$, where the set $B^{(n)} = \bigsqcup_{Z \in F^{0,n}} B_Z$ consists of blue edges, and $R^{(n)} = \bigsqcup_{Z \in F^{0,n}}R_Z$ consists of red edges. More precisely, for $Z \in F^{0,n}$,
$$B_Z = \{e(Z',Z) \mid Z' \in \calP_{n+1}\text{ such that } Z' \subseteq Z\}$$
and if $j \geq 0$, then
$$R_Z = \{f(Z',Z) \mid Z' \in \calP_{2j+1}\text{ such that } Z' \subseteq \varphi(Z)\}$$
for $Z \in \calP_{2j}$, and
$$R_{Z'} = \{f(Z'',Z') \mid Z'' \in \calP_{2j+2}\text{ such that } Z'' \subseteq \varphi^{-1}(Z')\}$$
for $Z' \in \calP_{2j+1}$.

For $j \geq 0$, the separations are given by
$$D_Z = \{B_Z,R_Z\}$$
in case $Z \in \calP_{2j}$, and by
$$D_{Z'} = \{B_{Z'}\} \cup \{R(f) \mid s(f) = Z', f \in R_{r(f)}\}$$
in case $Z' \in \calP_{2j+1}$.

Observe that, since $\varphi$ is a homeomorphism,
$$\calP_n^{\varphi} = \{\varphi(Z) \mid Z \in \calP_n\}.$$
Hence for $Z \in \calP_{2j}, Z' \in \calP_{2j+1}$, we have $Z' \subseteq \varphi(Z)$ if and only if $\varphi^{-1}(Z') \subseteq Z$. This means precisely that there is exactly one red edge with source $Z'$, namely $f = f(Z',Z) \in R_Z$. Hence
$$R_{Z'} = R(f)$$
and so the separation at $Z'$ becomes
$$D_{Z'} = \{B_{Z'},R_{Z'}\}.$$

This shows condition (a) in Definition \ref{definition-H.separated.Bratteli}. The first part of condition (b) in Definition \ref{definition-H.separated.Bratteli} follows from condition (c) in Definition \ref{definition-L.separated.Bratteli} together with Proposition \ref{proposition-decomposition.odd.rombs.even} (i) (note that the disjointness of the second union in Proposition \ref{proposition-decomposition.odd.rombs.even} (i) follows in this case by the observation in the above paragraph). The second part of condition (b) follows by the definition of the edges in $(F,D)$, and by the observation in the paragraph above. 

For condition (c) in Definition \ref{definition-H.separated.Bratteli}, we note that by Remark \ref{remarks-edges.unique} (2) it is enough to show that
$$\bigcup_{e \in B_Z} s(B_{s(e)}) = \bigcup_{f \in R_Z}s(R_{s(f)})$$
for each $n \geq 0$ and $Z \in \calP_n$. If $n$ is even, this follows directly from condition (e) of Definition \ref{definition-L.separated.Bratteli}, so we may assume that $n = 2j+1, j \geq 0$.

Let $Z'' \in s(B_{s(e)})$ for some $e \in B_Z$. Since $\calP_{2j+2}^{\varphi} = \varphi(\calP_{2j+2})$ is a partition which is coarser than $\calP_{2j+3}$, we can find $Z' \in \calP_{2j+2}$ such that $Z'' \subseteq \varphi(Z')$. This says that $Z'' \in s(R_{Z'})$. But $\varphi^{-1}(\calP_{2j+1})$ is coarser than the partition $\calP_{2j+2}$, so there exists $Z \in \calP_{2j+1}$ such that $Z' \subseteq \varphi^{-1}(Z)$. Hence $Z'' = s(f')$ with $f' = f(Z'',Z') \in R_{s(f)}$, where $f = f(Z',Z)$. This proves one inclusion; the other is analogous and we will omit it.

This concludes the proof.
\end{proof}

\begin{theorem}\label{theorem-H.diag.homeo}
Let $(F,D) \in \emph{\textsf{HDiag}}$. Then the resulting dynamical system $(X,\varphi) \in \emph{\textsf{LHomeo}}$ obtained from $(F,D)$ by applying the construction \ref{construction-local.homeo} belongs in fact to $\emph{\textsf{Homeo}}$.
\end{theorem}
\begin{proof}
We recall the construction of $(X,\varphi)$. The space $X$ is defined as the set of infinite blue paths
$$X = \{(e_0,e_1,e_2,\dots) \mid e_j \in B_{r(e_j)}, r(e_{j+1}) = s(e_j) \text{ for all } j \geq 0\}$$
endowed with a topology having as basis the cylinder sets $\{Z(e_0,\dots,e_n) \mid (e_0,\dots,e_n) \in P_{0,n}^B, n \geq 0\}$, which become clopen sets. The map $\varphi : X \ra X$ is defined as follows.

Let $(e_0,e_1,e_2,\dots) \in X$. For each $j \geq 0$ we complete the pair of blue edges $(e_{2j},e_{2j+1})$ to a romb $(e_{2j},e_{2j+1},f_{2j},f_{2j+1})$. Now, for each $j \geq 1$, we consider the red pair $(f_{2j-1},f_{2j})$ which satisfies $s(f_{2j-1}) = r(f_{2j})$, so we can complete it to another romb $(e_{2j-1}',e_{2j}',f_{2j-1},f_{2j})$. Finally, take $e_0'$ to be the unique blue edge in $F^{1,0}$ such that $s(e_0') = r(e_1')$. Then
$$\varphi(e_0,e_1,e_2,\dots) = (e_0',e_1',e_2',\dots).$$
We already know (Lemmas \ref{lemma-infinite.paths.compact.L} and \ref{construction-local.homeo}) that $X$ is a totally disconnected, compact metrizable space, and that $\varphi$ is a surjective local homeomorphism. In this particular case, we can explicitly construct the inverse of $\varphi$ by reversing the above process. Let us detail its construction.

Take $(e_0',e_1',e_2',\dots) \in X$. By Remarks \ref{remarks-edges.unique} there exists a unique red edge $f_0 \in R_{r(f_0)}$ such that $s(f_0) = s(e_0')$. By recursively applying Lemma \ref{lemma-building.rombs}, we construct a sequence of red edges $f_1, f_2,\dots$ such that, for all $j \geq 0$, $f_{2j+1} \in R_{s(f_{2j})}$ and the quadruple $(e_{2j+1}',e_{2j+2}',f_{2j+1},f_{2j+2})$ is a romb. We use again Lemma \ref{lemma-building.rombs} to build pairs of blue edges $(e_{2j},e_{2j+1})$ such that $(e_{2j},e_{2j+1},f_{2j},f_{2j+1})$ is a romb, for all $j \geq 0$. Define
$$\varphi^{-1}(e_0',e_1',e_2',\dots) = (e_0,e_1,e_2,\dots).$$
An easy inspection shows that $\varphi^{-1}$ is the inverse of $\varphi$ (the key part is the uniqueness property in Lemma \ref{lemma-building.rombs}). Since $\varphi$ was already a local homeomorphism, it must be a homeomorphism.

This proves the theorem.
\end{proof}

Putting everything together, we obtain the following bijective correspondence between homeomorphisms and $h$-diagrams.

\begin{theorem}\label{theorem-equivalence.homeo}
The constructions \ref{construction-l.diag} and \ref{construction-local.homeo} also establish a bijective correspondence between equivalence classes of $h$-diagrams (via isomorphism and contraction) and topological conjugacy classes of homeomorphisms on totally disconnected, compact metrizable spaces.
\end{theorem}
\begin{proof}
Note that, by Lemma \ref{lemma-F'D'.H.diagram}, $\textsf{HDiag}$ is a saturated class in $\textsf{LDiag}$, in the sense that $(F',D')\sim_l (F,D)$ and $(F,D)\in \textsf{HDiag}$ imply that $(F',D')\in \textsf{HDiag}$. 
Similarly, $\textsf{Homeo}$ is a saturated class in $\textsf{LHomeo}$ with respect to $\sim_{t.c.}$. Hence the result follows directly from Theorems \ref{theorem-equivalence}, \ref{theorem-homeo.H.diag} and \ref{theorem-H.diag.homeo}. 
\end{proof}

\begin{example}\label{example-cantor.3}
The example discussed in Examples \ref{examples-cantor.1} 1) and \ref{examples-cantor.2} 1) falls under the umbrella of this section, since in this case $\varphi$ is a homeomorphism. Therefore the $l$-diagram obtained in Examples \ref{examples-cantor.2} 1) gives in fact an example of an $h$-diagram.
\end{example}

\section{Generalizing shifts of finite type}\label{section-general.shifts.finite.type}

In this section we introduce the key concept of a generalized finite shift, a class of dynamical systems
that generalizes both one-sided and two-sided shifts of finite type. We also show that any $(X,\sigma)\in \textsf{LHomeo}$ is an inverse limit of generalized finite shifts.

\subsection{A quick review on shifts of finite type}\label{subsection-shifts.finite.type}

It is well-known that shifts of finite type are codified by finite graphs (see e.g. \cite{LM2021}). We recall the construction here.

Let $E = (E^0,E^1,r_E,s_E)$ be a finite graph, with adjacency matrix $A = A_E$. We will assume that $E^0 = \{1,\dots , n\}$. Here the adjacency matrix $A_E = (a_{ji})$ of $E$ is the square $n \times n$ matrix such that $a_{ji}$ is the number of edges from vertex $j$ to vertex $i$. Then we may build the associated space
$$X_E = \{(e_0,e_1,e_2,\dots) \mid e_i \in E^1 \text{ and } s_E(e_i) = r_E(e_{i+1}) \text{ for all } i \geq 0\}$$
with the topology generated by the cylinder sets $Z(e_0,\dots,e_n)$. In this space we may define the shift map
$$\sigma_E \colon X_E \to X_E$$
by $\sigma_E(e_0,e_1,e_2,\dots) = (e_1,e_2,\dots)$. Note that $X_E$ is a closed invariant subset of the compact space $\prod_{n \geq 0}E^1$, and that it is defined by the \textit{forbidden words} $ef$ for all $e,f \in E^1$ such that $s(e) \ne r(f)$, so that $(X_E,\sigma_E)$ is a \textit{one-sided shift of finite type}. An analogous definition gives the two-sided shift, which is a homeomorphism.

Assuming that $E$ has neither sinks nor sources (so that the matrix $A_E$ has no zero-rows and no zero-columns), we have that $\sigma_E$ is a surjective local homeomorphism.

For $i \in E^0$, write $V_i = \{(e_0, e_1,e_2,\dots)\in X_E \mid r_E(e_0) = i \}$. Let also, for $i,j \in E^0$, $\{f_{ji}^{(k)} \mid 1 \leq k \leq a_{ji}\}$ be the set of edges with initial vertex $j$ and terminal vertex $i$. Set
$$V_{ji}^{(k)} = Z(f_{ji}^{(k)}) = \{(e_0,e_1,e_2,\dots ) \in X_E \mid e_0 = f_{ji}^{(k)}\}.$$
Then $V_i = \bigsqcup_{j=1}^n \bigsqcup_{k=1}^{a_{ji}} V_{ji}^{(k)}$, and $\sigma_E(V_{ji}^{(k)}) = V_j$ for all $i,j,k \in E^0$. Observe that the restriction of $\sigma_E$ to each subset $V_{ji}^{(k)}$ is injective. It turns out that all the dynamical behavior of the one-sided shift $(X_E,\sigma_E)$ is completely determined by these decompositions.




\subsection{Generalized finite shifts}\label{subsection-generalized.finite.shifts}

We generalize the above situation as follows.

\begin{definition}\label{definition-sigma.cylinder.dec}
Let $(X,\sigma) \in \textsf{LHomeo}$. We say that $(X,\sigma)$ admits a \textit{$\sigma$-cylinder decomposition} if there exist
\begin{enumerate}[(I),leftmargin=0.7cm]
\item a finite set $\Gamma = \bigsqcup_{i=1}^{n} \Gamma_i$ where each $\Gamma_i \neq \emptyset$;
\item a finite collection $\{Z_{\gamma} \mid \gamma \in \Gamma\}$ of non-empty compact open, mutually disjoint subsets $Z_{\gamma} \subseteq X$;
\item non-negative integers $\{a_{\gamma,i} \mid \gamma \in \Gamma, 1 \leq i \leq n\}$; and
\item a finite collection $\{V_{\gamma, i}^{(k)} \mid \gamma \in \Gamma, 1 \leq i \leq n, 1 \leq k \leq a_{\gamma,i}\}$ of non-empty compact open sets;
\end{enumerate}
such that the following properties hold:
\begin{enumerate}[(a),leftmargin=0.7cm]
\item $X = \bigsqcup_{\gamma \in \Gamma} Z_{\gamma}$.
\item For each $i = 1,\dots,n$, if we write $V_i := \bigsqcup_{\gamma\in \Gamma_i} Z_{\gamma}$, then
$$V_i = \bigsqcup_{\gamma \in \Gamma} \bigsqcup_{k=1}^{a_{\gamma, i}} V_{\gamma, i}^{(k)}.$$
\item $\sigma|_{V_{\gamma, i}^{(k)}}$ is injective and $\sigma(V_{\gamma, i}^{(k)})= Z_{\gamma}$ for all allowable values of $\gamma,i,k$.
\end{enumerate}
The pair $\calZ := (\{Z_{\gamma}\},\{V_{\gamma,i}^{(k)}\})$ will be called a \textit{$\sigma$-cylinder decomposition} of $X$.
\end{definition}

Some immediate consequences of Definition \ref{definition-sigma.cylinder.dec} are the following. Since $V_i \ne \emptyset$ for $1\leq i \leq n$, we have that for each $i = 1,\dots,n$ there exists $\gamma \in \Gamma$ such that $a_{\gamma, i} \ne 0$. Moreover, since $\sigma$ is surjective, we have
$$X = \sigma\Big(\bigcup_{i=1}^n V_i\Big) = \bigcup_{i=1}^n \bigcup_{\gamma\in \Gamma} \bigcup_{k=1}^{a_{\gamma, i}} \sigma(V_{\gamma, i}^{(k)}) = \bigcup_{i=1}^n \bigcup_{\gamma\in \Gamma} \bigcup_{k=1}^{a_{\gamma, i}} Z_{\gamma}  = \bigcup_{\gamma\in \Gamma} \Big( \bigcup_{i=1}^n\bigcup_{k=1}^{a_{\gamma, i}} Z_{\gamma}\Big).$$
Since also $X = \bigsqcup_{\gamma\in \Gamma} Z_{\gamma}$, we have conversely that for each $\gamma \in \Gamma $ there exists $i \in \{ 1,\dots,n \}$ such that $a_{\gamma, i} \ne 0$. These observations show that all rows and columns of the matrix $A := (a_{\gamma, i}) \in M_{\Gamma \times n}(\Z^+)$ are non-zero.

Let also $I \in M_{\Gamma \times n}(\Z^+)$ be the matrix with all entries in $\{0,1\}$ such that in the $i^{th}$ column it has $1$'s exactly at the positions that belong to $\Gamma_i$, for $1 \leq i \leq n$. The pair $(A,I)$ will be called the pair of \textit{red-blue adjacency matrices}, respectively, of a prescribed $\sigma$-cylinder decomposition $\calZ$ of $X$.\\

We consider also a notion of \textit{refined} $\sigma$-cylinder decompositions of $X$. Basically this amounts to have refinements of the two partitions $\{Z_{\gamma}\}$ and $\{V_{\gamma,i}^{(k)}\}$ satisfying a compatibility condition, as follows.

\begin{definition}\label{definition-refined.sigma.cylinder.dec}
Suppose $(X,\sigma) \in \textsf{LHomeo}$. Let $\calZ = (\{Z_{\gamma}\}, \{V_{\gamma,i}^{(k)}\}), \calT = (\{T_{\lambda}\}, \{W_{\lambda,j}^{(l)}\})$ be two $\sigma$-cylinder decompositions; in more detail, suppose that
$$X = \bigsqcup_{\gamma \in \Gamma} Z_{\gamma} = \bigsqcup_{\lambda \in \Lambda} T_{\lambda}, \qquad \Gamma = \bigsqcup_{i=1}^n \Gamma_i, \qquad \Lambda = \bigsqcup_{j=1}^m \Lambda_j,$$
and that, with $V_i = \bigsqcup _{\gamma \in \Gamma_i} Z_{\gamma}$ and $W_j = \bigsqcup _{\lambda \in \Lambda_j} T_{\lambda}$, we have
$$V_i = \bigsqcup_{\gamma\in \Gamma} \bigsqcup_{k=1}^{a_{\gamma, i}} V_{\gamma, i}^{(k)}\quad \text{ and }\quad W_j = \bigsqcup_{\lambda\in \Lambda} \bigsqcup_{l=1}^{b_{\lambda, j}} W_{\lambda, j}^{(l)},$$
such that $\sigma|_{V_{\gamma, i}^{(k)}}$ is injective and $\sigma (V_{\gamma, i}^{(k)})= Z_{\gamma}$, and $\sigma|_{W_{\lambda, j}^{(l)}}$ is injective  and $\sigma (W_{\lambda,j}^{(l)})= T_{\lambda}$ for all allowable values of $\gamma,i,k$ and $\lambda,j,l$ respectively.

Then we say that $\calT$ \textit{refines} $\calZ$ if:
\begin{enumerate}[(a),leftmargin=0.7cm]
\item The partition $\{T_{\lambda}\}$ is a refinement of $\{Z_{\gamma}\}$, that is, we can write
$$Z_{\gamma}= \bigsqcup_{\lambda \in \Lambda (\gamma)} T_{\lambda}, \quad \text{ with } \Lambda = \bigsqcup_{\gamma\in \Gamma} \Lambda (\gamma).$$
\item The partition $\{W_{\lambda,j}^{(l)}\}$ is a refinement of $\{V_{\gamma,i}^{(k)}\}$, that is, with
$$\mathfrak A = \{(\gamma,i,k) \mid \gamma\in \Gamma, 1 \le i \le n, 1 \le k \le a_{\gamma,i}\},\quad \mathfrak B = \{(\lambda,j,l) \mid \lambda \in \Lambda, 1\le j\le m, 1\le l\le b_{\lambda,j} \},$$
we have
$$V_{\gamma,i}^{(k)} = \bigsqcup_{\eta \in \mathfrak B (\gamma,i,k)} W(\eta), \quad \text{ with }  \mathfrak B = \bigsqcup_{\nu\in \mathfrak A} \mathfrak B (\nu),$$
where $W(\eta) = W_{\lambda,j}^{(l)}$ whenever $\eta = (\lambda,j,l) \in \mathfrak B$.
\item The following {\it compatibility condition} is satisfied: let $\pi_{\Lambda } \colon \mathfrak B \to \Lambda$ be the projection onto the first component. Then for each $(\gamma,i,k)\in \mathfrak A$, $\pi_{\Lambda}$ restricts to a bijection from $\mathfrak B (\gamma,i,k)$ onto $\Lambda (\gamma)$.
\end{enumerate} 
\end{definition}

\begin{remark}\label{remark-explanation.of.compatibility}
To see the meaning of the compatibility condition in the above definition, let us apply $\sigma$ to the decomposition $V_{\gamma,i}^{(k)} = \bigsqcup _{\eta \in \mathfrak B (\gamma,i,k)} W(\eta)$ for a given $(\gamma,i,k )\in \mathfrak A$. Since $\sigma|_{V_{\gamma,i}^{(k)}}$ is injective, we have
	$$ Z_{\gamma} = \sigma (V_{\gamma,i}^{(k)}) = \bigsqcup_{\eta\in \mathfrak B (\gamma,i,k)} \sigma (W(\eta)) = \bigsqcup_{\eta\in \mathfrak B (\gamma,i,k)}  T_{\pi_{\Lambda}(\eta)}.$$
Hence the compatibility condition (c) says that this decomposition agrees with the decomposition of $Z_{\gamma}$ from (a).
\end{remark}

We immediately observe transitivity of refinement, which we prove in the next lemma.

\begin{lemma}\label{lemma-refinement.transitivity}
Let $\calZ, \calT$ and $\calS$ be three $\sigma$-cylinder decompositions for $(X,\sigma)$. If $\calS$ refines $\calT$ and $\calT$ refines $\calZ$, then $\calS$ refines $\calZ$.
\end{lemma}
\begin{proof}
Write $\calZ = (\{Z_{\gamma}\}, \{V_{\gamma,i}^{(k)}\}), \calT = (\{T_{\lambda}\}, \{W_{\lambda,j}^{(l)}\})$ and $\calS = (\{S_{\delta}\}, \{ U_{\delta,r}^{(s)}\})$. In more detail,
$$X = \bigsqcup_{\gamma \in \Gamma} Z_{\gamma} = \bigsqcup_{\lambda \in \Lambda} T_{\lambda} = \bigsqcup_{\delta \in \Delta} S_{\delta},$$
where $\Gamma = \bigsqcup_{i=1}^n \Gamma_i, \Lambda = \bigsqcup_{j=1}^m \Lambda_j$ and $\Delta = \bigsqcup_{r=1}^{o} \Delta_r$, and that, with $V_i = \bigsqcup_{\gamma \in \Gamma_i} Z_{\gamma}, W_j = \bigsqcup_{\lambda \in \Lambda_j} T_{\lambda}$ and $U_r = \bigsqcup_{\delta \in \Delta_r} S_{\delta}$, we have
$$V_i = \bigsqcup_{\gamma\in \Gamma} \bigsqcup_{k=1}^{a_{\gamma, i}} V_{\gamma, i}^{(k)}, \quad W_j = \bigsqcup_{\lambda\in \Lambda} \bigsqcup_{l=1}^{b_{\lambda, j}} W_{\lambda, j}^{(l)} \quad \text{ and } \quad U_r = \bigsqcup_{\delta \in \Delta} \bigsqcup_{s=1}^{c_{\delta,r}} U_{\delta,r}^{(s)},$$
such that $\sigma|_{V_{\gamma,i}^{(k)}}$ is injective and $\sigma(V_{\gamma, i}^{(k)}) = Z_{\gamma}$, $\sigma|_{W_{\lambda,j}^{(l)}}$ is injective  and $\sigma(W_{\lambda,j}^{(l)})= T_{\lambda}$, and $\sigma|_{U_{\delta,r}^{(s)}}$ is injective and $\sigma(U_{\delta,r}^{(s)})= S_{\delta}$ for all allowable values of $\gamma,i,k$, of $\lambda,j,l$ and of $\delta,r,s$ respectively.

By hypothesis, we can write
$$Z_{\gamma} = \bigsqcup_{\lambda \in \Lambda(\gamma)} T_{\lambda}, \quad \text{ with } \Lambda = \bigsqcup_{\gamma \in \Gamma} \Lambda(\gamma),$$
and
$$T_{\lambda} = \bigsqcup_{\delta \in \Delta(\lambda)} S_{\delta}, \quad \text{ with } \Delta = \bigsqcup_{\lambda \in \Lambda} \Delta(\lambda).$$
Then if we write $\Delta(\gamma) := \bigsqcup_{\lambda \in \Lambda(\gamma)} \Delta(\lambda)$ for $\gamma \in \Gamma$, we obtain that $\Delta = \bigsqcup_{\gamma \in \Gamma} \Delta(\gamma)$ and
$$Z_{\gamma} = \bigsqcup_{\lambda \in \Lambda(\gamma)} \bigsqcup_{\delta \in \Delta(\lambda)} S_{\delta} = \bigsqcup_{\delta \in \Delta(\gamma)} S_{\delta}.$$
This proves part (a) of Definition \ref{definition-refined.sigma.cylinder.dec}. Part (b) is proved similarly, as follows. By hypothesis, if
$$\mathfrak A = \{(\gamma,i,k) \mid \gamma\in \Gamma, 1 \le i \le n, 1 \le k \le a_{\gamma,i}\},$$
$$\mathfrak B = \{(\lambda,j,l) \mid \lambda \in \Lambda, 1\le j\le m, 1\le l\le b_{\lambda,j}\},$$
$$\mathfrak C = \{(\delta,r,s) \mid \delta \in \Delta, 1\leq r\le o, 1\le s\le c_{\delta,r}\},$$
we have
$$V_{\gamma,i}^{(k)} = \bigsqcup_{\eta \in \mathfrak B(\gamma,i,k)} W(\eta), \quad \text{ with }  \mathfrak B = \bigsqcup_{\nu \in \mathfrak A} \mathfrak B(\nu),$$
where $W(\eta) = W_{\lambda,j}^{(l)}$ whenever $\eta = (\lambda,j,l) \in \mathfrak B$, and
$$W_{\lambda,j}^{(l)} = \bigsqcup_{\rho \in \mathfrak C(\lambda,j,l)} U(\rho), \quad \text{ with }  \mathfrak C = \bigsqcup_{\eta \in \mathfrak B} \mathfrak C(\eta),$$
where $U(\rho) = U_{\delta,r}^{(s)}$ whenever $\rho = (\delta,r,s) \in \mathfrak C$. If we define, for $\nu \in \mathfrak A$, the set $\mathfrak C(\nu) := \bigsqcup_{\eta \in \mathfrak B(\nu)} \mathfrak C(\eta)$, then $\mathfrak C = \bigsqcup_{\nu \in \mathfrak A} \mathfrak C(\nu)$ and
$$V_{\gamma,i}^{(k)} = \bigsqcup_{\eta \in \mathfrak B(\gamma,i,k)} \bigsqcup_{\rho \in \mathfrak C(\eta)} U(\rho) = \bigsqcup_{\rho \in \mathfrak C(\gamma,i,k)} U(\rho).$$
This shows part (b), as promised. Finally it remains to prove part (c). Consider $\pi_{\Lambda} : \mathfrak B \ra \Lambda$ and $\pi_{\Delta} : \mathfrak C \ra \Delta$ the projections onto the respective first components. By hypothesis, for each $(\gamma,i,k) \in \mathfrak A$, the restriction of $\pi_{\Lambda}$ on $\mathfrak B(\gamma,i,k)$ gives a bijection $\mathfrak B(\gamma,i,k) \simeq \Lambda(\gamma)$ given by $(\lambda,j,l) \mapsto \lambda$. Also, for each $(\lambda,j,l) \in \mathfrak B$, the restriction of $\pi_{\Delta}$ on $\mathfrak C(\lambda,j,l)$ gives a bijection $\mathfrak C(\lambda,j,l) \simeq \Delta(\lambda)$ given by $(\delta,r,s) \mapsto \delta$.

So, for a fixed $(\gamma,i,k)$, we have
$$\mathfrak C(\gamma,i,k) = \bigsqcup_{(\lambda,j,l) \in \mathfrak B(\gamma,i,k)} \mathfrak C(\lambda,j,l)$$
which, under the above bijections, translates to a bijection
$$\mathfrak C(\gamma,i,k) \simeq \bigsqcup_{\lambda \in \Lambda(\gamma)} \Delta(\lambda) = \Delta(\gamma)$$
under the restriction of $\pi_{\Delta}$ on $\mathfrak C(\gamma,i,k)$. This concludes the proof of the lemma.
\end{proof}

We build a finite bipartite separated graph $(E,C) := (E(\calZ),C(\calZ))$ out of a $\sigma$-cylinder decomposition $\calZ$ of $(X,\sigma)$ as follows. Let $E^0 = E^{0,0} \sqcup E^{0,1}$, where $E^{0,0} = \{v^1,\dots,v^n\}$ and $E^{0,1} = \Gamma$. Now we set, for each $1 \le i \le n$,
$$C_i:=C_{v^i}= \{B_i, R_i\}$$
where $B_i = \{e_{\gamma, i} \mid \gamma \in \Gamma _i\}$, with $s(e_{\gamma,i}) = \gamma$ and $r(e_{\gamma, i}) = v^i$; and $R_i = \{f_{\gamma, i}^{(k)} \mid \gamma \in \Gamma, 1 \leq k \leq a_{\gamma, i}\}$, where $s(f_{\gamma,i}^{(k)}) = \gamma$ and $r(f_{\gamma, i}^{(k)}) = v^i$. Then we have that $(E,C)$ is a finite bipartite separated graph such that there are exactly two colors at each vertex in $E^{0,0}$, $s(e) \ne s(e')$ for all distinct $e \in B_i$ and $e' \in B_{i'}$, and
$$E^{0,1} = \bigsqcup_{i=1}^n s(B_i) = \bigcup_{i=1}^{n} s(R_i).$$
Therefore we see that $(E,C)$ satisfies all the conditions needed to be the first layer of an $l$-diagram.

Conversely, given an $l$-diagram $(F,D)$, we can build a $\sigma$-cylinder decomposition for its corresponding space $(X,\sigma) \in \textsf{LHomeo}$, such that the first layer of $(F,D)$ corresponds to the above construction applied to this $\sigma$-cylinder decomposition. Set $F^{0,0} =\{v^1,\dots , v^n\}$, and then $\Gamma = \bigsqcup_{i=1}^n \Gamma_i$, where $\Gamma_i= B_i:=B_{v^i}$.
Note that $\Gamma$ can be identified with the set $F^{0,1}$ by means of the correspondence $e\leftrightarrow s(e)$. We have
$$X = \bigsqcup_{e \in \Gamma} Z(e) \quad \text{ and } \quad V_i := Z(v^i) = \bigsqcup_{e \in \Gamma_i} Z(e).$$
It remains to build the sets $V_{e,i}^{(k)}$. For $e \in \Gamma$ and $i \in \{1,...,n\}$, let $a_{e,i}$ be the number of red edges from $s(e)$ to $v^i$. The set of red edges from $s(e)$ to $v^i$ can then be labeled as $\{f_{e,i}^{(k)} \mid 1\le k \le a_{e,i}\}$. Fix a red edge $f_{e,i}^{(k)}$. Then $s(f_{e,i}^{(k)}) = s(e)$ and $r(f_{e,i}^{(k)}) = v^i$. By the proof of Lemma \ref{lemma-sigma.surj.local.homeo}, for each $y \in Z(e)$ there exists a unique $x = (e_0,e_1,e_2,\dots) \in Z(v^i) = V_i$ such that $\sigma(x) = y$ and the first component $f_0$ of the romb $(e_0,e_1,f_0,f_1)$ determined by $(e_0,e_1)$ is precisely $f_{e,i}^{(k)}$. We define $V_{e,i}^{(k)}$ to be the set of all these pre-images $x \in Z(v^i)$ of elements of $Z(e)$. Observe that by definition we have that $\sigma|_{V_{e,i}^{(k)}}$ is injective, and that $\sigma(V_{e,i}^{(k)}) = Z(e)$. Moreover, we get
$$V_i = Z(v^i) = \bigsqcup_{e \in \Gamma} \bigsqcup_{k=1}^{a_{e,i}} V_{e,i}^{(k)}.$$
Therefore the decompositions
$$X = \bigsqcup_{e \in \Gamma} Z(e), \quad V_i = \bigsqcup_{e \in \Gamma_i} Z(e), \quad V_i = \bigsqcup_{e \in \Gamma} \bigsqcup_{k=1}^{a_{e,i}} V_{e,i}^{(k)}$$
satisfy all the necessary conditions to form a $\sigma$-cylinder decomposition $\calZ_0 (F,D) = (\{Z(e)\}, \{V_{e,i}^{(k)}\})$ of $X$, and the bipartite separated graph obtained from these decompositions as explained above coincides with the first layer of $(F,D)$.

As a consequence, we have the following result.

\begin{proposition}\label{proposition-existence.sigma.cylinder.dec}
An object $(X,\sigma) \in \emph{\textsf{LHomeo}}$ always admits a $\sigma$-cylinder decomposition.
\end{proposition}
\begin{proof}
This is immediate due to the above construction and Theorem \ref{theorem-equivalence}.
\end{proof}

By the above construction, it seems that first layers of $l$-diagrams play a special role. They deserve a name on its own right.

\begin{definition}\label{definition-general.finite.shift}
We say that a finite bipartite separated graph $(E,C)$ is a {\it generalized finite shift graph} if it satisfies the following properties:
\begin{enumerate}[(a),leftmargin=0.7cm]
\item $C_v = \{B_v,R_v\}$ for each $v\in E^{0,0}$, where $B_v,R_v$ are non-empty;
\item $$(1)\,\,  E^{0,1} = \bigsqcup_{v\in E^{0,0}} s(B_v), \quad \text{ and also } \quad (2) \,\, E^{0,1} = \bigcup_{v \in E^{0,0}} s(R_v).$$
Moreover, $s(e) \ne s(e')$ for all distinct edges $e,e'\in B_v$, for all $v\in E^{0,0}$.
\end{enumerate}
We call the elements of $B := \bigsqcup_{v \in E^{0,0}}B_v$ \textit{blue edges}, and the elements of $R := \bigsqcup_{v \in E^{0,0}}R_v$ \textit{red edges}. The \textit{red adjacency matrix} of the generalized finite shift graph $(E,C)$ is the matrix $A = (a_{w,v})\in M_{E^{0,1} \times E^{0,0}} (\Z^+)$ such that $a_{w,v}$ is the number of red edges from $w$ to $v$. We say that $(E,C)$ is {\it refined} if $A$ is a $\{0,1\}$ matrix. We also consider the \textit{blue adjacency matrix} $I \in M_{E^{0,1}\times E^{0,0}}(\{0,1\})$, with
$$I_{w,v}= \begin{cases}
1 & \text{ if } w=s(e) \text{ for (a unique) } e\in B_{v}, \\
0 & \text{ otherwise}.
\end{cases}$$
The pair of matrices $(A,I)$ characterizes the graph $(E,C)$ completely.
\end{definition}

It follows from the definition of $l$-diagram (Definition \ref{definition-L.separated.Bratteli}) that first layers of $l$-diagrams can be characterized by the pair $(A,I)$ as defined above.

\begin{example}\label{example-general.finite.shift}
The finite bipartite separated graph $(E,C)$ defined by $E^0 = E^{0,0} \sqcup E^{0,1}$ where
$$E^{0,0} = \{v^1,v^2,v^3,v^4\}, \quad E^{0,1} = \{w^1,w^2,w^3,w^4,w^5,w^6\},$$
and
$$B_{v^1} = \{e_{1,1},e_{2,1}\}, \quad B_{v^2} = \{e_{3,2}\}, \quad B_{v^3} = \{e_{4,3}\}, \quad B_{v^4} = \{e_{5,4},e_{6,4}\},$$
$$R_{v^1} = \{f_{1,1}\}, \quad R_{v^2} = \{f_{1,2},f_{2,2},f_{3,2}^{(1)},f_{3,2}^{(2)}\}, \quad R_{v^3} = \{f_{1,3},f_{4,3}^{(1)},f_{4,3}^{(2)},f_{5,3},f_{6,3}\}, \quad R_{v^4} = \{f_{6,4}\},$$
with $s(e_{j,i}) = s(f_{j,i}^{(k)}) = w^j$ and $r(e_{j,i}) = r(f_{j,i}^{(k)}) = v^i$, is a generalized finite shift graph. The graph looks as follows.

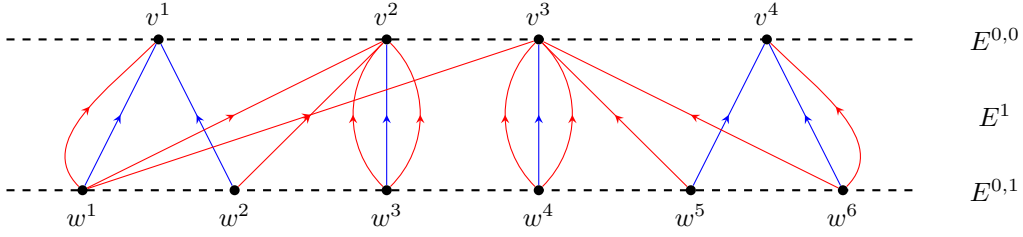
\begin{figure}[H]
\begin{tikzpicture}
	\draw[black,thick,dashed] (-6,0) -- (6,0);
	\draw[black,thick,dashed] (-6,2) -- (6,2);
	\node[label=below:$w^1$,circle,fill=black,scale=0.4] (W1) at (-5,0) {};
	\node[label=below:$w^2$,circle,fill=black,scale=0.4] (W2) at (-3,0) {};
	\node[label=below:$w^3$,circle,fill=black,scale=0.4] (W3) at (-1,0) {};
	\node[label=below:$w^4$,circle,fill=black,scale=0.4] (W4) at (1,0) {};
	\node[label=below:$w^5$,circle,fill=black,scale=0.4] (W5) at (3,0) {};
	\node[label=below:$w^6$,circle,fill=black,scale=0.4] (W6) at (5,0) {};
	\node[label=above:$v^1$,circle,fill=black,scale=0.4] (V1) at (-4,2) {};
	\node[label=above:$v^2$,circle,fill=black,scale=0.4] (V2) at (-1,2) {};
	\node[label=above:$v^3$,circle,fill=black,scale=0.4] (V3) at (1,2) {};
	\node[label=above:$v^4$,circle,fill=black,scale=0.4] (V4) at (4,2) {};
	\draw[-,blue,postaction={on each segment={mid arrow=blue}}] (W1) to (V1);
	\draw[-,blue,postaction={on each segment={mid arrow=blue}}] (W2) to (V1);
	\draw[-,blue,postaction={on each segment={mid arrow=blue}}] (W3) to (V2);
	\draw[-,blue,postaction={on each segment={mid arrow=blue}}] (W4) to (V3);
	\draw[-,blue,postaction={on each segment={mid arrow=blue}}] (W5) to (V4);
	\draw[-,blue,postaction={on each segment={mid arrow=blue}}] (W6) to (V4);
	\draw[-,red,postaction={on each segment={mid arrow=red}}] (W1) to [out=135,in=225] (V1);	
	\draw[-,red,postaction={on each segment={mid arrow=red}}] (W1) to (V2);
	\draw[-,red,postaction={on each segment={mid arrow=red}}] (W1) to (V3);
	\draw[-,red,postaction={on each segment={mid arrow=red}}] (W2) to (V2);
	\draw[-,red,postaction={on each segment={mid arrow=red}}] (W3) to [out=135,in=225] (V2);
	\draw[-,red,postaction={on each segment={mid arrow=red}}] (W3) to [out=45,in=315] (V2);
	\draw[-,red,postaction={on each segment={mid arrow=red}}] (W4) to [out=135,in=225] (V3);
	\draw[-,red,postaction={on each segment={mid arrow=red}}] (W4) to [out=45,in=315] (V3);
	\draw[-,red,postaction={on each segment={mid arrow=red}}] (W5) to (V3);
	\draw[-,red,postaction={on each segment={mid arrow=red}}] (W6) to (V3);
	\draw[-,red,postaction={on each segment={mid arrow=red}}] (W6) to [out=45,in=315] (V4);	
	\node at (7,2){$E^{0,0}$};
	\node at (7,1){$E^1$};	
	\node at (7,0){$E^{0,1}$};
\end{tikzpicture}
\caption{The generalized finite shift $(E,C)$.}
\label{figure-example1}
\end{figure}

Here the red-blue adjacency matrices are
$$A = \begin{pmatrix} 1 & 1 & 1 & 0 \\ 0 & 1 & 0 & 0 \\ 0 & 2 & 0 & 0 \\ 0 & 0 & 2 & 0 \\ 0 & 0 & 1 & 0 \\ 0 & 0 & 1 & 1 \end{pmatrix}, \quad I = \begin{pmatrix} 1 & 0 & 0 & 0 \\ 1 & 0 & 0 & 0 \\ 0 & 1 & 0 & 0 \\ 0 & 0 & 1 & 0 \\ 0 & 0 & 0 & 1 \\ 0 & 0 & 0 & 1 \end{pmatrix}.$$
\end{example}

\noindent We can now state the following result that summarizes and slightly generalizes our previous discussion.

\begin{proposition}\label{proposition-even.labels.sigma.cylinder.dec}
Let $(F,D)$ be an	$l$-diagram and let $(X,\sigma)$ be the associated surjective local homeomorphism, as defined in Subsection \ref{subsection-from.Ldiagram.to.local.homeo}. Then the first layer of $(F,D)$ is a generalized finite shift graph, and it gives rise to a $\sigma$-cylinder decomposition for $X$.

More generally, the layers corresponding to even vertices $(F^{0,2j}\cup F^{0,2j+1}, F^{1,2j})$, $j \geq 0$, in the $l$-diagram $(F,D)$ are generalized finite shift graphs, which are refined if $j>0$, and give rise to $\sigma$-cylinder decompositions for $X$. For each $j\ge 0$, the pair of red-blue adjacency matrices associated to this $\sigma$-cylinder decomposition agrees with the pair $(A,I)$ of red and blue adjacency matrices of the generalized finite shift graph $(F^{0,2j}\cup F^{0,2j+1}, F^{1,2j})$. Moreover, for $j'>j$,  the $\sigma$-cylinder decomposition corresponding to the level $2j'$ refines the $\sigma$-cylinder decomposition corresponding to the level $2j$.
\end{proposition}

\begin{proof}
By definition, it is clear that the even layers $(F^{0,2j}\cup F^{0,2j+1}, F^{1,2j})$, $j \geq 0$, are generalized finite shift graphs, which are refined for $j > 0$ by Remark \ref{remark-about.def.of.Ldiagram} (5). Let us now describe the associated $\sigma$-cylinder decompositions.

We have already defined the $\sigma$-cylinder decomposition for the first layer $(F^{0,0} \cup F^{0,1}, F^{1,0})$, namely $\calZ_0 = \calZ_0(F,D) = (\{Z(e)\}, \{ V_{e,i}^{(k)} \})$. We now define the $\sigma$-cylinder decomposition $\calZ_{2j} = \calZ_{2j}(F,D)$ for an even layer $(F^{0,2j}\cup F^{0,2j+1}, F^{1,2j})$, where $j \ge 1$.

We take as a finite set for the $2j^{\text{th}}$ layer the set
$$\Gamma^{2j} = P^B_{0,2j+1} = \bigsqcup_{\ol{e} \in P^B_{0,2j}} \Gamma^{2j}_{\ol{e}}, \quad \text{ with } \quad \Gamma^{2j}_{\ol{e}} =\{(\ol{e},e) \in P^B_{0,2j+1} \mid e \in B_{s(\ol{e})}\}.$$
Now we have $X = \bigsqcup_{(\ol{e},e) \in \Gamma^{2j}} Z(\ol{e},e)$, and we set
$$V_{\ol{e}} := \bigsqcup_{(\ol{e},e) \in \Gamma^{2j}_{\ol{e}}} Z(\ol{e},e) = \bigsqcup_{e \in B_{s(\ol{e})}} Z(\ol{e},e) = Z(\ol{e}).$$

For the construction of the remaining sets, we proceed as follows. First, recalling that the set of vertices $F^{0,n}$ can be identified with $P^B_{0,n}$ for every $n \ge 0$, we will label the red adjacency matrix of $(F^{0,2j} \cup F^{0,2j+1}, F^{1,2j})$ by the set $\Gamma^{2j} \times P^B_{0,2j}$, so that, for $(\ol{e},e) \in \Gamma^{2j}$ and $\ol{e}'\in P^B_{0,2j}$, $a_{(\ol{e},e),\ol{e}'} = 1$ if and only if there is a red edge from $s(e)$ to $s(\ol{e}')$, and $a_{(\ol{e},e),\ol{e}'} = 0$ otherwise (recall from Remark \ref{remark-about.def.of.Ldiagram} (5) that, since $j \ge 1$, the red adjacency matrix of $(F^{0,2j} \cup F^{0,2j+1}, F^{1,2j})$ is a binary matrix).

Fix $(\ol{e}, e) \in \Gamma^{2j}$ and $\ol{e}' \in P^B_{0,2j}$, with $\ol{e} = (e_0,e_1,\dots, e_{2j-2}, e_{2j-1})$, and $\ol{e}' = (e_0',e_1',\dots, e_{2j-2}', e_{2j-1}')$. We now check that $a_{(\ol{e},e), \ol{e}'} = 1$ if and only if there exists $x \in Z(\ol{e}')$ such that $\sigma(x) \in Z(\ol{e},e)$. Suppose first that there is some $x$ with the stated condition. Then we can write
$$x = (e_0',e_1',\dots, e_{2j-2}',e_{2j-1}',e_{2j}', e_{2j+1}',\dots)$$
for some $e_{2j}',e_{2j+1}',\dots$ blue edges. For each $k \ge 0$, let $(e_{2k}',e_{2k+1}', f_{2k}, f_{2k+1})$ be the romb associated to the pair of blue edges $(e_{2k}',e_{2k+1}')$ (see Proposition \ref{proposition-decomposition.odd.rombs.even} (ii)). Then by the definition of $\sigma$ (Construction \ref{construction-local.homeo}), and since $\sigma(x) \in Z(\ol{e},e)$ by assumption, we have that $(e_{2j-1},e,f_{2j-1},f_{2j})$ is the romb associated to the pair of red edges $(f_{2j-1},f_{2j})$. Hence we have that the red edge $f_{2j}$ satisfies $r(f_{2j}) = r(e_{2j}') = s(e_{2j-1}')$ and also $s(f_{2j}) = s(e)$, and we see that there is a red edge from $s(e)$ to $s(\ol{e}') = s(e_{2j-1}')$, and thus $a_{(\ol{e},e),\ol{e}'} = 1$.

Conversely, suppose that there is a red edge $f_{2j}$ such that $s(f_{2j}) = s(e)$ and $r(f_{2j}) = s(e_{2j-1}')$.  We will show that for each element
$$y = (e_0,e_1,\dots, e_{2j-1}, e, e_{2j+1}, e_{2j+2}, \dots) \in Z(\ol{e},e)$$
there is a unique $x \in Z(\ol{e}')$ such that $\sigma(x) = y$. The proof of this fact resembles the proof of surjectivity in Lemma \ref{lemma-sigma.surj.local.homeo}, although some additional manipulations are needed here. Use Definition \ref{definition-L.separated.Bratteli} (g) and induction to get unique red edges $f_n$ for $n \ge 2j+1$ such that $f_{2k+1} \in R(f_{2k})$ and $(e_{2k+1}, e_{2k+2},f_{2k+1},f_{2k+2})$ is a romb for each $k \ge j$. By Proposition \ref{proposition-decomposition.odd.rombs.even} (ii), there are blue edges $e'_n$ for $n \ge 2j$ such that $(e'_{2k},e'_{2k+1},f_{2k},f_{2k+1})$ is a romb for each $k \ge j$. Note that $r(e_{2j}') = r(f_{2j}) = s(e'_{2j-1})$. Hence we may consider the element
$$x = (e_0',e_1',\dots, e_{2j-2}',e_{2j-1}',e_{2j}',e_{2j+1}',e_{2j+2}',\dots) \in Z(\ol{e}').$$
Then by the construction of $x$ and the definition of $\sigma$, we have that
$$\sigma(x) = (e_0'',e_1'',\dots, e_{2j-2}'', e_{2j-1}'', e_{2j}'', e_{2j+1}, e_{2j+2}, e_{2j+3},\dots).$$
But now by uniqueness of blue edges with the same source we get that $e = e_{2j}''$ and $e_n = e_n''$ for $0 \leq n \leq 2j-1$, showing that $\sigma(x) = y$. The uniqueness of $x \in Z(\ol{e}')$ such that $\sigma(x) = y$ follows from the fact that $j \ge 1$ because, following the proof of Lemma \ref{lemma-sigma.surj.local.homeo}, the pre-images of $y$ are determined by the red edges departing from $s(e_0)$, and thus $x$ will necessarily be the pre-image of $y$ determined by the edge $f_0$, where $(e_0',e_1',f_0,f_1)$ is the romb associated to the pair of blue edges $(e_0',e_1')$, with $f_1 \in R(f_0)$.

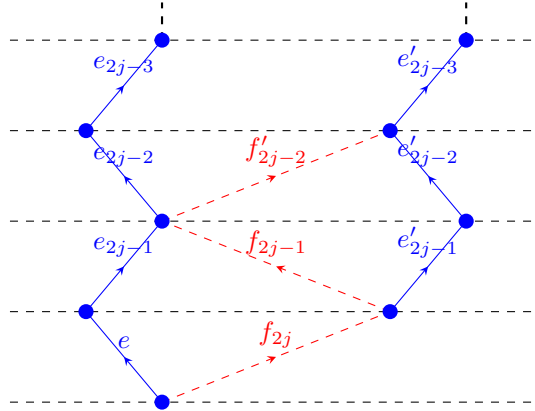
\begin{figure}[H]
	\begin{tikzpicture}
	\path [draw=blue,postaction={on each segment={mid arrow=blue}}]
	(-1,3.6) -- node[above,blue]{$e_{2j-3}$} (0,4.8)
	(0,2.4) -- node[above,blue]{$e_{2j-2}$} (-1,3.6)
	(-1,1.2) -- node[above,blue]{$e_{2j-1}$} (0,2.4)
	(0,0) -- node[above,blue]{$e$} (-1,1.2)
	;
	\path [draw=red,dashed,postaction={on each segment={mid arrow=red}}]
	(0,2.4) -- node[above,red]{$f'_{2j-2}$} (3,3.6)
	(3,1.2) -- node[above,red]{$f_{2j-1}$} (0,2.4)
	(0,0) -- node[above,red]{$f_{2j}$} (3,1.2)
	;
	\path [draw=blue,postaction={on each segment={mid arrow=blue}}]
	(3,3.6) -- node[above,blue]{$e'_{2j-3}$} (4,4.8)
	(4,2.4) -- node[above,blue]{$e'_{2j-2}$} (3,3.6)
	(3,1.2) -- node[above,blue]{$e'_{2j-1}$} (4,2.4)
	;
	\draw[black,thick,dashed] (0,4.8) -- (0,5.3);
	\draw[black,thick,dashed] (4,4.8) -- (4,5.3);
	\draw[black,dashed] (-2,0) -- (5,0);
	\draw[black,dashed] (-2,1.2) -- (5,1.2);
	\draw[black,dashed] (-2,2.4) -- (5,2.4);
	\draw[black,dashed] (-2,3.6) -- (5,3.6);
	\draw[black,dashed] (-2,4.8) -- (5,4.8);
	\node[circle,fill=blue,scale=0.6] (B0) at (0,0) {};
	\node[circle,fill=blue,scale=0.6] (B1) at (-1,1.2) {};
	\node[circle,fill=blue,scale=0.6] (B2) at (0,2.4) {};
	\node[circle,fill=blue,scale=0.6] (B3) at (-1,3.6) {};
	\node[circle,fill=blue,scale=0.6] (B4) at (0,4.8) {};
	\node[circle,fill=blue,scale=0.6] (B'0) at (3,1.2) {};
	\node[circle,fill=blue,scale=0.6] (B'1) at (4,2.4) {};
	\node[circle,fill=blue,scale=0.6] (B'2) at (3,3.6) {};
	\node[circle,fill=blue,scale=0.6] (B'3) at (4,4.8) {};
	\end{tikzpicture}
	\caption{Relation between $(\ol{e},e)$ and $\ol{e}'$ via red edges.}
	\label{figure-schematics5}
\end{figure}

\noindent Let us now continue the proof. If $a_{(\ol{e},e),\ol{e}'} = 1$, we set
$$V_{(\ol{e},e),\ol{e}'} = \{x \in Z(\ol{e}') \mid \sigma(x) \in Z(\ol{e},e)\}.$$
The set $V_{(\ol{e},e),\ol{e}'}$ is left undefined if $a_{(\ol{e},e),\ol{e}'} = 0$. Finally, we define
$$\calZ_{2j} = \calZ_{2j}(F,D) = (\{Z(\ol{e},e) \mid (\ol{e},e) \in \Gamma^{2j} \}, \{V_{(\ol{e},e), \ol{e}'} \mid (\ol{e},e) \in \Gamma^{2j}, \ol{e}' \in P^B_{0,2j} \text{ with } a_{(\ol{e},e),\ol{e}'} = 1\}).$$
Clearly conditions (a) and (b) in Definition \ref{definition-sigma.cylinder.dec} are satisfied. Moreover, by what we have proven above, we have that the restriction of $\sigma$ to each $V_{(\ol{e},e),\ol{e}'}$, with $a_{(\ol{e},e),\ol{e}'} = 1$, is injective, and that $\sigma(V_{(\ol{e},e),\ol{e}'}) = Z(\ol{e},e)$. Therefore condition (c) in Definition \ref{definition-sigma.cylinder.dec} is also satisfied. Hence we have defined $\sigma$-cylinder decompositions $\calZ_{2j}$, whose associated red-blue adjacency matrices agree with the pair $(A,I)$ of red and blue adjacency matrices of the generalized finite shift graph $(F^{0,2j} \cup F^{0,2j+1}, F^{1,2j})$, for each $j \ge 0$.

We now show the part on the refinement of the $\sigma$-cylinder decompositions. By Lemma \ref{lemma-refinement.transitivity}, it is enough to show it in the case where we have two consecutive even layers $(F^{0,2j} \cup F^{0,2j+1}, F^{1,2j})$ and $(F^{0,2j+2} \cup F^{0,2j+3}, F^{1,2j+2})$, $j \geq 0$.

We need to check the conditions in Definition \ref{definition-refined.sigma.cylinder.dec}. Note that
$$\Gamma^{2j+2} = P^B_{2j+3} = \bigsqcup_{\ol{e} \in \Gamma^{2j}} \Gamma^{2j+2}(\ol{e}),$$
with
$$\Gamma^{2j+2}(\ol{e}) = \{(\ol{e},e_{2j+1},e_{2j+2}) \mid e_{2j+1} \in B_{s(\ol{e})}, e_{2j+2} \in B_{s(e_{2j+1})}\},$$
and we have, for $\ol{e} \in \Gamma^{2j}$,
$$Z(\ol{e}) = \bigsqcup_{(\ol{e},e_{2j+1},e_{2j+2}) \in \Gamma^{2j+2}(\ol{e})} Z(\ol{e},e_{2j+1},e_{2j+2}).$$
This proves condition (a) in Definition \ref{definition-refined.sigma.cylinder.dec}. 

We prove (b) and (c) first for $j \ge 1$. In this case
$$\mathfrak A_{2j} = \{((\ol{e},e),\ol{e}') \mid \ol{e},\ol{e}' \in P^B_{0,2j}, e \in B_{s(\ol{e})}, a_{(\ol{e},e),\ol{e}'} = 1\},$$
and we have a similar expression for $\mathfrak A _{2j+2}$. We can then write
$$\mathfrak A_{2j+2} = \bigsqcup_{((\ol{e},e),\ol{e}') \in \mathfrak A_{2j}} \mathfrak A_{2j+2}((\ol{e},e),\ol{e}'),$$
where $\mathfrak A_{2j+2}((\ol{e},e),\ol{e}')$ is the set of all elements of $\mathfrak A _{2j+2}$ of the form $((\ol{e},e,e_{2j+1},e_{2j+2}),(\ol{e}',e'_{2j},e'_{2j+1}))$, for suitable edges $e_{2j+1},e_{2j+2}, e'_{2j}, e_{2j+1}'$.

With these decompositions, we obtain the equality
\begin{equation}\label{equation-refinement.second.partition}
V_{(\ol{e},e),\ol{e}'} = \bigsqcup_{\nu \in \mathfrak A _{2j+2}((\ol{e}, e), \ol{e}')} V_{\nu},
\end{equation}
for each $((\ol{e},e), \ol{e}') \in \mathfrak A _{2j}$, which shows (b). To prove (c), one has to show that given $((\ol{e},e),\ol{e}') \in \mathfrak A_{2j}$, and given $e_{2j+1} \in B_{s(e)}$, $e_{2j+2} \in B_{s(e_{2j+1})}$, there are unique $e'_{2j} \in B_{s(\ol{e}')}$ and $e'_{2j+1} \in B_{s(e'_{2j})}$ such that $((\ol{e},e,e_{2j+1},e_{2j+2}),(\ol{e}',e'_{2j},e'_{2j+1})) \in \mathfrak A_{2j+2}$. The blue edges $e_{2j}',e_{2j+1}'$ are constructed as follows. Let $f_{2j}$ be the unique red edge such that $s(f_{2j}) = s(e)$ and $r(f_{2j}) = s(\ol{e}')$. This edge must exist because $a_{(\ol{e},e),\ol{e}'} = 1$. Now by Definition \ref{definition-L.separated.Bratteli} (g), there is a unique pair of red edges $(f_{2j+1},f_{2j+2})$ such that $f_{2j+1} \in R(f_{2j})$ and $(e_{2j+1},e_{2j+2},f_{2j+1},f_{2j+2})$ is a romb. Now use Proposition \ref{proposition-decomposition.odd.rombs.even} (ii) to get a unique pair of blue edges $(e_{2j}',e_{2j+1}')$ such that $(e_{2j}',e_{2j+1}',f_{2j},f_{2j+1})$ is a romb. Note that $f_{2j+2}$ is a red edge such that $s(f_{2j+2}) = s(e_{2j+2})$ and $r(f_{2j+2}) = s(e_{2j+1}')$. It follows that $(\ol{e}',e_{2j}',e_{2j+1}')$ is the unique path in $P^B_{0,2j+2}$ such that $((\ol{e},e,e_{2j+1},e_{2j+2}),(\ol{e}',e_{2j}',e_{2j+1}')) \in \mathfrak A_{2j+2}$.

It remains to show (b) and (c) in Definition \ref{definition-refined.sigma.cylinder.dec} for $j=0$. In this case, we have the sets $V_{e_0,i}^{(k)}$, where $e_0$ is a blue edge in $F^{1,0}$, $v^i\in F^{0,0}$ and $1\le k\le a_{e_0,i}$. Given such a triple $(e_0,i,k)$, the set $\mathfrak A _2(e_0,i,k)$ is the set of all elements $((e_0,e_1,e_2),(e_0',e_1')) \in \mathfrak A_2$ such that $r(e_0')= v^i$ and the romb associated to the pair of blue edges $(e_0',e_1')$ is of the form $(e_0',e_1',f_0,f_1)$ with $f_0= f_{e_0,i}^{(k)}$. We then have
$$\mathfrak A_2 = \bigsqcup_{(e_0,i,k) \in \mathfrak A_0} \mathfrak A_2 (e_0,i,k)$$
and
$$ V_{e_0,i}^{(k)} = \bigsqcup_{\nu \in \mathfrak A_2(e_0,i,k)} V_{\nu}$$
for $(e_0,i,k) \in \mathfrak A_0$. This shows (b) for $j = 0$. Part (c) for $j=0$ is proven in a similar way as in the case where $j > 0$. Namely, given a triple $(e_0,i,k)$ as before, and given $e_1 \in B_{s(e_0)}$ and $e_2\in B_{s(e_1)}$, take first the unique pair of red edges $(f_1,f_2)$ such that $f_1 \in R(f_{e_0,i}^{(k)})$ and $(e_1,e_2,f_1,f_2)$ is a romb. Then take the unique blue pair $(e_0',e_1')$ so that $(e_0',e_1',f^{(k)}_{e_0,i},f_1)$ is a romb. Then $((e_0,e_1,e_2),(e_0',e_1'))$ is the unique element in $\mathfrak A_2 (e_0,i,k)$ projecting to $(e_0,e_1,e_2)$. This completes the proof.
\end{proof}

\begin{remark}\label{remark-relation.adjacencies}
Note the following relation between red adjacency matrices. For $j \geq 0$, let $A_{2j} = (a_{(\ol{e},e),\ol{e}'}) \in M_{\Gamma^{2j} \times P^B_{0,2j}}(\Z^+)$ be the red adjacency matrix of the generalized finite shift graph $(F^{0,2j} \cup F^{0,2j+1},F^{1,2j})$. Let also, for $n \geq 0$, $P_{n+1}$ be the characteristic matrix of the partitions $\{Z(\ol{e}) \mid \ol{e} \in P^B_{0,n}\}$ and $\{Z(\ol{e},e_0,e_1) \mid (\ol{e},e_0,e_1) \in P^B_{0,n+2}\}$, which is defined by
$$(P_{n+1})_{\ol{e}',(\ol{e},e_0,e_1)} = \begin{cases} 1 & \text{ if } \ol{e}' = \ol{e}, \\ 0 & \text{ otherwise}. \end{cases}$$
Then we have the relation
$$D_{2j} A_{2j} = P_{2j+2} A_{2j+2} P_{2j+1}^T,$$
where $D_{2j}$ is the diagonal $\Gamma^{2j} \times \Gamma^{2j}$ matrix having $|P^B_{2j+1,2j+3}(s(\ol{e}))|$ at the position $(\ol{e},\ol{e})$, for each $\ol{e} \in \Gamma^{2j}$. 

The formula holds because for each $(\ol{e},e_0,e_1) \in \Gamma^{2j+2}$ and each red edge $f_0$ such that $s(f_0)= s(\ol{e})$, there is a unique $(\ol{e}',e_0',e_1') \in P^B_{0,2j+2}$ such that $a_{(\ol{e},e_0,e_1),(\ol{e}',e_0',e_1')} = 1$ and the romb associated to $(e_0',e_1')$ is of the form $(e_0',e_1',f_0,f_1)$, where $(e_0,e_1,f_1,f_2)$ is the romb associated to $(e_0,e_1)$.

\begin{figure}[H]
\begin{tikzpicture}
	\path [draw=blue,postaction={on each segment={mid arrow=blue}}]
	(-1,3.6) -- node[above,blue]{} (0,4.8)
	(0,2.4) -- node[above,blue]{$\ol{e}$} (-1,3.6)
	(-1,1.2) -- node[above,blue]{$e_0$} (0,2.4)
	(0,0) -- node[above,blue]{$e_1$} (-1,1.2)
	;
	\path [draw=red,dashed,postaction={on each segment={mid arrow=red}}]
	(0,2.4) -- node[above,red]{$f_0$} (3,3.6)
	(3,1.2) -- node[above,red]{$f_1$} (0,2.4)
	(0,0) -- node[above,red]{$f_2$} (3,1.2)
	;
	\path [draw=blue,postaction={on each segment={mid arrow=blue}}]
	(3,3.6) -- node[above,blue]{$\ol{e}'$} (4,4.8)
	(4,2.4) -- node[above,blue]{$e'_0$} (3,3.6)
	(3,1.2) -- node[above,blue]{$e'_1$} (4,2.4)
	;
	\draw[black,thick,dashed] (0,4.8) -- (0,5.3);
	\draw[black,thick,dashed] (4,4.8) -- (4,5.3);
	\draw[black,dashed] (-2,0) -- (5,0);
	\draw[black,dashed] (-2,1.2) -- (5,1.2);
	\draw[black,dashed] (-2,2.4) -- (5,2.4);
	\draw[black,dashed] (-2,3.6) -- (5,3.6);
	\draw[black,dashed] (-2,4.8) -- (5,4.8);
	\node[circle,fill=blue,scale=0.6] (B0) at (0,0) {};
	\node[circle,fill=blue,scale=0.6] (B1) at (-1,1.2) {};
	\node[circle,fill=blue,scale=0.6] (B2) at (0,2.4) {};
	\node[circle,fill=blue,scale=0.6] (B3) at (-1,3.6) {};
	\node[circle,fill=blue,scale=0.6] (B4) at (0,4.8) {};
	\node[circle,fill=blue,scale=0.6] (B'0) at (3,1.2) {};
	\node[circle,fill=blue,scale=0.6] (B'1) at (4,2.4) {};
	\node[circle,fill=blue,scale=0.6] (B'2) at (3,3.6) {};
	\node[circle,fill=blue,scale=0.6] (B'3) at (4,4.8) {};
\end{tikzpicture}
\caption{Visual explanation of the formula $D_{2j}A_{2j} = P_{2j+2} A_{2j+2}P_{2j+1}^T$.}
\label{figure-schematics6}
\end{figure}
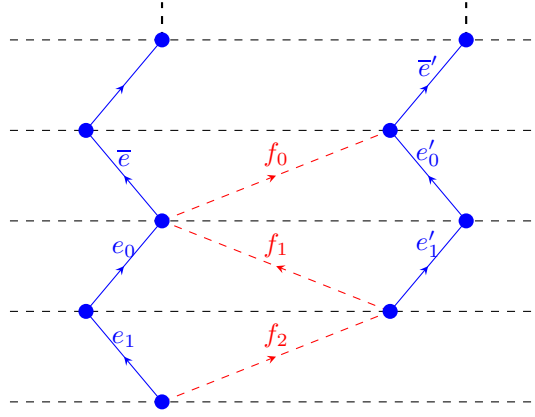
\end{remark}

We will now show that any generalized finite shift graph gives rise to a dynamical system $(X,\sigma) \in \textsf{LHomeo}$. This dynamical system will be called a \textit{generalized finite shift}, and includes both one-sided and two-sided finite shifts. For this purpose we will use the construction of the canonical resolution of a bipartite separated graph. We first recall its definition (see \cite{AL2018} and \cite{AE2014}).

\begin{definition}\label{definition-F.infty.and.others}
Let $(E,C)$ be any finite bipartite separated	graph, and write
$$C_u = \{X_1^u,\dots,X_{k_u}^u\}$$
for all $u \in E^{0,0}$. Then the \textit{$1$-step resolution} of $(E,C)$ is the finite bipartite separated graph denoted by $(E_1,C^1)$, and defined by
\begin{enumerate}[(a),leftmargin=0.7cm]
\item $E_1^{0,0} := E^{0,1}$ and $E_1^{0,1} := \{v(x_1,\dots,x_{k_u}) \mid u \in E^{0,0}, x_j \in X_j^u \text{ for all } 1 \leq j \leq k_u\}$.
\item $E_1^1 := \{\alpha^{x_i}(x_1,\dots,\widehat{x_i},\dots,x_{k_u}) \mid u \in E^{0,0}, 1 \leq i \leq k_u, x_j \in X_j^u \text{ for all } 1 \leq j \leq k_u \}$.
\item The range and source for the edges in $E_1^1$ are defined by
$$r_1(\alpha^{x_i}(x_1,\dots,\widehat{x_i},\dots,x_{k_u})) := s(x_i),$$
$$s_1(\alpha^{x_i}(x_1,\dots,\widehat{x_i},\dots,x_{k_u})) := v(x_1,\dots,x_{k_u}).$$
\item $C^1_v := \{X(x) \mid x \in s^{-1}(v)\}$ for $v \in E^{0,0}_1$, where for $u \in E^{0,0}$, $1 \leq i \leq k_u$, and $x_i\in X_i^u$,
$$X(x_i) := \{\alpha^{x_i}(x_1,\dots,\widehat{x_i},\dots,x_{k_u}) \mid x_j \in X_j^u \text{ for } j \ne i\}.$$
\end{enumerate}
A sequence of finite bipartite separated graphs $\{(E_n,C^n)\}_{n \ge 0}$ with $(E_0,C^0) := (E,C)$ is then defined inductively by letting $(E_{n+1},C^{n+1})$ denote the $1$-step resolution of $(E_n,C^n)$. The bipartite separated graph $(E_n,C^n)$ is called the \textit{$n$-step resolution} of $(E,C)$. Finally $(F,D)$ is the infinite layer graph
$$(F,D) := \bigcup_{n=0}^{\infty} (E_n,C^n).$$
It is clear by construction that $(F,D)$ is a separated Bratteli diagram, called the {\it canonical resolution} of the finite bipartite graph $(E,C)$.
\end{definition}

We now apply this construction to a generalized finite shift graph $(E,C)$ (see Definition \ref{definition-general.finite.shift}).

\begin{proposition}\label{proposition-multiresolution.is.Ldiagram}
Let $(E,C)$ be a generalized finite shift graph. Then the canonical resolution $(F,D)$ of $(E,C)$ is an $l$-diagram.
\end{proposition}
\begin{proof}
Let $(E,C)$ be a generalized finite shift graph, with $E^{0,0}= \{ v^1,\dots , v^n\}$ and red adjacency matrix $A= (a_{w,i})$, where $a_{w,i}$ is the number of red edges from $w \in E^{0,1}$ to $v^i$.

Observe that, by its natural expression as a separated Bratteli diagram, we have $F^{0,n} =E_n^{0,0}$ and $F^{1,n}= E_n^{1}$ for all $n\ge 0$. We also have $D_v = C_v^n$ for $v\in F^{0,n}$. We need to check conditions (a)-(g) in Definition \ref{definition-L.separated.Bratteli}.

Note first that, by definition, property (a) for $j=0$, formulas (1) and (2) in (d), and the second part of (d) for $j=0$ hold. 

We analyze first the graph $(E_1,C^1)$. Observe that by Definition \ref{definition-F.infty.and.others} we have exactly $|s^{-1}(v)|$ elements in $D_v$ for each $v \in E_1^{0,0} = F^{0,1}$. These are of the form
$$D_v = \{B_v\} \cup \{R(f_{v,j}^{(k)}) \mid 1 \leq j \leq n, \text{ } 1 \leq k \leq a_{v,j}\},$$
where $\{f_{v,j}^{(k)}\mid 1 \leq k \leq a_{v,j}\}$ is the set of red edges from $v$ to $v^j$. For $v\in F^{0,1}$, we write
$$R_v = \bigcup_{j=1}^n \bigcup_{k=1}^{a_{v,j}} R(f_{v,j}^{(k)}).$$
These are the red edges at the vertices in $E^{0,1}= F^{0,1}$. Condition (b) is therefore verified for $j=1$.

To verify condition (c) at $j=1$, observe that $F^{0,2}$ is the disjoint union of the vertices corresponding to the different vertices $v^i$, for $1 \leq i \leq n$, where a vertex $w$ in $F^{0,2}$ corresponds to $v^i$ in case there is a (necessarily unique) blue path going from $w$ to $v^i$. So it is enough to concentrate on one such index $i$. The set of vertices corresponding to $v^i$ is 
$$\{v(e_{v,i}, f_{w,i}^{(k)}) \mid e_{v,i} \in B_{v^i}, w \in F^{0,1}, 1 \leq k \leq a_{w,i} \}.$$
Here $B_{v^i} = \{e_{v,i}\}_{v\in s(B_{v^i})}$, with $s(e_{v,i}) = v$ and $r(e_{v,i}) = v^i$. Clearly, for a given vertex $v(e_{v,i}, f_{w,i}^{(k)})$ of this form, there are exactly two edges with source $v(e_{v,i}, f_{w,i}^{(k)})$, namely $\alpha^{e_{v,i}}(f_{w,i}^{(k)})\in B_v$ and $\alpha^{f_{w,i}^{(k)}}(e_{v,i}) \in R(f_{w,i}^{(k)}) \subseteq R_w$. Hence condition (c) holds for $j=1$.

The validity of condition (e) for $j=0$ is clear from the above considerations.

We have verified all the conditions up to the second layer. Now, using induction, we suppose that, for $n>0$, all conditions hold up to the layer $2n$, that is, suppose that all the conditions in Definition \ref{definition-L.separated.Bratteli} involving vertices from $F^{0,j}$ for $0\le j\le  2n$ and edges between them, hold. We will check the conditions for the next two layers. For $v\in F^{0,2n}$ we have $D_v= \{B_v,R_v\}$ because by induction hypothesis (condition (c)), there are exactly two edges departing from $v$, one of them a blue edge and the other a red edge, giving rise (by Definition \ref{definition-F.infty.and.others}) to the two sets $B_v$ and $R_v$ respectively. This gives condition (a) for $j=n$. Moreover, the elements of $B_v$ are of the form $\alpha^e(f_1,\dots , f_p)$, where $e$ is the unique blue edge such that $s(e)= v$, and $f_i\in R(g_i)$, where $g_1,\dots , g_p$ are the different red edges such that $s(g_i)=r(e)$. We then have that
$$s(\alpha^e (f_1,\dots , f_p))= v(e,f_1,\dots , f_p),$$
so clearly we have that the sources of two distinct blue edges are distinct. This verifies the second part of condition (d) for $j=n$. Observe that $(F^{0,2n} \cup F^{0, 2n+1}, F^{1,2n})$ is again a generalized finite shift graph, and so we can prove the same properties as before for the layers between $2n$ and $2n+2$.

It remains to check conditions (f) and (g).

We check condition (f) for $j= n$. Suppose that we have a pair $(f_0,f_1)$ of red edges such that $r(f_0)= v\in F^{0,2n-1}$ and $r(f_1)= s(f_0)$. Then $f_1$ will be of the form
$$f_1 = \alpha^{f_0'}(e_0, g_1, \dots , g_p)$$
where $e_0 \in B_{v}$, $f_0'\in R(h_0)$, and $g_i \in R(h_i)$ for $1 \leq i \leq p$, where $\{h_0,h_1,\dots , h_p\}$ are the different red edges departing from $v$. Observe that $s(f_0') = r(f_1) = s(f_0)$. Hence by property (c) for $j = n$ we need to have $f_0' = f_0$. 
We define $e_1 := \alpha ^{e_0}(f_0, g_1,\dots , g_p)$. We have that $r(e_1) = s(e_0)$, and that $e_1\in B_{r(e_1)}$. Moreover, $s(e_1) = s(f_1) = v(e_0, f_0, g_1, \dots, g_p)$ and $r(e_0) = r(f_0) = v$, as desired. It is clear that the pair $(e_0,e_1)$ is the unique pair of blue edges satisfying these conditions.

We finally deal with condition (g). Suppose that we are given one of the red edges $h_i$, where as before $\{h_0,h_1,\dots , h_p\}$ are the different red edges departing from $v$, and we are also given a pair $(e_0,e_1)$ of blue edges such that $r(e_0) = v$ and $r(e_1) = s(e_0)$. This means that $e_1$ is of the form
$$e_1= \alpha ^{e_0}(g_0,g_1,\dots , g_p),$$
where $g_j\in R(h_j)$ for $0 \leq j \leq p$. We then define $f_0= g_i$ and
$$f_1= \alpha^{g_i}(e_0, g_0,\dots , \hat{g_i},\dots , g_p).$$
Then we get that $(e_0,e_1,f_0,f_1)$ is a romb, and that $(f_0,f_1)$ is unique satisfying this property and the property that $f_0\in R(h_i)$. This shows that condition (g) is satisfied for $j=n$.

This completes the proof of the fact that $(F,D)$ is an $l$-diagram.
\end{proof}

We have observed in the above proof that the separated graphs $(E_{2n},C^{2n})$, $n \geq 0$, are also generalized finite shift graphs. In particular, we can talk about their red-blue adjancency matrices $(A_{2n},I_{2n})$. A similar relation between the different matrices $A_n$ as the one described in Remark \ref{remark-relation.adjacencies} can also be given. Of course, all graphs $(E_{2n},C^{2n})$ describe exactly the same dynamical system.

\begin{definition}\label{definition-generalized.finite.shift}
Let $(E,C)$ be a generalized finite shift graph, and let $(F,D)$ be its canonical resolution, which is an $l$-diagram by Proposition \ref{proposition-multiresolution.is.Ldiagram}. The associated dynamical system $(X,\sigma)$, as defined in Subsection \ref{subsection-from.Ldiagram.to.local.homeo}, will be called a {\it generalized finite shift}.
\end{definition}

The rest of this section is devoted to show how these generalized finite shifts provide common generalizations of one-sided and two-sided shifts.\\

We come back to the original situation where we have $E = (E^0,E^1,r_E,s_E)$ a finite graph with adjacency matrix $A = A_E$ with no zero-rows and no zero-columns, and set of vertices $E^0 = \{1,\dots , n\}$.

We can canonically build a finite bipartite separated graph $(E',C')$ associated to $E$ as follows. The graph $E'$ has set of vertices $(E')^0 = (E')^{0,0} \sqcup (E')^{0,1}$ with $(E')^{0,0}$ and $(E')^{0,1}$ disjoint copies of $E^0$. Let $i \mapsto v_i$ be a bijection from $E^0$ onto $(E')^{0,1}$, and $i \mapsto v^i$ a bijection of $E^0$ onto $(E')^{0,0}$. We define $C'_i= \{B_i,R_i\}$, where $B_i= \{e_i\}$ with $s(e_i)= v_i$ and $r(e_i)= v^i$, and where $R_i$ consists of a copy of $r_E^{-1}(i)$. More precisely, for each $f \in E^1 $ such that $r_E(f) = i$, we define an edge $\ol{f} \in (E')^1$ such that $s(\ol{f}) = s_E(f)_i$ and $r(\ol{f}) = r_E(f)^i$. Then we define $R_i$ as the set of such edges $\ol{f}$ with $r_E(f)= i$.

\begin{example}\label{example-forbidden.word.000}
Let $E$ be the finite graph
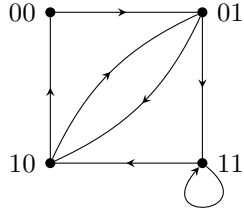
\begin{figure}[H]
\begin{tikzpicture}
	\path [draw=black,postaction={on each segment={mid arrow=black}}]
	(-1,2) -- node[above,red]{} (1,2)
	(-1,0) -- node[above,red]{} (-1,2)
	(1,2) -- node[above,red]{} (1,0)
	(1,0) -- node[above,red]{} (-1,0)
	;
	\node[label=left:$00$,circle,fill=black,scale=0.4] (00) at (-1,2) {};
	\node[label=right:$01$,circle,fill=black,scale=0.4] (01) at (1,2) {};
	\node[label=left:$10$,circle,fill=black,scale=0.4] (10) at (-1,0) {};
	\node[label=right:$11$,circle,fill=black,scale=0.4] (11) at (1,0) {};
	\draw[-,black,postaction={on each segment={mid arrow=black}}] (01) to [out=245,in=25] (10);
	\draw[-,black,postaction={on each segment={mid arrow=black}}] (10) to [out=65,in=205] (01);
	\draw[->,>= stealth] (11) edge [out=320,in=230,distance=10mm] (11);
\end{tikzpicture}
\label{figure-example2}
\end{figure}

The associated separated graph $(E',C')$ is given by
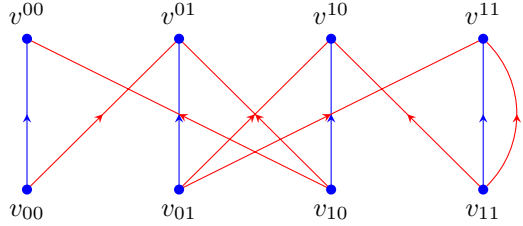
\begin{figure}[H]
\begin{tikzpicture}
	\path [draw=blue,postaction={on each segment={mid arrow=blue}}]
	(-3,0) -- node[above,blue]{} (-3,2)
	(-1,0) -- node[above,blue]{} (-1,2)
	(1,0) -- node[above,blue]{} (1,2)
	(3,0) -- node[above,blue]{} (3,2)
	;
	\path [draw=red,postaction={on each segment={mid arrow=red}}]
	(-3,0) -- node[above,red]{} (-1,2)
	(-1,0) -- node[above,red]{} (1,2)
	(-1,0) -- node[above,red]{} (3,2)
	(1,0) -- node[above,red]{} (-3,2)
	(1,0) -- node[above,red]{} (-1,2)
	(3,0) -- node[above,red]{} (1,2)
	;
	\node[label=below:$v_{00}$,circle,fill=blue,scale=0.4] (00bottom) at (-3,0) {};
	\node[label=below:$v_{01}$,circle,fill=blue,scale=0.4] (01bottom) at (-1,0) {};
	\node[label=below:$v_{10}$,circle,fill=blue,scale=0.4] (10bottom) at (1,0) {};
	\node[label=below:$v_{11}$,circle,fill=blue,scale=0.4] (11bottom) at (3,0) {};
	\node[label=above:$v^{00}$,circle,fill=blue,scale=0.4] (00top) at (-3,2) {};
	\node[label=above:$v^{01}$,circle,fill=blue,scale=0.4] (01top) at (-1,2) {};
	\node[label=above:$v^{10}$,circle,fill=blue,scale=0.4] (10top) at (1,2) {};
	\node[label=above:$v^{11}$,circle,fill=blue,scale=0.4] (11top) at (3,2) {};
	\draw[-,red,postaction={on each segment={mid arrow=red}}] (11bottom) to [out=45,in=315] (11top);
\end{tikzpicture}
\label{figure-example3}
\end{figure}
It is easy to show that the associated one-sided shift $(X_E,\sigma_E)$ is precisely the subset of the set of all sequences in $\{0,1\}^{\N}$ where the word $000$ is forbidden. \qed
\end{example}

The above procedure defines a separated graph $(E',C')$ that, in fact, is a generalized finite shift graph since $E$ has neither sinks nor sources. Let $(X,\sigma)$ be the associated generalized finite shift.

\begin{proposition}\label{proposition-generalized.shifts.one.sided.conjugate}
The pair $(X,\sigma)$ is topologically conjugate to the one-sided shift $(X_E,\sigma_E)$ as defined in Subsection \ref{subsection-shifts.finite.type}.
\end{proposition}
\begin{proof}
Let $(F,D)$ denote the canonical resolution of $(E',C')$. We first show that the red adjacency matrix $A_2$ associated to the separated graph $(F_2,D^2)$ is exactly the incidence matrix of $E$, and hence equal to the adjacency matrix of $D(E)$, the dual graph of $E$. For this, we first observe that we can identify the vertices in $F^{0,2}$ with the edges of $E$. Indeed, each vertex in $F^{0,2}$ is of the form $v(e_i,\ol{f})$, where $f \in E^1$ with $r_E(f) = i$. Hence there is a bijection between $F^{0,2}$ and $E^1$ sending $v(e_i,\ol{f})$ to $f$. We will identify $F^{0,2}$ with $E^1$ through this bijection. Now for $v_i \in F^{0,1}$ we have
$$B_{v_i} = \{\alpha^{e_i}(\ol{f}) \mid f \in r_E^{-1}(i)\},$$
and
$$R_{v_i} = \{\alpha^{\ol{f}}(e_j) \mid f \in s_E^{-1}(i), j = r_E(f)\}.$$
Denoting by $f^*$ the edge $\alpha ^{\ol{f}}(e_j)$, we have $r(f^*) = v_{s_E(f)}$ and $s(f^*) = f$. Further, we denote by $f'$ the edge $\alpha^{e_i}(\ol{f})$, so $s(f') = f$ and $r(f') = v_{r_E(f)}$.

Now we move to the graph $(F_2, D^2)$. Recall that we are identifying $F^{0,2}$ with $E^1$. For $f \in E^1$, we have only one element in $R(\ol{f})$, namely $R(\ol{f}) = \{f^*\}$. Therefore $F^{0,3}$ can be identified with $F^{0,2} = E^1$ through the bijection $v(f',f_1^*, f_2^*.\dots ,f_p^*) \leftrightarrow f$, where $s_E^{-1}(i) = \{ f_1,f_2,\dots , f_p \}$ and $r_E(f) = i$. Observe that
$$r(\alpha ^{f'}(f_1^*, f_2^*,\dots ,f_p^*)) = s(f') = f \in F^{0,2} = E^1$$
and $s(\alpha^{f'}(f_1^*,f_2^*,\dots , f_p^*)) = v(f', f_1^*,f_2^*,\dots , f_p^*)$, which we identify with $f$ through the above bijection. Hence $B_f$ consists of a single edge, which connects $f\in F^{0,3}$ with $f\in F^{0,2}$.

It remains to identify the red edges in $(F_2,D^2)$. These are of the form
$$\alpha^{g^*}(f',g_2^*,\dots , g_p^*)$$
for $\{ g,g_2,\dots ,g_p \} = s_E^{-1}(i)$ and $r_E(f) = i$ for some $i \in \{1,\dots,n\}$. We then have
$$r(\alpha^{g^*}(f',g_2^*,\dots , g_p^*)) = s(g^*) = g \in F^{0,2}, \quad s(\alpha ^{g^*}(f',g_2^*,\dots , g_p^*)) = f \in F^{0,3}$$
under the above identifications. So there is at most one edge from $f \in F^{0,3}$ to $g \in F^{0,2}$, and there is such edge if and only if $r_E(f) = s_E(g)$. We conclude that the matrix $A_2$ is the incidence matrix of the graph $E$.

More generally, it follows by induction that for each $N \ge 1$ the layer $(F^{0,2N}\cup F^{0,2N+1}, F^{1,2N})$ gives exactly the \textit{$(N+1)^{\text{th}}$ higher edge graph} $E^{[N+1]}$ of $E$, which is defined as the graph having as vertices all the paths of length $N$ in $E$, and having exactly one edge from $e_Ne_{N-1}\cdots e_2e_1$ to $f_Nf_{N-1}\cdots f_2f_1$ if and only if 
$e_N \cdots e_{2} = f_{N-1}\cdots f_{1}$, and none otherwise. The edge is named $f_Ne_N\cdots e_2e_1$ (see e.g. \cite[Definition 2.3.4]{LM2021}).\\

Using these identifications it is now easy to define a conjugation between $(X,\sigma)$ and $(X_E,\sigma_E)$. An element $x = (e_0,e_1,e_2,\dots) \in X_E$ gives rise to a finite path $e_Ne_{N-1} \cdots e_1e_0$ for each $N \geq 0$, which corresponds to a unique vertex in the layer $(F^{0,2N+2}\cup F^{0,2N+3}, F^{1,2N+2})$ of $(F,D)$. By the construction of $(X,\sigma)$ given in Subsection \ref{subsection-from.Ldiagram.to.local.homeo}, it is clear that $x$ determines a unique element $\psi(x) \in X$. The map $\psi : X_E \ra X$ gives the desired conjugacy between $(X_E,\sigma_E)$ and $(X,\sigma)$.
\end{proof}

\begin{remark}\label{remark-true.for.two.sided}
Note that one-sided shifts $(X,\sigma)$ of infinite type are not objects in our category $\textsf{LHomeo}$, because in this case $\sigma$ is not an open map (see \cite[Theorem 1]{IT1974}), and therefore it is not a local homeomorphism. Hence there is no $l$-diagram realizing a one-sided shift of infinite type. Of course, the situation is different with the two-sided shifts, of any type, which are always homeomorphisms, and can be concretely realized by $h$-diagrams (recall Section \ref{section-homeos}). This has been studied in  \cite[6.7--6.10]{AL2018}, and we refer the reader there for further details. Note in particular that, by \cite[Proposition 6.8 (2)]{AL2018}, two-sided shifts of finite type correspond to a particular class of generalized finite shift graphs.
\end{remark}

\subsection{A configurational point of view}\label{subsection-configurations}  

We will now obtain an explicit realization of the generalized finite shifts. To obtain this, we rely on the work done in \cite{AE2014} (see also \cite{AL2018}) to describe the dynamics inherent to a bipartite separated graph. We will use the notation of \cite{AL2018}.

\begin{definition}\cite[Definition 2.11]{AL2018}\label{definition-local.configurations}
Let $(E,C)$ be a finite bipartite separated graph, and let $\mathbb F$ be the free group on $E^1$. Given $\xi \subseteq \mathbb F$ and $\alpha \in \xi$, the {\it local configuration} $\xi_{\alpha}$ of $\xi$ at $\alpha$ is
$$\xi_{\alpha} = \{ \sigma \in E^1\cup (E^1)^{-1} \mid \sigma \in \xi \cdot \alpha^{-1} \}.$$
The space $\Omega(E,C)$ is the space of all {\it configurations} $\xi \subseteq \mathbb F$ such that:
\begin{enumerate}[(a),leftmargin=0.7cm]
\item $1 \in \xi$;
\item $\xi$ is {\it right convex}, which means that $\xi$ is suffix closed: if $x_n\cdots x_2x_1 \in \xi$, where $x_n\cdots x_2x_1$ is a reduced word (so that $x_i\in E^1\cup (E^1)^{-1}$ and $x_i\ne x_{i+1}^{-1}$ for $1\le i< n$), then $x_m\cdots x_2x_1\in \xi $ for all $1 \leq m \leq n$;
\item for each $\alpha \in \xi$, one of the following holds:
\begin{enumerate}
\item[(c.1)] $\xi_{\alpha} = s^{-1}(v)$ for some $v\in E^{0,1}$;
\item[(c.2)] $\xi_{\alpha} = \{ e_X^{-1} \mid X\in C_v\}$ for some $v\in E^{0,0}$ and $e_X \in X$.
\end{enumerate}
\end{enumerate}
In words, $\Omega(E,C)$ is the space of all the configurations $\xi \subseteq \mathbb F$ containing $1$, being right convex, and whose local configuration at each point consists of either emitting all the edges departing from a vertex in $E^{0,1}$, or emitting the inverses of a choice of exactly one edge $e_X \in X$ for each $X \in C_v$, $v \in E^{0,0}$. With the induced topology from $2^{\mathbb F}$, $\Omega (E,C) \subseteq 2^{\mathbb F}$ is a totally disconnected compact metrizable space.
\end{definition}

There is a partial action of $\mathbb F$ on $\Omega(E,C)$ with open compact domains, which is defined by setting, for each $\alpha \in \mathbb F\setminus \{1\}$, the domains as
$$\Omega(E,C)_{\alpha}= \{ \xi \in \Omega(E,C) \mid \alpha^{-1} \in \xi\}$$
and the maps as
$$\theta_{\alpha} : \Omega(E,C)_{\alpha^{-1}} \ra \Omega(E,C)_{\alpha}, \quad \theta_{\alpha}(\xi)= \xi \cdot \alpha^{-1}.$$
Following \cite{AL2018}, we will use the notation $\alpha . \xi$ to denote $\theta_{\alpha}(\xi)$. We thus have $\alpha . \xi = \theta_{\alpha}(\xi)= \xi \cdot \alpha^{-1}$.

For $v \in E^0$, the space $\Omega(E,C)_v$ is defined to be the space of all configurations ``starting" at vertex $v$, that is, $\Omega(E,C)_v$ is the set of configurations $\xi \in \Omega(E,C)$ such that $\xi_1 = s^{-1}(v)$ in case $v \in E^{0,1}$, and the set of configurations $\xi \in \Omega (E,C)$ such that $\xi_1 = \{e_X^{-1} \mid X \in C_v\}$ for a choice of an edge $e_X\in X$ for each $X\in C_v$, in case $v \in E^{0,0}$.

We will also need the notion of an $n$-ball in $\Omega (E,C)$.

\begin{definition}\cite[Definition 3.4]{AL2018}\label{definition-ball.configurations}
Given $\xi \in \Omega (E,C)$ and $n\in \N$, we define the \textit{$n$-ball $\xi^n$} by
$$\xi^n = \{ \alpha \in \xi \mid |\alpha| \le n \}.$$
Here we use the notation $|\alpha|$ to denote the length of the (reduced) word $\alpha \in \mathbb F$.
 
Of course, the $1$-ball $\xi^1$ of $\xi$ is just the local configuration of $\xi$ at $1$, together with the element $1 \in \mathbb F$, so that $\xi^1 = \xi_1 \cup \{1\}$. We denote by $\calB_n(\Omega(E,C))$ the collection of $n$-balls $\xi^n$ associated to elements $\xi \in \Omega (E,C)$.
\end{definition}

Suppose now that $(E,C)$ is a generalized finite shift graph. In our next result we relate the pair $(\Omega(E,C),\theta)$ with the generalized finite shift $(X,\sigma)$ associated with $(E,C)$. Recall from Definition \ref{definition-F.infty.and.others} the notion of the $n$-step resolution $(E_n,C^n)$ of $(E,C)$.  

\begin{proposition}\label{proposition-relating.spaces}
Let $(E,C)$ be a generalized finite shift graph, let $(X,\sigma)$ be the corresponding generalized finite shift, as in Definition \ref{definition-generalized.finite.shift}, and let $\Omega(E,C)$ be the space of configurations on $\mathbb F$ defined above. Then $X$ can be identified with the subspace of configurations $\xi$ whose local configuration $\xi_1$ at $1$ is of the form (c.2). Moreover, we have
$$\sigma(\xi) = \theta_{a_1f_0^{-1}}(\xi),$$
where $\xi_1 = \{e_0^{-1},f_0^{-1}\}$, with $e_0 \in B_v$, $f_0 \in R_{v}$, $v = r(e_0) = r(f_0)$, and $a_1$ is the unique blue edge such that $s(a_1) = s(f_0)$.
\end{proposition}
\begin{proof}
We will use \cite[Theorem 3.22]{AL2018} which, for every finite bipartite separated graph $(E,C)$, establishes a conjugacy
$$\theta^{(E_n,C^n)} \cong (\theta^{(E,C)})^{[n]},$$
where $(\theta^{(E,C)})^{[n]}$ is the $n$-ball convex subshift associated to the convex subshift $\theta^{(E,C)}$ (see \cite[Construction 3.14]{AL2018}). More precisely, there is a canonical bijective correspondence
$$E_n^0 \to \calB _n (\Omega(E,C)), \quad v \mapsto B(v)$$
and a homeomorphism $\varphi_n \colon \Omega (E_n,C^n)\to \Omega (E,C)$ which restricts to a homeomorphism
$$\Omega(E_n,C^n)_v  \to \{ \xi \in \Omega(E,C) \mid \xi^n = B(v) \}$$
for each $v \in E_n^0$. It is clear that $X$ can be identified with the subspace of configurations $\xi$ whose local configuration $\xi_1$ at $1$ is of the form (c.2), through the following correspondence. If $x=(e_0,e_1,\dots )\in X$, then $\xi$ is the unique configuration such that $\xi^{2j-1} = B(e_0,e_1,\dots , e_{2j-1})$, where we interpret $(e_0,e_1,\dots , e_{2j-1})$ as a vertex in $E_{2j-1}^{0,1}$, for all $j\ge 1$.
Note that $(e_0,e_1,\dots , e_{2j-1})$ can also be interpreted as a vertex in $E_{2j}^{0,0}$, and then it correponds to $\xi^{2j}$ through the above bijection.

Furthermore, a group isomorphism $\Phi_n\colon \mathbb F^{[n:\Omega]} \to \mathbb F (E_n^1)$ is established in the proof of \cite[Theorem 3.22]{AL2018}, for all $n\ge 1$, where $\Omega:= \Omega (E,C)$, and $\mathbb F ^{[n:\Omega]} := \mathbb F (A^{[n:\Omega]})$ is the free group defined in \cite[Construction 3.14]{AL2018}. 

Let $\xi \in X \subseteq \Omega(E,C)$. The element $\xi$ is determined by a sequence $(e_0,e_1,e_2,\dots )$ in the Bratteli diagram $(F,D)$. Let $(e_{2j},e_{2j+1},f_{2j},f_{2j+1})$ be the rombs associated to the pairs  $(e_{2j},e_{2j+1})$ of blue edges, for each $j\ge 0$. By definition of $\sigma$ (Construction \ref{construction-local.homeo}), we have that $\sigma (\xi)= \xi'$, with $\xi'= (a_1,e_1',e_2',\dots )$, where $(e_{2j-1}',e_{2j}',f_{2j-1},f_{2j})$ are the rombs associated to  the pairs $(f_{2j-1},f_{2j})$ of red edges, for $j\ge 1$, and $a_1$ is the unique blue edge such that $s(a_1) = r(e_1')= r(f_1)= s(f_0)$. 

Note that, for $j \ge 1$, the edges $e_{2j-1}$ and $e_{2j}$ correspond to the edges 
$$[\xi^{2j-1} \overset{e_0}{\longleftarrow} (e_0^{-1}  .\,  \xi)^{2j-1}] \quad \text{and} \quad [(e_0^{-1}  .\,  \xi)^{2j}  \overset{e_0^{-1}}{\longleftarrow}   \xi^{2j}]$$
in $(\theta^{(E,C)})^{[2j-1]}$ and $(\theta^{(E,C)})^{[2j]}$ under the isomorphisms $\Phi_{2j-1}$ and $\Phi_{2j}$, respectively. Similarly, the edges $f_{2j-1}$ and $f_{2j}$ correspond to the edges
$$[\xi^{2j-1} \overset{f_0}{\longleftarrow} (f_0^{-1}  .\,  \xi)^{2j-1}] \quad \text{and} \quad [(f_0^{-1}  .\,  \xi)^{2j}  \overset{f_0^{-1}}{\longleftarrow}   \xi^{2j}]$$
in $(\theta^{(E,C)})^{[2j-1]}$ and $(\theta^{(E,C)})^{[2j]}$ under the isomorphisms $\Phi_{2j-1}$ and $\Phi_{2j}$, respectively. It follows that the edges $e'_{2j-1}$ and $e'_{2j}$ correspond to the edges
$$[(a_1f_0^{-1} .\, \xi)^{2j-1} \overset{a_1}{\longleftarrow} (f_0^{-1}  .\,  \xi)^{2j-1}] \quad \text{and} \quad [(f_0^{-1}  .\,  \xi)^{2j}  \overset{a_1^{-1}}{\longleftarrow}   (a_1f_0^{-1} .\, \xi)^{2j}]$$
in $(\theta^{(E,C)})^{[2j-1]}$ and $(\theta^{(E,C)})^{[2j]}$, respectively. Therefore $\xi' = (a_1,e_1',e_2',\dots )$ corresponds to $(a_1f_0^{-1}).\xi = \theta _{a_1f_0^{-1}}(\xi)$, that is, $\sigma(\xi) = \theta_{a_1f_0^{-1}}(\xi)$.
\end{proof}

With the help of the above proposition, we can give a more concrete form of a generalized finite shift.

\begin{construction}\label{construction-GFS.concretedescription}
Let $(E,C)$ be a generalized finite shift graph, and set $R = \bigsqcup_{v\in E^{0,0}} R_v$. For each $f \in R$, consider a symbol $a_f$, and form the alphabet $A = \{ a_f \mid f \in R\}$. Moreover, let $s,r \colon A \to E^{0,0}$ be the maps defined by $s(a_f) = r(f)$ and $r(a_f) = r(e)$, where $e$ is the unique blue edge in $E^1$ such that $s(e)= s(f)$. Let $\mathfrak X$ be the space of configurations $\xi \subseteq \mathbb F (A)$ on the free group $\mathbb F (A)$ on $A$ such that
\begin{enumerate}[(a),leftmargin=0.7cm]
\item $1\in \xi$;
\item $\xi$ is right convex;
\item for each $\alpha \in \xi$, the local configuration of $\xi$ at $\alpha $ is of the following form: there exists a vertex $v\in E^{0,0}$ such that
$$(\alpha.\xi)^1 = \{1\}\cup \{a_f\} \cup \{a_{f_i}^{-1} \mid 1 \leq i \leq n\},$$
where $f \in R_v$ and there is a blue edge $e\in B_v$ such that $\{ f_1,\dots , f_n\}$ is the set of all red edges $g\in R$ such that $s(g)= s(e)$.
\end{enumerate}
The local homeomorphism $\tau$ on $\mathfrak X$ is defined by
$$\tau(\xi) = \xi \cdot a_{f}^{-1},$$
where $f$ is the unique red edge such that $a_{f}\in \xi^1$. 

Observe that any configuration $\xi \in \mathfrak X$ contains only elements of the form
$$a_{f_1}^{-1}\cdots a_{f_m}^{-1}a_{g_n}\cdots a_{g_1}$$
for $n,m\ge 0$, where $f_i,g_j\in R$,  $f_m\ne g_n$, $s(f_m)= s(g_n)$, $r(g_{j+1})=r(e_j)$, where $e_j$ is the unique blue edge such that $s(e_j)= s(g_j)$ for $1 \leq j \leq n-1$, and $r(f_{i+1})=r(e_i')$, where $e_i'$ is the unique blue edge such that $s(e_i')= s(f_i)$ for $1 \leq i \leq m-1$.   
\end{construction}

Note that, for $\xi \in \mathfrak X$, and with $\{f_1,\dots ,f_n\}$ the red edges of $E$ such that $a_{f_i}^{-1}\in \xi^1$ for all $1 \leq i \leq n$, we have $\xi \cdot a_{f_i} \in \mathfrak X$, and $\tau (\xi \cdot a_{f_i})= \xi$. Hence $\tau $ is surjective, but not necessarily injective. 

\begin{proposition}\label{proposition-description.of.X}
Let $(E,C)$ be a generalized finite shift graph. Then the dynamical systems $(X,\sigma)$ and $(\mathfrak X,\tau)$ are topologically conjugate. 
\end{proposition}
\begin{proof}
This follows from Proposition \ref{proposition-relating.spaces}. Indeed, if $\xi \in \Omega(E,C)$ is a configuration whose local configuration $\xi_1$ at $1$ is of the form (c.2), then replacing the arrows $ef^{-1}$, where $f$ is the unique red edge in $E^1$ such that $f^{-1}\in (\alpha \, . \, \xi)^1$, and $e$ is the unique blue edge such that $s(e) = s(f)$, by $a_f$, for each $\alpha \in \xi$ of even length, we obtain a configuration in $\mathfrak X$. And clearly, the process can be reversed. This also gives the desired intertwining between $\sigma$ and $\tau$.  
\end{proof}

We now consider a concrete example of a generalized finite shift which is neither a one-sided nor a two-sided shift. 

\begin{example}\label{example-generalized.shift.config}
Let $(E,C)$ be the generalized finite shift graph with associated red-blue adjacency matrices
$$A = \begin{pmatrix} 2 \\ 1 \end{pmatrix} \quad \text{and} \quad I = \begin{pmatrix} 1 \\ 1 \end{pmatrix}.$$
Concretely, we set $E^{0,0} = \{v\}$, $E^{0,1}= \{w_1,w_2\}$, $C_v= \{B_v,R_v\}$, where $B_v= \{ \alpha_0,\alpha_1\}$, $R_v=\{ \beta_0,\beta_1,\beta_2\}$, with all edges having range $v$, and with $s(\alpha_0)= s(\beta_0)= s(\beta_1)= w_1$, $s(\alpha_1)= s(\beta_2)= w_2$. 

\begin{figure}[H]
\begin{tikzpicture}
	\node[label=below:$w_1$,circle,fill=blue,scale=0.4] (W1) at (-2,0) {};
	\node[label=below:$w_2$,circle,fill=blue,scale=0.4] (W2) at (2,0) {};
	\node[label=above:$v$,circle,fill=blue,scale=0.4] (V) at (0,2) {};
	\draw[-,blue,postaction={on each segment={mid arrow=blue}}] (W1) to [out=110,in=200] node[above] {$\alpha_0$} (V);
	\draw[-,red,postaction={on each segment={mid arrow=red}}] (W1) to [out=80,in=210] node[below] {$\beta_0$} (V);
	\draw[-,red,postaction={on each segment={mid arrow=red}}] (W1) to [out=20,in=260] node[below] {$\beta_1$} (V);
	\draw[-,blue,postaction={on each segment={mid arrow=blue}}] (W2) to [out=100,in=340] node[above] {$\alpha_1$} (V);
	\draw[-,red,postaction={on each segment={mid arrow=red}}] (W2) to [out=160,in=280] node[below] {$\beta_2$} (V);
\end{tikzpicture}
\caption{The generalized finite shift graph $(E,C)$.}
\label{figure-example4}
\end{figure}
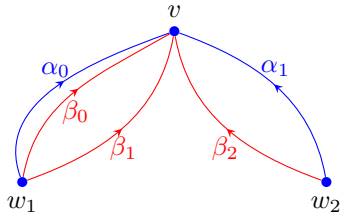

The possible local configurations are of the following form.

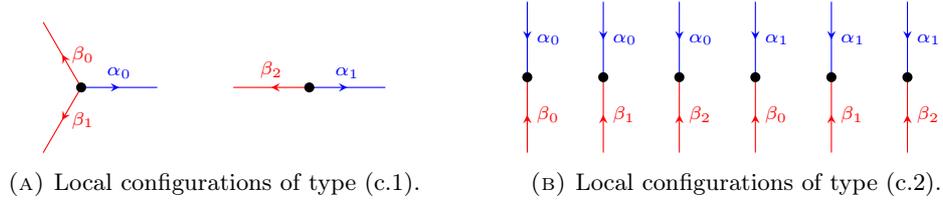
\begin{figure}[H]
\centering
	\begin{subfigure}[t]{0.4\textwidth}
	\centering
		\begin{tikzpicture}
			\path [draw=blue,postaction={on each segment={mid arrow=blue}}]
			(-1.5,0) -- node[above,blue]{\scriptsize $\alpha_0$} (-0.5,0)
			(1.5,0) -- node[above,blue]{\scriptsize $\alpha_1$} (2.5,0)
			;
			\path [draw=red,postaction={on each segment={mid arrow=red}}]
			(1.5,0) -- node[above,red]{\scriptsize $\beta_2$} (0.5,0)
			(-1.5,0) -- node[right,red]{\scriptsize $\beta_0$} (-2,0.866)
			(-1.5,0) -- node[right,red]{\scriptsize $\beta_1$} (-2,-0.866)
			;
			\node[circle,fill=black,scale=0.4] () at (-1.5,0) {};
			\node[circle,fill=black,scale=0.4] () at (1.5,0) {};
		\end{tikzpicture}
	\caption{Local configurations of type (c.1).}
	\end{subfigure}
	\begin{subfigure}[t]{0.4\textwidth}
	\centering
		\begin{tikzpicture}
			\path [draw=blue,postaction={on each segment={mid arrow=blue}}]
			(-2.5,1) -- node[right,blue]{\scriptsize $\alpha_0$} (-2.5,0)
			(-1.5,1) -- node[right,blue]{\scriptsize $\alpha_0$} (-1.5,0)
			(-0.5,1) -- node[right,blue]{\scriptsize $\alpha_0$} (-0.5,0)
			(0.5,1) -- node[right,blue]{\scriptsize $\alpha_1$} (0.5,0)
			(1.5,1) -- node[right,blue]{\scriptsize $\alpha_1$} (1.5,0)
			(2.5,1) -- node[right,blue]{\scriptsize $\alpha_1$} (2.5,0)
			;
			\path [draw=red,postaction={on each segment={mid arrow=red}}]
			(-2.5,-1) -- node[right,red]{\scriptsize $\beta_0$} (-2.5,0)
			(-1.5,-1) -- node[right,red]{\scriptsize $\beta_1$} (-1.5,0)
			(-0.5,-1) -- node[right,red]{\scriptsize $\beta_2$} (-0.5,0)
			(0.5,-1) -- node[right,red]{\scriptsize $\beta_0$} (0.5,0)
			(1.5,-1) -- node[right,red]{\scriptsize $\beta_1$} (1.5,0)
			(2.5,-1) -- node[right,red]{\scriptsize $\beta_2$} (2.5,0)
			;
			\node[circle,fill=black,scale=0.4] () at (-2.5,0) {};
			\node[circle,fill=black,scale=0.4] () at (-1.5,0) {};
			\node[circle,fill=black,scale=0.4] () at (-0.5,0) {};
			\node[circle,fill=black,scale=0.4] () at (0.5,0) {};
			\node[circle,fill=black,scale=0.4] () at (1.5,0) {};
			\node[circle,fill=black,scale=0.4] () at (2.5,0) {};
		\end{tikzpicture}
	\caption{Local configurations of type (c.2).}
	\end{subfigure}
\caption{The different possible local configurations for the graph $(E,C)$.}
\label{figure-example5}
\end{figure}
In particular, the possible local configurations at $1$ are the ones shown in Figure \ref{figure-example5} (B). In Figure \ref{figure-example6} below we give an example of the local description of the shift map $\sigma$.

\begin{figure}[H]
\begin{tikzpicture}
	\node[label=above:$1$,circle,fill=green,draw=black!80,scale=0.6] (1) at (0,0) {};
	\node[circle,fill=black,scale=0.4] (R) at (1,0) {};
	\node[circle,fill=green,draw=black!80,scale=0.4] (R1) at (1.5,0.866) {};
	\node[circle,fill=black,scale=0.4] (R11) at (2,1.732) {};
	\node[circle,fill=green,draw=black!80,scale=0.4] (R111) at (1.5,2.598) {};
	\node[circle,fill=green,draw=black!80,scale=0.4] (R112) at (3,1.732) {};
	\node[circle,fill=green,draw=black!80,scale=0.4] (R2) at (1.5,-0.866) {};
	\node[circle,fill=black,scale=0.4] (R21) at (2,-1.732) {};
	\node[circle,fill=green,draw=black!80,scale=0.4] (R211) at (3,-1.732) {};
	\node[circle,fill=green,draw=black!80,scale=0.4] (R212) at (1.5,-2.598) {};
	\node[circle,fill=black,scale=0.4] (L) at (-1,0) {};
	\node[circle,fill=green,draw=black!80,scale=0.4] (L1) at (-1.5,0.866) {};
	\node[circle,fill=black,scale=0.4] (L11) at (-2,1.732) {};
	\node[circle,fill=green,draw=black!80,scale=0.4] (L111) at (-2.5,2.598) {};
	\node[circle,fill=green,draw=black!80,scale=0.4] (L2) at (-1.5,-0.866) {};
	\node[circle,fill=black,scale=0.4] (L21) at (-2,-1.732) {};
	\node[circle,fill=green,draw=black!80,scale=0.4] (L211) at (-3,-1.732) {};
	\node[circle,fill=green,draw=black!80,scale=0.4] (L212) at (-1.5,-2.598) {};
	\node[circle,fill=green,draw=black!80,scale=0.4] (1') at (8,0) {};
	\node[circle,fill=black,scale=0.4] (R') at (9,0) {};
	\node[circle,fill=green,draw=black!80,scale=0.4] (R'1) at (9.5,0.866) {};
	\node[circle,fill=black,scale=0.4] (R'11) at (10,1.732) {};
	\node[circle,fill=green,draw=black!80,scale=0.4] (R'111) at (9.5,2.598) {};
	\node[circle,fill=green,draw=black!80,scale=0.4] (R'112) at (11,1.732) {};
	\node[circle,fill=green,draw=black!80,scale=0.4] (R'2) at (9.5,-0.866) {};
	\node[circle,fill=black,scale=0.4] (R'21) at (10,-1.732) {};
	\node[circle,fill=green,draw=black!80,scale=0.4] (R'211) at (11,-1.732) {};
	\node[circle,fill=green,draw=black!80,scale=0.4] (R'212) at (9.5,-2.598) {};
	\node[circle,fill=black,scale=0.4] (L') at (7,0) {};
	\node[circle,fill=green,draw=black!80,scale=0.4] (L'1) at (6.5,0.866) {};
	\node[circle,fill=black,scale=0.4] (L'11) at (6,1.732) {};
	\node[circle,fill=green,draw=black!80,scale=0.4] (L'111) at (5.5,2.598) {};
	\node[label=right:$1$,circle,fill=green,draw=black!80,scale=0.6] (L'2) at (6.5,-0.866) {};
	\node[circle,fill=black,scale=0.4] (L'21) at (6,-1.732) {};
	\node[circle,fill=green,draw=black!80,scale=0.4] (L'211) at (5,-1.732) {};
	\node[circle,fill=green,draw=black!80,scale=0.4] (L'212) at (6.5,-2.598) {};
		\path [draw=blue,postaction={on each segment={mid arrow=blue}}]
	(R) -- node[above,blue]{\footnotesize $\alpha_0$} (1)
	(R11) -- node[above,blue,xshift=-5pt]{\footnotesize $\alpha_0$} (R1)
	(R21) -- node[above,blue,xshift=5pt]{\footnotesize $\alpha_0$} (R2)
	(L) -- node[above,blue,xshift=-5pt]{\footnotesize $\alpha_0$} (L2)
	(L21) -- node[below,blue,xshift=-5pt]{\footnotesize $\alpha_0$} (L212)
	(L11) -- node[above,blue,xshift=5pt]{\footnotesize $\alpha_1$} (L1)
	;
	\path [draw=red,postaction={on each segment={mid arrow=red}}]
	(L) -- node[above,red]{\footnotesize $\beta_1$} (1)
	(L) -- node[above,red,xshift=5pt]{\footnotesize $\beta_0$} (L1)
	(L11) -- node[above,red,xshift=5pt]{\footnotesize $\beta_2$} (L111)
	(L21) -- node[above,red,xshift=-5pt]{\footnotesize $\beta_0$} (L2)
	(L21) -- node[above,red]{\footnotesize $\beta_1$} (L211)
	(R) -- node[above,red,xshift=-5pt]{\footnotesize $\beta_0$} (R1)
	(R) -- node[above,red,xshift=5pt]{\footnotesize $\beta_1$} (R2)
	(R11) -- node[above,red,xshift=5pt]{\footnotesize $\beta_0$} (R111)
	(R11) -- node[above,red]{\footnotesize $\beta_1$} (R112)
	(R21) -- node[above,red]{\footnotesize $\beta_1$} (R211)
	(R21) -- node[below,red,xshift=5pt]{\footnotesize $\beta_0$} (R212)
	;
	\draw[black,dashed] (L111) -- (-2.75,3.031);
	\draw[black,dashed] (R111) -- (1.25,3.031);
	\draw[black,dashed] (R112) -- (3.5,1.732);
	\draw[black,dashed] (L211) -- (-3.5,-1.732);
	\draw[black,dashed] (L212) -- (-1.25,-3.031);
	\draw[black,dashed] (R212) -- (1.25,-3.031);
	\draw[black,dashed] (R211) -- (3.5,-1.732);
	\draw[black,dashed] (L'111) -- (5.25,3.031);
	\draw[black,dashed] (R'111) -- (9.25,3.031);
	\draw[black,dashed] (R'112) -- (11.5,1.732);
	\draw[black,dashed] (L'211) -- (4.5,-1.732);
	\draw[black,dashed] (L'212) -- (6.75,-3.031);
	\draw[black,dashed] (R'212) -- (9.25,-3.031);
	\draw[black,dashed] (R'211) -- (11.5,-1.732);
	\path [draw=blue,postaction={on each segment={mid arrow=blue}}]
	(R') -- node[above,blue]{\footnotesize $\alpha_0$} (1')
	(R'11) -- node[above,blue,xshift=-5pt]{\footnotesize $\alpha_0$} (R'1)
	(R'21) -- node[above,blue,xshift=5pt]{\footnotesize $\alpha_0$} (R'2)
	(L'21) -- node[below,blue,xshift=-5pt]{\footnotesize $\alpha_0$} (L'212)
	(L'11) -- node[above,blue,xshift=5pt]{\footnotesize $\alpha_1$} (L'1)
	;
	\path [draw=blue,line width=0.6mm,postaction={on each segment={mid arrow=blue}}]
	(L') -- node[above,blue,xshift=-5pt]{\footnotesize $\alpha_0$} (L'2)
	;
	\path [draw=red,postaction={on each segment={mid arrow=red}}]
	(L') -- node[above,red,xshift=5pt]{\footnotesize $\beta_0$} (L'1)
	(L'11) -- node[above,red,xshift=5pt]{\footnotesize $\beta_2$} (L'111)
	(L'21) -- node[above,red,xshift=-5pt]{\footnotesize $\beta_0$} (L'2)
	(L'21) -- node[above,red]{\footnotesize $\beta_1$} (L'211)
	(R') -- node[above,red,xshift=-5pt]{\footnotesize $\beta_0$} (R'1)
	(R') -- node[above,red,xshift=5pt]{\footnotesize $\beta_1$} (R'2)
	(R'11) -- node[above,red,xshift=5pt]{\footnotesize $\beta_0$} (R'111)
	(R'11) -- node[above,red]{\footnotesize $\beta_1$} (R'112)
	(R'21) -- node[above,red]{\footnotesize $\beta_1$} (R'211)
	(R'21) -- node[below,red,xshift=5pt]{\footnotesize $\beta_0$} (R'212)
	;
	\path [draw=red,line width=0.6mm,postaction={on each segment={mid arrow=red}}]
	(L') -- node[above,red]{\footnotesize $\beta_1$} (1')
	;
	\draw[-,black,line width=0.7mm,postaction={on each segment={mid arrow=black}}] (3.1,0) to [out=45,in=135] node [above] (S) {$\sigma$} (4.9,0);
	\draw[-,black,dotted,line width=0.4mm,postaction={on each segment={mid arrow=black}}] (7.93,-0.07) to [out=190,in=50] (6.6,-0.766);
	\node[label=above:$\xi$,scale=0.9] () at (0,1.732) {};
	\node[label=above:$\sigma(\xi)$,scale=0.9] () at (8,1.732) {};
\end{tikzpicture}
\caption{Description of the shift map $\sigma$ on $X \subseteq \Omega(E,C)$. Here the configuration $\xi$ is given by $\xi = \{1,\alpha_0^{-1},\beta_1^{-1},\beta_0\alpha_0^{-1},\beta_1\alpha_0^{-1},\beta_0\beta_1^{-1},\alpha_0\beta_1^{-1},\dots\}$, which has local configuration at $1$ given by $\{\alpha_0^{-1},\beta_1^{-1}\}$. The shifted configuration $\sigma(\xi)$ has local configuration at $1$ given by $\{\alpha_0^{-1},\beta_0^{-1}\}$.}
\label{figure-example6}
\end{figure}
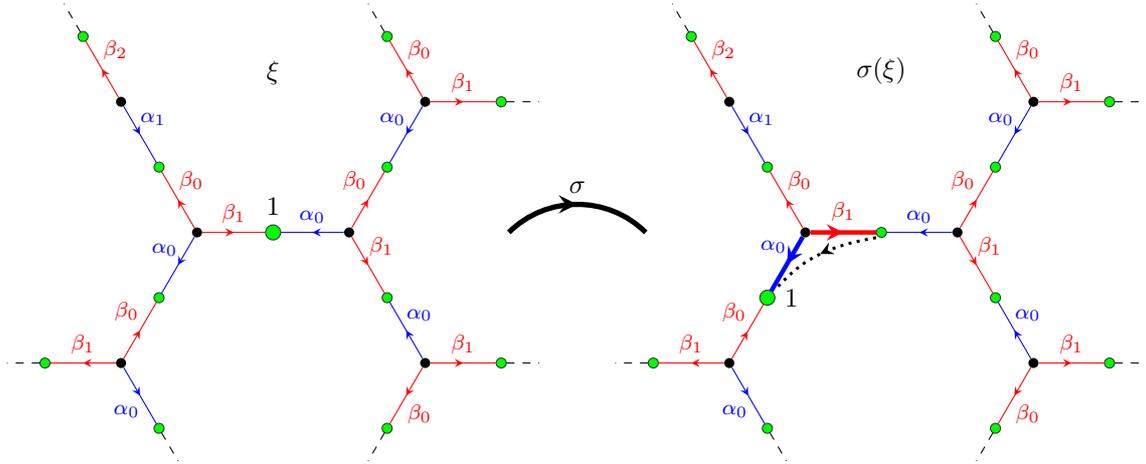

Taking $a = \alpha_0 \beta_0^{-1}$, $b = \alpha_0 \beta_1^{-1}$, $c = \alpha_1 \beta_2^{-1}$, we can further describe the generalized finite shift $(X,\sigma)$ associated with $(E,C)$ as follows.

Let $\mathbb F$ be the free group on $\{a,b,c\}$. Then $X$ is identified as the set $\mathfrak X$ of all configurations $\xi \subseteq \mathbb F$ satisfying the conditions (a)--(c) in Construction \ref{construction-GFS.concretedescription}. In particular, the local configuration of $\xi$ at each point is either of the form $\{z,a^{-1},b^{-1}\}$ or of the form $\{z,c^{-1}\}$, where $z \in \{a,b,c\}$. The local homeomorphism $\sigma$ is identified with $\tau$, which consists in shifting one step in the direction pointed by the directed edges.

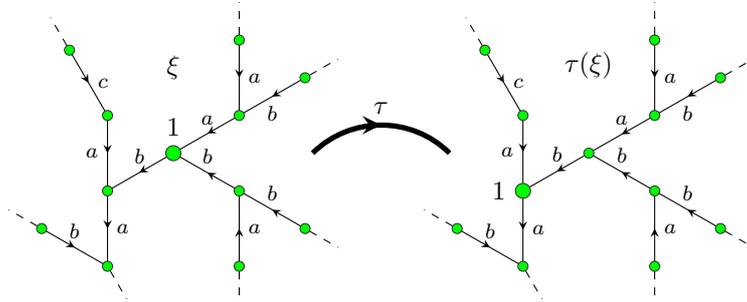
\begin{figure}[H]
\begin{tikzpicture}
	\node[label=above:$1$,circle,fill=green,draw=black!80,scale=0.6] (1) at (0,0) {};
	\node[circle,fill=green,draw=black!80,scale=0.4] (R1) at (0.866,0.5) {};
	\node[circle,fill=green,draw=black!80,scale=0.4] (R111) at (0.866,1.5) {};
	\node[circle,fill=green,draw=black!80,scale=0.4] (R112) at (1.732,1) {};
	\node[circle,fill=green,draw=black!80,scale=0.4] (R2) at (0.866,-0.5) {};
	\node[circle,fill=green,draw=black!80,scale=0.4] (R211) at (1.732,-1) {};
	\node[circle,fill=green,draw=black!80,scale=0.4] (R212) at (0.866,-1.5) {};
	\node[circle,fill=green,draw=black!80,scale=0.4] (L1) at (-0.866,0.5) {};
	\node[circle,fill=green,draw=black!80,scale=0.4] (L111) at (-1.366,1.366) {};
	\node[circle,fill=green,draw=black!80,scale=0.4] (L2) at (-0.866,-0.5) {};
	\node[circle,fill=green,draw=black!80,scale=0.4] (L211) at (-1.732,-1) {};
	\node[circle,fill=green,draw=black!80,scale=0.4] (L212) at (-0.866,-1.5) {};
	\node[circle,fill=green,draw=black!80,scale=0.4] (1') at (5.464,0) {};
	\node[circle,fill=green,draw=black!80,scale=0.4] (R'1) at (6.33,0.5) {};
	\node[circle,fill=green,draw=black!80,scale=0.4] (R'111) at (6.33,1.5) {};
	\node[circle,fill=green,draw=black!80,scale=0.4] (R'112) at (7.196,1) {};
	\node[circle,fill=green,draw=black!80,scale=0.4] (R'2) at (6.33,-0.5) {};
	\node[circle,fill=green,draw=black!80,scale=0.4] (R'211) at (7.196,-1) {};
	\node[circle,fill=green,draw=black!80,scale=0.4] (R'212) at (6.33,-1.5) {};
	\node[circle,fill=green,draw=black!80,scale=0.4] (L'1) at (4.598,0.5) {};
	\node[circle,fill=green,draw=black!80,scale=0.4] (L'111) at (4.098,1.366) {};
	\node[label=left:$1$,circle,fill=green,draw=black!80,scale=0.6] (L'2) at (4.598,-0.5) {};
	\node[circle,fill=green,draw=black!80,scale=0.4] (L'211) at (3.732,-1) {};
	\node[circle,fill=green,draw=black!80,scale=0.4] (L'212) at (4.598,-1.5) {};
	\path [draw=black,postaction={on each segment={mid arrow=black}}]
	(1) -- node[above,black]{\footnotesize $b$} (L2)
	(L2) -- node[right,black]{\footnotesize $a$} (L212)
	(L211) -- node[above,black]{\footnotesize $b$} (L212)
	(L1) -- node[left,black]{\footnotesize $a$} (L2)
	(L111) -- node[right,black]{\footnotesize $c$} (L1)
	(R1) -- node[above,black]{\footnotesize $a$} (1)
	(R2) -- node[above,black]{\footnotesize $b$} (1)
	(R111) -- node[right,black]{\footnotesize $a$} (R1)
	(R112) -- node[below,black]{\footnotesize $b$} (R1)
	(R211) -- node[above,black]{\footnotesize $b$} (R2)
	(R212) -- node[right,black]{\footnotesize $a$} (R2)
	;
	\draw[black,dashed] (L111) -- (-1.616,1.799);
	\draw[black,dashed] (L211) -- (-2.165,-0.75);
	\draw[black,dashed] (L212) -- (-0.616,-1.933);
	\draw[black,dashed] (R111) -- (0.866,2);
	\draw[black,dashed] (R112) -- (2.165,1.25);
	\draw[black,dashed] (R211) -- (2.165,-1.25);
	\draw[black,dashed] (R212) -- (0.866,-2);
	;
	\path [draw=black,postaction={on each segment={mid arrow=black}}]
	(1') -- node[above,black]{\footnotesize $b$} (L'2)
	(L'2) -- node[right,black]{\footnotesize $a$} (L'212)
	(L'211) -- node[above,black]{\footnotesize $b$} (L'212)
	(L'1) -- node[left,black]{\footnotesize $a$} (L'2)
	(L'111) -- node[right,black]{\footnotesize $c$} (L'1)
	(R'1) -- node[above,black]{\footnotesize $a$} (1')
	(R'2) -- node[above,black]{\footnotesize $b$} (1')
	(R'111) -- node[right,black]{\footnotesize $a$} (R'1)
	(R'112) -- node[below,black]{\footnotesize $b$} (R'1)
	(R'211) -- node[above,black]{\footnotesize $b$} (R'2)
	(R'212) -- node[right,black]{\footnotesize $a$} (R'2)
	;
	\draw[black,dashed] (L'111) -- (3.848,1.799);
	\draw[black,dashed] (L'211) -- (3.299,-0.75);
	\draw[black,dashed] (L'212) -- (4.848,-1.933);
	\draw[black,dashed] (R'111) -- (6.33,2);
	\draw[black,dashed] (R'112) -- (7.629,1.25);
	\draw[black,dashed] (R'211) -- (7.629,-1.25);
	\draw[black,dashed] (R'212) -- (6.33,-2);
	;
	\draw[-,black,line width=0.7mm,postaction={on each segment={mid arrow=black}}] (1.832,0) to [out=45,in=135] node [above] (T) {$\tau$} (3.632,0);
	\node[label=above:$\xi$,scale=0.9] () at (0,0.75) {};
	\node[label=above:$\tau(\xi)$,scale=0.9] () at (5.464,0.75) {};
\end{tikzpicture}
\caption{Description of the shift map $\tau$ on $\mathfrak X$. Here the configuration $\xi$ is given by $\xi = \{1,b,a^{-1},b^{-1},ab,a^{-1}b,a^{-2},b^{-1}a^{-1},b^{-2},a^{-1}b^{-1},\dots\}$, which has local configuration at $1$ given by $\{b,a^{-1},b^{-1}\}$. The shifted configuration $\sigma(\xi)$ has local configuration at $1$ given by $\{a,a^{-1},b^{-1}\}$. The picture conveys the same idea of a river basin  present for instance in \cite[p. 5041]{DE2017}. From any vertex there is only one way forward, and this fact enables us to define the shift map.}
\label{figure-example7}
\end{figure}

\end{example}

We now establish the universal property of the dynamical system $(X,\sigma)$ associated to a generalized finite shift graph $(E,C)$. Given such a graph, we set $B := \bigsqcup_{v\in E^{0,0}} B_v$ and $R:= \bigsqcup_{v\in E^{0,0}} R_v$.

\begin{definition}\label{definition-universal.for.generalizedSFT}
Let $(E,C)$ be a generalized finite shift graph with red adjacency matrix $A = (a_{e,v}) \in M_{B \times E^{0,0}}(\Z^+)$, where as usual $a_{e,v}$ is the number of red edges from $s(e)$ to $v$. Let $(Y,\rho)$ be an object from $\textsf{LHomeo}$. An {\it $(E,C)$-structure} on $(Y,\rho)$ is given by a $\rho$-cylinder decomposition
$$\calZ = (\{Z_e \mid e \in B\}, \{W_{e,v}^{(k)} \mid e \in B, v \in E^{0,0}, 1\le k\le a_{e,v}\}).$$
That is, we have
$$Y = \bigsqcup_{e \in B} Z_e, \qquad W_v:= \bigsqcup_{e\in B_v} Z_e = \bigsqcup_{e\in B}\bigsqcup_{k=1}^{a_{e,v}} W_{e,v}^{(k)},$$    
$\rho|_{W_{e,v}^{(k)}}$ is injective, and $\rho(W_{e,v}^{(k)})= Z_e$ for all allowable indices $e,v,k$.

We say that an $(E,C)$-structure $\calZ = (\{Z_e\},\{W_{e,v}^{(k)}\})$ on $(Y,\rho)$ is {\it universal} if for any other $(E,C)$-structure $\tilde{\calZ} = (\{\tilde{Z}_e\},\{ \tilde{W}_{e,v}^{(k)} \})$ on $(\tilde{Y},\tilde{\rho})\in \textsf{LHomeo}$, there is a unique $(E,C)$-equivariant continuous map $\psi\colon \tilde{Y}\to Y$. Here $(E,C)$-equivariant means that $\psi(\tilde{Z}_e) \subseteq Z_e$ and $\psi(\tilde{W}_{e,v}^{(k)}) \subseteq W_{e,v}^{(k)}$ for all allowable indices $e,v,k$, and that $\psi \circ \tilde{\rho} = \rho \circ \psi$.
\end{definition}

By Propositions \ref{proposition-multiresolution.is.Ldiagram} and \ref{proposition-even.labels.sigma.cylinder.dec}, every generalized finite shift $(X,\sigma)$, associated to a generalized finite shift graph $(E,C)$, has a natural $(E,C)$-structure. It turns out that this structure is universal.

\begin{theorem}\label{theorem-universal.(E,C).structure}
Let $(X,\sigma)$ be the generalized finite shift associated to the generalized finite shift graph $(E,C)$. Then the natural $(E,C)$-structure on $X$ is universal.
\end{theorem}
\begin{proof}
The result follows from the universality of the $(E,C)$-dynamical system $\Omega(E,C)$ (\cite[Corollary 6.11]{AE2014}). We thus recall the definition of this dynamical system. We will identify $X$ as a subspace of $\Omega(E,C)$ throughout the proof, as in Proposition \ref{proposition-relating.spaces}.

The $(E,C)$-dynamical system on $\Omega(E,C)$ consists of the collection of open compact subsets $\Omega(E,C)_v$, for $v\in E^0$, and the open compact subsets $H_e$ and $H_f$, for $e\in B$ and $f\in R$, where
$$H_e = \{\xi \in \Omega(E,C) \mid e^{-1} \in \xi_1\}, \qquad H_f = \{\xi \in \Omega(E,C) \mid f^{-1} \in \xi_1\},$$
together with the homeomorphisms $\theta_e \colon \Omega(E,C)_{s(e)} \to H_e$ and $\theta_f \colon \Omega(E,C)_{s(f)} \to H_f$, which we defined just after Definition \ref{definition-local.configurations}. Note that for each $v\in E^{0,0}$, we have $\Omega(E,C)_v = \bigsqcup_{e\in B_v} H_e = \bigsqcup_{f\in R_v} H_f \subseteq X$.

Let us also identify the natural $(E,C)$-structure on $(X,\sigma)$, adopting the picture of $(X,\sigma)$ from Proposition \ref{proposition-relating.spaces}. We have that the space corresponding to $e \in B$ is just $H_e=Z(e)$, and the space corresponding to $f \in R$ is just $H_f$. Of course we have $X = \bigsqcup_{v\in E^{0,0}} \Omega(E,C)_v$ and $\Omega(E,C)_v = \bigsqcup_{e\in B_v} H_e$ for each $v\in E^{0,0}$. We also have
$$\Omega(E,C)_v = \bigsqcup_{f \in R_v} H_f = \bigsqcup_{e \in B} \bigsqcup_{k=1}^{a_{e,v}} H_{e,v}^{(k)},$$
where $H_{e,v}^{(k)} := H_{f_{e,v}^{(k)}}$, since for each $f \in R_v$ there is a unique $e \in B$ such that $s(e) = s(f)$, and then there is a unique $1 \le k\le a_{e,v}$ such that $f = f_{e,v}^{(k)}$. By the proof of Proposition \ref{proposition-relating.spaces}, the space $H_{e,v}^{(k)}$ agrees with the space $V_{e,v}^{(k)}$ defined just before Proposition \ref{proposition-existence.sigma.cylinder.dec}, namely the space of all $x=(e_0,e_1,\dots )\in X$ such that the romb corresponding to the pair of blue edges $(e_0,e_1)$ is of the form $(e_0,e_1,f_0,f_1)$ with $f_0=f_{e,v}^{(k)}$. Hence, the $(E,C)$-structure $\calH :=(\{H_e\}, \{H_{e,v}^{(k)}\})$ on $(X,\sigma)$ agrees with the previously defined $\sigma$-cylinder decomposition $\calZ_0(F,D)$, where $(F,D)$ is the canonical resolution of $(E,C)$.

Now let $\calZ = (\{Z_e\}, \{W_{e,v}^{(k)}\})$ be an $(E,C)$-structure on $(Y,\rho) \in \textsf{LHomeo}$. We need to show that there is a unique $(E,C)$-equivariant continuous map $\psi \colon Y \to X$. We first build an $(E,C)$-dynamical system
$$(\{\ol{Y}_v \mid v \in E^{0}\}, \{\ol{H}_x \mid x \in E^1\},\{\ol{\theta}_x \mid x \in E^1\})$$
on $\ol{Y} = Y^0 \sqcup Y^1$, where $Y^0$ and $Y^1$ are disjoint copies of $Y$. For $i=0,1$ we denote by $y^i$ the canonical image of $y \in Y$ in $Y^i$, and by $N^i$ the canonical image of a subset $N$ of $Y$ in $Y^i$, with $W^{(k,i)}_{e,v}: = (W^{(k)}_{e,v})^i$.

First we set $\ol{Y}_v = \bigsqcup_{e \in B_v} Z_{e}^0$ for $v \in E^{0,0}$, and $\ol{Y}_w = Z_e^1$ for $w \in E^{0,1}$, where $e$ is the unique blue edge such that $s(e) = w$. Now we set $\ol{H}_e = Z^0_e$ for $e \in B$, and $\ol{H}_f = W_{e,v}^{(k,0)}$ for $f \in R$ with $r(f) = v$, where $e \in B$ is the unique blue edge such that $s(e) = s(f)$, and $k$ is the unique index in $\{1,\dots,a_{e,v}\}$ such that $f = f_{e,v}^{(k)}$. Finally we define
$$\ol{\theta}_e \colon \ol{Y}_{s(e)} \to \ol{H}_e, \qquad \ol{\theta}_e(z^1) = z^0, \text{ for } z\in Z_e,$$
and
$$\ol{\theta}_f \colon \ol{Y}_{s(f)} \to \ol{H}_f, \qquad \ol{\theta}_f(z^1) = \Big( \rho|_{W_{e,v}^{(k)}}\Big)^{-1} (z)^0, \text{ for } z\in Z_e, \text{ where } f = f_{e,v}^{(k)}.$$
By \cite[Corollary 6.11]{AE2014}, there exists a unique $(E,C)$-equivariant continuous map $\ol{\psi}\colon \ol{Y}\to \Omega (E,C)$. This means that $\ol{\psi}(\ol{Y}_v)\subseteq \Omega(E,C)_v$ for all $v\in E^0$, $\ol{\psi} (\ol{H}_x)\subseteq H_x$ for all $x\in E^1$, and that $\ol{\psi}(\ol{\theta}_x(y)) = \theta_x (\ol{\psi}(y))$ for all $y\in \ol{Y}_{s(x)}$, $x\in E^1$. Note that, for $y^0 \in Y^0$, we have $\ol{\psi}(y^0) \in \Omega(E,C)_v$ for some $v \in E^{0,0}$, and hence $\ol{\psi}(y^0) \in X$. Thus the map
$$\psi \colon Y \to X, \quad \psi(y) = \ol{\psi}(y^0)$$
is well-defined. We have
$$\psi(Z_e) = \ol{\psi}(Z_e^0) = \ol{\psi}(\ol{H}_e) \subseteq H_e, \qquad \psi(W_{e,v}^{(k)}) = \ol{\psi}(W_{e,v}^{(k,0)}) = \ol{\psi}(\ol{H}_{f_{e,v}^{(k)}}) \subseteq H_{e,v}^{(k)}.$$
We finally check that $\psi \circ \rho = \sigma \circ \psi$. For this it is enough to check that $\psi \circ \rho|_{W_{e,v}^{(k)}} = \sigma|_{H_{e,v}^{(k)}} \circ \psi$ for all $e \in B$, $v \in E^{0,0}$ and $1 \leq k \leq a_{e,v}$. For $z \in Z_e$, we compute
\begin{align*}
\psi((\rho|_{W_{e,v}^{(k)}})^{-1}(z)) & = \ol{\psi}((\rho|_{W_{ev}^{(k)}})^{-1}(z)^0) \\
& = \ol{\psi}(\ol{\theta}_{f_{e,v}^{(k)}}(z^1)) \\
& = \theta_{f_{e,v}^{(k)}}(\ol{\psi}(z^1)) \\
& = \theta_{f_{e,v}^{(k)}}(\ol{\psi}(\ol{\theta}_e^{-1}(z^0))) \\
& = \theta_{f_{e,v}^{(k)}} \theta_e^{-1}(\ol{\psi}(z^0)) = (\sigma|_{H^{(k)}_{e,v}})^{-1} (\psi (z)),
\end{align*}
where we have used Proposition \ref{proposition-relating.spaces} in the last equality. Hence $\psi$ is an $(E,C)$-equivariant continuous map from $Y$ to $X$. Uniqueness of $\psi$ follows from uniqueness of $\ol{\psi}$. Indeed, the universality of the $(E,C)$-dynamical system $\Omega(E,C)$ turns out to be equivalent to the universality of the natural $(E,C)$-structure on $X$.
\end{proof}
	
\begin{remark}\label{remark-universal.map.described}
For a dynamical system $(Y,\rho) \in \textsf{LHomeo}$ with an $(E,C)$-structure $(\{\ol{H}_e\} _{e\in B}, \{ \ol{H}_f \}_{f\in R})$, the unique $(E,C)$-equivariant map $\psi \colon Y\to X$ from Theorem \ref{theorem-universal.(E,C).structure} can be explicitly described. Indeed, observe that we have decompositions
$$Y =\bigsqcup_{e\in B} \ol{H}_e = \bigsqcup _{f\in R} \ol{H}_f$$
and homeomorphisms
$$\rho_f \colon \ol{H}_f \to \ol{H}_e,$$
for each $f\in R$, where $e$ is the unique blue edge such that $s(e) = s(f)$. Then, for $y\in Y$, $\psi(y)$ consists of the set of elements
$$a_{f_1}^{-1}\cdots a_{f_m}^{-1}a_{g_n}\cdots a_{g_1}\in \mathbb F(A),$$
in the notation of Construction \ref{construction-GFS.concretedescription}, such that
$$y \in \text{Dom}(\rho_{f_1}^{-1}\circ \cdots \circ \rho_{f_m}^{-1}\circ \rho_{g_n}\circ \cdots \circ \rho_{g_1}),$$
where the above is a composition of partial maps (see e.g. \cite[Chapter 2]{Exel2017}).
\end{remark}

We end this section by proving that any $(X,\sigma) \in \textsf{LHomeo}$ can be expressed as an inverse limit of a sequence of generalized finite shifts.\\

Let $(X,\sigma) \in \textsf{LHomeo}$, and fix an $l$-diagram $(F,D)$ corresponding to $(X,\sigma)$ under the bijection established in Theorem \ref{theorem-equivalence}. For each $i \geq 0$, write $E(i):= (F^{0,2i}\cup F^{0,2i+1}, F^{1,2i})$ and let $C(i)$ be the partition of $E(i)^1$ inherited from $D$. Then, by Proposition \ref{proposition-even.labels.sigma.cylinder.dec}, $(E(i),C(i))$ is a generalized finite shift graph. We let $(X_i,\sigma_i)$ be the generalized finite shift (Definition \ref{definition-generalized.finite.shift}) associated to the canonical resolution $(F(i),D(i))$ of the generalized finite shift graph $(E(i),C(i))$. Moreover, by Proposition \ref{proposition-even.labels.sigma.cylinder.dec}, $(X,\sigma)$ admits an $(E(i),C(i))$-structure for each $i \ge 0$. It follows from Theorem \ref{theorem-universal.(E,C).structure} that there is a unique continuous $(E(i),C(i))$-equivariant map $\psi_i \colon (X,\sigma)\to (X_i,\sigma_i)$ for each $i \ge 0$, where we consider the natural $(E(i),C(i))$-structure on $(X_i,\sigma_i)$.

\begin{proposition}\label{proposition-equivariant.maps.factors}
The space $(X_{i+1},\sigma_{i+1})$ admits an $(E(i),C(i))$-structure, in such a way that the natural $(E(i+1),C(i+1))$-structure on $(X_{i+1},\sigma_{i+1})$ is a refinement (in the sense of Definition \ref{definition-refined.sigma.cylinder.dec}) of the mentioned $(E(i),C(i))$-structure.

In particular, the map $\psi_i$ factors through $\psi_{i+1}$.
\end{proposition}
\begin{proof}
Let $(G,H)$ be the Bratteli diagram whose first two layers consist of the layers $E(i)= (F^{0,2i}\cup F^{0,2i+1}, F^{1,2i})$ and $(F^{0,2i+1}\cup F^{0,2i+2}, F^{1,2i+1})$, respectively, with the separations inherited from $D$, and whose other layers coincide with the layers of $(F(i+1),D(i+1))$, moved down in two units, so that for instance $G^{0,k}= F(i+1)^{0,k-2}$ for all $k\ge 2$. Using Proposition \ref{proposition-multiresolution.is.Ldiagram} it is easy to show that $(G,H)$ is an $l$-diagram. By Theorem \ref{theorem-independence.of.Lclass}, we have that the dynamical system associated to the $l$-diagram $(G,H)$ is conjugate to $(X_{i+1},\sigma_{i+1})$. In particular, we see from Proposition \ref{proposition-even.labels.sigma.cylinder.dec} that $(X_{i+1},\sigma_{i+1})$ admits an $(E(i),C(i))$-structure, and moreover the natural $(E(i+1),C(i+1))$-structure on $(X_{i+1},\sigma_{i+1})$ is a refinement of the $(E(i),C(i))$-structure.

Now Theorem \ref{theorem-universal.(E,C).structure} provides us with a unique equivariant continuous map
$$\psi_{i,i+1} \colon X_{i+1} \to X_i,$$
with respect to the respective $(E(i),C(i))$-structures. Since the three maps $\psi_i$, $\psi_{i+1}$ and $\psi_{i,i+1}$ preserve the $(E(i),C(i))$-structure, we get by uniqueness that $\psi_i = \psi_{i,i+1}\circ \psi_{i+1}$. This concludes the proof of the proposition.
\end{proof}

We therefore have an inverse system $\{(X_i,\sigma_i), \psi_{i,j}\}$ in $\textsf{LHomeo}$, and hence we get a dynamical system $(\varprojlim X_i, \tilde{\sigma})$ together with an equivariant continuous map
$$\psi \colon (X,\sigma)\to (\varprojlim X_i, \tilde{\sigma}), \quad \psi(x) = (\psi_0(x),\psi_1(x),\dots).$$

\begin{theorem}\label{theorem-inverse.limit}
With the notation as above, we have that the natural equivariant map
$$\psi \colon X \to \varprojlim X_i$$
is a homeomorphism.
\end{theorem}
\begin{proof}
Let $x = (e_0,e_1,e_2,\dots )\in X$, where all $e_i\in F^{1,i}$ are blue edges. Since $\psi_i$ is $(E(i),C(i))$-equivariant, we have
$$\psi_i (x)= (e_{2i}, \tilde{e}_{i,1},\tilde{e}_{i,2},\dots ),$$
where $\tilde{e}_{i,j} \in F(i)^{1, j}$ for $j\ge 1$. It follows that if $x = (e_i), x' = (e'_i) \in X$ and $\psi(x) = \psi(x')$, then $e_{2i} = e'_{2i}$ for all $i \ge 0$. But then $e_{2i+1} = e'_{2i+1}$ for all $i \ge 0$, since there is a unique blue edge in $F$ joining $r(e_{2i+2})$ with $s(e_{2i})$. It follows that $\psi$ is injective.

We now show that $\psi$ is surjective. Let $(x_0,x_1,x_2,\dots ) \in \varprojlim X_i$. Write
$$x_i = (e_{2i}, \tilde{e}_{i,1},\tilde{e}_{i,2},\dots ),$$
where $e_{2i}\in F^{1,2i}$ and $\tilde{e}_{i,j}\in F(i)^{1, j}$ for all $j\ge 1$. Since $\psi_{i,i+1}(x_{i+1})= x_i$ and all the maps $\psi_{i,i+1}$ are $(E(i),C(i))$-equivariant, it follows by a simple induction that there is a unique blue path 
$$x= (e_0,e_1,e_2, \dots , e_{2i}, e_{2i+1},e_{2i+2}, \dots)\in X .$$
We show that $\psi (x)= (x_0,x_1,x_2,\dots )$. For this we need the description of the maps $\psi_j$ given in Remark \ref{remark-universal.map.described}. We will concentrate in showing that $\psi_0(x) = x_0$. Consider some $i \ge 0$. Then the first $2i$ coordinates of $\psi_0 (x)$ are determined by the set of all elements
$$a_{f_1}^{-1}\cdots a_{f_m}^{-1}a_{g_n}\cdots a_{g_1}\in \mathbb F (A)$$
where $f_1,\dots ,f_m,g_1,\dots ,g_n$ are red edges in $E(0)$ satisfying the requirements given in Construction \ref{construction-GFS.concretedescription}, such that
$$x \in \text{Dom}(\theta_{f_1}^{-1} \circ \cdots \circ \theta_{f_m}^{-1} \circ \theta_{g_n} \circ \cdots \circ \theta_{g_1})$$
and with $n+m \le i$. Here $\{\theta_f\}$ is the set of partial maps as in Remark \ref{remark-universal.map.described} associated to the $(E(0),C(0))$-structure on $X$. Now it follows from the construction of $\sigma$ (given in Construction \ref{construction-local.homeo}) and from the proof of surjectivity of $\sigma$ (given in Lemma \ref{lemma-sigma.surj.local.homeo}) that this information is contained in the blue path $(e_0,e_1,\dots , e_{2i-1})$, which is the initial segment of $x$ of length $2i$. Note that, by the same argument as before and the definition of the maps $\psi_{j,j+1}$, we see that, for $y \in X_i$, the first $2i$ components of $\psi _{0,i}(y)$ are determined by the unique blue path $(e_0', e_1',\dots ,e'_{2i-1})$ joining $r(y)\in F^{0,{2i}}$ with a vertex in $F^{0,0}$ in the graph $F$, because the $(E(0),C(0))$-structure on $X_i$ is determined by the canonical $(E(0),C(0))$-structure on the $l$-diagram obtained by topping $(F(i),D(i))$ with the first $2i$ layers of $(F,D)$. Since $\psi_i(x)= (e_{2i}, \dots )$ and $x_i =  (e_{2i}, \dots )$, that unique blue path is precisely $(e_0,e_1,\dots ,e_{2i-1})$ for both $\psi_i(x)$ and $x_i$, and thus the first $2i$ components of $\psi_{0,i}(\psi_i(x)) $ agree with the first $2i$ components of $\psi_{0,i}(x_i)$. But $\psi_{0,i}(\psi_i(x))= \psi_0(x)$ and $\psi_{0,i}(x_i) = x_0$, hence the first $2i$ coordinates of $\psi_0(x)$ must agree with the first $2i$ coordinates of $x_0$. Since this holds for all $i \ge 0$, we get that $\psi_0(x) = x_0$.

In the same way we may prove that $\psi_n(x) = x_n$ for all $n\ge 1$, completing the proof.
\end{proof}

\section{Approximating algebras associated to a dynamical system}\label{section-approx.algebras}

Let $(K,\ast)$ be a field endowed with an involution $\ast$. In this section we analyze both the Steinberg algebra $A_K(\calG)$ and the $C^*$-algebra $C^*(\calG)$ of the Deaconu-Renault groupoid $\calG = \calG(X,\sigma)$ associated to an object $(X,\sigma) \in \textsf{LHomeo}$. We will describe approximating processes of such algebras by means of the corresponding algebras associated with the sequence of generalized finite shifts $(X_i,\sigma_i)$ described at the end of Section \ref{section-general.shifts.finite.type}.

\subsection{The general approximation construction}\label{subsection-general.approx}

Fix $(X,\sigma) \in \textsf{LHomeo}$ and a $\sigma$-refined sequence of partitions $\{ \calP_n \}_{n\ge 0}$ of $X$. Let $(F,D)$ be the $l$-diagram associated to  $\{\calP_n\}_{n\ge 0}$, see Construction \ref{construction-l.diag}.
 We adopt the terminology and notation introduced in Section \ref{section-general.shifts.finite.type}.

For each $n \geq 0$ we define bipartite finite separated graphs $(F_n,D^n)$ as follows:
$$F_n^0 := F_n^{0,0} \sqcup F_n^{0,1} \text{ with } F_n^{0,0} := F^{0,n} \text{ and } F_n^{0,1} := F^{0,n+1},$$
$$F_n^1 := F^{1,n},$$
and $D_v^n := D_v$ for $v \in F_n^{0,0}$. Note that $r(F_n^1) = F_n^{0,0}$ and $s(F_n^1) = F_n^{0,1}$. With this notation, the generalized finite shift graphs $(E(n),C(n))$ introduced in Section \ref{section-general.shifts.finite.type} are precisely $(F_{2n},C^{2n})$. 

For notational convenience, we write $L_n := L_K(F_n,D^n)$, the Leavitt path algebra of the separated graph $(F_n,D^n)$. Also $p_n$ will denote the projection in $L_n$ on $F_n^{0,0}$, so that
$$p_n = \sum_{v \in F_n^{0,0}}v.$$

\begin{lemma}\label{lemma-technical.1}
Let $j \geq 0$ be an integer.
\begin{enumerate}[$i)$,leftmargin=0.7cm]
\item For any vertex $v \in F_{2j}^{0,0}$,
$$\textbf{S}(v) := \sum_{e \in B_v} \sum_{e' \in B_{s(e)}} s(e') = \sum_{f \in R_v} \sum_{f' \in R(f)} s(f') \in L_{2j+1}.$$
\item For any vertex $v \in F_{2j+1}^{0,0}$ and any $g \in R_{r(g)}$ such that $s(g)= v$,
$$\textbf{S}(v) := \sum_{e \in B_v} \sum_{e' \in B_{s(e)}} s(e') = \sum_{f \in R(g)} \sum_{f' \in R_{s(f)}} s(f') \in L_{2j+2}.$$
\end{enumerate}
\end{lemma}
\begin{proof}
Note that in both $i)$ and $ii)$ the sum $\sum_{e \in B_v} \sum_{e' \in B_{s(e)}} s(e')$ is a sum of distinct vertices in $F^{n+2}$, where $n$ is either $2j$ or $2j+1$.
\begin{enumerate}[i),leftmargin=0.7cm]
\item The set of vertices $s(e')$ in the above expression is in bijective correspondence with the set of pairs $(e,e')$, where $e \in B_v$ and $e' \in B_{s(e)}$. By Proposition \ref{proposition-decomposition.odd.rombs.even} (ii), the set of pairs $(e,e')$ where $e \in B_v$ and $e' \in B_{s(e)}$ is in bijective correspondence with the set of pairs $(f,f')$ where $f \in R_v$, $f' \in R(f)$, and $(f,f')$ corresponds to $(e,e')$ if and only if $s(f') = s(e')$. The result follows.
\item The argument is similar to the one given for the first part, taking into account conditions (f) and (g) in the definition of $l$-diagram (Definition \ref{definition-L.separated.Bratteli}).\qedhere
\end{enumerate}
\end{proof}

\begin{remark}\label{remark-unified.formula.forS}
Note that each $e \in s^{-1}(v)$ determines a unique set $X(e) \in D_v$. This notation is compatible with the notation introduced in Definition \ref{definition-F.infty.and.others}. With this notation, the formulas in Lemma \ref{lemma-technical.1} take a uniform shape:
$$\textbf{S}(v) = \sum_{e\in X} \sum _{e'\in X(e)} s(e')$$
for any $X \in D_v$, $v \in F^{0}$. 
\end{remark}

\begin{proposition}\label{proposition-graph.homo}
There exist unital $*$-homomorphisms $\wt{\phi}_n \colon  L_n \ra L_{n+1}$ such that $\wt{\phi}_n(p_n)= p_{n+2}$ and $\wt{\phi}_n(p_{n+1})=p_{n+1}$. 
\end{proposition}
\begin{proof}
We define $\wt{\phi}_n : L_n \ra L_{n+1}$ by the rules
$$\wt{\phi}_n(\alpha) =
\begin{cases}
w & \text{ if } \alpha = w \in F_n^{0,1}; \\
\textbf{S}(v) & \text{ if } \alpha = v \in F_n^{0,0}; \\
\sum_{\ol{e} \in B_{s(e)}} \ol{e}^* & \text{ if } \alpha = e \in F_n^1 \text{ is a blue edge}; \\
\sum_{\ol{f} \in R(f)} \ol{f}^* & \text{ if } n= 2j \text{ and } \alpha = f \in F_n^1 \text{ is a red edge}; \\
\sum_{\ol{f} \in R_{s(f)}} \ol{f}^* & \text{ if }  n= 2j+1 \text{ and }\alpha = f \in F_n^1 \text{ is a red edge}.
\end{cases}$$
We must check that the defining relations for $L_n$ in Definition \ref{definition-leavitt.alg} are satisfied. The relations (V), (E) are easily verified. For (SCK1) take first $e,f \in B_v$, where $v\in F_n^{0,0}= F^{0,n}$. We compute
\begin{equation*}
\begin{split}
\wt{\phi}_n(e)^* \wt{\phi}_n(f) = \Big( \sum_{\ol{e} \in B_{s(e)}} \ol{e} \Big) \Big( \sum_{\ol{f} \in B_{s(f)}} \ol{f}^* \Big) = \sum_{\ol{e} \in B_{s(e)}} \sum_{\ol{f} \in B_{s(f)}} \ol{e} \ol{f}^*.
\end{split}
\end{equation*}
The product $\ol{e}\ol{f}^*$ is zero unless $s(\ol{e}) = s(\ol{f})$. Since the edges are blue, this implies that $\ol{e} = \ol{f}$ (see Definition \ref{definition-L.separated.Bratteli}). In particular, $s(e) = r(\ol{e}) = r(\ol{f}) = s(f)$, and again we deduce that $e = f$. Hence
$$\wt{\phi}_n(e)^* \wt{\phi}_n(f) = \delta_{e,f} \Big( \sum_{\ol{e} \in B_{s(e)}} \ol{e} \ol{e}^* \Big) = \delta_{e,f} s(e)=\wt{\phi}_n (\delta_{e,f} s(e)),$$
where we have used the relation (SCK2) inside $L_{n+1}$.

Assume now that $n = 2j$ for $j \geq 0$ and that $e,f\in R_v$ for $v\in F_{2j}^{0,0}$. Then
\begin{equation*}
\begin{split}
\wt{\phi}_n(e)^* \wt{\phi}_n(f) =  \sum_{\ol{e} \in R(e)} \sum_{\ol{f} \in R(f)} \ol{e} \ol{f}^*.
\end{split}
\end{equation*}
Again, the product $\ol{e}\ol{f}^*$ is zero unless $s(\ol{e}) = s(\ol{f})$. We claim that in that case we necessarily have $(e,\ol{e}) = (f,\ol{f})$. This follows from Proposition \ref{proposition-decomposition.odd.rombs.even} (ii). Indeed, a first application of this proposition gives that there is a unique pair $(e_0,e_1)$ of blue edges such that $r(e_0) = r(e) = v$, $r(e_1) = s(e_0)$ and $s(e_1) = s(\ol{e})$. But now, since we are assuming that $s(\ol{e}) = s(\ol{f})$, we have that $(e,\ol{e})$ and $(f, \ol{f})$ are two pairs of red edges with $\ol{e}\in R(e)$ and $\ol{f}\in R(f)$ which correspond to the same pair $(e_0,e_1)$ of blue edges, so the uniqueness part of Proposition \ref{proposition-decomposition.odd.rombs.even} (ii) gives that $(e,\ol{e}) = (f,\ol{f})$, as desired. Hence
$$\wt{\phi}_n(e)^* \wt{\phi}_n(f) = \delta_{e,f} \Big( \sum_{\ol{e} \in R(e)} \ol{e} \ol{e}^* \Big) = \delta_{e,f} s(e)= \wt{\phi}_n (\delta_{e,f} s(e)),$$
where again we have used the relation (SCK2) inside $L_{n+1}$.

Assume finally that $n = 2j+1$ for $j\ge 0$ and that $e,f \in R(g)$ for $g \in R_{r(g)}$ such that $s(g) \in F_{2j+1}^{0,0}$. Then we have
\begin{equation*}
\begin{split}
\wt{\phi}_n(e)^* \wt{\phi}_n(f) =  \sum_{\ol{e} \in R_{s(e)}} \sum_{\ol{f} \in R_{s(f)}} \ol{e} \ol{f}^*.
\end{split}
\end{equation*}
Again, the product $\ol{e}\ol{f}^*$ is zero unless $s(\ol{e}) = s(\ol{f})$. We claim that, as in the previous case, we necessarily have $(e,\ol{e}) = (f,\ol{f})$. The argument is similar to the above, now using conditions (f) and (g) in Definition \ref{definition-L.separated.Bratteli}. Namely, by (f) there is a unique blue pair $(e_0,e_1)$ such that $(e_0,e_1,e,\ol{e})$ is a romb. Now by the uniqueness part of (g) (and since $e,f \in R(g)$), we deduce that $(e,\ol{e})= (f,\ol{f})$. Hence
$$\wt{\phi}_n(e)^* \wt{\phi}_n(f) = \delta_{e,f} \Big( \sum_{\ol{e} \in R_{s(e)}} \ol{e} \ol{e}^* \Big) = \delta_{e,f} s(e) = \wt{\phi}_n (\delta_{e,f} s(e))$$
again by the relation (SCK2) inside $L_{n+1}$.\\

For (SCK2) note that, for $e\in B_v$, $v\in F_n^{0,0}$, we have
$$\wt{\phi}_n(e) \wt{\phi}_n(e)^* = \sum_{f,g \in B_{s(e)}} f^*g = \sum_{f \in B_{s(e)}} s(f),$$
so that
$$\sum_{e \in B_v} \wt{\phi}_n(e) \wt{\phi}_n(e)^* = \sum_{e \in B_v} \sum_{f \in B_{s(e)}} s(f) = \textbf{S}(v)= \wt{\phi}_n(v)$$
because of Lemma \ref{lemma-technical.1}. Similarly, we get that (SCK2) is also preserved for red edges.

This analysis gives the well-definiteness of the maps $\wt{\phi}_n \colon L_n \ra L_{n+1}$. Now note that
$$\wt{\phi}(p_{n+1}) = \sum_{w\in F_n^{0,1}} w = p_{n+1} \quad \text{ and } \quad \wt{\phi}_n(p_n) = \sum_{v \in F_n^{0,0}} \textbf{S}(v) = \sum_{w \in F_{n+1}^{0,1}} w = p_{n+2}.$$
The proposition is now proved.
\end{proof}

For $m > n$, denote by $\wt{\phi}_{n,m} \colon L_n \ra L_m$ the composition $\wt{\phi}_{m-1}\circ \cdots \circ \wt{\phi}_n$. Set
$$L_{\infty} : =\varinjlim (L_n, \wt{\phi}_{n,m}),$$
with canonical maps $\wt{\phi}_{n,\infty} \colon L_n \to L_{\infty}$.

We will need the following observations.

\begin{lemma}\label{lemma-observations.for.phinm}
With the above notation, and with $k=m-n$, the following hold.
\begin{enumerate}[i),leftmargin=0.7cm]
\item If $v_1,\dots , v_l$ are mutually distinct vertices in $F_n$, then each of $\wt{\phi}_{n,m}(v_i)$ is a sum of mutually distinct vertices in $F_m$, and moreover the vertices appearing in the decompositions of two different vertices $v_i$ and $v_j$ are also mutually distinct.
\item If $f \in F_n^1$, then $\wt{\phi}_{n,m}(f)$ is a sum of edges $g \in F_m^1$ whose sources are mutually distinct if $k$ is even, and is a sum of ``ghost'' edges $g^*$ for edges $g \in F_m^1$ whose sources are mutually distinct if $k$ is odd. Moreover, for each $w \in F_m^{0,0}$, all the edges $g \in F_m^1$ appearing in the above sum with range $w$ belong to the same set $Z \in D_w$.
\end{enumerate}
\end{lemma}
\begin{proof}
The easy proof of $i)$ is left to the reader. For $ii)$, we proceed by induction on $k$. If $k = 1$ then $\wt{\phi}_n(f) = \sum _{\ol{f}\in X(f)} \ol{f}^*$, where we use the notation introduced in Remark \ref{remark-unified.formula.forS}. Since $r(\ol{f}) = r(\ol{f}')$ for all $\ol{f},\ol{f}'\in X(f)$, the result follows from the fact that, except possibly at the first layer, distinct blue (respectively, distinct red) edges with the same range must have distinct sources (see Definition \ref{definition-L.separated.Bratteli} and Remark \ref{remark-about.def.of.Ldiagram} (5)). 

Now suppose that the result holds for $k \ge 1$. Assume first that $k$ is even. Then we have that $\wt{\phi}_{n,n+k}(f)  = \sum _{g\in L} g$, where $L$ is a set of edges in $F_{n+k} = F_m$ whose sources are mutually distinct, and such that all members of $L$ with the same range belong to the same element of the partition $D$. Using the notation introduced in Remark \ref{remark-unified.formula.forS}, we have
$$\wt{\phi}_{n,n+k+1}(f) = \sum _{g \in L} \wt{\phi}_{n+k}(g) = \sum _{g \in L} \sum_{\ol{g} \in X(g)} \ol{g}^*,$$
where $X(g) \in D_{s(g)}$. Let $w \in F_m^0$ and let $X^w\in D_w$ be the unique set in the partition $D_w$ to which all elements $g$ of $L$ with range $w$ belong. Using Lemma \ref{lemma-technical.1} (and Remark \ref{remark-unified.formula.forS}), we have
$$\sum_{\{g\in L \mid r(g)= w\}} \sum_{\ol{g} \in X(g)} s(\ol{g}) \le \sum_{g\in X^w}\sum_{\ol{g}\in X(g)} s(\ol{g}) = \textbf{S}(w).$$
Since in addition $\textbf{S}(w) \cdot \textbf{S}(w') = 0$ for distinct $w,w' \in F_m^{0,0}$, we conclude that $\{s(\ol{g}) \mid \ol{g}\in X(g), g\in L\}$ is a set of mutually distinct vertices in $F_{m+1}^{0,1}$.
Moreover, if $r(\ol{g})= r(\ol{h})$ for $\ol{g}\in X(g)$ and $\ol{h}\in X(h)$, with $g,h\in L$, then we get
$$s(g) = r(\ol{g}) = r(\ol{h}) = s(h)$$
and by induction hypothesis we conclude that $g=h$, so that $\ol{g},\ol{h}\in X(g)= X(h)$, as wanted.

The case where $k$ is odd is treated in the same way.
\end{proof}

We now recall some basic definitions, see \cite[Definition 12.9]{Exel2017}.

\begin{definition}\label{definition-tame.stuff}
A set $F$ of partial isometries of a $*$-algebra $A$ is said to be {\it tame} if for all two elements $u,u'\in U$, where $U$ is the multiplicative $*$-subsemigroup of $A$ generated by $F$, we have that $e(u)$ and $e(u')$ are commuting elements of $A$, where $e(v)= vv^*$ for $v\in U$. Note that if $F$ is a tame set of partial isometries, then all elements of $U$ are indeed partial isometries, and therefore the elements $e(u)$, with $u\in U$, are mutually commuting projections in $A$.   

A $*$-algebra $A$ is said to be {\it tame} if it is generated as a $*$-algebra  by a tame set of partial isometries. If $A$ is a $*$-algebra generated by a subset $F$ of partial isometries, there is a universal tame $*$-algebra $A^{\text{ab}}$ associated to $F$. By definition, there is a surjective $*$-homomorphism 
$\pi \colon A \to A^{\text{ab}}$ so that $\pi (F)$ is a tame set of partial isometries generating $A^{\text{ab}}$, and such that for any $*$-homomorphism 
$\psi \colon A \to B$ of $A$ to a $*$-algebra $B$ such that $\psi(F)$ is a tame set of partial isometries in $B$, there is a unique $*$-homomorphism
$\ol{\psi}\colon A^{\text{ab}} \to B$ such that $\psi = \ol{\psi}\circ \pi$. 
Indeed, we have 
$A^{\text{ab}} = A/J$, where $J$ is the ideal of $A$ generated by all the commutators $[e(u),e(u')]$, where $u,u'\in U$.

If $(E,C)$ is a separated graph, then the universal tame $*$-algebra associated to the generating set $E^0 \sqcup E^1$ of partial isometries of $L_K(E,C)$ is denoted by $\Lab _K(E,C)$. In addition, we denote by $\calO (E,C)$ the $C^*$-algebra $C^*(E,C)/\calJ$, where $\calJ$ is the closed ideal of $C^*(E,C)$ generated by all the commutators $[e(u),e(u')]$, for $u,u'\in U$, where $U$ is the multiplicative $*$-subsemigroup of $C^*(E,C)$ generated by $E^0\cup E^1$. Clearly $\calO (E,C)$ is universal with respect to $*$-homomorphisms from $C^*(E,C)$ to a tame $C^*$-algebra. 
\end{definition}

The following lemma will be crucial for our main result. We continue with the notation introduced at the beginning of this section. 

\begin{lemma}\label{lemma-tame.aftersteps}
Let $U_n$ be the multiplicative $*$-subsemigroup of $L_n = L_K(F_n,D^n)$ generated by the set $F_n^1$. Let $\omega$ be a word in $U_n$ of length $k$. Then $\wt{\phi}_{n,n+k}(\omega \omega^*)$ is a sum of distinct vertices of $F_{n+k}$.
\end{lemma}
\begin{proof}
We proceed by induction on $k$. If $k = 1$, then either $\omega = e$ or $\omega = e^*$ for some $e \in F_n^1$. In the first case $\omega = e$, we have
$$\wt{\phi}_n (\omega \omega ^*) = \wt{\phi}_n (ee^*) = \Big(\sum_{\ol{e}\in X(e)} \ol{e}^*\Big) \Big(\sum _{\ol{e}'\in X(e)} \ol{e}'\Big) = \sum _{\ol{e}\in X(e)}  s(\ol{e}),$$
which is a sum of mutually distinct vertices of $F_{n+1}$. In the second case $\omega = e^*$, we have
$$\wt{\phi}_n (\omega \omega^*) = \wt{\phi}_n(e^*e) =\wt{\phi}_n(s(e)) = s(e).$$
The result follows for $k = 1$.

Suppose now that the result is true for words of length $k$ and let $\omega \in U_n$ be a word of length $k+1$. Write $\omega = x\omega_0$, where $\omega_0$ is a word of length $k$ and $x$ is a word of length $1$ (so that $x= f$ or $x= f^*$ for $f\in F_n^1$). By induction hypothesis we can write
$$\wt{\phi}_{n,n+k}(\omega_0 \omega_0^*)= \sum_{i=1}^t v_i$$
for mutually distinct vertices $v_i$ of $F_{n+k}$. We now use Lemma \ref{lemma-observations.for.phinm} (ii). Suppose first that $\wt{\phi}_{n,n+k}(x)= \sum _{g \in L} g^*$, where $L$ is a family of edges in $F_{n+k}$ whose sources are mutually distinct, such that, for each $w \in F_{n+k}^{0,0}$, all the edges $g \in L$ with range $w$ belong to the same set $Z \in D_w$. We then have
\begin{align*}
\wt{\phi}_{n,n+k}(\omega \omega^*) & = 	\wt{\phi}_{n,n+k}(x)\wt{\phi}_{n,n+k}(\omega_0 \omega_0^*)\wt{\phi}_{n,n+k}(x)^* \\
& = \Big(\sum _{g \in L} g^*\Big) \Big(\sum_{i=1}^t v_i\Big) \Big(\sum_{h \in L} h\Big) \\
& = \sum_{i=1}^t \sum_{\{g,h \in L \mid r(g) = r(h) =v_i \}} g^*h \\
& = \sum_{i=1}^t \sum_{\{g \in L \mid r(g)= v_i \}} s(g),
\end{align*}
where for the last equality we have used the fact that any two edges $g,h\in L$ with the same range $w$ belong to the same element of $D_w$. Now since the sources of edges in $L$ are mutually distinct, we see that  $\wt{\phi}_{n,n+k}(\omega \omega^*)$ is a sum of mutually distinct vertices in $F_{n+k}^{0,1}$. Since $\wt{\phi}_{n,n+k+1}(\omega \omega^*)= \wt{\phi}_{n+k}(\wt{\phi}_{n,n+k}(\omega \omega^*))$, we obtain the result in this case by Lemma \ref{lemma-observations.for.phinm} (i). 

Suppose now that $\wt{\phi}_{n,n+k}(x)$ is a sum of edges in $F_{n+k}$. Then necessarily
$$\wt{\phi}_{n,n+k+1}(x)= \wt{\phi}_{n+k}(\wt{\phi}_{n,n+k}(x))$$
will be of the form $\sum _{g\in L} g^*$, where $L$ is a family of edges in $F_{n+k+1}$ whose sources are mutually distinct, such that, for each $w \in F_{n+k+1}^{0,0}$, all the edges $g \in L$ with range $w$ belong to the same set $Z\in D_w$. Since $\wt{\phi}_{n,n+k+1}(\omega _0\omega_0^*)$ is a sum of mutually distinct vertices by induction hypothesis and by Lemma \ref{lemma-observations.for.phinm} (i), the same argument as before gives us the result also in this case.
\end{proof}

We are now in a position to show that the algebra $L_{\infty} = \varinjlim L_n$ is tame, and compute it by means of the universal tame $*$-algebras of $L_n$ with respect to the sets $F_n$.

\begin{theorem}\label{theorem-tame.at.infinity}
The $*$-algebra $L_{\infty}$ is tame, with generating set of partial isometries $F_{\infty}^1 := \bigcup_{n=0}^{\infty} \wt{\phi}_{n,\infty} (F_n^1)$. Moreover, the natural $*$-homomorphisms $\wt{\phi}_{n, \infty} \colon L_n \to L_{\infty}$ factor through $\Lab _n:= \Lab_K(F_n,D^n)$ with homomorphisms $\ol{\phi}_{n,\infty} \colon \Lab _n \to L_{\infty}$, and we have
$$L_{\infty} = \bigcup _{n=0}^{\infty} \wt{\phi}_{n,\infty} (L_n) = \bigcup _{n=0}^{\infty} \ol{\phi}_{n,\infty} (\Lab_n).$$
\end{theorem}
\begin{proof}
Let $n,n'$ be non-negative integers, and let $\omega, \omega'$ be words in $U_n$ and $U_{n'}$ respectively.  By Lemmas \ref{lemma-tame.aftersteps} and \ref{lemma-observations.for.phinm} (i), there exists $m\ge n, n'$ such that both $\wt{\phi}_{n,m}(e(\omega))$ and $\wt{\phi}_{n',m}(e(\omega'))$ are finite sums of mutually distinct vertices of $F_{m}$.  Therefore they are commuting elements of $L_{m}$. It follows that $F_{\infty}^1 = \bigcup_{n=0}^{\infty} \wt{\phi}_{n,\infty} (F_n^1)$ is a tame set of partial isometries of $L_{\infty}$. By the universal property of $\Lab _n$, $\wt{\phi}_{n,\infty}$ factors (uniquely) through	a $*$-homomorphism $\ol{\phi}_{n,\infty} \colon \Lab_n\to L_{\infty}$, and obviously $\wt{\phi}_{n,\infty} (L_n) = \ol{\phi}_{n,\infty}(\Lab _n)$. This completes the proof.   
\end{proof}

Let $\wt{\phi}_n\colon L_n \to L_{n+1}$ be the $*$-homomorphism defined in Proposition \ref{proposition-graph.homo}. Observe that 
$$\wt{\phi}_{n+1}\circ \wt{\phi}_n (p_n) = \wt{\phi}_{n+1} \Big( \sum_{w\in F_{n+1}^{0,1}} w \Big) = p_{n+2}\in L_{n+2}.$$
From now on we will write $\calA_n = p_n L_n p_n$, and we will denote by $\phi_n \colon  \calA_{2n} \ra \calA_{2n+2}$ the restriction to $\calA_{2n}$ of the composition $\wt{\phi}_{2n+1} \circ \wt{\phi}_{2n}$. We consider the inductive limit $\calA_{\infty} := \varinjlim (\calA_{2n}, \phi_n)$, with canonical maps $\phi_{n,\infty} \colon \calA_{2n}\ra \calA_{\infty}$. For $m>n$, we set $\phi_{n,m}:= \phi_{m-1} \circ \cdots \circ \phi_n\colon \calA_{2n}\to \calA_{2m}$. 

\begin{lemma}\label{lemma-ind.limit.OK}
We have a natural $*$-isomorphism
$$\calA_{\infty} \cong  pL_{\infty}p,$$
where $p = \wt{\phi}_{2n,\infty} (p_{2n})$ for all $n \ge 0$. 
\end{lemma}
\begin{proof}
We have obvious compatible $*$-homomorphisms $\wt{\phi}_{2n,\infty} \colon \calA_{2n} \ra  pL_{\infty}p$, and hence we get a $*$-homomorphism
$$\lambda \colon \calA_{\infty} \to pL_{\infty}p$$
defined by $\lambda (\phi_{n,\infty} (a)) = \wt{\phi}_{2n,\infty} (a)$ for $a\in \calA_{2n}$. It is straightforward to check that $\lambda $ is a $*$-isomorphism.
\end{proof}

Our aim for the rest of this subsection is to find a $*$-isomorphism $A_K(\calG(X,\sigma)) \cong \calA_{\infty}$, thus providing a graph-theoretic model for the Steinberg algebra $A_K(\calG(X,\sigma))$ of the Deaconu-Renault groupoid $\calG(X,\sigma)$ associated to $(X,\sigma)$. We will also deal with the corresponding $C^*$-algebras. 

To start with, we define $*$-homomorphisms $\varphi _n\colon \calA_{2n} \ra A_K(\calG(X,\sigma))$. For this, we use another presentation for $\calA_{2n}$, given by \cite[Proposition 2.8]{Ara2022}. We warn the reader about the different convention used for sources and ranges in \cite{Ara2022}.

\begin{definition}\cite[Definition 2.7]{Ara2022}\label{definition-universal}
Let $(E,C)$ be a finite bipartite separated graph with $E^0 = E^{0,0} \sqcup E^{0,1}$, $s(E^1) = E^{0,1}$ and $r(E^1) = E^{0,0}$. Let $LV_K(E,C)$ be the universal $*$-algebra with generators $P_V \sqcup T$ where $P_V = \{ p_v \mid v \in E^{0,0}\}$ and $T = \{\tau_{(e,f)} \mid e,f \in E^1, s(e) = s(f)\}$, and subject to the relations
\begin{enumerate}
\item[(V')] $p_v p_w = \delta_{v,w} p_v$, and $p_v = p_v^*$ for $v,w \in E^{0,0}$;
\item[(T)] $\tau_{(e,f)}^* = \tau_{(f,e)}$ for $e,f \in E^1$ with $s(e) = s(f)$;
\item[(E')] $\tau_{(e,f)} \cdot p_{r(f)} = p_{r(e)} \cdot \tau_{(e,f)} = \tau_{(e,f)}$ for $e,f \in E^1$ with $s(e) = s(f)$;
\item[(SCK1')] $\tau_{(e,f)} \cdot \tau_{(g,h)} = \delta_{f,g} \tau_{(e,h)}$ for $e,f,g,h \in E^1$ with $s(e) = s(f)$, $s(g) = s(h)$ and $f,g \in Y$ belonging to the same set $Y \in C$;
\item[(SCK2')] $p_v = \sum_{e \in Y} \tau_{(e,e)}$ for all $v \in E^{0,0}$ and all $Y \in C_v$.
\end{enumerate}
We denote by $\LVab_K(E,C)$ the abelianized Leavitt path algebra of $LV_K(E,C)$ with respect to its generating family $T = \{\tau_{(e,f)} \mid e,f\in E^1, s(e)=s(f)\}$ of partial isometries.
\end{definition}

\begin{proposition}\cite[Proposition 2.8]{Ara2022}\label{proposition-presentation.corner}
Let $(E,C)$ be a finite bipartite separated graph with $E^0 = E^{0,0} \sqcup E^{0,1}$, $s(E^1) = E^{0,1}$ and $r(E^1) = E^{0,0}$, and set $V = \sum_{v \in E^{0,0}} v \in L_K(E,C)$. Then there is a unique $*$-isomorphism
$$\varphi_V \colon LV_K(E,C)\to VL_K(E,C)V$$
such that $\varphi_V (p_v) = v$ for all $v \in E^{0,0}$ and $\varphi_V(\tau_{(e,f)}) = ef^*$ for all $e,f \in E^1$ with $s(e) = s(f)$.
\end{proposition}

Let us now define the $*$-homomorphisms $\varphi _n \colon \calA_{2n} \ra A_K(\calG(X,\sigma))$. We use the presentation of $\calA_{2n} = p_{2n}L_K(F_{2n},D^{2n})p_{2n} \cong LV_K(F_{2n},D^{2n})$ provided by Definition \ref{definition-universal} (see Proposition \ref{proposition-presentation.corner} above). Recall that the vertices of $F_{2n}^{0,0}$ correspond to the sets in the partition $\calP_{2n}$ of $X$, and similarly, the vertices in $F_{2n}^{0,1}$ correspond to the sets in the partition $\calP_{2n+1}$ of $X$.

Given $p_v$, with $v \in F_{2n}^{0,0}$, there is a unique clopen subset $U \in \calP_{2n}$ such that $v = U$. Define
$$\varphi_n(p_v) = \varphi_n(p_U) = \mathds{1}_U \in A_K(\calG(X,\sigma)),$$
where we identify $Z(U,0,0,U)$ with $U$. Now we want to define the images of elements of the form $\tau_{(e,f)}$, where $e$ and $f$ are edges in $F_{2n}$ such that $s(e) = s(f)$. We write $s(e) = s(f) = W \in \calP_{2n+1}$, $r(e) = U \in \calP_{2n}$ and $r(f) = V \in \calP_{2n}$. We need to distinguish various cases.
\begin{enumerate}[(a),leftmargin=0.7cm]
\item $e,f$ are both blue edges. By definition, we have then that $W \subseteq U$ and $W \subseteq V$. In this case, define
$$\varphi_n(\tau_{(e,f)}) = \mathds{1}_W \in A_K(\calG(X,\sigma)),$$
where we identify $Z(W,0,0,W)$ with $W$.
\item $e$ and $f$ are both red edges. In this case we have $W \subseteq \sigma(U)$ and $W \subseteq \sigma (V)$. Define
$$\varphi_n(\tau_{(e,f)}) = \mathds{1}_{Z((\sigma|_U)^{-1}(W), 1, 1, (\sigma|_V)^{-1}(W))} \in A_K(\calG(X,\sigma)).$$
Observe that all sets $Z$ in the partitions $\calP_m$ satisfy the property that the restriction of $\sigma$ to $Z$ is injective. Hence $Z((\sigma|_U)^{-1}(W), 1, 1, (\sigma|_V)^{-1}(W))$ is well-defined.
\item $e$ is a red edge and $f$ is a blue edge. In this case we have $W \subseteq \sigma(U)$ and $W \subseteq V$. Define
$$\varphi_n(\tau_{(e,f)}) = \mathds{1}_{Z((\sigma|_U)^{-1}(W), 1, 0, W)} \in A_K(\calG(X,\sigma)).$$
\item $e$ is a blue edge and $f$ is a red edge. Then $W \subseteq U$ and $W \subseteq \sigma(V)$. Define
$$\varphi_n(\tau_{(e,f)}) = \mathds{1}_{Z(W, 0, 1, (\sigma|_V)^{-1}(W))} \in A_K(\calG(X,\sigma)).$$
\end{enumerate}
It is straightforward to check that the defining relations from Definition \ref{definition-universal} are preserved, and hence we obtain well-defined $*$-homomorphisms
$$\varphi_n \colon \calA_{2n} \to A_K(\calG(X,\sigma))$$
for all $n\ge 0$.

Observe that the algebras $L_{2n} = L_K(F_{2n},D^{2n})$ are $\Z$-graded for all $n \ge 0$, with grading determined by setting $d(v) = 0$ for each vertex $v \in F_{2n}^0$, $d(e) = 0$ if $e \in F_{2n}^1$ is a blue edge, and $d(f) = 1$ if $f \in F_{2n}^1$ is a red edge. For each $n \ge 0$, $\calA_{2n}$ is a graded $*$-subalgebra of $L_{2n}$, and the maps $\phi_n\colon \calA_{2n}\to \calA_{2n+2}$ are graded $*$-homomorphisms, so we get the structure of a $\Z$-graded $*$-algebra in $\calA_{\infty}$.

\begin{proposition}\label{proposition-morph.compat.steinberg}
The morphisms $\varphi_n \colon \calA_{2n} \to A_K(\calG(X,\sigma))$ are graded and compatible with the maps $\phi_n \colon \calA_{2n} \to \calA_{2n+2}$ from the inductive system $(\calA_{2n},\phi_n)$. Therefore they define a well-defined graded $*$-homomorphism
$$\varphi \colon \calA_{\infty} \to A_K(\calG(X,\sigma)).$$
\end{proposition}

\begin{proof} It is clear from the definitions that the maps $\varphi_n$ are graded $*$-homomorphisms for all $n \ge 0$, so it only remains to show compatibility with the maps $\phi_n$.

First, if $v \in F_{2n}^{0,0}$, then $v = U$ for a unique $U \in \calP_{2n}$, and $\phi_n(v) = \textbf{S}(v)$ corresponds to the decomposition
$$U = \bigsqcup_{\{\ol{U} \in \calP_{2n+2} \mid \ol{U} \subseteq U\}} \ol{U}$$
of $U$ as a disjoint union of sets from $\calP_{2n+2}$. Hence we obtain that $(\varphi_{n+1} \circ \phi_n)(p_v)= \varphi_n(p_v)$.

There are four cases to consider concerning the elements $\tau_{(e,f)}$, all of them similar. We will only deal with the case where $e$ is a red edge and $f$ is a blue edge. So, let $e$ be a red edge and $f$ a blue edge, with $s(e) = s(f) = W$ and $r(e) = U$, $r(f) = V$, where $U,V \in \calP_{2n}$ and $W \in \calP_{2n+1}$. We then have
\begin{equation}\label{equation-tau.decomposition}
\phi_n(\tau_{(e,f)}) = \phi_n (ef^*) = \sum_{s(g) = s(h)} gh^* = \sum_{s(g) = s(h)} \tau_{(g,h)},
\end{equation}  
where $g$ ranges on the set of red edges $g \in R_{s(\ol{e})}$, where $\ol{e} \in R(e)$, and $h$ ranges on the set of blue edges $h \in B_{s(\ol{f})}$, where $\ol{f} \in B_{s(f)}$. Note that
$$S: = \{s(g) \mid g \in R_{s(\ol{e})}, \ol{e} \in R(e)\} = \{ s(h) \mid h \in B_{s(\ol{f})}, \ol{f} \in B_{s(f)}\}$$
is exactly the set of vertices appearing in the expression of $\textbf{S}(W)$ (see Lemma \ref{lemma-technical.1} (ii)), and just as before, this set corresponds to the set of clopen subsets $\ol{W} \in \calP_{2n+3}$ such that $\ol{W} \subseteq W$. Given $\ol{W} \in \calP_{2n+3}$ such that $\ol{W} \subseteq W$, we have inclusions $\ol{W} \subseteq W \subseteq \sigma(U)$ and so by Lemma \ref{lemma-group.of.f}(ii) there exists a unique $\ol{U} \in \calP_{2n+2}$ such that
\begin{equation}\label{equation-tau.decomp.first}
\ol{U} \subseteq U \quad \text{ and } \quad \ol{W} \subseteq \sigma(\ol{U}) \subseteq W \subseteq \sigma(U).
\end{equation}
In addition, there is a unique $\ol{V} \in \calP_{2n+2}$ such that $\ol{W} \subseteq \ol{V}$, and since $\ol{W} \subseteq W \subseteq V$, we must have
\begin{equation}\label{equation-tau.decomp.second}
\ol{W} \subseteq \ol{V} \subseteq W \subseteq V.
\end{equation}
It follows that each of the terms $gh^*$ in the sum appearing in \eqref{equation-tau.decomposition} corresponds exactly to a triple $(\ol{U}, \ol{W}, \ol{V})$, where $\ol{W} \in \calP_{2n+3}$, $\ol{U},\ol{V} \in \calP_{2n+2}$, and \eqref{equation-tau.decomp.first} and \eqref{equation-tau.decomp.second} hold. We infer that the decomposition \eqref{equation-tau.decomposition} corresponds, under $\varphi_{n+1}$, to the decomposition
$$\bigsqcup_{\{\ol{W} \in \calP_{2n+3} \mid \ol{W} \subseteq W\}} Z((\sigma|_{\ol{U}})^{-1}(\ol{W}), 1, 0, \ol{W}) = Z((\sigma|_U)^{-1}(W), 1, 0, W).$$
Hence we get that $(\varphi_{n+1} \circ \phi_n)(\tau_{(e,f)}) = \varphi_n(\tau_{(e,f)})$, as desired.
\end{proof}

We now start a study of elements in $\calA_{\infty}$.

\begin{definition}\label{definition-n.triple}
Let $n \ge 0$ and $U,V \in \calP_{2n}$, $W \in \calP_{2n+1}$. We say that $(U, W, V)$ is an {\it $n$-triple} if $W \subseteq V$ and $W \subseteq \sigma (U)$. Note that such a triple determines an element $fe^* = \tau_{(f,e)}$ in $\calA_{2n}$, where $f \in R_U$ and $e \in B_V$, such that $\varphi_n(fe^*) = \mathds{1}_{Z((\sigma|_U)^{-1}(W), 1, 0, W)}$.
\end{definition}

In the following lemma, we describe the images by $\varphi_n$ of some elements of $\calA_{2n}$.

\begin{lemma}\label{lemma-images.in.AG}
Let $n \ge 0$.
\begin{enumerate}[i),leftmargin=0.7cm]
\item Suppose that $x = f_1e_1^* \cdots f_ke_k^* \in \calA_{2n}$, with $f_i$ red edges and $e_i$ blue edges, such that $s(f_i) = s(e_i) = W_i$, $r(f_1) = U_1$ and $r(e_i) = r(f_{i+1}) = U_{i+1}$ for $1 \leq i \leq k-1$. Consider the set
$$T_1 := (\sigma|_{U_1})^{-1}(W_1 \cap (\sigma|_{U_2})^{-1}(W_2 \cap \cdots \cap (\sigma|_{U_{k-1}})^{-1}(W_{k-1} \cap (\sigma|_{U_k})^{-1}(W_k)) \cdots )).$$
Then
$$\varphi_n(x) = \mathds{1}_{Z(T_1, k, 0, \sigma^k(T_1))}$$
and
$$\varphi_n(xx^*) = \mathds{1}_{T_1}.$$
\item Let $x = (f_1e_1^*) \cdots (f_{k-1}e_{k-1}^*) (f_kf_{k+1}^*)(e_{k+2}f_{k+2}^*) \cdots (e_{k+l}f_{k+l}^*) \in \calA_{2n}$, with $e_i$ blue edges and $f_i$ red edges, such that $s(e_i) = s(f_i) = W_i$ for $1 \leq i \leq k-1$ and $k+2 \leq i \leq k+l$, $s(f_k) = s(f_{k+1}) = W_k = W_{k+1}$, $f_k \ne f_{k+1}$, $r(f_1) = U_1$, $r(e_i) = r(f_{i+1}) = U_{i+1}$ for $1 \leq i \leq k-1$, $r(f_{k+i}) = r(e_{k+i+1}) = U_{k+i}$ for $1 \leq i \leq l-1$ and $r(f_{k+l}) = U_{k+l}$. Let $T_1$ be as in i) and set
$$T_2 := (\sigma|_{U_{k+l}})^{-1}(W_{k+l} \cap (\sigma|_{U_{k+l-1}})^{-1}(W_{k+l-1} \cap \cdots \cap (\sigma|_{U_{k+2}})^{-1}(W_{k+2} \cap (\sigma|_{U_{k+1}})^{-1}(W_{k+1})) \cdots )).$$
Then $T := \sigma^k(T_1) \cap \sigma^l(T_2) \subseteq W_k$, $(\sigma|_{T_1})^k$ and $(\sigma|_{T_2})^l$ are injective, and, with $U := (\sigma|_{T_1})^{-k}(T)$ and $V := (\sigma|_{T_2})^{-l}(T)$, we have
$$\varphi_n(x) = \mathds{1}_{Z(U, k, l, V)}$$
and
$$\varphi_n(xx^*) = \mathds{1}_{U}.$$
\end{enumerate}
\end{lemma}
\begin{proof}
$i)$ Note that, for $y \in T_1$, we have $y \in U_1$ and $\sigma(y) \in W_1 \subseteq U_2$. Moreover $\sigma^2(y) = \sigma(\sigma(y)) \in W_2 \subseteq U_3$. In general we have $\sigma^i(y) \in W_i \subseteq U_{i+1}$ for $1 \leq i \leq k-1$, and $\sigma^k(y) \in W_k$. Since the restriction of $\sigma$ to each set $U_i$ is injective, we see that $\sigma^k$ is injective on $T_1$. We have
$$\varphi_n(x) = \mathds{1}_{\calZ},$$
where
$$\calZ = Z((\sigma|_{U_1})^{-1}(W_1), 1, 0, W_1) \cdots Z((\sigma|_{U_k})^{-1}(W_k), 1, 0, W_k).$$
The elements of the right-hand side of the above equality have the form $(y_0,1,y_1)(y_1,1,y_2) \cdots (y_{k-1},1,y_k)$, where $y_0 \in U_1 $, $y_1 = \sigma(y_0) \in W_1 \subseteq U_2$, and in general $y_i = \sigma^i(y_0) \in W_i \subseteq U_{i+1}$ for $1 \leq i \leq k-1$, and $\sigma^k(y_0)= y_k \in W_k$. Then we get that $y_0 \in T_1$ and that
$$(y_0,1,y_1) \cdots (y_{k-1},1,y_k) = (y_0,k,y_k) \in Z(T_1, k, 0, \sigma^k(T_1)).$$
Conversely, if $(y,k,z) \in Z(T_1, k, 0, \sigma^k(T_1))$, by the definition of $T_1$ we can write $(y,k,z)$ as a product of elements in $Z((\sigma|_{U_i})^{-1}(W_i), 1, 0, W_i)$ for all $1 \leq i \leq k$.

Now,
\begin{align*}
\varphi_n(xx^*) & = \varphi_n(x)\varphi_n(x)^* \\
& = \mathds{1}_{Z(T_1, k, 0, \sigma^k(T_1))} \cdot \mathds{1}_{Z(T_1, k, 0, \sigma^k(T_1))^{-1}} \\
& = \mathds{1}_{Z(T_1, k, 0, \sigma^k(T_1)) \cdot Z(T_1, k, 0, \sigma^k(T_1))^{-1}} \\
& = \mathds{1}_{Z(T_1, k, 0, \sigma^k(T_1)) \cdot Z(\sigma^k(T_1), 0, k, T_1)} \\
& = \mathds{1}_{T_1}.
\end{align*}

The proof of $ii)$ is similar to the proof of $i)$, and is left to the reader.
\end{proof}

This result is a first step towards finding a canonical expression for an element of $\calA_{\infty}$. The next step is to show a partial injectivity result for $\varphi$.

\begin{lemma}\label{lemma-injectivity.on.projections}
The following statements hold.
\begin{enumerate}[i),leftmargin=0.7cm]
\item Let $x = f_1e_1^* \cdots f_ke_k^* \in \calA_{2n}$ be as in Lemma \ref{lemma-images.in.AG} (i). Then $\phi_{n,\infty}(x) \ne 0$ if and only if $\varphi_n(x) \ne 0$.
\item Let $x = (f_1e_1^*) \cdots (f_{k-1}e_{k-1}^*)(f_kf_{k+1}^*)(e_{k+2}f_{k+2}^*) \cdots (e_{k+l}f_{k+l}^*) \in \calA_{2n}$ be as in Lemma \ref{lemma-images.in.AG} (ii). Then $\phi_{n,\infty}(x) \ne 0$ if and only if $\varphi_n(x) \ne 0$.
\end{enumerate}
\end{lemma}
\begin{proof}
First, observe that for any term $x$ as in $i)$ or $ii)$, we have $\phi_{n,\infty} (x)\ne 0$ if and only if $\phi_{n,\infty} (xx^*)\ne 0$. This is due to the fact that $\calA_{\infty}$ is tame (Theorem \ref{theorem-tame.at.infinity}), so that $\phi_{n,\infty}(x) $ is a partial isometry in $\calA_{\infty}$. Similarly, $\varphi _n(x)\ne 0$ if and only if $\varphi_n(xx^*)\ne 0$.  
	
Let us prove $i)$. If $\varphi_n(x) \ne 0$ then $\phi_{n,\infty}(x) \ne 0$, because $\varphi_n(x) = \varphi(\phi_{n,\infty}(x))$.

Suppose now that $\phi_{n,\infty}(xx^*) \ne 0$, so in particular $\phi_{n,m}(xx^*) \neq 0$ for all $m> n$. We will show that $T_1$ is non-empty. We proceed by induction on $k$.

For $k = 1$, we have $x = f_1e_1^*$ and $xx^* = f_1f_1^*$. In this case $T_1 = (\sigma|_{U_1})^{-1}(W_1)\ne \emptyset$, because $W_1\ne \emptyset$ and $W_1\subseteq \sigma (U_1)$. 

Suppose now that the result is true for all lengths up to $k-1$. Take $x = f_1e_1^* \cdots f_ke_k^*$. Observe that
$$xx^* = f_1e_1^* \cdots e_{k-1}^*f_kf_k^* e_{k-1} \cdots e_1f_1^*= (f_1e_1^*)yy^*(e_1^*f_1),$$
where $y= f_2e_2^* \cdots f_ke_k^*$. Recall from Lemma \ref{lemma-tame.aftersteps} that $\phi_{n,n+k} (xx^*)= \widetilde{\phi}_{2n,2n+2k}(xx^*)$ is a sum of a finite number of distinct vertices in $F_{2n+2k}^{0,0}$, and likewise $\phi_{n,n+k-1} (yy^*)= \widetilde{\phi}_{2n,2n+2k-2}(yy^*)$ is a sum of a finite number of distinct vertices in $F_{2n+2k-2}^{0,0}$. We will show, using induction, that any of the open compact sets corresponding to the vertices appearing in the expression of $\phi_{n,n+k} (xx^*)$ is contained in $T_1$. Our assumption that $\phi_{n, \infty} (xx^*)\ne0$ would then imply that $T_1\ne \emptyset$. 

Suppose that $\phi_{n,n+k-1} (yy^*) = \sum_v v $ for distinct vertices $v\in F_{2n+2k-2}^{0,0}$ and that each of the open compact subsets of $X$ represented by these vertices is contained in
$$T' :=   (\sigma|_{U_2})^{-1}(W_2 \cap (\sigma|_{U_3})^{-1}(W_3 \cap \cdots \cap (\sigma|_{U_{k-1}})^{-1}(W_{k-1} \cap (\sigma|_{U_k})^{-1}(W_k)) \cdots )).$$

By Lemma \ref{lemma-observations.for.phinm} (ii), we can write $\phi_{n,n+k-1}(e_1) = \sum _{g\in N} g$, where $N$ is a set of blue edges in $F_{2n+2k-2}$ with mutually distinct sources. So we get
\begin{align*}
\phi_{n,n+k-1} (xx^*) & = \phi_{n,n+k-1} (f_1)\Big(\sum_{g\in N} g^*\Big)\Big(\sum_ v v\Big) \Big(\sum_{g\in N} g\Big) \phi_{n,n+k-1} (f_1^*) \\
& = \phi_{n,n+k-1} (f_1)\Big( \sum_{g\in N'} s(g) \Big)  \phi_{n,n+k-1} (f_1^*),
\end{align*}
where $N'$ is a non-empty subset of $N$. Again by Lemma \ref{lemma-observations.for.phinm} (ii), we can write $\widetilde{\phi}_{2n,2n+2k-1}(f_1) = \sum _{h\in L} h^*$, where $L$ is a subset of red edges in $F_{2n+2k-1}$, whose sources are mutually disjoint, and such that all edges in $L$ with the same range belong to the same element of the partition $D$. It follows that
$$\widetilde{\phi}_{2n,2n+2k-1} (xx^*) = \Big(\sum_{h\in L} h^*\Big) \Big( \sum_{g\in N'} s(g)\Big) \Big(\sum_{h\in L} h\Big) = \sum_{h\in L'} s(h),$$
where $L'$ is a subset of $L$. Note that we have $0\ne \phi_{n,n+k} (xx^*)= \sum_{h\in L'} s(h)$, so $L'\ne \emptyset$. Now take any $h\in L'$ and set $Z= s(h)$, $W=r(h)$. There is $g\in N'$ such that $r(h)= s(g)$. Moreover we have that $r(g)= V$, where $V\subseteq T'$ by our induction hypothesis. Since $g$ is a blue edge we have $W=r(h)= s(g)\subseteq V\subseteq T'$, so that $W\subseteq T'$. Since $h$ is a red edge in $F_{2n+2k-1}$, we must have $Z\subseteq \sigma^{-1}(W)$. On the other hand, since $xx^*\le f_1f_1^*\le r(f_1)=U_1$, we have that $s(h)\le \phi_{n,n+k} (xx^*)\le \phi_{n,n+k}(r(f_1))$, and we conclude that $Z\subseteq U_1$. Similarly $e_1^*yy^*e_1\le s(e_1)= W_1$ and thus $s(g)\le \phi_{n,n+k-1}(e_1^*yy^*e_1)\le \phi_{n,n+k-1}(s(e_1))$ implies that $W\subseteq W_1$. We thus have that 
$$W\subseteq W_1\cap T'\qquad \text{and} \qquad  Z\subseteq U_1 \cap \sigma^{-1}(W) \subseteq  U_1 \cap \sigma^{-1}(W_1\cap T'), $$
so we conclude that $Z\subseteq (\sigma|_{U_1})^{-1}(W_1\cap T') = T_1$, as desired.

The similar proof of $ii)$ is left to the reader.
\end{proof}

We can now state our result on canonical forms.

\begin{theorem}\label{theorem-four.types.theorem}
Each element $a \in \calA_{\infty}$ can be represented as $a = \phi_{n,\infty}(b)$ for some $n \ge 0$ and $b \in \calA_{2n}$ such that $b$ is a finite linear combination of the following types of elements:
\begin{enumerate}[(A),leftmargin=0.7cm]
\item Vertices $p_U$, with $U \in \calP_{2n}$.
\item Terms of the form $x = f_1e_1^* \cdots f_ke_k^*$ as in Lemma \ref{lemma-images.in.AG} (i), with $\varphi_n(x) \ne 0$.
\item Adjoints of the terms from (B).
\item Terms of the form $x = (f_1e_1^*) \cdots (f_{k-1}e_{k-1}^*)(f_kf_{k+1}^*)(e_{k+2}f_{k+2}^*) \cdots (e_{k+l}f_{k+l}^*)$ as in Lemma \ref{lemma-images.in.AG} (ii), where $k, l \ge 1$ and with $\varphi_n(x) \ne 0$. When $k = l = 1$, the corresponding term must be interpreted as $f_1f_2^*$, where $f_1, f_2$ are distinct red edges with the same source.
\end{enumerate}
Note that the family (D) is closed under adjoints.
\end{theorem}
\begin{proof} Let $a$ be a nonzero element of $\calA_{\infty}$. There exists $m \ge 0$ and $c \in \calA_{2m}$ such that $\phi_{m,\infty}(c) = a$. The element $c$ will be a linear combination of vertices $p_U$, for $U \in \calP_{2m}$, and terms which are products of elements $ef^*$ , where $e$ and $f$ are red or blue edges such that $s(e) = s(f) \in F_{2m}^{0,1}$.

Observe that a term $ee^*$, where $e$ is either a blue or a red edge, can be simplified by applying $\phi_m$, since $\phi_m(ee^*)$ is a sum of distinct vertices in $\calA_{2m+2}$ by Lemma \ref{lemma-tame.aftersteps}, so of the form (A). We can thus assume that we do not have terms of the form $ee^*$ in our product. Moreover, since $s(e_1) = s(e_2)$ implies $e_1 = e_2$ for $e_1,e_2$ blue edges, we may, in view of the above, assume that no terms $e_1e_2^*$, with both $e_1$ and $e_2$ blue edges, appear in the product.

Suppose that we have a product $(fe^*)(gh^*)$ such that $e, g$ are both blue edges. By our assumptions, we necessarily have that both $f$ and $h$ are red edges. Then we get that
$$(fe^*)(gh^*) = \delta_{e,g}(fh^*).$$
So either the product is $0$ or it can be simplified to a product containing a term $fh^*$, with both $f$ and $h$ red edges. By the previous paragraph, we can assume that $f \ne h$. We can thus simplify these terms.

Suppose now that we have a product $(ef^*)(gh^*)$ such that $f, g$ are both red edges. Then we get, as before,
$$(ef^*)(gh^*) = \delta _{f,g}(eh^*).$$
If both $e$ and $h$ are red edges, then we can assume that they satisfy $e \ne h$ (otherwise we can reduce the length of the word by applying $\phi_m$). If $e$ and $h$ are both blue edges, then either $eh^* = 0$ or $e = h$. In the latter case, we can simplify by applying $\phi_m$. If $e$ is blue and $h$ is red the product will be either $0$ or of the form $eh^*$, and similarly when $e$ is red and $h$ is blue. Hence these terms can be simplified too. 

In conclusion, the only products that cannot be simplified are those where all the factors are of the form $e_if_i^*$, where all $e_i$ are blue and all $f_i$ are red (terms of the form (C)), or all $e_i$ are red and all $f_i$ are blue (terms of the form (B)), and those of the form (D).
\end{proof}

We are now ready to show the main result of this section.

\begin{theorem}\label{theorem-main.iso.theorem}
With the previous notation, we have that the map
$$\varphi \colon \calA_{\infty} \to A_K(\calG(X,\sigma))$$
is a graded $*$-isomorphism.
\end{theorem}
\begin{proof}
In view of Proposition \ref{proposition-morph.compat.steinberg}, it only remains to show that $\varphi $ is bijective. We first show it is surjective. For this, it is enough to check that all the characteristic functions of sets of the form $Z(U,k,l,V)$ belong to the image of $\varphi$. Suppose therefore that $U$ and $V$ are open compact subsets of $X$ such that $\sigma^k|_{U}$ and $\sigma^l|_V$ are injective and that $\sigma^k(U) = \sigma^l(V)$. Observe that
$$Z(U,k,l,V) = Z(U,k,0,\sigma^k(U)) \cdot Z(V,l,0,\sigma^l(V))^{-1},$$
so that it suffices to show that $Z(U,k,0,\sigma^k(U))$ belongs to the image of $\varphi$.

There is some $n \ge 0$ such that $U$ can be written as a disjoint union of sets in $\calP_{2n}$. Hence we can assume that $U \in \calP_{2n}$. We claim that
\begin{equation}\label{equation-decomposition.of.U}
U = \bigsqcup_{(W_1,\dots,W_k) \in \calB} (\sigma|_{U_1})^{-1}(W_1 \cap (\sigma|_{U_2})^{-1}(W_2 \cap \cdots \cap (\sigma|_{U_{k-1}})^{-1}(W_{k-1} \cap (\sigma|_{U_k})^{-1}(W_k)) \cdots )),
\end{equation}
where
$$\calB = \{(W_1,\dots ,W_k) \in \calP_{2n+1}^k \mid W_i \subseteq \sigma(U_i) \text{ and } W_i \subseteq U_{i+1} \text{ for } 1 \leq i \leq k, \text{ for unique } U_i\text{'s} \text{ with } U_1 = U\}.$$
First, we observe that $T \cap T' = \emptyset$, where $T$ and $T'$ are sets as in the right-hand side of \eqref{equation-decomposition.of.U} corresponding to different choices $(W_1, \dots, W_k)$ and $(W_1', \dots, W_k')$. Indeed, suppose that $T \cap T' \ne \emptyset$ and take $x \in T \cap T'$. Then $x \in U_1$ and $\sigma (x) \in W_1 \cap W_1'$. Since $W_1, W_1'$ belong to the partition $\calP_{2n+1}$ of $X$, we obtain that $W_1 = W_1'$. This in turn implies that $U_2 = U_2'$. Now note that $\sigma^2(x) \in W_2 \cap W_2'$ and thus $W_2 = W_2'$ and $U_3 = U_3'$. Continuing in this way, we obtain that $(W_1, \dots, W_k) = (W_1', \dots, W_k')$.

Take now $x \in U = U_1$. We have
$$\sigma(U_1) = \bigsqcup_{\{ W \in \calP_{2n+1} \mid W \subseteq \sigma(U_1)\}} W,$$
so that there is a unique $W_1 \in \calP_{2n+1}$ such that $\sigma(x) \in W_1$. Then there is a unique $U_2 \in \calP_{2n}$ such that $W_1 \subseteq U_2$. Since $\sigma(x) \in U_2$, by the same argument we get $W_2 \in \calP_{2n+1}$ such that $W_2 \subseteq \sigma(U_2)$ and $\sigma^2(x) = \sigma(\sigma(x)) \in W_2$. In this way, we build a family of sets $W_1, \dots, W_k$ and $U_2, \dots, U_{k}$ satisfying the specified conditions and such that
$$x \in (\sigma|_{U_1})^{-1}(W_1 \cap (\sigma|_{U_2})^{-1}(W_2 \cap \cdots \cap (\sigma|_{U_{k-1}})^{-1}(W_{k-1} \cap (\sigma|_{U_k})^{-1}(W_k)) \cdots )).$$
This shows our claim.

Note that each $(U_i,W_{i},U_{i+1})$ is an $n$-triple in the sense of Definition \ref{definition-n.triple}, hence it gives rise to an element $f_ie_i^*$ in $\calA_{2n}$, where $f_i\in R_{U_i}$, $e_i\in B_{U_{i+1}}$ and $\varphi _n (f_ie_i^*)= \mathds{1}_{Z((\sigma|_{U_i})^{-1}(W_i),1,0,W_i)}$. Therefore we see that each sequence $(W_1, W_2, \dots, W_k) \in \calP_{2n+1}^k$ satisfying the conditions stated above, and such that the corresponding set $T$ is non-empty, determines an element $x = f_1e_1^* \cdots f_ke_k^*$ of the form (B) from Theorem \ref{theorem-four.types.theorem}, with $\varphi_n(x) = \mathds{1}_{Z(T,k,0,\sigma^k(T))}$ by Lemma \ref{lemma-images.in.AG}. Using this observation and \eqref{equation-decomposition.of.U}, the surjectivity of $\varphi$ follows.\\

Finally we show that $\varphi$ is injective. Since $\varphi$ is a graded homomorphism, it is enough to show that it is injective in each graded component. Suppose that we have a nonzero homogeneous element, say $a$, in $\calA_{\infty}$. By Theorem \ref{theorem-four.types.theorem}, there exists some $n \ge 0$ and an element $b \in \calA_{2n}$ which is a linear combination of terms of the forms (A), (B), (C), (D) such that $a = \phi_{n,\infty}(b)$. Of course we can assume that $\phi_{n,\infty}$ does not annihilate any of these terms, which by Lemma \ref{lemma-injectivity.on.projections} implies that $\varphi_n$ also does not annihilate any of them. It is sufficient to show that the compact open bisections corresponding to two distinct terms of the same degree are disjoint.  We will treat only the case where the two terms are of the form (D), i.e. satisfy the conditions in Lemma \ref{lemma-images.in.AG} (ii). The cases where other types of terms arise are treated in a similar way.

Suppose that we have two terms
$$\alpha = (f_1e_1^*) \cdots (f_{k-1}e_{k-1}^*) (f_k f_{k+1}^*) (e_{k+2}f_{k+2}^*) \cdots (e_{k+l}f_{k+l}^*),$$
$$\beta = f'_1(e'_1)^* \cdots f'_{k'-1}(e'_{k'-1})^* f'_{k'}(f'_{k'+1})^* e'_{k'+2}(f'_{k'+2})^* \cdots e'_{k'+l'}(f'_{k'+l'})^*$$
in $\calA_{2n}$ as in Lemma \ref{lemma-images.in.AG} (ii), and such that $k-l = k'-l'$. We will assume that $k' \ge k$, so that $k' = k + r$ and $l' = l + r$ for some $r \ge 0$. Let
$$W_1, \dots, W_k = W_{k+1}, W_{k+2}, \dots, W_{k+l} \quad \text{and} \quad U_1, \dots, U_k, U_{k+1}, \dots, U_{k+l}$$
be the sets associated to $\alpha$, and
$$W_1', \dots, W'_{k'} = W'_{k'+1}, W'_{k'+2}, \dots, W'_{k'+l'} \quad \text{and} \quad U'_1, \dots, U'_{k'}, U'_{k'+1}, \dots, U'_{k'+l'}$$
be the sets associated to $\beta$. Let $T_1$, $T_2$, $T$, $U = (\sigma|_{T_1})^{-k}(T)$ and $V = (\sigma|_{T_2})^{-l}(T)$ be the sets associated to $\alpha$ defined in Lemma \ref{lemma-images.in.AG} (ii), and let $T_1'$, $T_2'$, $T'$, $U' = (\sigma|_{T'_1})^{-k'}(T')$ and $V' = (\sigma|_{T'_2})^{-l'}(T')$ be the ones corresponding to $\beta$. Suppose, by way of contradiction, that
$$Z(U,k,l,V) \cap Z(U',k',l',V') \ne \emptyset,$$
and let $(x,k-l,y)$ be an element in this intersection. Then $x \in U \cap U'$, $y \in V \cap V'$, and $\sigma^k(x) = \sigma^l(y)$, which automatically implies $\sigma^{k'}(x) = \sigma^{l'}(y)$ because $k' = k + r$ and $l' = l + r$ with $r \ge 0$. Since $U \cap U' \neq \emptyset$, necessarily $U_1 = U_1'$, and also $U_{k+l} = U'_{k'+l'}$ because $V \cap V' \neq \emptyset$. Now observe that $\sigma(x) \in W_1 \cap W_1'$, so that $W_1 = W_1'$ and $U_2 = U_2'$. Then we get $\sigma ^2(x) \in W_2 \cap W_2'$, so $W_2 = W_2'$ and $U_3 = U_3'$, and so on up to $U_{k} = U'_{k}$ and $W_k = W_k'$. Starting at the other end, we have $\sigma(y) \in W_{k+l} \cap W'_{k'+l'}$, so that $W_{k+l} = W'_{k'+l'}$ and $U_{k+l-1} = U'_{k'+l'-1}$. Then since $\sigma^2(y) \in W_{k+l-1} \cap W'_{k'+l'-1}$, we get $W_{k+l-1} = W'_{k'+l'-1}$ and $U_{k+l-2} = U'_{k'+l'-2}$. Continuing in this way, we get that $W_{k+l-i} = W'_{k'+l'-i}$ and $U_{k+l-i}= U'_{k'+l'-i}$ for all $0 \leq i \leq l-1$. We now distinguish two cases.
\begin{enumerate}[$\cdot$,leftmargin=0.5cm]
\item In case $r = 0$ we conclude that all sets $W_i$ and all sets $U_i$ agree with $W_i'$ and $U_i'$ respectively, so that $\alpha = \beta$. 
\item If $r > 0$, then observe that
$$\sigma^k(x) = \sigma^l(y) \in W_k' \cap W'_{k'+l'-l+1},$$
so that $W_k' = W'_{k'+l'-l+1} = W'_{k+2r+1}$. Therefore $U'_{k+1} = U'_{k+2r}$. Since $\sigma^{k+1}(x) = \sigma^{l+1}(y) \in W'_{k+1}\cap W'_{k+2r}$, we get that $W'_{k+1} = W'_{k+2r}$ and thus $U'_{k+2} = U'_{k+2r-1}$. Continuing in this way, we see that $W'_{k+r-1} = W'_{k+r+2}$ and thus $U'_{k+r} = U'_{k+r+1}$, that is, we obtain that $U'_{k'} = U'_{k'+1}$. But since $f'_{k'} \ne f'_{k'+1}$ and $s(f'_{k'}) = s(f'_{k'+1})$, we necessarily have that $U'_{k'} = r(f'_{k'}) \ne r(f'_{k'+1})= U'_{k'+1}$, and we obtain a contradiction. Therefore the case where $r > 0$ is not possible.
\end{enumerate}
The conclusion is that, whenever $Z(U,k,l,V) \cap Z(U',k',l',V') \neq \emptyset$ holds, we necessarily have $\alpha = \beta$, as desired. The proof is complete.
\end{proof}

We consider now the situation for the corresponding $C^*$-algebras. Writing $C_n= C^*(F_n,D^n)$, we have that, by definition, $C_n$ is the enveloping $C^*$-algebra of $L_n$. Hence we get a sequence of $*$-homomorphisms $C_n\to C_{n+1}$, also denoted by $\tilde{\phi}_n$. Set 
$C_{\infty} =  \varinjlim (C_n,\tilde{\phi}_n )$,  where the above inductive limit is the colimit of $\{(C_n,\tilde{\phi}_n)\}$ in the cateory of $C^*$-algebras. 
We now set
$$\calB_{2n} := p_{2n}C_{2n}p_{2n}\quad \text{ and } \quad \calB_{\infty} = \varinjlim (\calB_{2n},\phi_n),$$
where, as before, $\phi_n$ is the restriction of $\tilde{\phi}_{2n+1}\circ \tilde{\phi}_{2n}$ to $\calB_{2n}$. 

We can now state our main result for $C^*$-algebras.

\begin{theorem}\label{theorem-isomorphism.for.Cstars}
With the above notation, there is a unique $*$-isomorphism $\ol{\varphi} \colon \calB_{\infty} \to C^*(\calG (X,\sigma))$ such that the following diagram
\begin{equation}\label{equation-commutative.diagram.Ainfty.Binfty}
\vcenter{
	\xymatrix{
		\calA _{\infty} \ar@{->}[r]^{\kern-20pt{\varphi}}_{\kern-20pt{\cong}} \ar@{->}[d] & A_{\C}(\calG (X,\sigma))   \ar@{->}[d] \\
		\calB_{\infty}  \ar@{->}[r]^{\kern-20pt{\ol{\varphi}}}_{\kern-20pt{\cong}}  & C^*(\calG (X,\sigma))
	}
}
\end{equation}
is commutative.
\end{theorem}
\begin{proof}
It is a simple matter to check that $C_{\infty}$ is the enveloping $C^*$-algebra of $L_{\infty}$.

Just as in Lemma \ref{lemma-ind.limit.OK}, we have a natural $*$-isomorphism $\ol{\lambda} \colon \calB_{\infty} \to pC_{\infty}p$, such that the following diagram 
\begin{equation}\label{equation-commutative.diagram.Ainfty.Binfty.BIS}
\vcenter{
	\xymatrix{
		\calA _{\infty} \ar@{->}[r]^{\kern-7pt{\lambda}}_{\kern-7pt{\cong}} \ar@{->}[d] & pL_{\infty}p   \ar@{->}[d] \\
		\calB_{\infty}  \ar@{->}[r]^{\kern-7pt{\ol{\lambda}}}_{\kern-7pt{\cong}}  & pC_{\infty}p
	}
}
\end{equation} 	
is commutative. Since there is a partial isometry $s\in L_{\infty}$ such that $ss^*= p$ and $s^*s= 1-p$ (take for instance $s= \wt{\phi}_{0,\infty}(\sum_{e\in B} e)$, where $B$ is the set of blue edges in $F^{1,0}$), it follows that $L_{\infty} \cong pL_{\infty}p\otimes M_2(\C)$, and $C_{\infty}\cong pC_{\infty}p\otimes M_2(\C)$ in a compatible way. Since $C_{\infty}$ is the enveloping $C^*$-algebra of $L_{\infty}$, it follows that $pC_{\infty}p$ is the enveloping $C^*$-algebra of $pL_{\infty}p$. We infer from \eqref{equation-commutative.diagram.Ainfty.Binfty.BIS} that $\calB_{\infty}$ is the enveloping $C^*$-algebra of $\calA_{\infty}$. The result follows from this and Theorem \ref{theorem-main.iso.theorem}, because $C^*(\calG (X,\sigma))$ is the enveloping $C^*$-algebra of $A_{\C} (\calG (X,\sigma))$ (see Definition \ref{definition-Steinberg.algebra}).
\end{proof}

\subsection{Functoriality properties}\label{subsection-functoriality}

In this section we show that the Deaconu-Renault construction is indeed a functor from $\textsf{LHomeo}$ to $\textsf{AHGrd}$, where $\textsf{AHGrd}$ is the category whose objects are all the second countable, ample Hausdorff groupoids, and whose morphisms are the continuous groupoid homomorphisms. We further show that, whenever $\psi\colon Y\to X$ is the universal continuous $(E,C)$-equivariant map induced by an $(E,C)$-structure on $Y$, there are induced $*$-homomorphisms at the levels of Steinberg algebras and $C^*$-algebras of the corresponding groupoids, respectively.

The following lemma, which gives the mentioned functoriality, may be well-known to experts.

\begin{lemma}\label{lemma-morphismsOK}
Let $(X,\sigma)$ and $(Y,\rho)$ be objects in $\emph{\textsf{LHomeo}}$, and let $\psi \colon X \to Y$ be an equivariant continuous map. Then the induced map $\psi_* \colon \calG(X,\sigma)\to \calG(Y,\rho)$ defined by
$$\psi_*(x,m-n,y) = (\psi(x),m-n,\psi(y))$$
for $(x,m-n,y) \in \calG(X,\sigma)$, is a morphism in $\emph{\textsf{AHGrd}}$. Hence the Deaconu-Renault groupoid construction can be extended to a functor $\emph{\textsf{DR}} \colon \emph{\textsf{LHomeo}} \to \emph{\textsf{AHGrd}}$.
\end{lemma}

\begin{proof}
It is clear that $\psi_*$ is well-defined and an algebraic morphism of groupoids. Thus we only need to show that $\psi_*$ is continuous.

Observe that the restriction of $\psi_*$ to the space of units $X$ agrees with $\psi$, hence it is continuous. Let $\alpha = (x,m-n,y)\in \calG(X,\sigma)$, and let $U$ be an open compact subset of $\calG(Y,\rho)$ containing $\psi_*(\alpha)$. Since $\rho ^i$ is a local homeomorphism for each $i\ge 0$, we can find open neighborhoods $U_1$ of $\psi(x)$ and $U_2$ of $\psi(y)$ such that $\rho^m|_{U_1}$ and $\rho^n|_{U_2}$ are injective. Replacing $U$ with $U \cap Z(U_1,m,n,U_2)$, we can assume that $U \subseteq Z(U_1,m,n,U_2)$, and in particular that $U$ is an open compact bisection. Note that $s(U) \subseteq U_2$ and $r(U) \subseteq U_1$.

By continuity of $\psi$, we can take open neighborhoods $U_1'$ and $U_2'$ of $x$ and $y$ in $X$, respectively, such that $\psi(U_1')\subseteq r(U)$ and $\psi(U_2') \subseteq s(U)$. Then $Z(U_1',m,n,U_2')$ is an open subset of $\calG(X,\sigma)$, and clearly $\alpha \in Z(U_1',m,n,U_2')$. Now if $\gamma \in Z(U_1',m,n,U_2')$, then $r(\psi_*(\gamma)) \in r(U)$, and since $U$ is contained in the bisection $Z(U_1,m,n,U_2)$ we conclude that $\psi_*(\gamma) \in U$. This shows that
$$\psi_*(Z(U_1',m,n,U_2')) \subseteq U,$$
which gives the continuity of $\psi_*$. The last statement follows straightforwardly.
\end{proof}

\begin{notation}\label{notation-set.red.edges}
Let $(E,C)$ be an $l$-diagram, or a layer in an $l$-diagram, and let $v$ be a vertex in $E$ . We will denote by $\sred^{-1}(v)$ the set of red edges $f$ in $E^1$ such that $s(f)= v$.
\end{notation}

If $\calG$ is a groupoid and $x\in \calG^0$, we denote by $\calG x$ the set of elements $\gamma \in \calG$ such that $s(\gamma) = x$.  

\begin{lemma}\label{lemma-preimagesOK}
Let $(Y,\rho) \in \emph{\textsf{LHomeo}}$ and let $(E,C)$ be a generalized finite shift graph. Suppose that we have an $(E,C)$-structure on $Y$, and let $\psi \colon Y \to X$ be the unique $(E,C)$-equivariant continuous map, where $(X,\sigma)$ is the generalized finite shift associated to $(E,C)$. Then for $y\in Y$, the map $\psi_*$ from Lemma \ref{lemma-morphismsOK} induces a bijection from $\calG(Y,\rho) y$ onto $\calG(X,\sigma)\psi(y)$.
\end{lemma}

\begin{proof}
We start by fixing the notation. We will denote the natural $(E,C)$-structure on $X$ by $(\{H_e\}_{e\in B}, \{H_f\}_{f\in R})$, where $B$ and $R$ are the sets of blue and red edges of $E$, respectively. For $f \in R$, we have a homeomorphism $\theta_f\colon H_f \to H_e$, where $e$ is the unique blue edge such that $s(e) = s(f)$, and where $\theta _f$ is just the restriction of $\sigma$ to $H_f$. We will also denote by $(\{Z_e\}_{e\in B}, \{Z_f\}_{f\in R})$ the open compact sets corresponding to the $(E,C)$-structure on $Y$, and $\eta _f\colon Z_f\to Z_e$ for the corresponding homeomorphisms, which are the restrictions of $\rho$ to $Z_f$. We recall from Definition \ref{definition-universal.for.generalizedSFT} that we have decompositions $Y=\bigsqcup_{e\in B} Z_e= \bigsqcup_{f\in R} Z_f$ such that, for every $v\in E^{0,0}$, we have $\bigsqcup_{e\in B_v} Z_e = \bigsqcup_{f\in R_v} Z_f$, and, for each $f\in R$, $\rho|_{Z_f}$ is injective and $\rho (Z_f)=Z_e$, where $e$ is the unique blue edge such that $s(e)=s(f)$. 

We use the description of the map $\psi \colon Y\to X$ given in Remark \ref{remark-universal.map.described}. It is clear that, given $y\in Y$, the map $\psi_*$ induces a well-defined map $\calG (Y,\rho)y  \to \calG (X,\sigma) \psi (y)$. Observe that
$$\calG(Y,\rho) y = \{ (z, m-n , y) \in \calG(Y,\rho) \mid \rho^m(z) = \rho^n(y)\}.$$
Moreover, when $\rho^m(z)= \rho^n(y)$ and $n,m>0$, we can obviously assume that $\rho^{m-1}(z)\ne \rho^{n-1}(y)$.

We now give a description of the elements $(z,m-n,y)$ as above. Suppose that $y\in Y$ and that $\rho^m(z) = \rho^n(y)$, where $n,m>0$ and $\rho^{m-1}(z) \ne \rho^{n-1}(y)$ (the cases where $n=0$ or $m=0$ are simpler and left to the reader). Let $g_1 \in R$ be the unique red edge such that $y \in Z_{g_1}$. Then we have $\eta_{g_1}(y) = \rho(y) \in Z_{g_2}$ for a unique $g_2 \in R$, and continuing in this way we find a sequence $g_1,\dots, g_n$ of red edges such that $\rho^i(y) \in Z_{g_{i+1}}$ for $0 \leq i \leq n-1$. Now $\rho^n(y) = \rho(\rho^{n-1}(y)) \in \rho(Z_{g_n}) = Z_e$, where $e$ is the unique blue edge such that $s(e)= s(g_n)$. Set $e_m:=e$ for notational convenience. We also have $\rho(\rho^{m-1}(z)) = \rho^m(z) = \rho^n(y)$, and hence there exists a unique $f_{m}\in \sred^{-1}(s(e_m))$ such that $\rho^{m-1}(z) \in Z_{f_m}$. Note that $f_m \ne g_n$ because $\rho^{m-1}(z) \ne \rho^{n-1}(y)$. Now $\rho(\rho^{m-2}(z)) \in Z_{e_{m-1}}$ for a unique blue edge $e_{m-1}$, and so $\rho^{m-2}(z)\in Z_{f_{m-1}}$ for a unique red edge $f_{m-1}\in \sred^{-1}(s(e_{m-1}))$. In this way, we get unique sequences of blue edges $e_1,\dots ,e_m$ and red edges $f_1,\dots , f_m$, such that $f_i \in \sred^{-1} (s(e_i))$ for $1 \leq i \leq m$, and $\rho^i(z)\in Z_{f_{i+1}}$ for $0 \leq i \leq m-1$. Note that the sequence of red edges $f_1,\dots , f_m,g_1,\dots ,g_n$ satisfies the conditions given in Construction \ref{construction-GFS.concretedescription}, so that we may consider the element $a_{f_1}^{-1} \cdots  a_{f_m}^{-1} a_{g_n} \cdots a_{g_1}\in \mathbb F (A)$.

Now observe that
\begin{equation}\label{equation-relation.for.z.and.thetas}
\eta_{f_1}^{-1}\circ \cdots \circ \eta_{f_m}^{-1}\circ \eta_{g_n} \circ \cdots \circ \eta_{g_1}(y) = z,
\end{equation}
so that certainly $y\in \text{Dom}(\eta_{f_1}^{-1}\circ \cdots \circ \eta_{f_m}^{-1}\circ \eta_{g_n} \circ \cdots \circ \eta_{g_1})$.
Conversely, given a sequence $f_1,\dots , f_m,g_1,\dots ,g_n$ satisfying the conditions given in Construction \ref{construction-GFS.concretedescription} and such that $y \in \text{Dom}(\eta_{f_1}^{-1} \circ \cdots \circ \eta_{f_m}^{-1} \circ \eta_{g_n} \circ \cdots \circ \eta_{g_1})$, the element $z$ defined by equation \eqref{equation-relation.for.z.and.thetas} satisfies that $\rho^m(z)= \rho^n(y)$ and that $\rho^{m-1}(z)\ne \rho^{n-1}(y)$, thus $(z,m-n,y)\in \calG(Y,\rho)$.

We now have, for $(z,m-n,y) \in \calG(Y,\rho)y$ and corresponding red edges $f_1,\dots ,f_m,g_1,\dots , g_n$ as above,
$$a_{f_1}^{-1}  \cdots  a_{f_m}^{-1}  a_{g_n}  \cdots  a_{g_1} \in \psi (y)$$
and, by the $(E,C)$-equivariance of $\psi$,
$$\psi(z) =  \theta_{f_1}^{-1}\circ \cdots \circ \theta_{f_m}^{-1}\circ \theta_{g_n} \circ \cdots \circ \theta_{g_1} (\psi (y)).$$
To show that $\psi_*|_{\calG(Y,\rho) y}$ is surjective onto $\calG(X,\sigma) \psi(y)$, observe that the above argument applied to $(X,\sigma)$ gives that any element in $\calG(X,\sigma) \psi(y)$ is of the form
$$(\theta_{f_1}^{-1}\circ \cdots \circ \theta_{f_m}^{-1}\circ \theta_{g_n} \circ \cdots \circ \theta_{g_1} (\psi(y)), m-n, \psi (y)),$$
where $a_{f_1}^{-1}\cdots a_{f_m}^{-1}a_{g_n}\cdots a_{g_1} \in \psi (y)$. Then by the description of $\psi$ in Remark \ref{remark-universal.map.described} we have that
$$y \in  \text{Dom}(\eta_{f_1}^{-1}\circ \cdots \circ \eta_{f_m}^{-1}\circ \eta_{g_n} \circ \cdots \circ \eta_{g_1}),$$
and then $\psi_*(z,m-n,y) = (\theta_{f_1}^{-1}\circ \cdots \circ \theta_{f_m}^{-1}\circ \theta_{g_n} \circ \cdots \circ \theta_{g_1} (\psi(y)),m-n, \psi (y))$, where $z$ is given by equation \eqref{equation-relation.for.z.and.thetas}. So we get the surjectivity.

Finally, to show that $\psi_*|_{\calG(Y,\rho) y}$ is injective, suppose that
\begin{align*}
\psi_* (z_1,m-n,y) &  = (\theta_{f_1}^{-1}\circ \cdots \circ \theta_{f_m}^{-1}\circ \theta_{g_n} \circ \cdots \circ \theta_{g_1}(\psi(y)), m-n, \psi (y)) \\
& = (\theta_{f_1'}^{-1}\circ \cdots \circ \theta_{f_{m'}'}^{-1}\circ \theta_{g_{n'}'} \circ \cdots \circ \theta_{g_1'}(\psi(y)), m'-n', \psi (y)) \\
& = \psi _* (z_2,m'-n',y),
\end{align*}
for sequences of red edges $f_1,\dots ,f_m,g_1,\dots , g_n$ and $f_1',\dots ,f_{m'}',g_1',\dots ,g_{n'}'$ satisfying the conditions in Construction \ref{construction-GFS.concretedescription}, and which correspond to $(z_1, m-n,y), (z_2,m'-n',y)\in \calG (Y,\rho) y$ respectively. We have $m-n= m'-n'$ and we may assume that $m'\ge m$, so that setting $m'=m+k$ we have $k\ge 0$ and $n'= n+k$. We emphasize that the conditions in Construction \ref{construction-GFS.concretedescription} include in particular that $f_m\ne g_n$ and that $f'_{m'}\ne g'_{n'}$.

Setting $\xi = \psi (y)$, we have
$$\theta_{f_1}^{-1}\circ \cdots \circ \theta_{f_m}^{-1}\circ \theta_{g_n} \circ \cdots \circ \theta_{g_1} (\xi)
=  \theta_{f_1'}^{-1}\circ \cdots \circ \theta_{f_{m'}'}^{-1}\circ \theta_{g_{n'}'} \circ \cdots \circ \theta_{g_1'} (\xi).$$
Hence $\xi \in \text{Dom} (\theta_{g_1}) \cap \text{Dom} (\theta_{g'_1}) = H_{g_1}\cap H_{g_1'}$ and so $g_1= g_1'$. Now we get
$$\theta_{f_1}^{-1}\circ \cdots \circ \theta_{f_m}^{-1}\circ \theta_{g_n} \circ \cdots \circ \theta_{g_2} (\xi_1)
 =  \theta_{f_1'}^{-1}\circ \cdots \circ \theta_{f_{m'}'}^{-1}\circ \theta_{g_{n'}'} \circ \cdots \circ \theta_{g_2'} (\xi_1),$$
where $\xi _1 = \theta_{g_1} (\xi) = \sigma (\xi)$. The same argument gives that $g_2= g_2'$. Proceeding in this way, we obtain that $g_i= g_i'$ for $1 \leq i \leq n$. Set $\xi_n = \theta_{g_n}\circ \cdots \circ \theta_{g_1}(\xi)$. Then we have
$$\xi' := \theta_{f_1}^{-1}\circ \cdots \circ \theta_{f_m}^{-1} (\xi_n) = \theta_{f_1'}^{-1}\circ \cdots \circ \theta _{f'_m}^{-1}\circ \theta_{f'_{m+1}}^{-1} \circ \cdots \circ \theta_{f_{m+k}'}^{-1}\circ \theta_{g_{n+k}'} \circ \cdots \circ \theta_{g_{n+1}'} (\xi_n).$$
Hence $\xi' \in H_{f_1} \cap H_{f_1'}$, and thus $f_1 = f_1'$. Applying the same argument to $\theta_{f_1}(\xi')$, we get that $f_2 = f_2'$, and continuing in this way, that $f_i = f_i'$ for $1 \leq i \leq m$. Therefore, we get
$$\xi_n = \theta_{f_m}\circ \cdots \circ \theta_{f_1} (\xi') = \theta_{f'_{m+1}}^{-1} \circ \cdots \circ \theta_{f_{m+k}'}^{-1}\circ \theta_{g_{n+k}'} \circ \cdots \circ \theta_{g_{n+1}'} (\xi_n).$$
If $k>0$, then we have
$$\theta_{f_{m+k}'}\circ \cdots \circ \theta_{f_{m+1}'} (\xi_n) =  \theta_{g_{n+k}'} \circ \cdots \circ \theta_{g_{n+1}'} (\xi_n),$$
and proceeding as above we get $g_{n+i}' = f_{n+i}'$ for $1 \leq i \leq k$. In particular we have $f'_{m'} = f_{m+k}' = g_{n+k}' = g'_{n'}$, which is a contradiction with the condition $f'_{m'} \ne g'_{n'}$. Therefore $k = 0$, and thus we have $m' = m$, $n' = n$, $f_i = f_i'$ for $1 \leq i \leq m$, and $g_j = g_j'$ for $1 \leq j \leq n$. Hence we obtain that $(z_1, m-n, y) = (z_2, m'-n',y)$, as desired.
\end{proof}

\begin{theorem}\label{theorem-extending.to.Steinberg}
Let $(Y,\rho) \in \emph{\textsf{LHomeo}}$ and let $(E,C)$ be a generalized finite shift graph. Suppose that we have an $(E,C)$-structure on $Y$, and let $\psi \colon Y \to X$ be the unique $(E,C)$-equivariant continuous map, where $(X,\sigma)$ is the generalized finite shift associated to $(E,C)$. Then the associated groupoid homomorphism
$$\psi_* \colon \calG(Y,\rho) \to \calG(X,\sigma)$$
given in Lemma \ref{lemma-morphismsOK} induces a $K$-algebra $*$-homomorphism
$$\psi^*\colon A_K(\calG(X,\sigma))\to A_K(\calG(Y,\rho))$$
by the rule $\psi^*(f)= f\circ \psi_*$ for $f\in A_K(\calG(X,\sigma))$.
\end{theorem}
\begin{proof}
We first show that $\psi_*$ is proper. For this it is enough to shown that $\psi_*^{-1}(Z(U,m,n,V))$ is compact, where $U$ and $V$ are open compact subsets of $X$ such that $\sigma^m|_{U}$ and $\sigma^n|_{V}$ are injective, and $\sigma^m(U)= \sigma^n(V)$.

For such open compact subset $Z(U,m,n,V)$ of $\calG(X,\sigma)$, we claim that the set $Z(\psi^{-1}(U),m,n,\psi^{-1}(V))$ is a bisection homeomorphic to $\psi^{-1}(V)$, and thus a compact set. For this, we show that the source map $s$ induces a bijection
$$Z(\psi^{-1}(U),m,n,\psi^{-1}(V)) \to \psi^{-1}(V).$$
Let $y \in \psi^{-1}(V)$. Then $\psi(y) \in V$ and since $\sigma^m(U)= \sigma^n(V)$, there is $x \in U$ such that $\sigma^m(x)= \sigma^n(\psi(y))$. Hence $(x,m-n,\psi (y))\in \calG(X,\sigma)\psi (y)$. Now by the proof of surjectivity in Lemma \ref{lemma-preimagesOK}, there is $z\in Y$ such that $\psi_*(z,m-n,y) = (x,m-n,\psi (y))$. But then $(z,m-n,y)\in Z(\psi^{-1}(U),m,n,\psi^{-1}(V))$ and $s(z,m-n,y) = y$, showing surjectivity.

To show injectivity, suppose that $(z_1,m-n,y), (z_2,m-n,y)$ are two elements in $Z(\psi^{-1}(U),m,n,\psi^{-1}(V))$. Then, $(\psi (z_1),m-n,\psi (y)), (\psi (z_2),m-n,\psi (y))\in Z(U,m,n,V)$, and thus $\sigma^m (\psi (z_1)) = \sigma^n (\psi (y)) = \sigma^m(\psi (z_2))$. But $\sigma^m|_{U}$ is injective, so $\psi (z_1)= \psi (z_2)$. By Lemma \ref{lemma-preimagesOK}, we conclude that $z_1=z_2$.

We have completed the proof that $\psi_*$ is proper. Hence, since it is also continuous by Lemma \ref{lemma-morphismsOK}, we obtain that $\psi^*$ is a well-defined linear map. We have to show that $\psi^*$ preserves the convolution product. For $f,g \in A_K(\calG(X,\sigma))$ and an element $\gamma = (z,m-n,y) \in \calG(Y,\rho)$, we have
\begin{align*}
(\psi^*(f) * \psi^*(g))(\gamma) & = \sum_{\alpha \in \calG(Y,\rho)y} \psi^*(f)(\gamma \alpha^{-1}) \psi^*(g) (\alpha) \\
& = \sum_{\alpha \in \calG(Y,\rho)y} f(\psi_*(\gamma) \psi_*(\alpha)^{-1}) g (\psi_*(\alpha)) \\
& = \sum_{\beta \in \calG(X,\sigma)\psi(y)} f(\psi_*(\gamma) \beta^{-1}) g(\beta) = \psi^*(f*g)(\gamma),
\end{align*}
where we have used Lemma \ref{lemma-preimagesOK} in the third equality. We thus have $\psi^*(f) * \psi^*(g) = \psi^* (f * g)$. It is straightforward to check that $\psi^*$ is a $*$-map.
\end{proof}

We also have a similar result for groupoid $C^*$-algebras.

\begin{theorem}\label{theorem-extending.to.Cstar}
Let $(Y,\rho) \in \emph{\textsf{LHomeo}}$ and $(X,\sigma)$ and $\psi$ be as in Theorem \ref{theorem-extending.to.Steinberg}. Then the groupoid homomorphism $\psi_* \colon \calG(Y,\rho) \to \calG(X,\sigma)$ induces a $*$-homomorphism
$$\psi^* \colon C^*(\calG(X,\sigma))\to C^*(\calG(Y,\rho))$$
such that $\psi^*(f) = f \circ \psi_*$ for $f \in C_c(\calG(X,\sigma))$.
\end{theorem}
\begin{proof}
As in the proof of Theorem \ref{theorem-extending.to.Steinberg}, we see that $\psi^*$ gives a $*$-homomorphism $C_c(\calG(X,\sigma)) \to C_c(\calG(Y,\rho))$. Since $C^*(\calG)$ is the enveloping $C^*$-algebra of $C_c(\calG)$ for any second-countable locally compact Hausdorff \'etale groupoid $\calG$ (see \cite[Theorem 9.2.3 and Lemma 9.2.4]{SSW2020}), we conclude the result. Alternatively, we may also use that $C^*(\calG)$ is the enveloping $C^*$-algebra of $A_{\C}(\calG)$, and use Theorem \ref{theorem-extending.to.Steinberg}.
\end{proof}

\subsection{Algebras as colimits of generalized finite shift algebras}\label{subsection-functors.continuous}

We now consider the situation treated at the end of Section \ref{section-general.shifts.finite.type}, so that $(X,\sigma) \in \textsf{LHomeo}$, $(F,D)$ is a fixed $l$-diagram for $(X,\sigma)$, and we set $(E(n),C(n)) = (F_{2n},D^{2n})$ for all $n \geq 0$. Recall that $(E(n),C(n))$ is a generalized finite shift graph, with associated generalized finite shift $(X_n,\sigma_n)$. We denote by $(F(n),D(n))$ the canonical resolution of $(E(n),C(n))$, which is an $l$-diagram for $(X_n,\sigma_n)$. By Theorem \ref{theorem-inverse.limit}, we have $(X,\sigma) = \varprojlim (X_n,\sigma_n)$, and the inverse limit maps $\psi_n \colon X \to X_n$ and the connecting maps $\psi_{n,n+1} \colon X_{n+1} \to X_n$ are universal maps induced by the corresponding $(E(n),C(n))$-structures. It therefore follows from Theorems \ref{theorem-extending.to.Steinberg} and \ref{theorem-extending.to.Cstar} that we have directed systems
$$(A_K(\calG(X_n,\sigma_n)),\psi_{n,n+1}^*) \quad \text{ and } \quad (C^*(\calG(X_n,\sigma_n)),\psi_{n,n+1}^*)$$
together with compatible $*$-homomorphisms
$$ A_K(\calG(X_n,\sigma_n)) \overset{\psi_n^*}{\longrightarrow} A_K(\calG(X,\sigma)) \quad \text{ and } \quad C^*(\calG(X_n,\sigma_n)) \overset{\psi_n^*}{\longrightarrow} C^*(\calG(X,\sigma)).$$
We show in this section that indeed $(A_K(\calG (X,\sigma)),\psi_n^*)$ and $(C^*(\calG (X,\sigma)),\psi_n^*)$ are the colimits of the above directed systems in the categories of $*$-algebras and $C^*$-algebras, respectively.\\

We introduce some useful notation. Recall the definitions of the $*$-algebra $\calA_{\infty}$ and the $C^*$-algebra $\calB_{\infty}$ associated to the $l$-diagram $(F,D)$ from Subsection \ref{subsection-general.approx}.

Given a generalized finite shift graph $(E,C)$, we identify $LV_K(E,C)$ with the corner $VL_K(E,C)V$, where $V=\sum_{v\in E^{0,0}} v \in L_K(E,C)$, see Proposition \ref{proposition-presentation.corner}. Thinking of $V$ as a projection in $C^*(E,C)$, we set $C^*V(E,C):= VC^*(E,C)V$.

\begin{notation}\label{notation-LVab}
Let $(E,C)$ be a generalized finite shift graph. We denote by $\LVab_K(E,C)$ the universal tame $*$-algebra associated to the canonical generating set of partial isometries of $LV_K(E,C)$. Similarly $\calO V(E,C)$ will denote the universal tame $C^*$-algebra associated to $C^*V(E,C)$. By using \cite[Lemma 4.3]{Ara2022}, we will identify $\LVab_K(E,C)$ with the algebra $\calA_{\infty}$ of the canonical resolution of $(E,C)$. Moreover, the arguments in \cite[Lemma 4.3]{Ara2022} apply {\it verbatim} to the $C^*$-algebra counterparts, therefore we will also identify $\calO V(E,C)$ with the $\calB_{\infty}$ $C^*$-algebra associated to the canonical resolution of $(E,C)$.
\end{notation}

Applying Theorem \ref{theorem-main.iso.theorem} to $(F(n),D(n))$ we obtain a canonical $*$-isomorphism
$$ \varphi^{(n)}\colon \LVab_K(E(n),C(n)) = \calA_{\infty}^{(n)} \longrightarrow A_K(\calG(X_n,\sigma_n)),$$
where $\LVab_K(E(n),C(n)) = \calA_{\infty}^{(n)} = \varinjlim_i p_{2i}^{(n)} L_K(F(n)_{2i},D(n)^{2i})p_{2i}^{(n)}$ is the $\calA_{\infty}$ algebra corresponding to $(F(n),D(n))$ (see Notation \ref{notation-LVab}). In particular, since the first layer of $(F(n),D(n))$ is $(E(n),C(n))$, we have the $*$-homomorphism
$$\varphi^{(n)}_0\colon LV_K(E(n),C(n)) \to A_K(\calG(X_n,\sigma_n))$$
which satisfies 
\begin{equation}\label{equation-main.varphi.property}
\varphi^{(n)}_0 = \varphi^{(n)} \circ \pi_n,
\end{equation}
where $\pi _n \colon LV_K(E(n),C(n)) \to \LVab_K(E(n),C(n))$ is the canonical projection map.

In a completely analogous fashion, using Theorem \ref{theorem-isomorphism.for.Cstars}, we get $*$-isomorphisms
$$\ol{\varphi}^{(n)} \colon \calO V(E(n),C(n)) = \calB_{\infty}^{(n)}\longrightarrow C^*(\calG (X_n,\sigma_n))$$
satisfying corresponding properties.

We can now obtain the following result, which links the two developments in Subsections \ref{subsection-general.approx} and \ref{subsection-functoriality}.

\begin{theorem}\label{theorem-combined.FINAL.RESULT}
Let $(X,\sigma)\in \emph{\textsf{LHomeo}}$ and let $(F,D)$ be an $l$-diagram for $(X,\sigma)$. With the above notation, for each $n \ge 0$ there exists a unique $*$-homomorphism $\phi_n^{\mathrm{ab}} \colon \LVab_K(E(n),C(n)) \to \LVab_K(E(n+1),C(n+1))$ such that the following diagram
\begin{equation}\label{equation-commutative.diagram.FINAL}
\vcenter{
	\xymatrix{
	LV_K(E(n),C(n)) \ar@{->}[r]^{\pi_n} \ar@{->}[d]^{\phi_n} & \LVab_K(E(n),C(n)) \ar@{->}[d]^{\phi_n^{\rm ab}} \ar@{->}[r]^{\kern5pt{\varphi^{(n)}}}_{\kern5pt{\cong}} & A_K(\calG (X_n,\sigma_n))  \ar@{->}[d]^{\psi_{n,n+1}^*} \\
	LV_K(E(n+1),C(n+1)) \ar@{->}[r]^{\pi_{n+1}} \ar@{->}[d]^{\phi_{n+1,\infty}} & \LVab_K(E(n+1),C(n+1)) \ar@{->}[r]^{\kern10pt{\varphi^{(n+1)}}}_{\kern10pt{\cong}} \ar@{->}[d]^{\phi_{n+1,\infty}^{{\rm ab}}} & A_K(\calG (X_{n+1},\sigma_{n+1})) \ar@{->}[d]^{\psi_{n+1}^*} \\
	\calA_{\infty} \ar@{->}[r]^{{\rm id}} & \calA_{\infty} \ar@{->}[r]^{\kern-15pt{\varphi}}_{\kern-15pt{\cong}} & A_K(\calG (X,\sigma))  
	}
}
\end{equation}
is commutative, where $\pi_n \colon LV_K(E(n),C(n))\to \LVab_K(E(n),C(n))$ are the canonical quotient maps. Hence we get $*$-isomorphisms
\begin{equation}\label{equation-colimits.FINAL}
A_K( \calG (X,\sigma)) \cong \calA_{\infty} \cong \varinjlim \LVab_K(E(n),C(n))  \quad \text{ and } \quad A_K(\calG (X,\sigma))\cong \varinjlim A_K(\calG (X_n,\sigma_n)).
\end{equation}
Similarly, we have $*$-isomorphisms
\begin{equation}\label{equation-colimits.Cstar.FINAL}
C^*(\calG (X,\sigma)) \cong \calB_{\infty} \cong \varinjlim \OV (E(n),C(n))  \quad \text{ and } \quad C^*(\calG (X,\sigma))\cong \varinjlim C^*(\calG (X_n,\sigma_n)).
\end{equation}
\end{theorem}
\begin{proof}
Since $\varphi^{(n)}$ and $\varphi^{(n+1)}$ are isomorphisms, there is a unique $*$-homomorphism
$$\phi_n^{{\rm ab}} \colon \LVab_K(E(n),C(n)) \to \LVab_K(E(n+1),C(n+1))$$
making the upper right square in \eqref{equation-commutative.diagram.FINAL} commutative.

By \eqref{equation-main.varphi.property}, in order to show that the upper left square in \eqref{equation-commutative.diagram.FINAL} is commutative, it suffices to show the identity $\varphi_0^{(n+1)} \circ \phi_n = \psi_{n,n+1}^* \circ \varphi_0^{(n)}$. To do so, we have to fix some notation.

For $u\in E(n)^0$, we will denote by $Z^{(n)}(u)$ the corresponding cylinder set in $X_n$. Let $(G,D)$ be the $l$-diagram obtained by placing the two layers $(F_{2n},D^{2n})=(E(n),C(n))$ and $(F_{2n+1},D^{2n+1})$ on top of $(F(n+1),D(n+1))$, and denote by $\calP_i^G$, $i\ge 0$, the corresponding partitions of $X_{n+1}$. Note that the first three layers of $(G,D)$ are identical to $(F_{2n},D^{2n})\cup (F_{2n+1},D^{2n+1})\cup (F_{2n+2},D^{2n+2})$, so in particular we have a natural bijection
$$\gamma \colon E(n)^{0,0}\cup E(n)^{0,1} \cup E(n+1)^{0,0}\cup E(n+1)^{0,1} \longrightarrow \calP^G_0\cup \calP^G_1\cup \calP^G_2\cup \calP^G_3.$$
Note that $\gamma(u) = Z^{(n+1)}(u)$ for $u\in E(n+1)^0$. Using the definition of $\psi_{n,n+1}$, we get, for $u\in E(n)^{0,0}$, that $\psi_{n,n+1}^{-1}(Z^{(n)}(u)) = \gamma(u)$, so setting $U :=\gamma (u)$, we have
$$\psi_{n,n+1}^{-1}(Z^{(n)}(u)) = U = \bigsqcup_{\{ U'\in \calP^G_{2} \mid U'\subseteq U\}} U'.$$
Similarly, for $w\in E(n)^{0,1}$, we have $\psi_{n+1,n}^{-1}(Z^{(n)}(w)) = \bigsqcup_{W'} W'$ where the union is extended to all $W'\in \calP^G_{3}$ such that $W'\subseteq W$, where $W :=\gamma(w) \in \calP^G_1$. It follows that, for $u\in E(n)^{0,0}$, and with $U :=\gamma(u)$, we have
\begin{align*}
\psi_{n,n+1}^*(\varphi_0^{(n)}(u)) & = \psi_{n,n+1}^*(\mathds{1}_{Z^{(n)}(u)}) \\
& = \mathds{1}_{Z^{(n)}(u)}\circ \psi_{n,n+1} \\
& = \mathds{1}_{\psi_{n,n+1}^{-1}(Z^{(n)}(u))} \\
&	= \sum_{\{U'\in \calP^G_2 \mid U'\subseteq U\}} \mathds{1}_{U'} \\
& = \varphi_0^{(n+1)}(\mathbf{S}(u)) = \varphi_0^{(n+1)}(\phi_n(u)).
\end{align*}
Similarly $\psi_{n,n+1}^*(\varphi_0^{(n)}(w)) = \varphi_0^{(n+1)}(\phi_n(w))$ for all $w \in E(n)^{0,1}$.

Let now $f \in R_u$, $e\in B_v$, where $e,f\in E(n)^1$, $u,v\in E(n)^{0,0}$, and $s(e) = s(f) =: w\in E(n)^{0,1}$. Using Lemma \ref{lemma-technical.1} and the definition of $\phi_n$, we have
$$\phi_n(fe^*) = \sum_{(f',e')\in A} f'(e')^*,$$
where $A$ is the set of all pairs $(f',e') \in E(n+1)^1 \times E(n+1)^1$ such that $f' \in R_{u'}$, $e'\in B_{v'}$, $u' = r(f') \le \mathbf{S}(u)$, $v'= r(e')\le \mathbf{S}(v)$ and $w' := s(f')= s(e') \le \mathbf{S}(w)$. Each pair $(f',e')\in A$ corresponds exactly to a triple $(U',W',V')$ such that $U',V'\in \calP^G_2$ and $W'\in \calP^G_3$, with $W'\subseteq \sigma_{n+1}(U')$ and $W' \subseteq V'$, with $U' = \gamma(u') \subseteq U = \gamma(u)$, $V' = \gamma(v') \subseteq V := \gamma(v)$ and $W' = \gamma(w') \subseteq W = \gamma(w)$. Moreover, observe that
$$\varphi^{(n+1)}_0 (f'(e')^*) = \mathds{1}_{Z((\sigma|_{U'})^{-1}(W'),1,0,W')} \in A_K(\calG(X_{n+1},\sigma_{n+1})).$$
We denote by $B$ the family of all such triples $(U',W',V')$. Using this, we compute
\begin{align*}
\psi_{n,n+1}^*(\varphi_0^{(n)}(fe^*)) & = \psi_{n,n+1}^*(\mathds{1}_{Z((\sigma_n|_{Z^{(n)}(u)})^{-1}(Z^{(n)}(w)),1,0,Z^{(n)}(w))}) \\
& = \mathds{1}_{Z((\sigma_n|_{Z^{(n)}(u)})^{-1}(Z^{(n)}(w)),1,0,Z^{(n)}(w))} \circ (\psi_{n,n+1})_* \\
& = \mathds{1}_{({\psi_{n+1,n}}_*)^{-1} Z((\sigma_n|_{Z^{(n)}(u)})^{-1}(Z^{(n)}(w)),1,0,Z^{(n)}(w))} \\
& = \mathds{1}_{( Z((\sigma_{n+1}|_{U})^{-1}(W)),1,0,W)} \\
& = \sum_{(U',W',V')\in B} \mathds{1}_{Z((\sigma_{n+1}|_{U'})^{-1}(W'),1,0,W')} \\
& = \sum_{(f',e')\in A} \varphi_0^{(n+1)}(f'(e')^*) = \varphi_0^{(n+1)}(\phi_n(fe^*)).
\end{align*}
Since the family of partial isometries $fe^*$, with $f\in R_u$, $e\in B_v$, $u,v\in E(n)^{0,0}$, generates $LV_K(E(n),C(n))$ as a $*$-algebra, we get that $\psi_{n,n+1}^* \circ \varphi_0^{(n)} = \varphi_0^{(n+1)}\circ \phi_n$, and hence $\phi_n^{\text{ab}}\circ \pi_n = \pi_{n+1}\circ \phi_n$, so that the upper left square of \eqref{equation-commutative.diagram.FINAL} is commutative.

Using that $\varphi$ is an isomorphism, we see that there is a unique $*$-homomorphism
$$\phi_{n+1,\infty}^{\text{ab}} \colon \LVab_K(E(n+1),C(n+1)) \longrightarrow \calA_{\infty}$$
making the lower right square of \eqref{equation-commutative.diagram.FINAL} commutative. Now a computation similar to the above gives that
$$\psi_{n+1}^*\circ \varphi_0^{(n+1)} = \varphi_{n+1} = \varphi \circ \phi_{n+1,\infty},$$
and so the lower left square of \eqref{equation-commutative.diagram.FINAL} is also commutative.

Now recall that, by its definition, $(\calA_{\infty},\phi_{n,\infty})$ is the colimit of the directed system
$$(LV_K(E(n),C(n)),\phi_n).$$
Since all the maps $\pi_n$ are surjective, the commutativity of \eqref{equation-commutative.diagram.FINAL} gives first that the maps $\phi_{n,\infty}^{{\rm ab}}$ are compatible with the transition maps $\phi_n^{{\rm ab}}$, that is, $\phi_{n,\infty}^{{\rm ab}} = \phi_{n+1,\infty}^{{\rm ab}}\circ \phi_n^{{\rm ab}}$, and second that $(\calA_{\infty},\phi_{n,\infty}^{{\rm ab}})$ is the colimit of the directed system
$$(\LVab_K(E(n),C(n)),\phi_n^{{\rm ab}}).$$
This gives the isomorphism $\calA_{\infty} \cong \varinjlim LV_K^{\text{ab}}(E(n),C(n))$. Using the commutativity of the right squares of \eqref{equation-commutative.diagram.FINAL} and the fact that $\varphi$ and all $\varphi^{(n)}$ are isomorphisms, we obtain \eqref{equation-colimits.FINAL}.

The proof of \eqref{equation-colimits.Cstar.FINAL} uses a diagram of $C^*$-algebras obtained by taking $C^*$-completions in diagram \eqref{equation-commutative.diagram.FINAL}. We leave the routine details to the reader. This concludes the proof.
\end{proof}

\section*{Acknowledgments}

The authors would like to thank the anonymous referee for her/his many suggestions, which have improved the exposition of the paper.

\providecommand{\bysame}{\leavevmode\hbox to3em{\hrulefill}\thinspace}
\providecommand{\MR}{\relax\ifhmode\unskip\space\fi MR }
\providecommand{\MRhref}[2]{%
  \href{http://www.ams.org/mathscinet-getitem?mr=#1}{#2}
}
\providecommand{\href}[2]{#2}

\end{document}